\input colordvi

\documentclass[12pt]{sat}

\newif\ifsattoc\sattoctrue
\newread\testfl\immediate\openin\testfl=\jobname.toc
    \ifeof\testfl\sattocfalse\fi\closein\testfl

\title{\Huge Summability of Multi-Dimensional \\ Trigonometric Fourier Series}
\def\shorttitle{\ifx\sectiontitle\empty\shortauthor\else\thesection. \sectiontitle\fi}

\author{Ferenc Weisz}

\def\shortauthor{Ferenc Weisz: Summation of multi-dimensional Fourier series}

\def\versiondate{\today}

\def\abstracttext{We consider the summability of one- and multi-dimensional
trigonometric Fourier series. The Fej{\'e}r and Riesz summability
methods are investigated in detail. Different types of summation and
convergence are considered. We will prove that the maximal operator
of the summability means is bounded from the Hardy space $H_p$ to
$L_p$, for all $p>p_0$, where $p_0$ depends on the summability
method and the dimension. For $p=1$, we obtain a weak type
inequality by interpolation, which ensures the almost everywhere
convergence of the summability means. Similar results are formulated
for the more general $\theta$-summability and for Fourier
transforms.}

\def\MSCnumbers{Primary 42B08, 42A38, Secondary 42B30.} 

\def\keywords{Hardy spaces, $H_p$-atom, interpolation, Fourier series, circular, triangular, cubic and rectangular summability.} 

\usepackage{latexsym}
\usepackage{amssymb}
\usepackage{graphicx}
\usepackage{epstopdf}
\usepackage{subfigure}
\usepackage{multind}

\newcommand{\file}{\jobname}
\makeindex{\file} \makeindex{\file-1}

\newcommand{\idword}[1]{{\dword{#1}}{\index{\file}{#1}}}

\newcommand{\ieword}[1]{{\eword{#1}}{\index{\file}{#1}}}

\newcommand{\ind}[1]{{#1}{\index{\file}{#1}}}
\newcommand{\inda}[1]{{#1}{\index{\file-1}{#1}}}

\def\cF{\mathcal{F}}

\def\cI{\mathcal{I}}

\def\cL{\mathcal{L}}
\def\cM{\mathcal{M}}
\def\cP{\mathcal{P}}

\def\cR{\mathcal{R}}
\def\cS{\mathcal{S}}

\def\bN{{\bf N}}

\def\bR{{\bf R}}
\def\bS{{\bf S}}

\def\bX{{\bf X}}
\def\bY{{\bf Y}}

\def\C{{\mathbb C}}

\def\N{{\mathbb N}}

\def\R{{\mathbb R}}
\def\T{{\mathbb T}}
\def\X{{\mathbb X}}
\def\Y{{\mathbb Y}}
\def\Z{{\mathbb Z}}

\def\nn{{n \in \N}}

\def\n{\nonumber}

\def\soc{{\rm soc \, }}

\newtheorem{thm}{Theorem}[section]
\newtheorem{cor}[thm]{Corollary}
\newtheorem{lem}[thm]{Lemma}
\newtheorem{exa}[thm]{Example}
\newtheorem{rem}[thm]{Remark}
\newtheorem{dfn}[thm]{Definition}

\def\sectiontitle{}
\newcommand{\sect}[1]{\gdef\sectiontitle{#1}\section{#1}\setcounter{equation}{0}}

\newenvironment{proof}{\begin{trivlist} \item[] \textbf{Proof.}}{\quad \rule{2mm}{2mm} \end{trivlist}}
\newenvironment{proof*}[1]{\begin{trivlist} \item[] \textbf{Proof of #1.} }
{\quad \rule{2mm}{2mm} \end{trivlist}}

\def\startpagenumber{1}
\def\volumenumber{7} 
\def\year{2012}

\def\dword#1{{\bf #1}} \def\eword#1{{\it #1}}
\def\dd{\,{\rm d}}  
\def\ee{{\rm e}}  
\def\ii{{\rm i}}  

\newcommand{\beginddoc}{
\maketitle
\begin{abstract}
\abstracttext \vskip1pt MSC: \MSCnumbers
\ifx\keywords\empty\else\vskip1pt Keywords: \keywords\fi
\end{abstract}
\insert\footins{\scriptsize
\medskip
\baselineskip 8pt \leftline{Surveys in Approximation Theory}
\leftline{Volume \volumenumber, \year.
pp.~\thepage--\pageref{endpage}.} \leftline{\copyright\ \year\
Surveys in Approximation Theory.} \leftline{ISSN 1555-578X}
\leftline{All rights of reproduction in any form reserved.}
\smallskip
\par\allowbreak}
\bigskip\bigskip
\begin{figure}[ht]
\centering \subfigure{
    \includegraphics[scale=0.3]{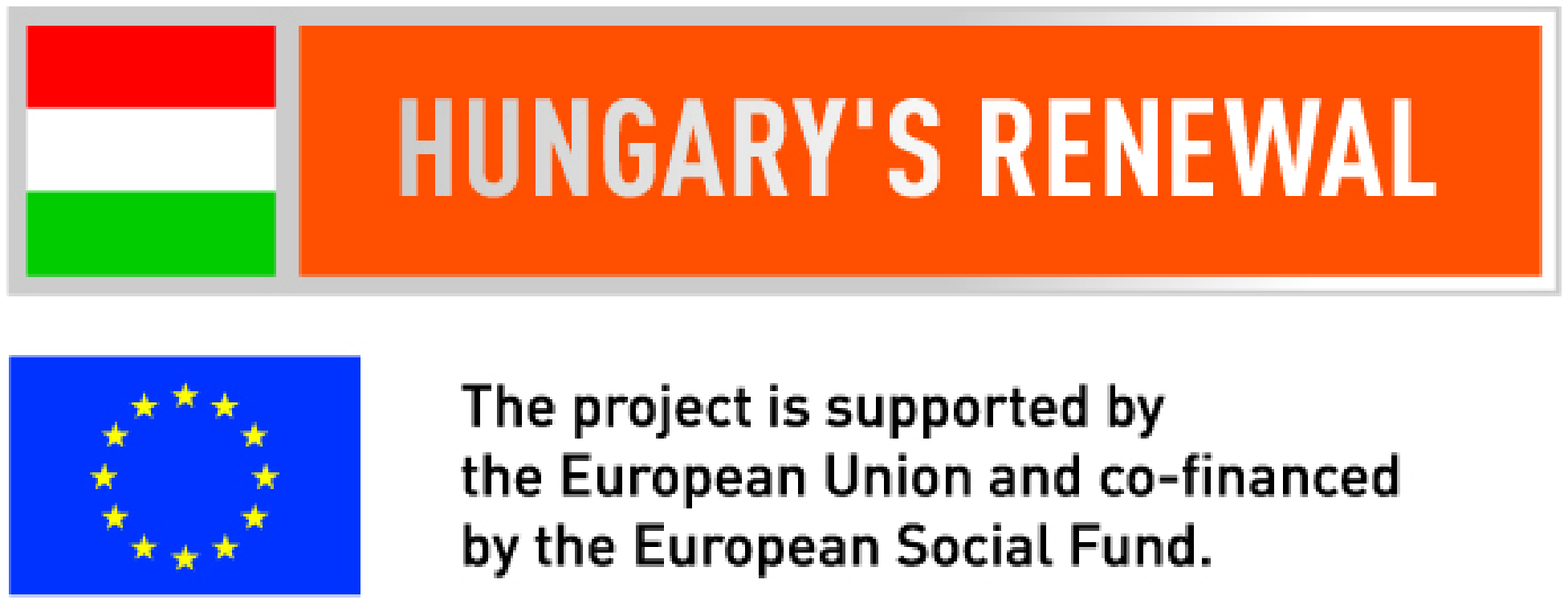}
} \subfigure{
    \includegraphics[scale=0.25]{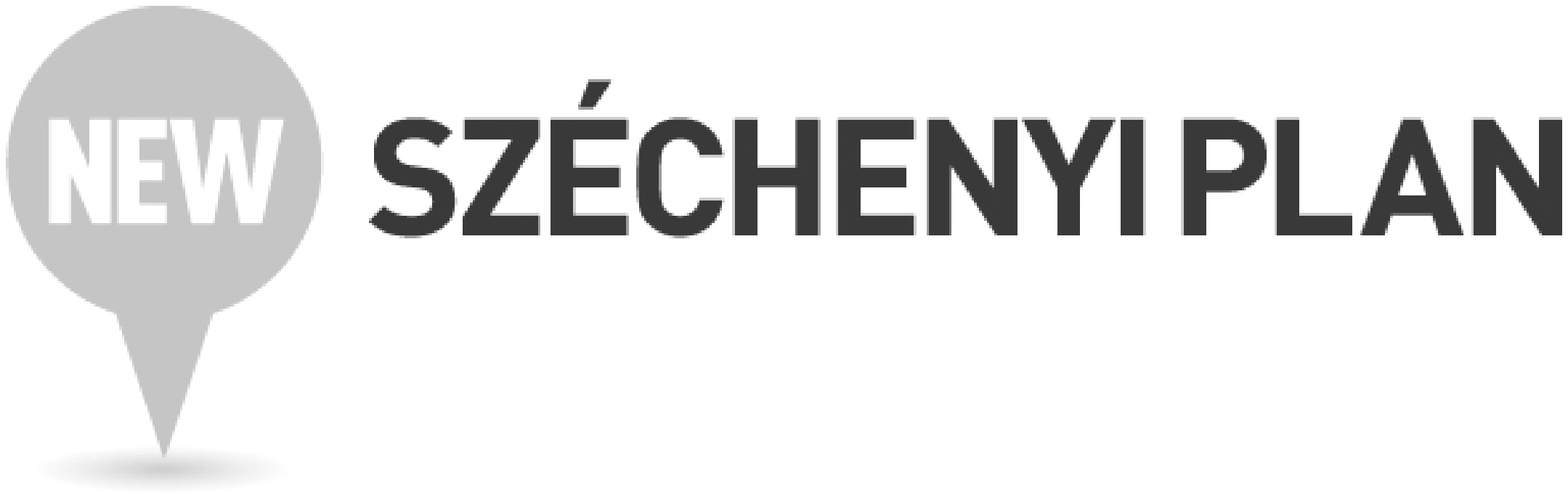}
}
\end{figure}
This Project is supported by the European Union and co-financed by
the European Social Fund (grant agreement no. TAMOP
4.2.1/B-09/1/KMR-2010-0003).
\thispagestyle{empty}
\newpage
\ifsattoc\else\tableofcontents\fi}
\renewcommand\rightmark{\ifodd\thepage{\it \hfill\shorttitle\hfill}\else {\it \hfill\shortauthor\hfill}\fi}
\def\endddoc{\label{endpage}\end{document}}
\date{{\small \versiondate}}
 \RequirePackage{hyperref}
\begin{document}
\beginddoc
\leftline{\normalfont\large\bfseries Contents} \ifsattoc
\bigskip
\def\toczer{0}\def\tochalf{.5}\def\tocone{1}
\def\tocindent{0}
\def\ection{section}\def\ubsection{subsection}
\def\numberline#1{\hskip\tocindent truecm{} #1\hskip1em}
\newread\testfl
\def\inputifthere#1{\immediate\openin\testfl=#1
    \ifeof\testfl\message{(#1 does not yet exist)}
    \else\input#1\fi\closein\testfl}
\countdef\counter=255
\def\diamondleaders{\global\advance\counter by 1
  \ifodd\counter \kern-10pt \fi
  \leaders\hbox to 15pt{\ifodd\counter \kern13pt \else\kern3pt \fi
  \hss.\hss}\hfill}
\newdimen\lextent
\newtoks\writestuff
\medskip
\begingroup
\small\baselineskip12.51pt 
\def\contentsline#1#2#3#4{
\def\argu{#1}
\ifx\argu\ection\let\tocindent\toczer\else
\ifx\argu\ubsection\let\tocindent\tochalf\else\let\tocindent\tocone\fi\fi
\setbox1=\hbox{#2}\ifnum\wd1>\lextent\lextent\wd1\fi}
\lextent0pt\inputifthere{\jobname.toc}\advance\lextent by 2em\relax
\def\contentsline#1#2#3#4{
\def\argu{#1}
\ifx\argu\ection\let\tocindent\toczer\else
\ifx\argu\ubsection\let\tocindent\tochalf\else\let\tocindent\tocone\fi\fi
\writestuff={#2}
\centerline{\hbox to \lextent{\rm\the\writestuff%
\ifx\empty#3\else\diamondleaders{} \hfil\hbox to 2 em\fi{\hss#3}}}}
\inputifthere{\jobname.toc}\endgroup
\immediate\openout\testfl=\jobname.toc 
\immediate\closeout\testfl             
\renewcommand{\contentsname}{}         
\tableofcontents\newpage               
\fi




\listoffigures

\addcontentsline{toc}{section}{List of Figures}

\newpage

\sect{Introduction}

We will consider different summation methods for multi-dimensional
trigonometric Fourier series. Basically two types of summations will
be introduced. In the first one we take the sum in the partial sums
and in the summability means over the balls of $\ell_q$, it is
called $\ell_q$-summability. In the literature the cases
$q=1,2,\infty$, i.e., the triangular, circular and cubic summability
are investigated. In the second version of summation we take the sum
over rectangles, it is called rectangular summability. In this case
two types of convergence and maximal operators are considered, the
restricted (convergence over the diagonal or more generally over a
cone), and the unrestricted (convergence over $\N^d$). In each
version, three well known summability methods, the Fej{\'e}r, Riesz
and Bochner-Riesz means will be investigated in detail. The
Fej{\'e}r summation is a special case of the Riesz method. We
consider norm convergence and almost everywhere convergence of the
summability means.

We introduce different types of Hardy spaces $H_p$ and prove that
the maximal operators of the summability means are bounded from
$H_p$ to $L_p$, whenever $p>p_0$ for some $p_0<1$. The critical
index $p_0$ depends on the summability method and the dimension. For
$p=1$, we obtain a weak type inequality by interpolation, which
implies the almost everywhere convergence of the summability means.
The one-dimensional version of the almost everywhere convergence and
the weak type inequality are proved usually with the help of a
Calderon-Zygmund type decomposition lemma. However, in higher
dimensions, this lemma can not be used for all cases investigated in
this monograph. Our method, that can also be applied well in higher
dimensions, can be regarded as a new method to prove the almost
everywhere convergence and weak type inequalities.

Similar results are also formulated for summability of Fourier
transforms. The so called $\theta$-summability, which is a general
summability method generated by a single function $\theta$, and the
Ces{\`a}ro summability are also considered.

We will prove all results except the ones that can be found in the
books Grafakos \cite{gra} and Weisz \cite{wk2}. For example, the
results about the circular Riesz summability below the critical
index and the results about Hardy spaces and interpolation can be
found in these books, so their proofs are omitted.

I would like to thank the editors and the referees for their efforts
they undertook to read the manuscript carefully and give useful
comments and suggestions. I also thank Levente L{\'o}csi for his
helpful assistance in creating the figures.

\sect{Partial sums of one-dimensional Fourier series}\label{s2}

In this and the next section we briefly present some theorems for
one-dimensional Fourier series. Later we will give their
generalization to higher dimensions in more details.

The set of the real numbers is denoted by \inda{$\R$}, the set of
the integers by \inda{$\Z$} and the set of the non-negative integers
by \inda{$\N$}. For a set ${\Y}\neq \emptyset$, let ${\Y}^d$ be its
Cartesian product ${\Y} \times \cdots\times {\Y}$ involving it $d$
times $(d\geq 1, d\in\N)$. We briefly write \inda{$L_p(\T^d)$}
instead of the \idword{$L_p(\T^d,\lambda)$ space} equipped with the
norm (or quasi-norm)
$$
\|f\|_{p}:=\left\{
             \begin{array}{ll}
               \Big(\int_{\T^d}|f|^p \dd \lambda\Big)^{1/p}, & \hbox{$0<p<\infty$;} \\
               \sup_{\T^d} |f|, & \hbox{$p=\infty$,}
             \end{array}
           \right.
$$
where $\T:=[-\pi,\pi]$ is the torus and $\lambda$ is the Lebesgue
measure.\index{\file-1}{$\T$} We use the notation $|I|$ for the
Lebesgue measure of the set $I$. The \idword{weak $L_p$ space},
\inda{$L_{p,\infty}(\T^d)$} $(0<p<\infty)$ consists of all
measurable functions $f$ for which
$$
\|f\|_{{p, \infty}} := \sup_{\rho >0} \rho \lambda(|f|>\rho)^{1/p}
<\infty.
$$
Note that $L_{p, \infty}(\T^d)$ is a quasi-normed space (see Bergh
and L{\"o}fstr{\"o}m \cite{belo}). It is easy to see that for each
$0<p<\infty$,
$$
L_p(\T^d) \subset L_{p, \infty}(\T^d) \qquad \mbox{and} \qquad
\|\cdot\|_{{p, \infty}} \leq \|\cdot\|_p.
$$
The space of continuous functions with the supremum norm is denoted
by \inda{$C(\T^d)$} and we will use \inda{$C_0(\R^d)$} for the space
of continuous functions vanishing at infinity.

For an integrable function $f\in L_1(\T)$, its $k$th \idword{Fourier
coefficient} is defined by
$$
\widehat f(k) = {1 \over 2\pi} \int_{\T} f(x) \ee^{-\ii kx} \dd x
\qquad (\ii:=\sqrt{-1}).
$$
The formal trigonometric series
$$
\sum_{k\in \Z} \widehat f(k) \ee^{\ii kx}\qquad (x\in \T)
$$
is called the \idword{Fourier series} of $f$. This definition can be
extended to distributions as well. Let \inda{$C^\infty(\T)$} denote
the set of all infinitely differentiable functions on $\T$. Then
$f\in C^\infty(\T)$ implies
$$
\sup_{\T} |f^{(k)}| < \infty \qquad \mbox{for all $k\in \N$}.
$$
We say that $f_n\to f$ in $C^\infty(\T)$ if
$$
\|f_n^{(k)}-f^{(k)}\|_\infty\to 0  \qquad \mbox{for all $k\in \N$}.
$$
$C^ \infty(\T)$ is also denoted by \inda{$\cS(\T)$}. A
\idword{distribution} $u:\cS(\T)\to\C$ (briefly $u \in \cS'(\T)$) is
a continuous linear functional on $\cS(\T)$, i.e., $u$ is linear and
$$
u(f_n)\to u(f) \quad \mbox{if} \quad f_n\to f \quad \mbox{in} \quad
C^\infty(\T).
$$
If $g\in L_p(\T)$ $(1\leq p\leq \infty)$, then
$$
u_g(f):=\int_{\T} fg \dd \lambda \qquad (f\in \cS(\T))
$$
is a distribution. So all functions from $L_p(\T)$ $(1\leq p\leq
\infty)$ can be identified with distributions $u \in \cS'(\T)$. We
say that the distributions $u_j$ tend to the distribution $u$
\idword{in the sense of distributions} or in \inda{$\cS'(\T)$} if
$$
u_j(f)\to u(f) \quad \mbox{for all} \quad f\in \cS(\T) \quad
\mbox{as} \quad j\to\infty.
$$

The next definition extends the Fourier coefficients to
distributions. For a distribution $f$, the $k$th \idword{Fourier
coefficient} is defined by $\widehat f(k):= f(e_{-k})$, where
$e_{k}(x):=\ee^{\ii kx}$ $(k\in\Z)$ (see e.g.~Edwards
\cite[p.~67]{ed2}).

For $f\in L_1(\T)$, the $n$th \idword{partial sum} \inda{$s_nf$} of
the Fourier series of $f$ is introduced by
\begin{equation}\label{e2}
s_{n} f(x) := \sum_{|k|\leq n} \widehat f(k) \ee^{\ii kx} =
\frac{1}{2\pi}\int_{\T} f(x-u) D_n(u) \dd u \qquad (\nn),
\end{equation}
where
$$
D_{n}(u) := \sum_{|k|\leq n} \ee^{\ii ku}
$$
is the $n$th \idword{Dirichlet kernel} \index{\file-1}{$D_n$}(see
Figure \ref{f13}).
\begin{figure}[htbp] 
   \centering
   \includegraphics[width=0.6\textwidth]{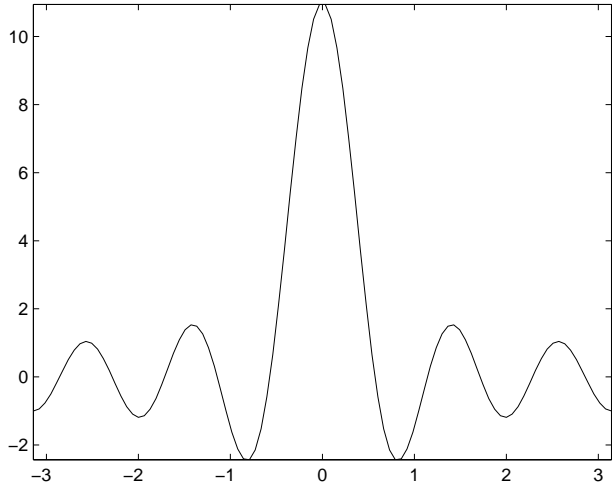}
   \caption{Dirichlet kernel $D_n$ for $n=5$.}
   \label{f13}
\end{figure}
Using some simple trigonometric identities, we obtain
\begin{eqnarray}\label{e43}
D_{n}(u) &=& 1+ 2\sum_{k=1}^n \cos(ku) \n\\
&=& \frac{1}{\sin(u/2)} \Big(\sin(u/2)+ 2\sum_{k=1}^n \cos(ku)\sin(u/2) \Big) \n\\
&=& \frac{1}{\sin(u/2)} \Big(\sin(u/2)+ \sum_{k=1}^n \Big(\sin((k+1/2)u) - \sin((k-1/2)u) \Big) \Big) \n\\
&=&\frac{\sin((n+1/2)u)}{\sin(u/2)}.
\end{eqnarray}
It is easy to see that $|D_{n}|\leq Cn$. The $L_1$-norms of $D_n$
are not uniformly bounded, more exactly $\|D_{n}\|_1 \sim \log n$.

It is a basic question as to whether the function $f$ can be
reconstructed from the partial sums of its Fourier series. It can be
found in most books about trigonometric Fourier series (e.g.~Zygmund
\cite{zy}, Bary \cite{ba}, Torchinsky \cite{to} or Grafakos
\cite{gra}) and is due to Riesz \cite{ri1}, that the partial sums
converge to $f$ in the $L_p$-norm if $1<p<\infty$.

\begin{thm}\label{t1}
If $f\in L_p(\T)$ for some $1<p< \infty$, then
$$
\|s_{n}f\|_p \leq C_p \|f\|_p \qquad (\nn)
$$
and
$$
\lim_{n\to\infty} s_{n}f=f \qquad \mbox{in the $L_p$-norm}.
$$
\end{thm}

The $L_1$-norms of $D_n$ are not uniformly bounded, Theorem \ref{t1}
is not true for $p=1$ and $p=\infty$.

Let us define the \idword{Riesz projection} with the formal series
$$
P^+f(x) \sim \sum_{k\in \N} \widehat {f}(k) \ee^{\ii kx}
$$
and let
$$
P_n^+f(x):=\sum_{k=0}^n \widehat {f}(k) \ee^{\ii kx} \qquad (\nn).
$$
Then Theorem \ref{t1} implies easily that $P^+:L_p(\T)\to L_p(\T)$
is bounded.\index{\file-1}{$P^+f$} \index{\file-1}{$P_n^+f$}

\begin{thm}\label{t53}
If $f\in L_p(\T)$ for some $1<p< \infty$, then
$$
\|P_n^+f\|_p \leq C_p \|f\|_p \qquad (\nn).
$$
Moreover,
$$
P^+f(x) = \sum_{k\in \N} \widehat {f}(k) \ee^{\ii kx} \qquad
\mbox{in the $L_p$-norm}
$$
and
$$
\|P^+f\|_p \leq C_p \|f\|_p.
$$
\end{thm}

\begin{proof}
Observe that
$$
\sum_{k=0}^{2n} \widehat {f}(k) \ee^{\ii kx} = \ee^{\ii nx}
\sum_{k=-n}^{n} \widehat {(f(\cdot)\ee^{-\ii n (\cdot)})}(k)
\ee^{\ii kx},
$$
in other words, $|P_{2n}^+f|=|s_n(f(\cdot)\ee^{-\ii n (\cdot)})|$.
Now the result follows from the Banach-Steinhaus theorem and from
Theorem \ref{t1}.
\end{proof}

One of the deepest results in harmonic analysis is \ind{Carleson's
theorem}, that the partial sums of the Fourier series converge
almost everywhere to $f\in L_p(\T)$ $(1<p\leq \infty)$ (see Carleson
\cite{ca} and Hunt \cite{hu} or recently Grafakos \cite{gra}).

\begin{thm}\label{t2}
If $f\in L_p(\T)$ for some $1<p< \infty$, then
$$
\|\sup_{n \in \N} |s_{n}f|\|_p  \leq C_p \|f\|_p
$$
and if $1<p\leq \infty$, then
$$
\lim_{n\to\infty} s_{n}f=f \qquad \mbox{a.e.}
$$
\end{thm}

The inequality of Theorem \ref{t2} does not hold if $p=1$ or
$p=\infty$, and the almost everywhere convergence does not hold if
$p=1$. du Bois Reymond proved the existence of a continuous function
$f\in C(\T)$ and a point $x_0\in \T$ such that the partial sums
$s_nf(x_0)$ diverge as $n\to \infty$. Kolmogorov gave an integrable
function $f\in L_1(\T)$, whose Fourier series diverges almost
everywhere or even everywhere (see Kolmogorov \cite{kol1,kol2},
Zygmund \cite{zy} or Grafakos \cite{gra}).

Since there are many function spaces contained in $L_1(\T)$ but
containing $L_p(\T)$ $(1<p\leq \infty)$, it is natural to ask
whether there is a "largest" subspace of $L_1(\T)$ for which almost
everywhere convergence holds. The next result, due to Antonov
\cite{{Antonov-1996}}, generalizes Theorem \ref{t2}.

\begin{thm}\label{t33}
If
\begin{equation}\label{e23}
\int_{\T} |f(x)| \log^+|f(x)| \log^+ \log^+ \log^+ |f(x)| \dd x
<\infty,
\end{equation}
then
$$
\lim_{n\to\infty} s_{n}f=f \qquad \mbox{a.e.}
$$
\end{thm}

Note that $\log^+u=\max(0,\log u)$.\index{\file-1}{$\log^+u$} It is
easy to see that if $f\in L_p(\T)$ $(1<p\leq \infty)$, then $f$
satisfies (\ref{e23}). If $f$ satisfies (\ref{e23}), then of course
$f\in L_1(\T)$. For the converse direction, Konyagin
\cite{{Konyagin-2000}} obtained the next result.

\begin{thm}\label{t34}
If the non-decreasing function $\phi:\R_+\to \R_+$ satisfies the
condition
$$
\phi(u)= o\Big(u \sqrt{\log u}/\sqrt{\log\log u}\Big) \qquad
\mbox{as $u\to\infty$},
$$
then there exists an integrable function $f$ such that
$$
\int_{\T} \phi(|f(x)|) \dd x<\infty
$$
and
$$
\limsup_{n\to\infty} s_{n}f(x)=\infty \qquad \mbox{for all $x\in
\T$},
$$
i.e., the Fourier series of $f$ diverges everywhere.
\end{thm}

For example, if $\phi(u)=u\, {\log^+\log^+ u}$, then there exists a
function $f$ such that its Fourier series diverges everywhere and
$$
\int_{\T} |f(x)| {\log^+\log^+ |f(x)|} \dd x<\infty.
$$

\sect{Summability of one-dimensional Fourier series}

Though Theorems \ref{t1} and \ref{t2} are not true for $p=1$ and
$p=\infty$, with the help of some summability methods they can be
generalized for these endpoint cases. Obviously, summability means
have better convergence properties than the original Fourier series.
Summability is intensively studied in the literature. We refer at
this time only to the books Stein and Weiss \cite{stwe}, Butzer and
Nessel \cite{bune}, Trigub and Belinsky \cite{trbe}, Grafakos
\cite{gra} and Weisz \cite{wk2} and the references therein.

The best known summability method is the Fej{\'e}r method. In 1904
Fej\'er \cite{fej} investigated the arithmetic means of the partial
sums, the so called Fej\'er means and proved that if the left and
right limits $f(x-0)$ and $f(x+0)$ exist at a point $x$, then the
Fej\'er means converge to $(f(x-0)+f(x+0))/2$. One year later
Lebesgue \cite{leb} extended this theorem and obtained that every
integrable function is Fej\'er summable at each Lebesgue point, thus
a.e. The Riesz means are generalizations of the Fej\'er means. M.
Riesz \cite{ri1} proved that the Riesz means of a function $f\in
L_1(\T)$ converge almost everywhere to $f$ as $n \to\infty$ (see
also Zygmund \cite[Vol. I, p.94]{zy}).

The \idword{Fej{\'e}r means} are defined by
$$
\sigma_n f(x) := {1 \over n} \sum_{j=0}^{n-1} s_{j}f(x).
$$
\index{\file-1}{$\sigma_nf$}It is easy to see that
$$
\sigma_n f(x) = \sum_{|k|\leq n} \Big(1-\frac{|k|}{n} \Big) \widehat
f(k)\ee^{\ii kx}=\frac{1}{2\pi}\int_{\T} f(x-u) K_n(u) \dd u,
$$
where the \idword{Fej{\'e}r kernels} are given by
$$
K_{n}(u) := \sum_{|k|\leq n} \Big(1-\frac{|k|}{n} \Big) \ee^{\ii k
u} = \frac{1}{n} \sum_{j=0}^{n-1} D_j(u)
$$
\index{\file-1}{$K_n$}(see Figure \ref{f14}).
\begin{figure}[htbp] 
   \centering
   \includegraphics[width=0.6\textwidth]{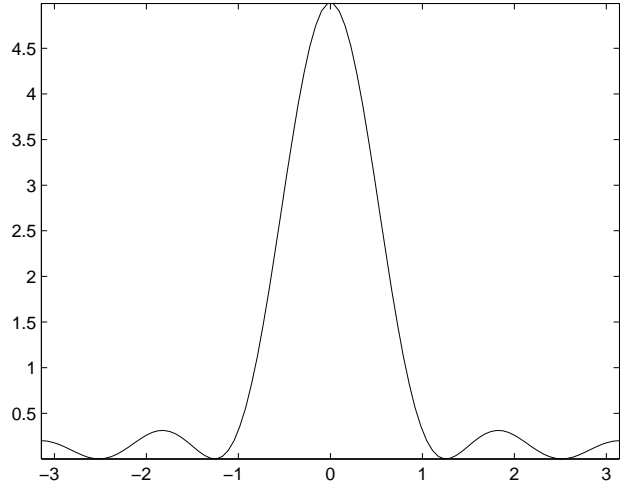}
   \caption{Fej{\'e}r kernel $K_n$ for $n=5$.}
   \label{f14}
\end{figure}
It is known (\cite{to}) that
$$
K_{n}(u)= {1\over n} \Big({\sin nu/2 \over u/2} \Big)^2.
$$
We consider also a generalization of the Fej{\'e}r means, the
\idword{Riesz means},
\begin{equation}\label{e3}
\sigma_n^\alpha f(x) := \sum_{|k|\leq n}
\Big(1-\Big(\frac{|k|}{n}\Big)^\gamma \Big)^\alpha \widehat
f(k)\ee^{\ii kx}=\frac{1}{2\pi}\int_{\T} f(x-u) K_n^\alpha(u) \dd u,
\index{\file-1}{$\sigma_n^\alpha f$}
\end{equation}
where
$$
K_{n}^\alpha(u) := \sum_{|k|\leq n}
\Big(1-\Big(\frac{|k|}{n}\Big)^\gamma \Big)^\alpha \ee^{\ii k u}
\index{\file-1}{$K_n^\alpha$}
$$
are the \idword{Riesz kernels}. Here, we suppose that
$0<\alpha<\infty,1\leq \gamma<\infty$. Since the results are
independent of $\gamma$, we omit the notation $\gamma$ in
$\sigma_n^\alpha f$ and $K_{n}^\alpha$. If $\alpha=\gamma=1$, then
we get the Fej{\'e}r means. It is known that
\begin{equation}\label{e1}
|K_n^\alpha(u)|\leq C \min (n, n^{-\alpha} u^{-\alpha-1}) \qquad
(\nn, u\neq 0),
\end{equation}
(see Zygmund \cite{zy}, Stein and Weiss \cite{stwe} or Weisz
\cite{wk2}). Note that this inequality follows from (\ref{e10.1}).

The \idword{maximal Riesz operator} is defined by
$$
\sigma_*^\alpha f := \sup_{n \in \N} |\sigma_{n}^\alpha
f|.\index{\file-1}{$\sigma_*^\alpha f$}
$$
Using (\ref{e1}), Zygmund and Riesz \cite{zy} proved

\begin{thm}\label{t3}
If $0<\alpha<\infty$ and $f \in L_1(\T)$, then
$$
\sup_{\rho >0} \rho \,\lambda(\sigma_*^\alpha f > \rho) \leq C
\|f\|_{1}.
$$
\end{thm}

This theorem will be proved right after Theorem \ref{t16.1} in
Section \ref{s16}. The next \ind{density theorem}, due to
Marcin\-kiewicz and Zygmund \cite{mazy}, is fundamental for the
almost everywhere convergence and is similar to the Banach-Steinhaus
theorem about the norm convergence of operators.

\begin{thm}\label{t4}
Suppose that $X$ is a normed space of measurable functions and
$X_0\subset X$ is dense in $X$. Let $T$ and $T_n$ $(n \in \N)$ be
linear operators such that
$$
Tf=\lim_{n\to \infty} T_nf \qquad \mbox{a.e.~for every $f \in X_0$}.
$$
If, for some $1\leq p<\infty$,
$$
\sup_{\rho>0} \rho \,\lambda(|Tf|>\rho)^{1/p} \leq C\|f\|_X \qquad
(f \in X)
$$
and
$$
\sup_{\rho>0} \rho \,\lambda(T_*f>\rho)^{1/p} \leq C\|f\|_X \qquad
(f \in X),
$$
where
$$
T_*f := \sup_{n \in \N} |T_nf| \qquad (f \in X),
$$
then
$$
Tf = \lim_{n \to \infty} T_nf \qquad \mbox{a.e.~for every $f \in
X$}.
$$

\end{thm}

\begin{proof}
Fix $f \in X$ and set
$$
\xi:= \limsup_{n \to \infty} |T_nf-Tf|.
$$
It is sufficient to show that $\xi=0$ a.e.

Choose $f_m \in X_0$ $(m \in \N)$ such that $\|f-f_m\|_X \to 0$ as
$m \to \infty$. Observe that
\begin{eqnarray*}
\xi &\leq& \limsup_{n \to \infty} |T_n(f-f_m)| +
\limsup_{n \to \infty} |T_nf_m-Tf_m| + |T(f_m-f)| \\
&\leq& T^*(f_m-f) + |T(f_m-f)|
\end{eqnarray*}
for all $m \in \N$. Henceforth, for all $\rho>0$ and $m \in \N$, we
have
\begin{eqnarray*}
\lambda(\xi>2\rho) &\leq& \lambda(T^*(f_m-f)>\rho) + \lambda(|T(f_m-f)|>\rho) \\
&\leq& C\rho^{-p} \|f_m-f\|_X^p + C \rho^{-p} \|f_m-f\|_X^p.
\end{eqnarray*}
Since $f_m \to f$ in $X$ as $m \to \infty$, it follows that
$$
\lambda(\xi>2\rho) =0
$$
for all $\rho>0$. So we can conclude that $\xi=0$ a.e.
\end{proof}

The weak type $(1,1)$ inequality of Theorem \ref{t3} and the density
argument of Theorem \ref{t4} will imply the almost everywhere
convergence of the Fej{\'e}r and Riesz means. We apply Theorem
\ref{t4} for $T=\cI$, the identity function, $T_n=\sigma_n^\alpha$,
$p=1$ and $X=L_1(\T)$. The dense set $X_0$ is the set of the
trigonometric polynomials. It is easy to see that
$\lim_{n\to\infty}\sigma_{n}^\alpha f=f$ {everywhere} for all $f\in
X_0$. This implies the next well known theorem, which is due to
Fej{\'e}r \cite{fej} and Lebesgue \cite{leb} for $\alpha=1$ and to
Riesz \cite{ri1} for other $\alpha$'s.

\begin{cor}\label{c5}
If $0<\alpha<\infty$ and $f \in L_1(\T)$, then
$$
\lim_{n\to\infty}\sigma_{n}^\alpha f=f \qquad \mbox{ a.e.}
$$
\end{cor}

Using (\ref{e1}) and the density of trigonometric polynomials, the
next corollary can be shown easily.

\begin{cor}\label{c6}
If $0<\alpha<\infty$ and $f \in C(\T)$, then
$$
\lim_{n\to\infty}\sigma_{n}^\alpha f=f \qquad \mbox{uniformly.}
$$
\end{cor}


\sect{Multi-dimensional partial sums}

For $x=(x_1,\ldots,x_d)\in \R^d$ and $u=(u_1,\ldots,u_d)\in \R^d$,
set
$$
u\cdot x := \sum_{k=1}^d u_k x_k, \qquad \|x\|_{q}:=\left\{
             \begin{array}{ll}
               \Big(\sum_{k=1}^d |x_k|^q \Big)^{1/q}, & \hbox{$0<q<\infty$;} \\
               \sup_{i=1,\ldots ,d} |x_i|, & \hbox{$q=\infty$.}
             \end{array}
           \right.\index{\file-1}{$u\cdot x$}
$$
The $d$-dimensional trigonometric system is introduced as a
Kronecker product by
$$
\ee^{\ii k\cdot x}=\prod_{j=1}^{d}\ee^{\ii k_jx_j},
$$
where $k=(k_1,\ldots,k_d)\in \Z^d$, $x=(x_1,\ldots,x_d) \in \T^d$.
The multi-dimensional \dword{Fourier
coefficients}\index{\file}{Fourier coefficient} of an integrable
function are given by
$$
\widehat f(k) = {1 \over (2\pi)^d} \int_{\T^d} f(x) \ee^{-\ii k\cdot
x} \dd x \qquad (k\in \N^d).
$$
The formal trigonometric series
$$
\sum_{k\in \Z^d} \widehat f(k) \ee^{\ii k \cdot x} \qquad (x\in
\T^d)
$$
defines the multi-dimensional \idword{Fourier series} of $f$.

We can generalize the partial sums (\ref{e2}) and summability means
(\ref{e3}) for multi-dimen\-sional functions basically in two ways.
In the first version, we replace the $|\cdot|$ in (\ref{e2}) and
(\ref{e3}) by $\|\cdot\|_q$. In the literature the most natural
choices $q=2$ (see e.g.~Stein and Weiss \cite{stwe,st1}, Davis and
Chang \cite{dach} and Grafakos \cite{gra}), $q=1$ (Berens, Li and Xu
\cite{bexu2,bexu1,bexu,xu2}, Weisz \cite{wel1-ft2,wel1-fs2}) and
$q=\infty$ (Marcinkiewicz \cite{ma}, Zhizhiashvili \cite{zi1} and
Weisz \cite{wk2,wmar6}) are investigated. In the second
generalization, we take the sum in each dimension, the so called
rectangular partial sum (Zygmund \cite{zy} and Weisz \cite{wk2}).

For $f\in L_1(\T^d)$, the $n$th \idword{$\ell_q$-partial sum}
\inda{$s_n^qf$}  $(\nn)$ is given by
$$
s_{n}^q f(x) := \sum_{k\in \Z^d,\, \|k\|_q \leq n} \widehat f(k)
\ee^{\ii k \cdot x} = \frac{1}{(2\pi)^d}\int_{\T^d} f(x-u) D_n^q(u)
\dd u
$$
where
$$
D_{n}^q(u) := \sum_{k\in \Z^d, \, \|k\|_q \leq n} \ee^{\ii k \cdot
u}\index{\file-1}{$D_n^q$}
$$
is the \idword{$\ell_q$-Dirichlet kernel}. The partial sums are
called  \dword{triangular} if $q=1$, \dword{circular} if $q=2$ and
\dword{cubic} if $q=\infty$ (see Figures
\ref{f1}--\ref{f17})\index{\file}{triangular partial
sum}\index{\file}{circular partial sum}\index{\file}{cubic partial
sum}.

\begin{figure}[htbp] 
   \centering
   \includegraphics[width=1\textwidth]{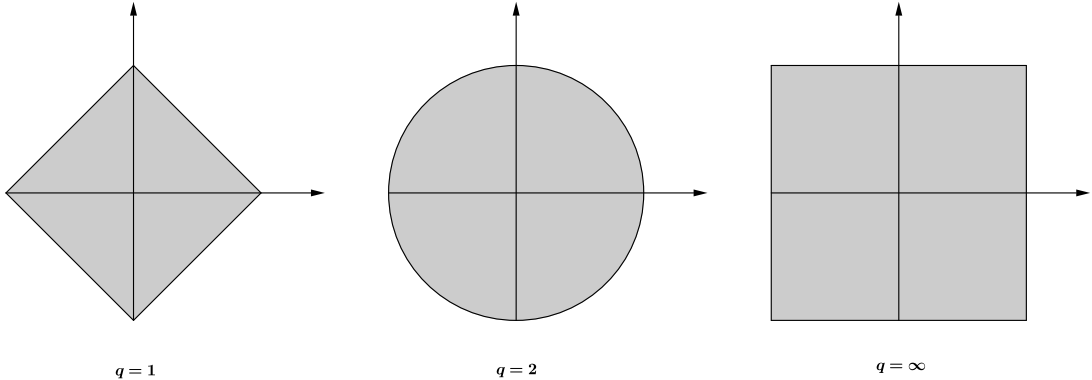}
   \caption{Regions of the $\ell_q$-partial sums for $d=2$.}
   \label{f1}
\end{figure}

\begin{figure}[htbp] 
   \centering
   \includegraphics[width=0.8\textwidth]{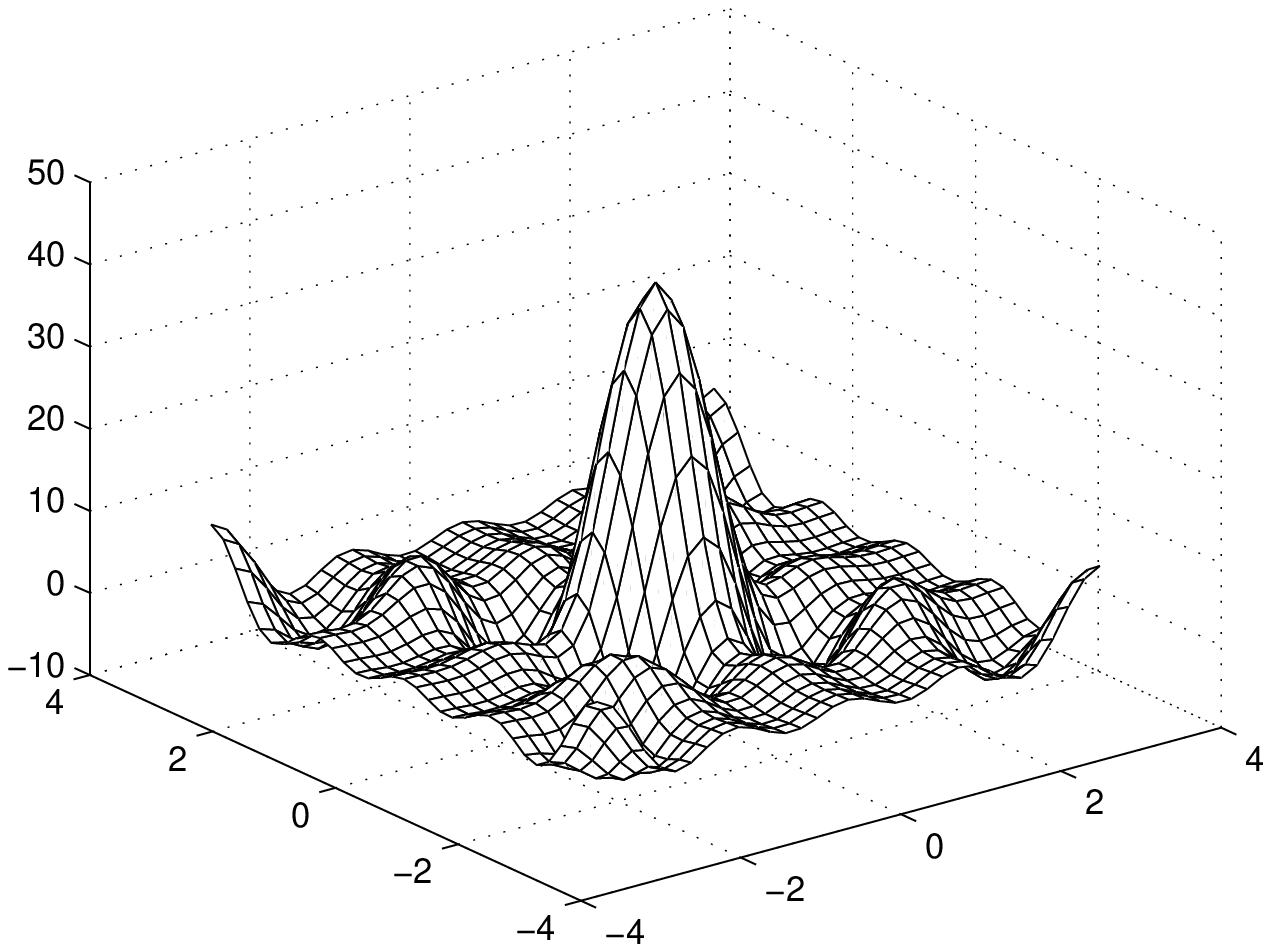}
   \caption{The Dirichlet kernel $D_n^q$ with $d=2$, $q=1$, $n=4$.}
   \label{f15}
\end{figure}

\begin{figure}[htbp] 
   \centering
   \includegraphics[width=0.8\textwidth]{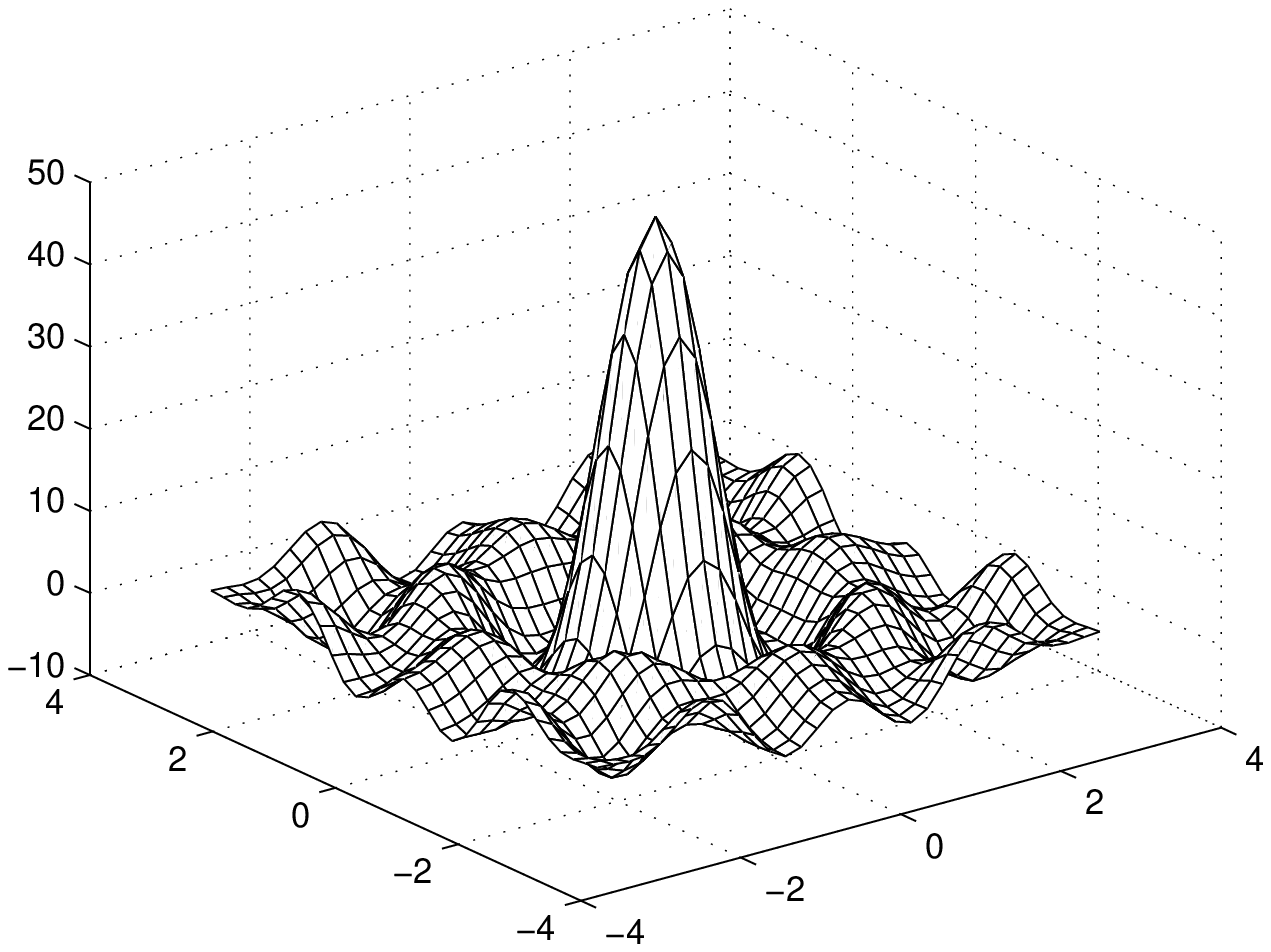}
   \caption{The Dirichlet kernel $D_n^q$ with $d=2$, $q=2$, $n=4$.}
   \label{f16}
\end{figure}

\begin{figure}[htbp] 
   \centering
   \includegraphics[width=0.8\textwidth]{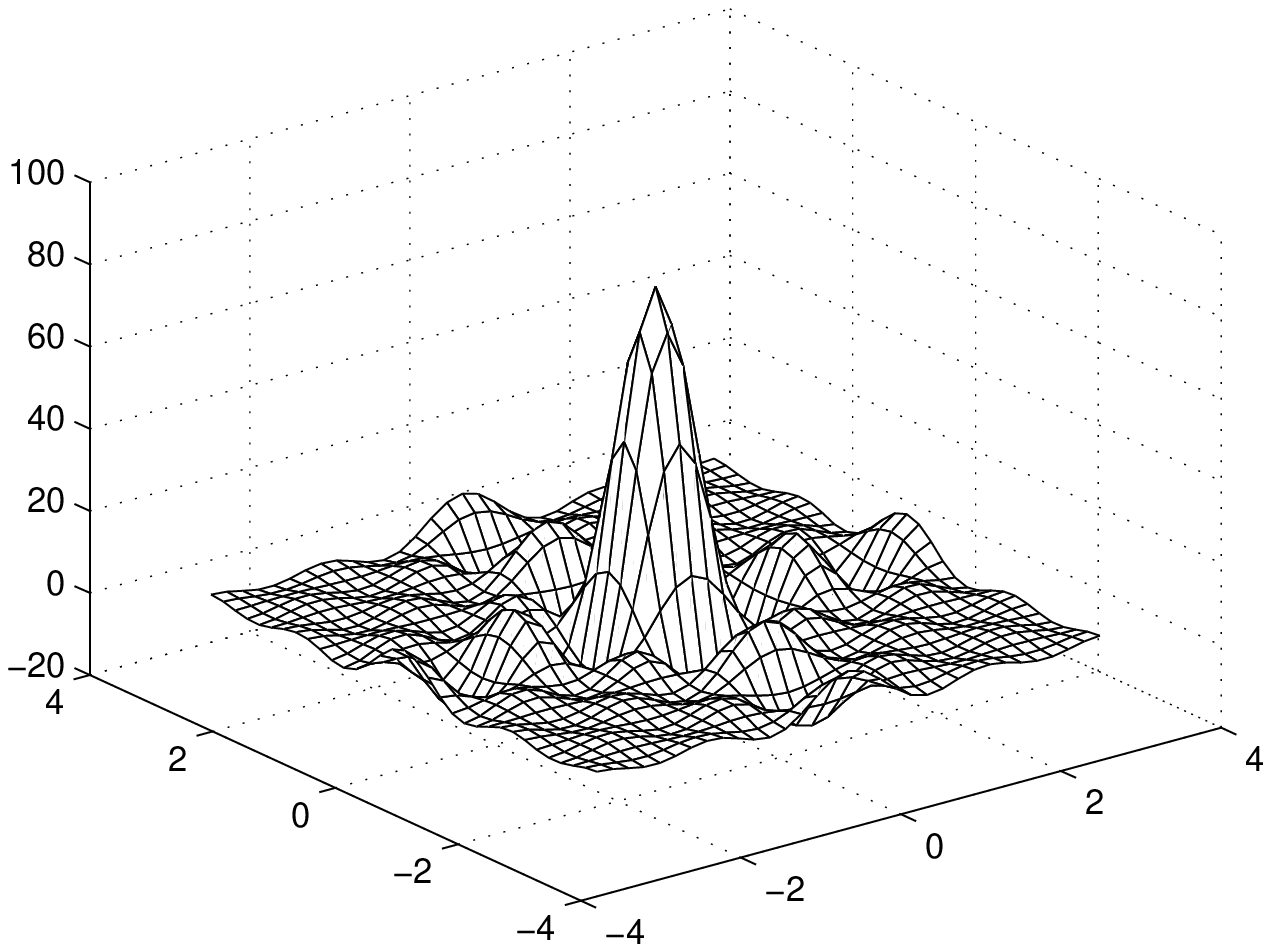}
   \caption{The Dirichlet kernel $D_n^q$ with $d=2$, $q=\infty$, $n=4$.}
   \label{f17}
\end{figure}

For $f\in L_1(\T^d)$, the $n$th \idword{rectangular partial sum}
\inda{$s_nf$} $(n\in \N^d)$ is introduced by
$$
s_{n} f(x) := \sum_{|k_1| \leq n_1}\cdots \sum_{|k_d| \leq n_d}
\widehat f(k) \ee^{\ii k \cdot x} = \frac{1}{(2\pi)^d}\int_{\T^d}
f(x-u) D_n(u) \dd u,
$$
where
$$
D_{n}(u) := \sum_{|k_1| \leq n_1}\cdots \sum_{|k_d| \leq n_d}
\ee^{\ii k \cdot u}\index{\file-1}{$D_n$}
$$
is the \idword{rectangular Dirichlet kernel} (see Figure \ref{f23}).
\begin{figure}[htbp] 
   \centering
   \includegraphics[width=0.8\textwidth]{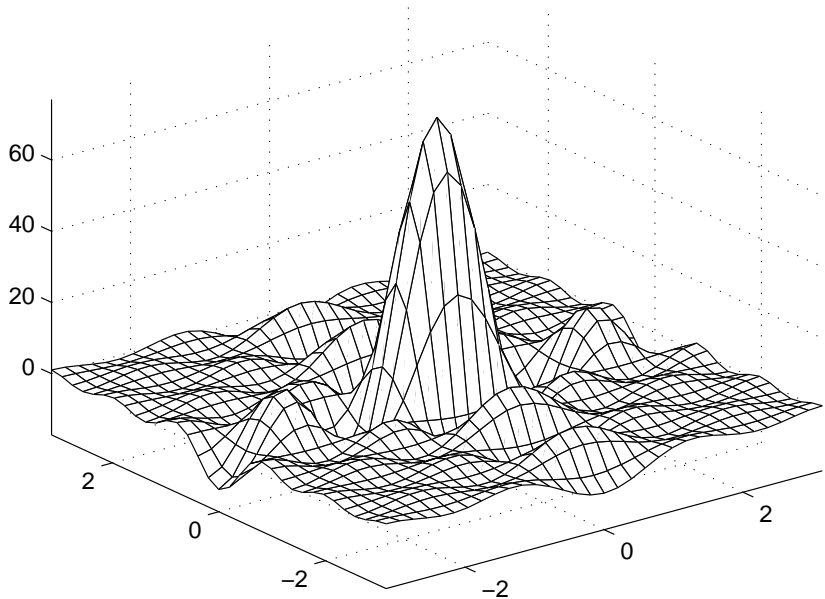}
   \caption{The rectangular Dirichlet kernel with $d=2$, $n_1=3$, $n_2=5$.}
   \label{f23}
\end{figure}

By iterating the one-dimensional result, we get easily the next
theorem.

\begin{thm}\label{t7}
If $f\in L_p(\T^d)$ for some $1<p< \infty$, then
$$
\|s_{n}f\|_p  \leq C_p \|f\|_p \qquad (n\in \N^d)
$$
and
$$
\lim_{n\to\infty} s_{n}f=f \qquad \mbox{in the $L_p$-norm}.
$$
\end{thm}
Here $n\to \infty$ means the \ind{Pringsheim convergence}, i.e.,
$\min(n_1,\ldots,n_d)\to \infty$.

\begin{proof}
By Theorem \ref{t1},
\begin{eqnarray*}
\int_\T |s_n f(x)|^p \dd x_1 &=& \int_\T \Big| \int_\T \Big(\int_{\T^{d-1}} f(t) \\
&&{} (D_{n_2}(x_2 + t_2) \cdots D_{n_d}(x_d + t_d)) \dd t_2\cdots
\dd t_d \Big)
D_{n_1}(x_1 + t_1) \dd t_1 \Big|^p \dd x_1 \\
&\leq& \int_\T \Big| \int_{\T^{d-1}} f(t) (D_{n_2}(x_2 + t_2) \cdots
D_{n_d}(x_d + t_d)) \dd t_2\cdots \dd t_d \Big|^p \dd t_1.
\end{eqnarray*}
Applying this inequality $(d-1)$-times, we get the desired
inequality of Theorem \ref{t7}. The convergence is a consequence of
this inequality and of the density of trigonometric polynomials.
\end{proof}

A similar result holds for the triangular and cubic partial sums.

\begin{thm}\label{t8}
If $q=1,\infty$ and $f\in L_p(\T^d)$ for some $1<p< \infty$, then
$$
\|s_{n}^qf\|_p  \leq C_p \|f\|_p \qquad (n\in \N)
$$
and
$$
\lim_{n\to\infty} s_{n}^qf=f \qquad \mbox{in the $L_p$-norm}.
$$
If $q=2$, then the same result is valid for $p=2$.
\end{thm}

\begin{proof}
The result for $q=\infty$ follows from Theorem \ref{t7}. For $q=2$,
it is a basic result of Fourier analysis. If $q=1$, then we will
prove the result for $d=2$, only. The general case can be proved in
the same way. Observe that
\begin{equation}\label{e35}
\int_{\T^2} f(x,y) \ee^{\ii kx+\ii ly} \dd x\dd y = 2 \int_{\T^2}
f(u-v,u+v) \ee^{\ii u(k+l)+\ii v(l-k)} \dd u\dd v.
\end{equation}
If $|k|+|l|\leq n$ on the left hand side, then $|k+l|\leq n$ and
$|l-k|\leq n$ on the right hand side, hence
\begin{equation}\label{e38} s_n^1f(x,y)=2s_n^\infty g(u,v),
\end{equation}
where
$$
g(u,v):=f(u-v,u+v),\quad x=u-v,\quad y=u+v.
$$
Thus
$$
\|s_n^1f\|_p  \leq 2^{1+1/p} \|s_n^\infty g\|_p \leq C_p \|g\|_p
\leq C_p \|f\|_p
$$
shows the result for $q=1$, too.
\end{proof}

Since the characteristic function of the unit ball is not an
$L_p(\R^d)$ $(1<p\neq 2<\infty,d\geq 2)$ multiplier (see Fefferman
\cite{cfe3} or Grafakos \cite[p.~743]{gra}), we have

\begin{thm}\label{t9}
If $d\geq 2$, $q=2$ and $1<p\neq 2 < \infty$, then the preceding
theorem is not true.
\end{thm}

The analogue of \ind{Carleson's theorem} does not hold in higher
dimensions for the rectangular partial sums. However, it is true for
the triangular and cubic partial sums (see Fefferman
\cite{cfe1,cfe2} and Grafakos \cite[p.~231]{gra}). Let us denote by
$$
s_*^qf:=\sup_{n \in \N} |s_{n}^qf|\index{\file-1}{$s_*^qf$}
$$
the \idword{maximal operator}.

\begin{thm}\label{t10}
If $q=1,\infty$ and $f\in L_p(\T^d)$ for some $1<p< \infty$, then
$$
\|s_*^qf\|_p  \leq C_p \|f\|_p
$$
and if $1<p\leq \infty$, then
$$
\lim_{n\to\infty} s_{n}^qf=f \qquad \mbox{a.e.}
$$
\end{thm}

\begin{proof}
We will prove the theorem for $d=2$ only. The proof for higher
dimensions is similar. Suppose first that $q=\infty$ and
\begin{equation}\label{e36}
\widehat {f}(k,l)=0 \qquad \mbox{for $l<k$ or $k<0$}.
\end{equation}
Let
$$
f_x(y):=f(x,y) \qquad (x,y\in \T)
$$
and observe by Fubini's theorem that $f_x$ belongs to $L_p(\T)$.
Hence, by Theorem \ref{t2},
\begin{equation}\label{e37}
\|s_* f_x\|_p  \leq C_p \|f_x\|_p
\end{equation}
for a.e.~$x\in \T$. Set
$$
h_l(x):= \int_{\T} f_x(y) \ee^{\ii ly} \dd y \qquad (l\in \N)
$$
and observe that
\begin{eqnarray*}
\|h_l\|_p &=& \Big( \int_{\T} \Big| \int_{\T} f_x(y) \ee^{\ii ly} \dd y \Big|^p \dd x \Big)^{1/p} \\
&\leq& C_p \Big( \int_{\T} \int_{\T} |f_x(y)|^p \dd y \dd x \Big)^{1/p} \\
&=&C_p\|f\|_p.
\end{eqnarray*}
Thus $h_l\in L_p(\T)$. Since
$$
\widehat {h}_l(k) = \int_{\T} h_l(x) \ee^{\ii kx} = \widehat
{f}(k,l),
$$
it is clear by (\ref{e36}) that each $h_l$ is a trigonometric
polynomial. More precisely, $\widehat {h}_l(k)$ vanishes if $k<0$ or
$k>l$. Consequently,
\begin{eqnarray*}
s_n f_x(y)&=& \sum_{|l|\leq n} h_l(x) \ee^{\ii ly} \\
&=& \sum_{|l|\leq n} \Big( \sum_{k=0}^{l} \widehat {f}(k,l) \ee^{\ii kx} \Big) \ee^{\ii ly} \\
&=& \sum_{0\leq k\leq l\leq n} \widehat {f}(k,l) \ee^{\ii kx+\ii ly}\\
&=& s_{n}^\infty f(x,y).
\end{eqnarray*}
Hence (\ref{e37}) implies
\begin{eqnarray*}
\|s_*^\infty f\|_p = \Big( \int_{\T} \int_{\T} |s_*f_x(y)|^p \dd y
\dd x \Big)^{1/p} \leq C_p \Big( \int_{\T} \int_{\T} |f_x(y)|^p \dd
y \dd x \Big)^{1/p} = C_p \|f\|_p,
\end{eqnarray*}
which proves the theorem if (\ref{e36}) holds. Obviously, the same
holds for functions $f$ for which $\widehat {f}(k,l)=0$ if $l>k$ or
$l<0$ and we could also repeat the proof for the other quadrants.

Let us define the projections
$$
P_1^+f(x,y):=\sum_{k\in \N} \sum_{l\in \Z} \widehat {f}(k,l)
\ee^{\ii kx+\ii ly},
$$
$$
P_2^+f(x,y):=\sum_{k\in \Z} \sum_{l\in \N} \widehat {f}(k,l)
\ee^{\ii kx+\ii ly},
$$
$$
Q_1f(x,y):=\sum_{l\geq |k|} \widehat {f}(k,l) \ee^{\ii kx+\ii ly}
$$
and
$$
Qf(x,y):=\sum_{l\geq k\geq 0} \widehat {f}(k,l) \ee^{\ii kx+\ii ly}
$$
(see Figure (\ref{f10})).

\begin{figure}[htbp] 
   \centering
   \includegraphics[width=1\textwidth]{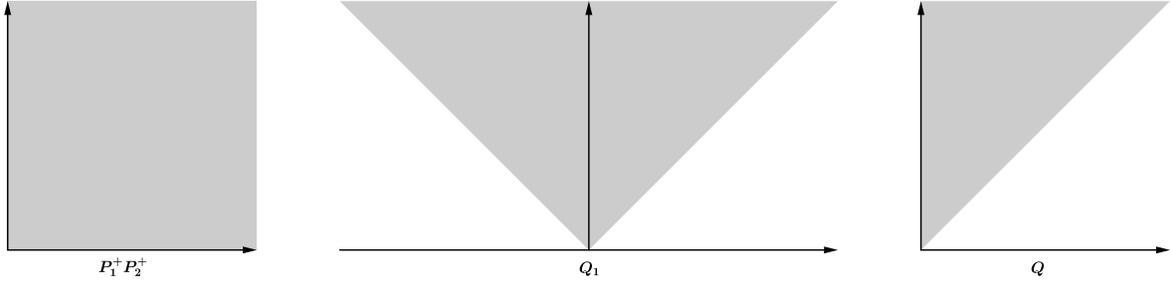}
   \caption{The projections $P_1^+P_2^+$, $Q_1$ and $Q$.}
   \label{f10}
\end{figure}

By (\ref{e35}) and Theorem \ref{t53}, we conclude that $Q_1f(x,y)=2
P_1^+P_2^+ g(u,v)$ and
$$
\|Q_1f\|_p = 2^{1+1/p} \|P_1^+P_2^+ g\|_p \leq C_p \|P_2^+ g\|_p
\leq C_p \|g\|_p \leq C_p \|f\|_p,
$$
where
$$
g(u,v):=f(u-v,u+v), \quad x=u-v,\quad y=u+v
$$
and $1<p<\infty$. Thus $Q_1$ is a bounded projection on $L_p(\T^2)$
and so is $Q=Q_1P_1^+P_2^+$. Since $Qf$ satisfies (\ref{e36}), we
obtain
$$
\|s_*^\infty (Qf)\|_p \leq C_p \|Qf\|_p \leq C_p \|f\|_p.
$$
Each function $f$ can be rewritten as the sum of eight similar
projections, which implies the theorem for $q=\infty$.

Equality (\ref{e38}) implies
$$
\|s_*^1f\|_p  \leq 2^{1+1/p} \|s_*^\infty g\|_p \leq C_p \|g\|_p
\leq C_p \|f\|_p,
$$
which also shows the result for $q=1$.
\end{proof}

The generalization of Theorem \ref{t33} for higher dimensions was
proved by Antonov \cite{{Antonov-2004}}.

\begin{thm}\label{t35}
If $q=\infty$ and
$$
\int_{\T^d} |f(x)| (\log^+|f(x)|)^{d} \log^+ \log^+ \log^+ |f(x)|
\dd x <\infty,
$$
then
$$
\lim_{n\to\infty} s_{n}^qf=f \qquad \mbox{a.e.}
$$
\end{thm}

Theorem \ref{t10} does not hold for circular partial sums (Stein and
Weiss \cite[p.~268]{stwe}).

\begin{thm}\label{t11}
If $q=2$ and $p<2d/(d+1)$, then there exists a function $f\in
L_p(\T^d)$ whose circular partial sums $s_n^qf$ diverge almost
everywhere.
\end{thm}

In other words, for a general function in $L_p(\T^d)$ $(p<2)$ almost
everywhere convergence of the circular partial sums is not true if
the dimension is sufficiently large. It is an open problem, whether
Theorem \ref{t10} holds for $p=2$ and for circular partial sums. As
in the one-dimensional case, Theorem \ref{t7}, Theorem \ref{t8} and
the inequality in Theorem \ref{t10} do not hold for $p=1$ and
$p=\infty$.


\sect{$\ell_q$-summability}

As we mentioned before, we define the \dword{$\ell_q$-Fej{\'e}r and
Riesz means} of an integrable function $f\in L_1(\T^d)$
by\index{\file}{$\ell_q$-Fej{\'e}r
means}\index{\file}{$\ell_q$-Riesz means}
$$
\sigma_n^qf(x) := \sum_{k\in \Z^d, \, \|k\|_q\leq n}
\Big(1-\frac{\|k\|_q}{n} \Big) \widehat f(k) \ee^{\ii k \cdot x}
=\frac{1}{(2\pi)^d}\int_{\T^d} f(x-u) K_n^q(u) \dd
u,\index{\file-1}{$\sigma_n^qf$}
$$
and
\begin{eqnarray}\label{e49}
\sigma_n^{q,\alpha}f(x) &:=& \sum_{k\in \Z^d, \, \|k\|_q\leq n} \Big(1-\Big(\frac{\|k\|_q}{n}\Big)^\gamma \Big)^\alpha \widehat f(k) \ee^{\ii k \cdot x} \nonumber \\
&=&\frac{1}{(2\pi)^d}\int_{\T^d} f(x-u) K_n^{q,\alpha}(u) \dd
u,\index{\file-1}{$\sigma_n^{q,\alpha}f$}
\end{eqnarray}
where
$$
K_{n}^q(u) := \sum_{k\in \Z^d, \, \|k\|_q\leq n}
\Big(1-\frac{\|k\|_q}{n} \Big) \ee^{\ii k \cdot
u}\index{\file-1}{$K_n^q$}
$$
and
$$
K_{n}^{q,\alpha}(u) := \sum_{k\in \Z^d, \, \|k\|_q\leq n}
\Big(1-\Big(\frac{\|k\|_q}{n}\Big)^\gamma \Big)^\alpha \ee^{\ii k
\cdot u}\index{\file-1}{$K_n^{q,\alpha}$}
$$
are the \dword{$\ell_q$-Fej{\'e}r and Riesz kernels} (see Figures
\ref{f18}--\ref{f22}).\index{\file}{$\ell_q$-Fej{\'e}r
kernels}\index{\file}{$\ell_q$-Riesz kernels}

\begin{figure}[htbp] 
   \centering
   \includegraphics[width=0.8\textwidth]{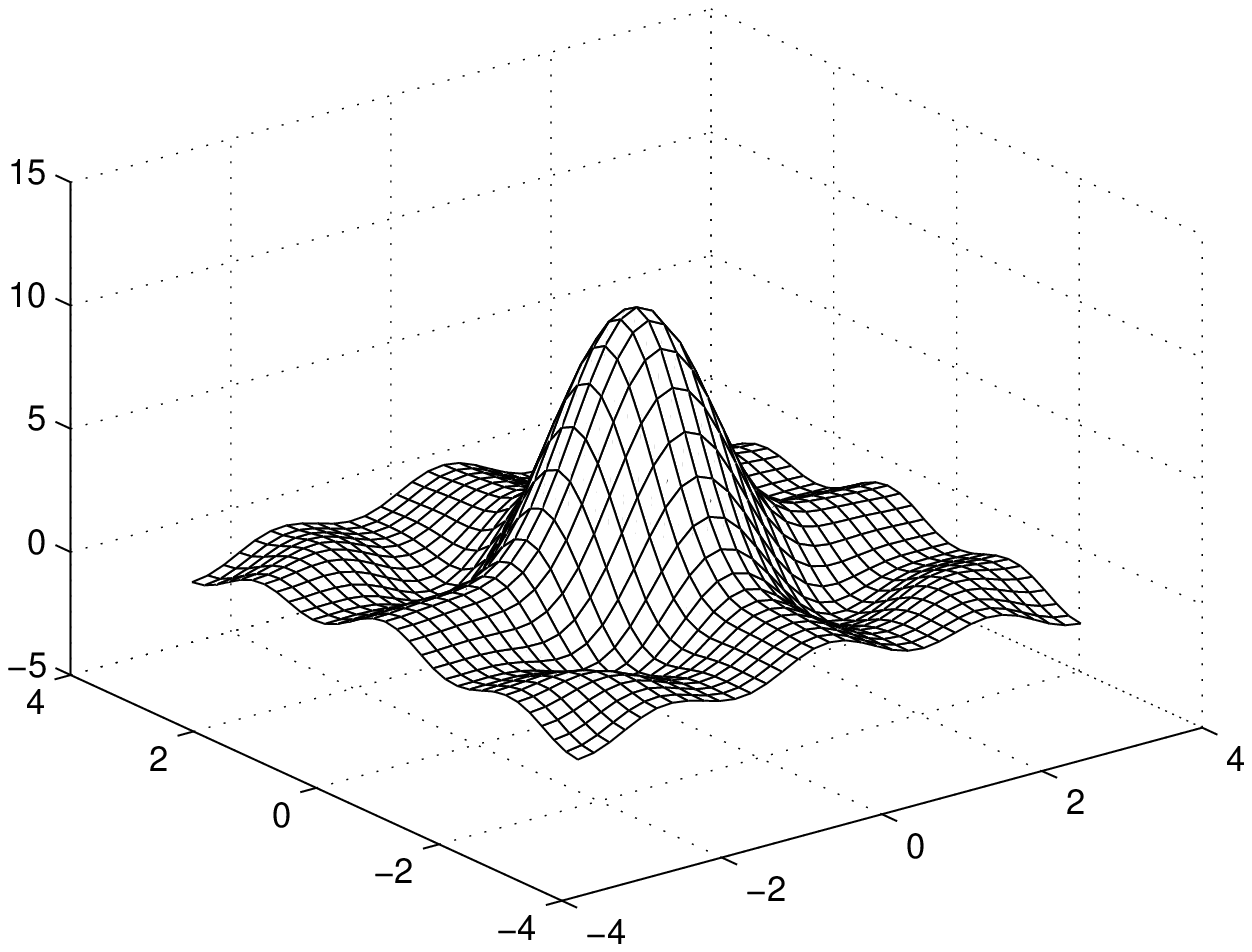}
   \caption{The Riesz kernel $K_{n}^{q,\alpha}$ with $d=2$, $q=1$, $n=4$, $\alpha=1$, $\gamma=1$.}
   \label{f18}
\end{figure}

\begin{figure}[htbp] 
   \centering
   \includegraphics[width=0.8\textwidth]{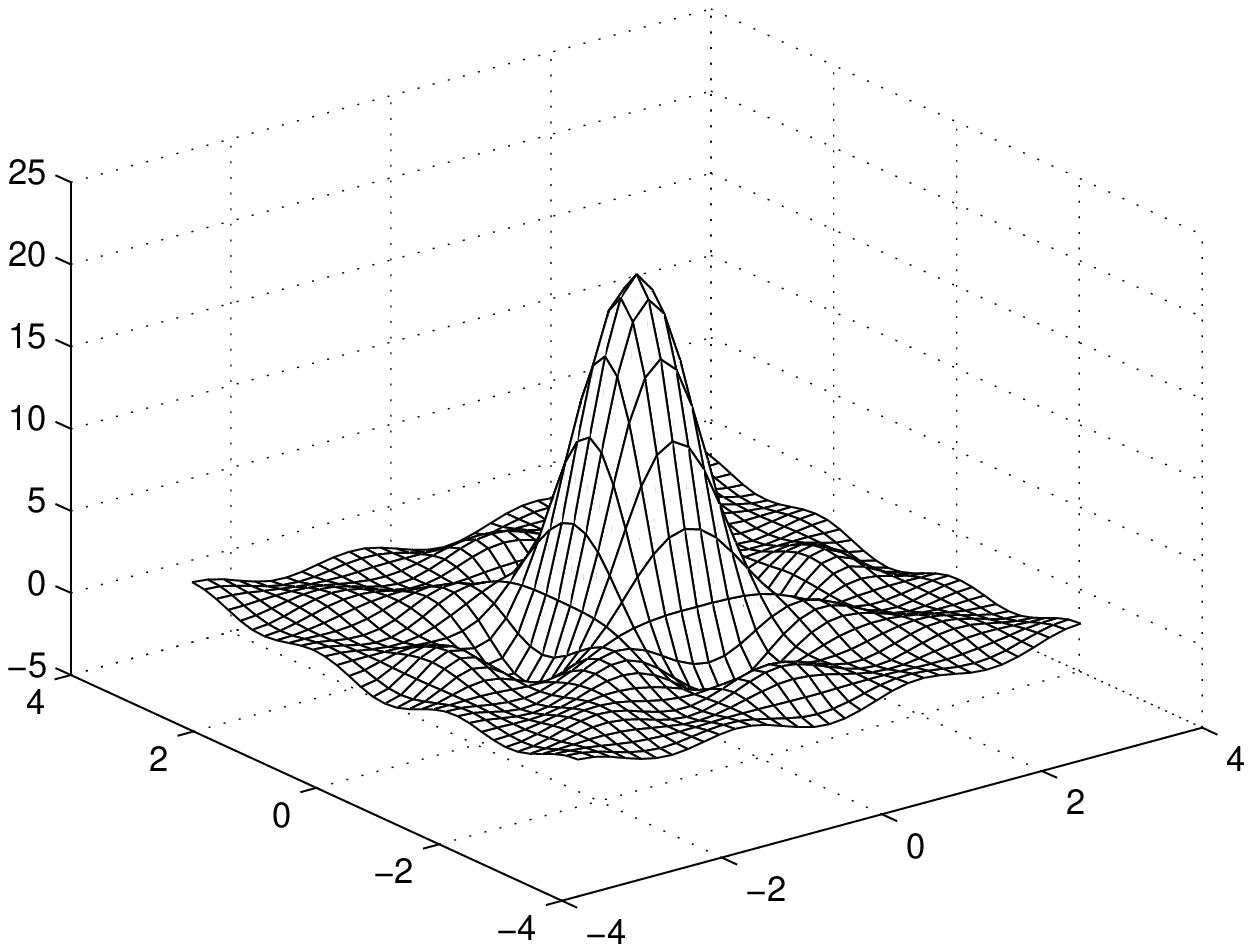}
   \caption{The Riesz kernel $K_{n}^{q,\alpha}$ with $d=2$, $q=\infty$, $n=4$, $\alpha=1$, $\gamma=1$.}
   \label{f19}
\end{figure}

\begin{figure}[htbp] 
   \centering
   \includegraphics[width=0.8\textwidth]{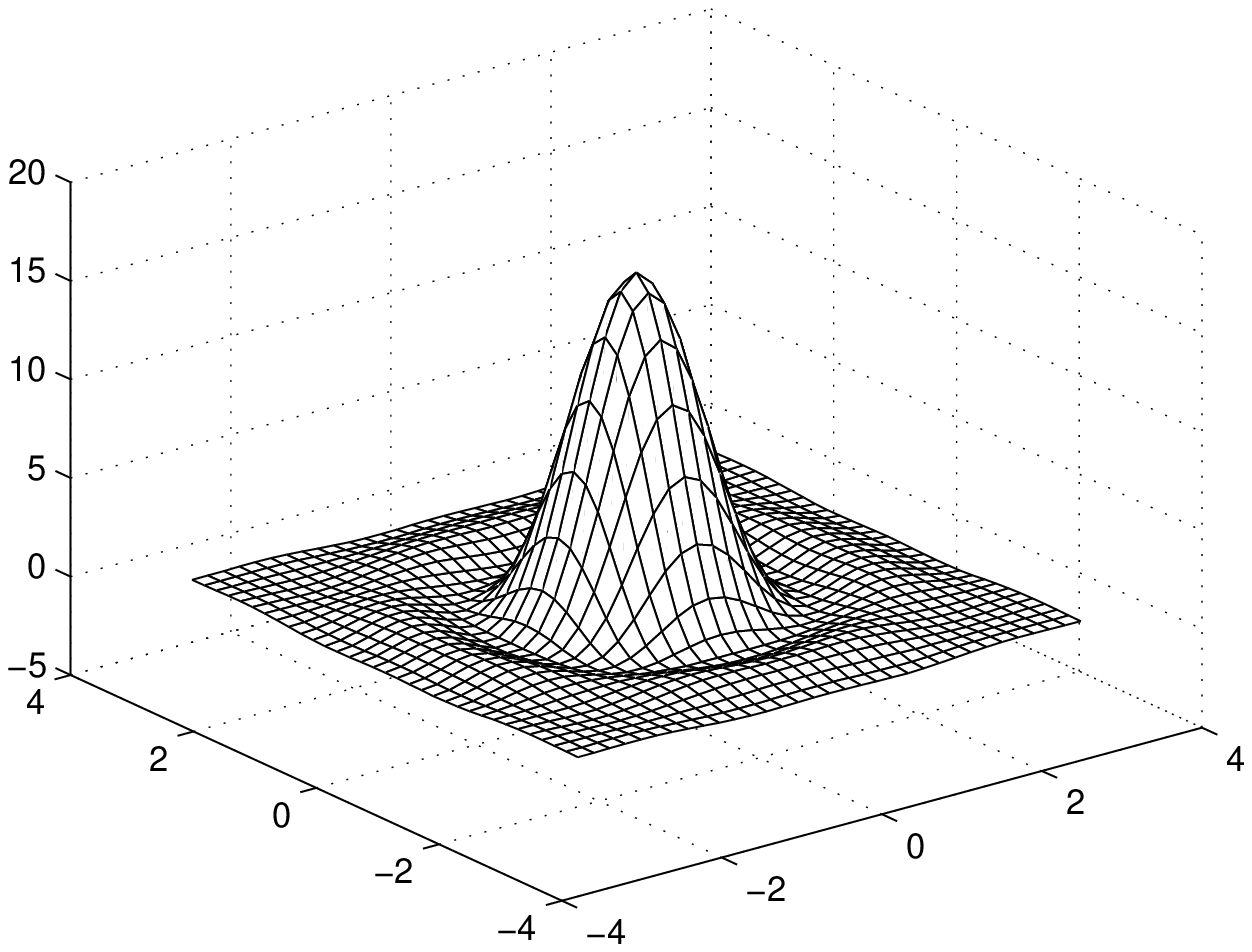}
   \caption{The Riesz kernel $K_{n}^{q,\alpha}$ with $d=2$, $q=2$, $n=4$, $\alpha=1$, $\gamma=1$.}
   \label{f20}
\end{figure}

\begin{figure}[htbp] 
   \centering
   \includegraphics[width=0.8\textwidth]{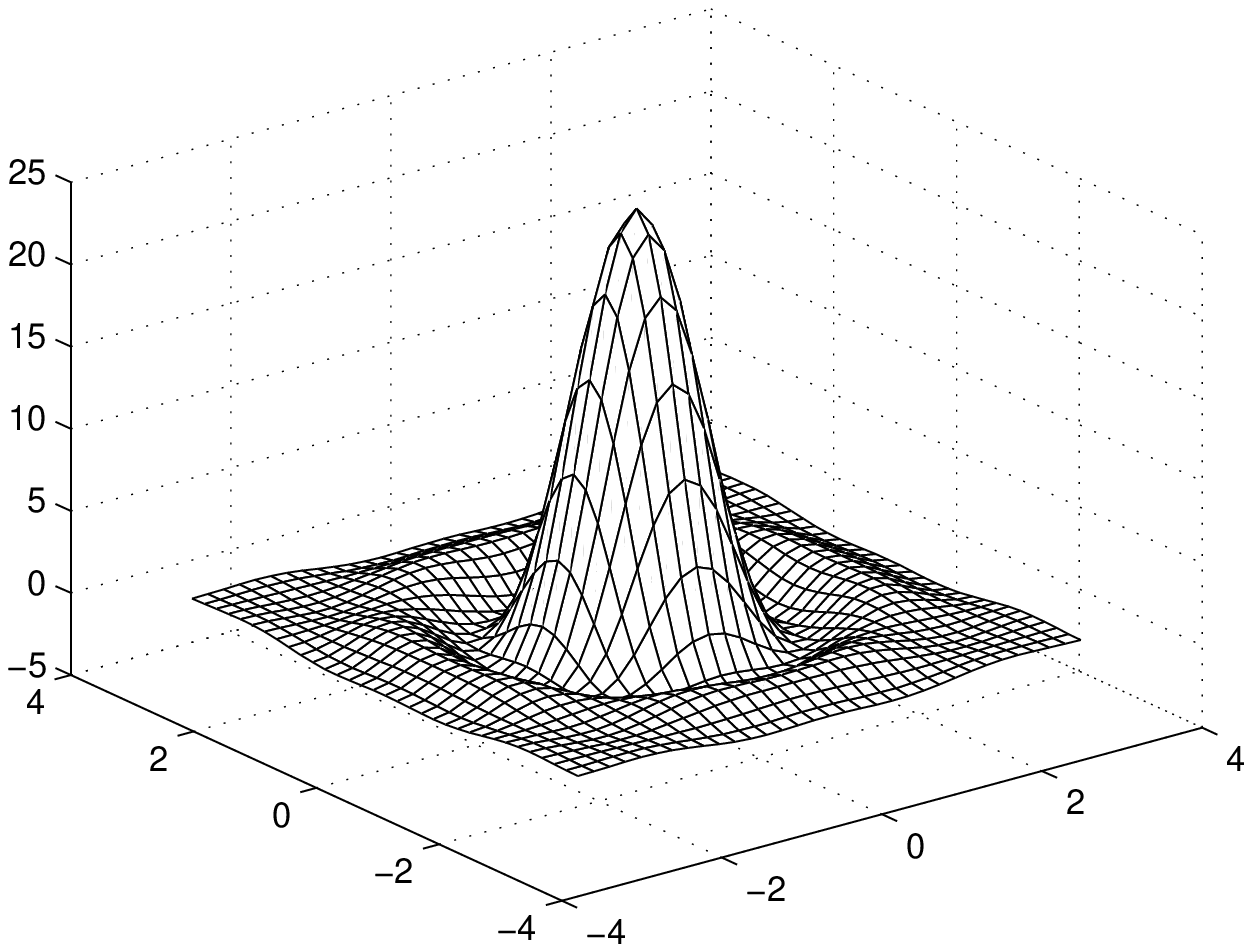}
   \caption{The Bochner-Riesz kernel $K_{n}^{q,\alpha}$ with $d=2$, $q=2$, $n=4$, $\alpha=1$, $\gamma=2$.}
   \label{f21}
\end{figure}

\begin{figure}[htbp] 
   \centering
   \includegraphics[width=0.8\textwidth]{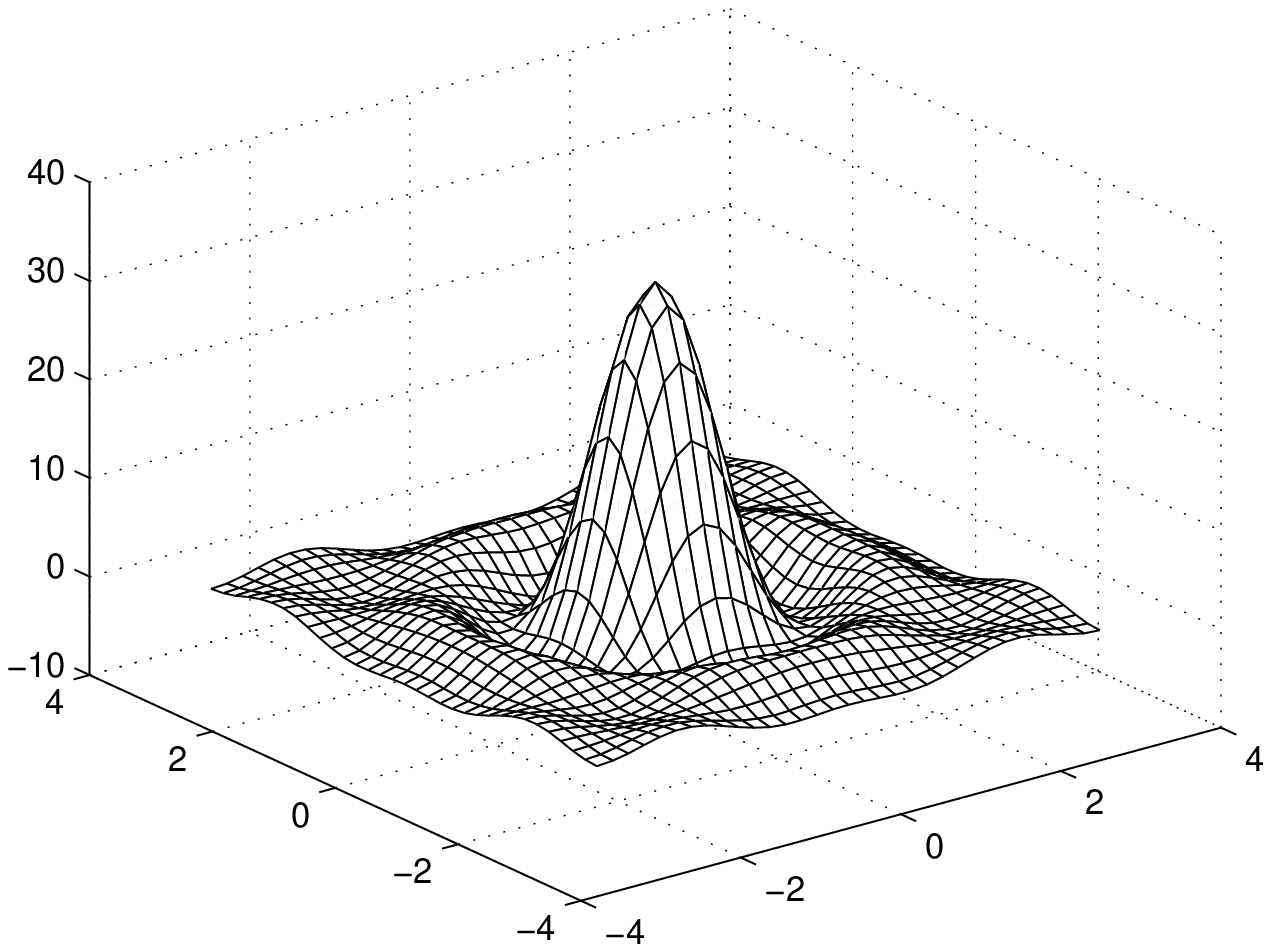}
   \caption{The Bochner-Riesz kernel $K_{n}^{q,\alpha}$ with $d=2$, $q=2$, $n=4$, $\alpha=1/2$, $\gamma=2$.}
   \label{f22}
\end{figure}

Observe that if $q=1,\infty$ (in this case $\|k\|_q$ is an integer),
then
$$
K_{n}^q(u) = \sum_{\|k\|_q\leq n} \sum_{j=\|k\|_q}^{n-1} \frac{1}{n}
\ee^{\ii k \cdot u} = \frac{1}{n} \sum_{j=0}^{n-1} D_j^q(u).
$$
Moreover,
\begin{equation}\label{e61.1}
|K_n^{q,\alpha}|\leq Cn^d \qquad (n\in \N^d).
\end{equation}
We will always suppose that $0\leq \alpha<\infty$, $1\leq
\gamma<\infty$. In the case $q=2$ let $\gamma\in \N$. If $\alpha=0$,
we get the partial sums and if $q=\gamma=2,\alpha>0$ the means are
called \idword{Bochner-Riesz means}. The cubic summability (when
$q=\infty$) is also called \idword{Marcinkiewicz summability}.
Obviously, the $\ell_q$-Fej{\'e}r means are the arithmetic means of
the $\ell_q$-partial sums when $q=1,\infty$:
$$
\sigma_n^q f(x) = \frac{1}{n} \sum_{k=0}^{n-1} s_{k}^q f(x).
$$

The proofs of the results presented later are very different for the
cases $q=1,2,\infty$, because the kernel functions are very
different. To demonstrate this, we present a few details about the
kernels in this section. For the triangular \ind{Dirichlet kernel},
we need the notion of the divided difference, which is usually used
in numerical analysis. The $n$th \idword{divided difference} of a
function $f$ at the (pairwise distinct) nodes $x_1,\ldots,x_n\in \R$
is introduced inductively as
\begin{equation}\label{e47}
[x_1]f:=f(x_1), \qquad
[x_1,\ldots,x_n]f:=\frac{[x_1,\ldots,x_{n-1}]f-[x_2,\ldots,x_n]f}{x_1-x_n}.
\end{equation}
One can see that the difference is a symmetric function of the
nodes. It is known (see e.g.~DeVore and Lorentz \cite[p.~120]{delo})
that
\begin{equation}\label{e39}
[x_1,\ldots,x_n]f= \sum_{k=1}^{n} \frac{f(x_k)} {\prod_{j=1,j\neq
k}^{n} (x_k-x_j)}.
\end{equation}
Moreover, if $f$ is $(n-1)$-times continuously differentiable on
$[a,b]$ and $x_i\in [a,b]$, then there exists $\xi\in [a,b]$ such
that
\begin{equation}\label{e42}
[x_1,\ldots,x_n]f=\frac{f^{(n-1)}(\xi)}{(n-1)!}.
\end{equation}
To give an explicit form of the triangular Dirichlet kernel, we will
need the following trigonometric identities.

\begin{lem}\label{l11}
For all $\nn$ and $0\leq x,y\leq \pi$,
\begin{eqnarray}\label{e40}
\lefteqn{\sum_{k=0}^{n} \epsilon_k \cos (ky) \sin((n-k+1/2)x) } \n\\
&&{}= \sin (x/2) \frac{\cos (x/2) \cos ((n+1/2)x)-\cos (y/2) \cos
((n+1/2)y)}{\cos x-\cos y}
\end{eqnarray}
and
\begin{eqnarray}\label{e41}
\lefteqn{\sum_{k=0}^{n} \epsilon_k \cos (ky) \cos((n-k+1/2)x) } \n\\
&&{}= \cos (x/2) \frac{\sin (y/2) \sin ((n+1/2)y)-\sin (x/2) \sin
((n+1/2)x)}{\cos x-\cos y},
\end{eqnarray}
where $\epsilon_0:=1/2$ and $\epsilon_k:=1$, $k\geq 1$.
\end{lem}

\begin{proof}
By trigonometric identities,
\begin{eqnarray*}
\lefteqn{\sum_{k=0}^{n} \epsilon_k \cos (ky) \sin((n-k+1/2)x) } \n\\
&=&
\sin((n+1/2)x) \sum_{k=0}^{n} \epsilon_k \cos (ky) \cos(kx) - \cos((n+1/2)x) \sum_{k=0}^{n} \epsilon_k \cos (ky) \sin(kx) \\
&=&
\frac{1}{2} \sin((n+1/2)x) \sum_{k=0}^{n} \Big( \epsilon_k \cos (k(x-y)) + \epsilon_k \cos(k(x+y) \Big)\\
&&{} - \frac{1}{2} \cos((n+1/2)x) \sum_{k=0}^{n} \Big( \epsilon_k
\sin (k(x-y)) + \epsilon_k\sin(k(x+y)) \Big).
\end{eqnarray*}
Similarly to (\ref{e43}), we can show that
$$
\sum_{k=0}^{n} \epsilon_k \sin (kx)
=\frac{\cos(x/2)-\cos((n+1/2)x)}{2\sin(x/2)}.
$$
Using this and (\ref{e43}), we conclude
\begin{eqnarray*}
\lefteqn{\sum_{k=0}^{n} \epsilon_k \cos (ky) \sin((n-k+1/2)x) } \n\\
&=&
\frac{1}{4} \sin((n+1/2)x) \Big( \frac{\sin((n+1/2)(x-y))}{\sin ((x-y)/2)} + \frac{\sin((n+1/2)(x+y))}{\sin ((x+y)/2)} \Big)\\
&&{} - \frac{1}{4} \cos((n+1/2)x) \Big( \frac{\cos ((x-y)/2)-\cos ((n+1/2)(x-y))}{\sin ((x-y)/2)} \\
&&{}+ \frac{1}{4} \frac{\cos ((x+y)/2)-\cos ((n+1/2)(x+y))}{\sin
((x+y)/2)} \Big).
\end{eqnarray*}
Now after some computation, we obtain (\ref{e40}).

Formula (\ref{e41}) can be shown in the same way.
\end{proof}

Define the function $G_n$ by
$$
G_n(\cos x):=(-1)^{[(d-1)/2]}2 \cos (x/2) (\sin x)^{d-2} \soc
((n+1/2)x)\index{\file-1}{$G_n$}
$$
where the $\soc$ function is defined by
$$
\soc x:= \left\{
           \begin{array}{ll}
             \cos x, & \hbox{if $d$ is even;} \\
             \sin x, & \hbox{if $d$ is odd.}
           \end{array}
         \right.\index{\file-1}{$\soc$}
$$
The following representation of the triangular \ind{Dirichlet
kernel} was proved by Herriot \cite{her1} and Berens and Xu
\cite{bexu,xu2}.

\begin{lem}\label{l7}
For $x\in \T^d$,
\begin{eqnarray}\label{e46}
D_n^1(x) &=& [\cos x_1,\ldots,\cos x_d]G_n \\
&=& (-1)^{[(d-1)/2]}2 \sum_{k=1}^{d} \frac{\cos (x_k/2) (\sin
x_k)^{d-2} \soc ((n+1/2)x_k)} {\prod_{j=1,j\neq k}^{d} (\cos
x_k-\cos x_j)}. \n
\end{eqnarray}
\end{lem}

\begin{proof}
First, we note that the second equality follows from the definition
of $G_n$ and from the property of the divided difference described
in (\ref{e39}). We prove the lemma by induction. Let us denote the
Dirichlet kernel in this proof by $D_{d,n}^1(x)=D_n^1(x)$. We have
seen in (\ref{e43}) that in the one-dimensional case
$$
D_{n}^1(x) = D_{1,n}^1(x) = \frac{\sin((n+1/2)x)}{\sin(x/2)} = 2
\cos (x/2) (\sin x)^{-1} \sin ((n+1/2)x),
$$
thus (\ref{e46}) holds for $d=1$. Suppose the lemma is true for
integers up to $d$ and let $d$ be even. It is easy to see that
\begin{eqnarray*}
D_{d+1,n}^1(x)&=& 2^{d+1}\sum_{j\in \N^d, \, \|j\|_1 \leq n} \epsilon_{j_1}\cos(j_1x_1)\cdots \epsilon_{j_{d+1}}\cos(j_{d+1}x_{d+1})\\
&=& 2\sum_{l=0}^n \epsilon_{l}\cos(lx_{d+1}) D_{d,n-l}(x_1,\ldots,x_d)\\
&=& (-1)^{[(d-1)/2]}4 \sum_{k=1}^{d} \frac{\cos (x_k/2) (\sin
x_k)^{d-2}}
{\prod_{j=1,j\neq k}^{d} (\cos x_k-\cos x_j)} \\
&&{}\sum_{l=0}^n \epsilon_{l}\cos(lx_{d+1}) \cos ((n-l+1/2)x_k),
\end{eqnarray*}
where $\epsilon_0:=1/2$ and $\epsilon_l:=1$, $l\geq 1$. Using
(\ref{e41}), we obtain
\begin{eqnarray*}
D_{d+1,n}^1(x)&=& -(-1)^{[(d-1)/2]}4 \sum_{k=1}^{d} \frac{\cos
(x_k/2) (\sin x_k)^{d-2}}
{\prod_{j=1,j\neq k}^{d+1} (\cos x_k-\cos x_j)} \\
&&{}\cos (x_k/2) \sin (x_{k}/2) \sin ((n+1/2)x_{k})\\
&&{} + (-1)^{[(d-1)/2]}4 \sum_{k=1}^{d} \frac{\cos (x_k/2) (\sin
x_k)^{d-2}}
{\prod_{j=1,j\neq k}^{d+1} (\cos x_k-\cos x_j)} \\
&&{}\cos (x_k/2) \sin (x_{d+1}/2) \sin ((n+1/2)x_{d+1})\\
&=& -(-1)^{[(d-1)/2]}2 \Big( \sum_{k=1}^{d} \frac{\cos (x_k/2) (\sin
x_k)^{d-1} \sin ((n+1/2)x_{k})}
{\prod_{j=1,j\neq k}^{d+1} (\cos x_k-\cos x_j)} \\
&&{} - \sin (x_{d+1}/2) \sin ((n+1/2)x_{d+1})\sum_{k=1}^{d}
\frac{(1+\cos x_k) (\sin x_k)^{d-2}} {\prod_{j=1,j\neq k}^{d+1}
(\cos x_k-\cos x_j)} \Big).
\end{eqnarray*}
If $d$ is even, then the function $h(t):=(1+t)(1-t^2)^{(d-2)/2}$ is
a polynomial of degree $d-1$. Then, by (\ref{e42}),
\begin{eqnarray*}
0&=&[\cos x_1,\ldots,\cos x_{d+1}]h \\
&=& \sum_{k=1}^{d} \frac{(1+\cos x_k) (\sin x_k)^{d-2}}
{\prod_{j=1,j\neq k}^{d+1} (\cos x_k-\cos x_j)} + \frac{(1+\cos
x_{d+1}) (\sin x_{d+1})^{d-2}} {\prod_{j=1,j\neq d+1}^{d+1} (\cos
x_{d+1}-\cos x_j)}.
\end{eqnarray*}
This implies
\begin{eqnarray*}
D_{d+1,n}^1(x) &=& -(-1)^{[(d-1)/2]}2 \Big( \sum_{k=1}^{d}
\frac{\cos (x_k/2) (\sin x_k)^{d-1} \sin ((n+1/2)x_{k})}
{\prod_{j=1,j\neq k}^{d+1} (\cos x_k-\cos x_j)} \\
&&{} + \sin (x_{d+1}/2) \sin ((n+1/2)x_{d+1}) \frac{(1+\cos x_{d+1})
(\sin x_{d+1})^{d-2}}
{\prod_{j=1,j\neq d+1}^{d+1} (\cos x_{d+1}-\cos x_j)} \Big)\\
&=& (-1)^{[d/2]}2 \sum_{k=1}^{d+1} \frac{\cos (x_k/2) (\sin
x_k)^{d-1} \sin ((n+1/2)x_{k})} {\prod_{j=1,j\neq k}^{d+1} (\cos
x_k-\cos x_j)},
\end{eqnarray*}
which proves the result if $d$ is even.

If $d$ is odd, the lemma can be proved similarly.
\end{proof}

We explicitly write out the result for $d=2$:
\begin{eqnarray}\label{e44}
D_{n}^1(x) &=& [\cos x_1,\cos x_2] G_n \n\\
&=&\frac{ [\cos x_1]G_n - [\cos x_2]G_n}{ \cos x_1-\cos x_2} \n\\
&=& 2 \frac{\cos (x_1/2) \cos ((n+1/2)x_1)-\cos (x_2/2) \cos
((n+1/2)x_2)}{\cos x_1-\cos x_2}.
\end{eqnarray}

The cubic \ind{Dirichlet kernel}s ($q=\infty$) are
$$
D_{n}^\infty(x) = \prod_{i=1}^{d} D_{n}^\infty(x_i ) =
\prod_{i=1}^{d} \frac{\sin((n+1/2)x_i)}{\sin(x_i/2)}.
$$

If $q=2$, then the continuous version of the Dirichlet kernel (see
Section \ref{s11})
$$
D_{t}^2(x) := \int_{\{\|v\|_2\leq t\}} \ee^{\ii x \cdot v} \dd v
$$
can be expressed as
$$
D_{t}^2(x) = \|x\|_2^{-d/2} t^{d/2} J_{d/2}(2\pi \|x\|_2t),
$$
where
$$
J_k(t):= \frac{(t/2)^k}{\sqrt{\pi}\,\Gamma(k+1/2)} \int_{-1}^{1}
\ee^{\ii ts} (1-s^2)^{k-1/2}\dd s \qquad (k>-1/2,
t>0)\index{\file-1}{$J_k$}
$$
are the \dword{Bessel functions} (see Subsection \ref{s6.2.4}).

\sect{Norm convergence of the $\ell_q$-summability means}\label{s6}

We introduce the \idword{reflection} and \idword{translation}
operators by
$$
\tilde f(x):=f(-x), \qquad
T_xf(t):=f(t-x).\index{\file-1}{$T_x$}\index{\file-1}{$\tilde f$}
$$
A Banach space $B$ consisting of measurable functions on $\T^d$ is
called a \idword{homogeneous Banach space} if

\begin{enumerate}
\item []
\begin{enumerate}
\item  for all $f\in B$ and $x\in \T^d$, $T_xf\in B$ and $\|T_xf\|_B=\|f\|_B$,
\item the function $x\mapsto T_xf$ from $\T^d$ to $B$ is continuous for all $f\in B$,
\item $\|f\|_1\leq C\|f\|_B$ for all $f\in B$.
\end{enumerate}
\end{enumerate}
For an introduction to homogeneous Banach spaces, see Katznelson
\cite{katz}. It is easy to see that the spaces $L_p(\T^d)$ $(1\leq
p<\infty)$, $C(\T^d)$, the Lorentz spaces $L_{p,q}(\T^d)$ $(1<
p<\infty,1\leq q<\infty)$ and the Hardy space $H_1(\T^d)$ are
homogeneous Banach spaces. The definition of the Fej{\'e}r and Riesz
means can be extended to distributions. Note first of all that
\begin{equation}\label{e24}
\sigma_{n}^{q,\alpha} f:= f * K_{n}^{q,\alpha} \qquad (\nn),
\end{equation}
where $*$ denotes \idword{convolution}, i.e., if $f,g\in L_1(\R^d)$,
then
$$
f*g(x):=\int_{\T^d} f(t)g(x-t)\dd t=\int_{\T^d} f(x-t)g(t)\dd
t.\index{\file-1}{$f*g$}
$$
Obviously, the convolution is well defined for all $g\in L_1(\T^d)$
and $f\in L_p(\T^d)$ $(1 \leq p\leq \infty)$ or $f\in B$, where $B$
is a homogeneous Banach space. The convolution can be extended to
distributions as follows. It is easy to see that
$$
\int_{\T^d} (f*g)(x)h(x)\dd x=\int_{\T^d} f(t) (\tilde g*h)(t)\dd t
$$
for all $f,g,h\in \cS(\T^d)$. For a distribution $u\in \cS'(\T^d)$
and $g\in \cS(\T^d)$, let us define the convolution $u*g$ by
\begin{equation}\label{e25}
u*g(h):= u(\tilde g*h) \qquad (h\in \cS(\T^d)).
\end{equation}
It is easy to see that $u*g$ is indeed a distribution. Moreover,
\begin{eqnarray*}
u*g(h)=u(\tilde g*h)&=&u\Big(\int_{\T^d} \tilde g(\cdot -x)h(x)\dd x \Big) \\
&=&u\Big(\int_{\T^d} T_x\tilde g(\cdot)h(x)\dd x \Big).
\end{eqnarray*}
The Riemann sums of the last integral are easily shown to converge
in the topology of $\cS(\T^d)$. Since $u$ is continuous,
$$
u*g(h) = \int_{\T^d} u(T_x\tilde g)h(x)\dd x.
$$
Thus the distribution $u*g$ equals the function $x\mapsto
u(T_x\tilde g)$ $(g\in \cS(\T^d))$. One can show that this function
is a $C^\infty$ function.

We can see easily that
$$
\int_{\T^d} (f*g)(x)h(x)\dd x=\int_{\T^d} (f*\tilde h)(t) \tilde
g(t)\dd t
$$
for all $f,g,h\in \cS(\T^d)$. If $g\in L_1(\T^d)$, then we define
$u*g$ by
\begin{equation}\label{e26}
u*g(h):= \langle u*\tilde h, \tilde g\rangle =\int_{\T^d} (u*\tilde
h)(x) \tilde g(x)\dd x \qquad (h\in \cS(\T^d)).
\end{equation}
The last integral is well defined, because $u*\tilde h\in
L_\infty(\T^d)$ and $g\in L_1(\T^d)$. We can show that $u*g$ is a
distribution if $u\in H_p^\Box(\T^d)$ (for the definition of the
Hardy space $H_p^\Box(\T^d)$ see Section \ref{s7}). Moreover, if
$\lim_{k\to\infty} u_k= u$ in the $H_p^\Box$-norm, then
$\lim_{k\to\infty} u_k*g=u*g$ in the sense of distributions. For
more details, see Stein \cite{st1}. Consequently, since
$K_{n}^{q,\alpha}$ is integrable, we obtain that
$\sigma_n^{q,\alpha} f$ is well defined in (\ref{e24}) for all
distributions $f\in \cS'(\T^d)$.

It was proved in Berens, Li and Xu \cite{bexu2}, Oswald \cite{os3}
and Weisz \cite{wel1-fs2,wmar6} for $q=1,\infty$ and in Bochner
\cite{boch} (see also Stein and Weiss \cite{stwe}) for $q=2$ that
the $L_1$-norms of the Riesz kernels are uniformly bounded.
Moreover, this theorem was proved by Li and Xu \cite{lixu} for
Jacobi polynomials.

\begin{thm}\label{t12}
If $q=1,\infty$ and $\alpha>0$, then
$$
\int_{\T^d} |K_n^{q,\alpha}(x)| \dd x \leq C \qquad (\nn).
$$
If $q=2$, then the same holds for $\alpha>(d-1)/2$.
\end{thm}

We will prove this theorem in Subsection \ref{s6.2}. This implies
easily

\begin{thm}\label{t13}
If $q=1,\infty$, $\alpha>0$ and $B$ is a homogeneous Banach space on
$\T^d$, then
$$
\|\sigma_n^{q,\alpha} f\|_B \leq C \|f\|_B \qquad (\nn)
$$
and
$$
\lim_{n\to\infty} \sigma_n^{q,\alpha} f=f \qquad \mbox{in the
$B$-norm for all $f\in B$}.
$$
If $q=2$, then the same holds for $\alpha>(d-1)/2$.
\end{thm}

\begin{proof}
Observe that
$$
\|\sigma_n^{q,\alpha}f\|_B \leq \frac{1}{(2\pi)^d}\int_{\T^d}
\|f(\cdot -u)\|_B K_n^{q,\alpha}(u) \dd u =
\frac{1}{(2\pi)^d}\int_{\T^d} \|f\|_B K_n^{q,\alpha}(u) \dd u.
$$
Since the trigonometric polynomials are dense in $B$ (see Katznelson
\cite{{katz}}), the theorem follows from Theorem \ref{t12}.
\end{proof}

Originally the theorem was proved in the case $q=1,\infty$ only for
$1\leq \alpha<\infty$. However, the analogous result for Fourier
transforms holds for every $\alpha$. Now Theorem \ref{t13} follows
from the transference theorem (see Grafakos
\cite[pp.~220-226]{gra}).

Since the $L_p(\T^d)$ $(1\leq p<\infty)$ spaces are homogeneous
Banach spaces, Theorem \ref{t13} holds for these spaces, too. The
situation is more complicated and not completely solved if $q=2$ and
$\alpha\leq (d-1)/2$. It is clear by the Banach-Steinhaus theorem
that $\lim_{n\to\infty} \sigma_n^{q,\alpha} f=f$ in the $L_p$-norm
for all $f\in L_p(\T^d)$ if and only if the operators
$\sigma_n^{q,\alpha}$ are uniformly bounded from $L_p(\T^d)$ to
$L_p(\T^d)$. We note that each operator $\sigma_n^{q,\alpha}$ is
bounded on $L_p(\T^d)$, because $K_n^{q,\alpha}\in L_1(\T^d)$.

\subsection{Further results for the Bochner-Riesz means}

The following results are all proved in the book of Grafakos
\cite[Chapter 10]{gra}, so we do not prove them here. Figures
\ref{f2}--\ref{f40} show the regions where the operators
$\sigma_n^{2,\alpha}$ are uniformly bounded or unbounded.

\begin{thm}\label{t14}
Suppose that $d\geq 2$ and $q=\gamma=2$. If $0\leq \alpha\leq
(d-1)/2$ and
$$
p\leq \frac{2d}{d+1+2\alpha} \qquad \mbox{or} \qquad p\geq
\frac{2d}{d-1-2\alpha},
$$
then the Bochner-Riesz operators $\sigma_n^{2,\alpha}$ are not
uniformly bounded on $L_p(\T^d)$ (see Figure \ref{f2}).
\end{thm}

\begin{figure}[h] 
   \centering
   \includegraphics[width=1\textwidth]{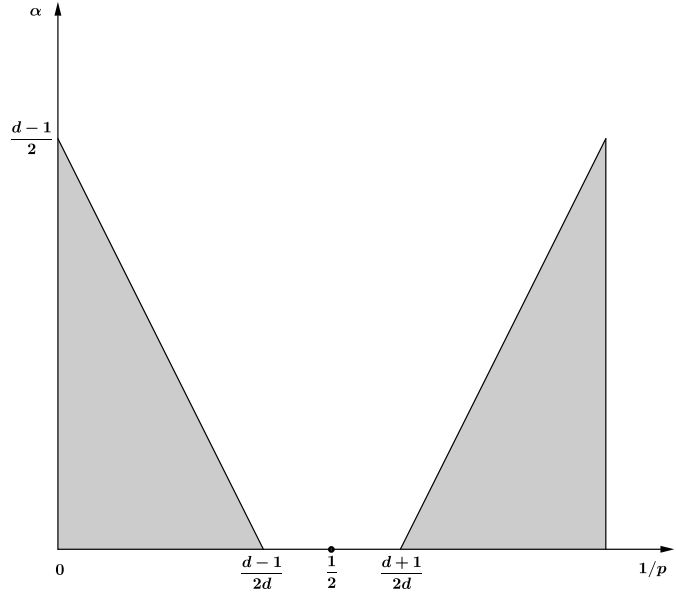}
   \caption{Uniform unboundedness of $\sigma_n^{2,\alpha}$.}
   \label{f2}
\end{figure}

The following result about the uniform boundedness of
$\sigma_n^{2,\alpha}$ was proved by Stein \cite[p.~ 276]{stwe}.

\begin{thm}\label{t260}
Suppose that $d\geq 2$ and $q=\gamma=2$. If $0<\alpha\leq
\frac{d-1}{2}$ and
$$
\frac{2(d-1)}{d-1+2\alpha}<p<\frac{2(d-1)}{d-1-2\alpha},
$$
then the Bochner-Riesz operators $\sigma_n^{2,\alpha}$ are uniformly
bounded on $L_p(\T^d)$ (see Figure \ref{f41}).
\end{thm}

\begin{figure}[h] 
   \centering
   \includegraphics[width=1\textwidth]{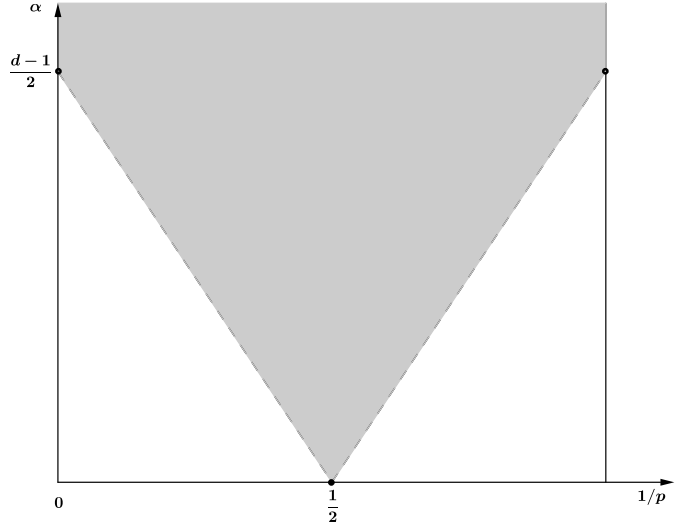}
   \caption{Uniform boundedness of $\sigma_n^{2,\alpha}$ when $d\geq 3$.}
   \label{f41}
\end{figure}

Carleson and Sj{\"o}lin \cite{casj} solved completely the uniform
boundedness of the Bochner-Riesz operators for $d=2$. They are
uniformly bounded for $p$'s which are excluded from Theorem
\ref{t14} (other proofs were given by Fefferman \cite{cfe4} and
H{\"o}rmander \cite{ho2}).

\begin{thm}\label{t15}
Suppose that $d=2$ and $q=\gamma=2$. If $0<\alpha\leq 1/2$ and
$$
\frac{4}{3+2\alpha}<p<\frac{4}{1-2\alpha},
$$
then the Bochner-Riesz operators $\sigma_n^{2,\alpha}$ are uniformly
bounded on $L_p(\T^d)$ (see Figure \ref{f3}).
\end{thm}

\begin{figure}[h] 
   \centering
   \includegraphics[width=1\textwidth]{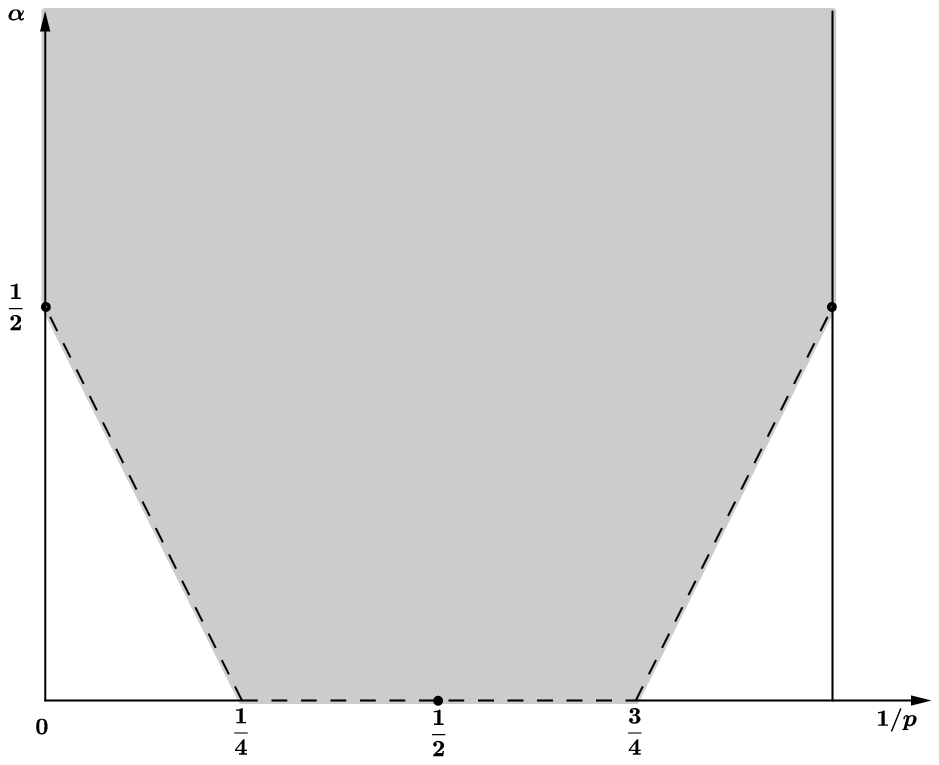}
   \caption{Uniform boundedness of $\sigma_n^{2,\alpha}$ when $d=2$.}
   \label{f3}
\end{figure}

Fefferman \cite{cfe5} generalized this result to higher dimensions.

\begin{thm}\label{t16}
Suppose that $d\geq 3$ and $q=\gamma=2$. If $\frac{d-1}{2(d+1)}\leq
\alpha\leq \frac{d-1}{2}$ and
$$
\frac{2d}{d+1+2\alpha}<p<\frac{2d}{d-1-2\alpha},
$$
then the Bochner-Riesz operators $\sigma_n^{2,\alpha}$ are uniformly
bounded on $L_p(\T^d)$ (see Figure \ref{f4}).
\end{thm}

Combining Theorems \ref{t16} and \ref{t260} and using analytic
interpolation (see e.g.~Stein and Weiss \cite[p.~276,
p.~205]{stwe}), we obtain

\begin{thm}\label{t160}
Suppose that $d\geq 3$ and $q=\gamma=2$. If
$0<\alpha<\frac{d-1}{2(d+1)}$ and
$$
\frac{2(d-1)}{d-1+4\alpha}<p<\frac{2(d-1)}{d-1-4\alpha},
$$
then the Bochner-Riesz operators $\sigma_n^{2,\alpha}$ are uniformly
bounded on $L_p(\T^d)$ (see Figure \ref{f4}).
\end{thm}

\begin{figure}[h] 
   \centering
   \includegraphics[width=1\textwidth]{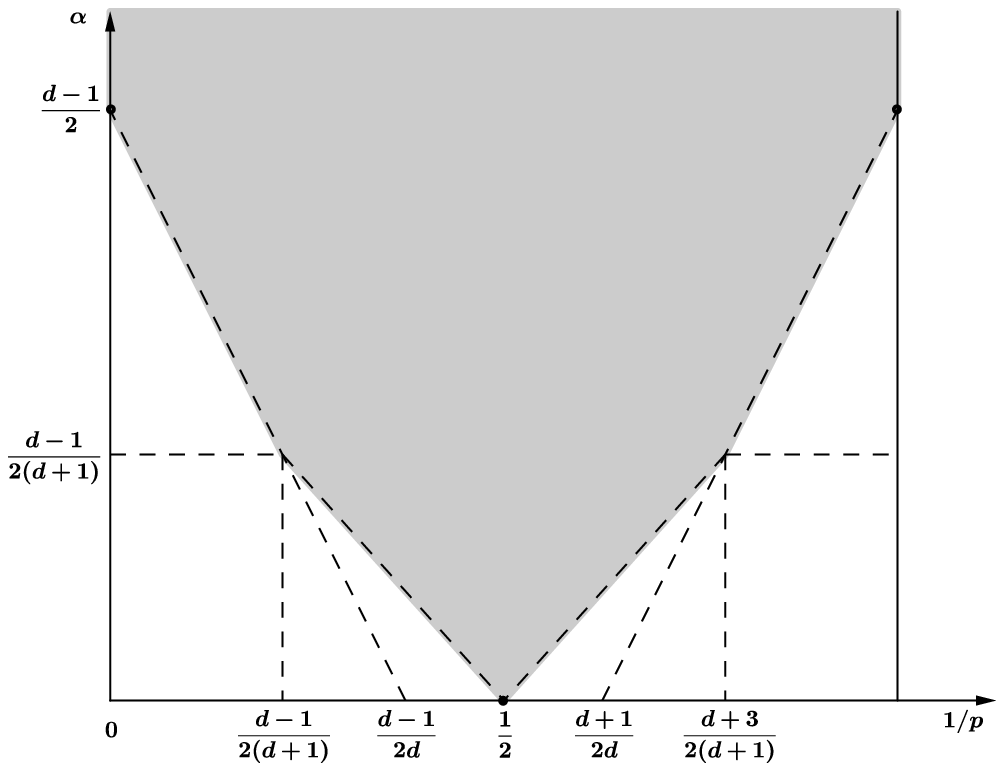}
   \caption{Uniform boundedness of $\sigma_n^{2,\alpha}$ when $d\geq 3$.}
   \label{f4}
\end{figure}

It is still an open question as to whether the operators
$\sigma_n^{2,\alpha}$ are uniformly bounded or unbounded in the
region of Figure \ref{f40}.

\begin{figure}[h] 
   \centering
   \includegraphics[width=1\textwidth]{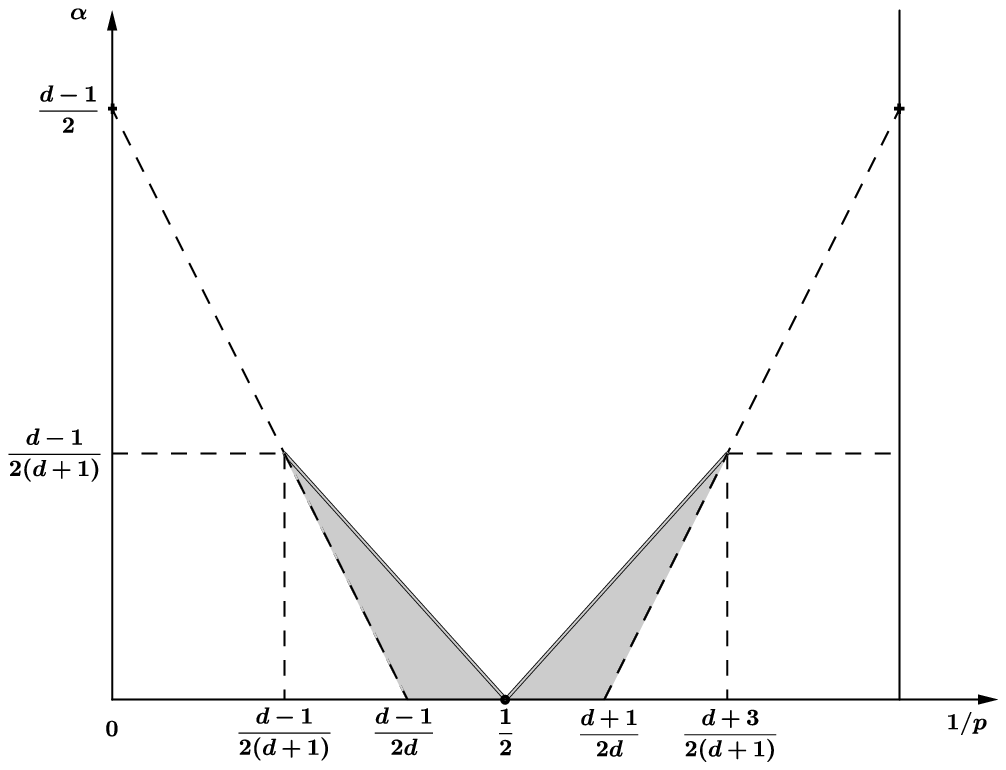}
   \caption{Open question of the uniform boundedness of $\sigma_n^{2,\alpha}$ when $d\geq 3$.}
   \label{f40}
\end{figure}

\subsection{Proof of Theorem \ref{t12}}\label{s6.2}

In this section, we will prove Theorem \ref{t12}. Since the kernel
functions and hence the proofs are very different for different
$q$'s, we prove the theorem in four subsections. If $q=1,\infty$,
then we suppose here that $\alpha=\gamma=1$. For other parameters,
see Subsection \ref{s10.1}. For $q=1$, we prove the theorem
separately for $d=2$ and $d\geq 3$.

\subsubsection{Proof for $q=1$ in the two-dimensional case}

In this section, we write $(x,y)$ instead of the two-dimensional
vector $x$. Recall from (\ref{e44}) that

\begin{eqnarray}\label{e61.13}
D_k^1(x,y) &=& 2 \frac{\cos (x/2) \cos ((k+1/2)x)-\cos (y/2) \cos ((k+1/2)y)}{\cos x-\cos y}\n\\
&=& -\frac{\cos (x/2) \cos ((k+1/2)x)-\cos (y/2) \cos
((k+1/2)y)}{\sin((x-y)/2) \sin((x+y)/2)}.
\end{eqnarray}

In what follows, we may suppose that $\pi>x>y>0$. We denote the
\idword{characteristic function} of a set $H$ by \inda{$1_H$}, i.e.,
$$
1_H(x) := \left\{
               \begin{array}{ll}
                 1, & \hbox{if $x\in H$;} \\
                 0, & \hbox{if $x\notin H$.}
               \end{array}
             \right.
$$

\begin{lem}\label{l61.1}
For $0<\beta<1$,
\begin{eqnarray}
\label{e61.2}  |K_{n}^1(x,y)| &\leq& C (x-y)^{-3/2}
\big(y^{-1/2}1_{\{y\leq \pi/2\}}+
(\pi-x)^{-1/2}1_{\{y>\pi/2\}} \big),\\
\label{e61.3}  &\leq& C n^{-1} (x-y)^{-1-\beta}
\big(y^{\beta-2}1_{\{y\leq \pi/2\}}+
(\pi-x)^{\beta-2}1_{\{y>\pi/2\}} \big),\\
\label{e61.4} &\leq& C y^{-2}1_{\{y\leq \pi/2\}}+C
(\pi-x)^{-2}1_{\{y>\pi/2\}}.
\end{eqnarray}
\end{lem}

\begin{proof}
In (\ref{e61.13}), we use that
$$
\sin (x\pm y)/2\sim x\pm y \quad \mbox{if} \quad y\leq \pi/2
$$
and
$$
\sin (x-y)/2\sim x-y,\quad \sin (x+y)/2\sim 2\pi-x-y \quad \mbox{if}
\quad y>\pi/2.
$$
The facts $x+y>x-y$, $x+y>y$ and $2\pi-x-y>x-y$, $2\pi-x-y>\pi-x$
imply (\ref{e61.2}). Using (\ref{e61.13}) and the formulas
\begin{equation}\label{e61.5}
\sum_{k=0}^{n-1} \cos(k+1/2)t= \frac{\sin(nt)}{2\sin(t/2)}, \qquad
\sum_{k=0}^{n-1} \sin(k+1/2)t= \frac{1-\cos(nt)}{2\sin(t/2)},
\end{equation}
we conclude
$$
|K_n^1(x,y)| \leq  C n^{-1} (x-y)^{-1}(x+y)^{-1}y^{-1}\leq C n^{-1}
(x-y)^{-1-\beta} y^{\beta-2}
$$
if $y\leq \pi/2$, which is exactly (\ref{e61.3}). The inequality for
$y>\pi/2$ can be proved in the same way.

Lagrange's mean value theorem and (\ref{e61.13}) imply that there
exists $x>\xi>y$ such that
$$
D_k^1(x,y) = -\frac{H_k'(\xi) (x-y)}{\sin((x-y)/2) \sin((x+y)/2)},
$$
where
$$
H_k(t):= \cos (t/2) \cos ((k+1/2)t).
$$
Then
$$
|K_n^1(x,y)| \leq  C n^{-1}(x-y) (n+1)
(x-y)^{-1}(x+y)^{-1}y^{-1}\leq  C y^{-2}
$$
shows (\ref{e61.4}) if $y\leq \pi/2$. The case $y>\pi/2$ is similar.
\end{proof}

In the next lemma, we estimate the partial derivatives of the kernel
function.

\begin{lem}\label{l61.6}
If $0<\beta<1$, then for $j=1,2$,
\begin{equation}\label{e61.6}
|\partial_jK_{n}^1(x,y)| \leq C
(x-y)^{-1-\beta}\big(y^{\beta-2}1_{\{y\leq \pi/2\}}+
(\pi-x)^{\beta-2}1_{\{y>\pi/2\}} \big).
\end{equation}
\end{lem}

\begin{proof}
By Lagrange's mean value theorem and (\ref{e61.13}),
\begin{eqnarray*}
\partial_1D_k^1(x,y) &=& \frac{1}{2}(\sin (x/2) \cos ((k+1/2)x)+\cos (x/2) (2k+1)\sin ((k+1/2)x))\\
&&{}
\sin((x-y)/2)^{-1} \sin((x+y)/2)^{-1} \\
&&{}+\frac{1}{2}H_k'(\xi) (x-y)(\sin((x-y)/2)^{-2} \sin((x+y)/2)^{-1}\cos((x-y)/2)\\
&&{}+\sin((x-y)/2)^{-1} \sin((x+y)/2)^{-2}\cos((x+y)/2)),
\end{eqnarray*}
where $y<\xi<x$ is a suitable number. Using the methods above,
$$
|\partial_1K_{n}^1(x,y)| \leq C (x-y)^{-1}(x+y)^{-1}y^{-1}+C
(x+y)^{-2}y^{-1} \leq C (x-y)^{-1-\beta}y^{\beta-2},
$$
which proves (\ref{e61.6}) if $y\leq \pi/2$. The case $y>\pi/2$ can
be shown similarly.
\end{proof}

Now we are ready to prove that the $L_1$-norm of the kernel
functions are uniformly bounded.

\begin{proof*}{Theorem \ref{t12} for $q=1$ and $d=2$}
It is enough to integrate the kernel function over the set $\{(x,y):
0<y<x<\pi\}$. Let us decompose this set into the union
$\cup_{i=1}^{10} A_i$, where
\begin{eqnarray*}
A_1&:=&\{(x,y): 0<x\leq 2/n, 0<y<x<\pi,y\leq \pi/2\}, \\
A_2&:=&\{(x,y): 2/n<x<\pi, 0<y\leq 1/n,y\leq \pi/2\}, \\
A_3&:=&\{(x,y): 2/n<x<\pi, 1/n<y\leq x/2,y\leq \pi/2\}, \\
A_4&:=&\{(x,y): 2/n<x<\pi, x/2<y\leq x-1/n,y\leq \pi/2\}, \\
A_5&:=&\{(x,y): 2/n<x<\pi, x-1/n<y<x,y\leq \pi/2\}\\
A_6&:=&\{(x,y): y>\pi/2,\pi-2/n\leq y<\pi, 0<y<x<\pi\}, \\
A_7&:=&\{(x,y): \pi/2<y<\pi-2/n, \pi-1/n<x<\pi\}, \\
A_8&:=&\{(x,y): \pi/2<y<\pi-2/n, (\pi+y)/2<x\leq \pi-1/n\}, \\
A_9&:=&\{(x,y): \pi/2<y<\pi-2/n, y+1/n<x\leq (\pi+y)/2\}, \\
A_{10}&:=&\{(x,y): \pi/2<y<\pi-2/n, y<x\leq y+1/n\}.
\end{eqnarray*}
The sets $A_i$ can be seen on Figure \ref{f11}.

\begin{figure}[htbp] 
   \centering
   \includegraphics[width=1\textwidth]{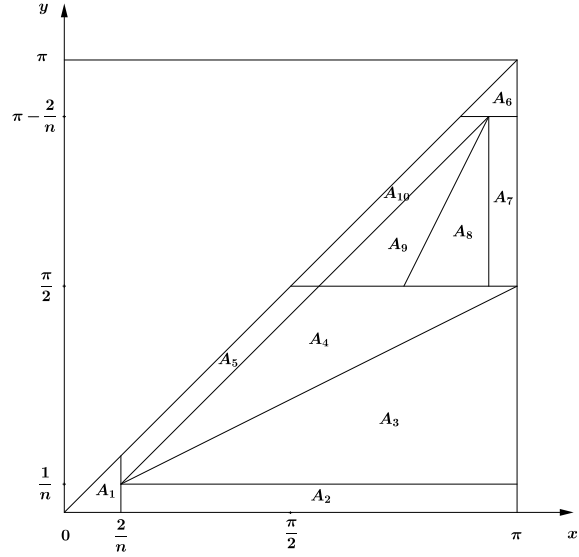}
   \caption{The sets $A_i$.}
   \label{f11}
\end{figure}

Inequality (\ref{e61.1}) implies
$$
\int_{A_1\cup A_6} |K_n^1(x,y)| \dd x\dd y \leq C.
$$
By (\ref{e61.2}),
\begin{eqnarray*}
\int_{A_2\cup A_7} |K_{n}^1(x,y)| \dd x\dd y &\leq& C \int_{2/n}^\pi
\int_0^{1/n} (x-1/n)^{-3/2} y^{-1/2} \dd x \dd y \\
&&{}+ \int_{\pi-1/n}^\pi
\int_{\pi/2}^{\pi-2/n} (\pi-1/n-y)^{-3/2} (\pi-x)^{-1/2} \dd x \dd y \\
&\leq& C.
\end{eqnarray*}
Since $x-y>x/2$ on the set $A_3$ and $x-y>(\pi-y)/2$ on the set
$A_8$, we get from (\ref{e61.3}) that
\begin{eqnarray*}
\int_{A_3\cup A_8} |K_{n}^1(x,y)| \dd x\dd y &\leq& C n^{-1}
\int_{2/n}^\pi
\int_{1/n}^{x/2} x^{-1-\beta} y^{\beta-2} \dd x \dd y \\
&&{}+C n^{-1} \int_{\pi/2}^{\pi-2/n} \int_{(\pi+y)/2}^{\pi-1/n}
(\pi-y)^{-1-\beta}
(\pi-x)^{\beta-2} \dd x \dd y \\
&\leq& C.
\end{eqnarray*}
Observe that $y>x/2$ on $A_4$ and $\pi-x>(\pi-y)/2$ on the set
$A_9$, hence (\ref{e61.3}) implies
\begin{eqnarray*}
\int_{A_4\cup A_9} |K_{n}^1(x,y)| \dd x\dd y &\leq& C n^{-1}
\int_{2/n}^\pi
\int_{x/2}^{x-1/n} (x-y)^{-1-\beta} x^{\beta-2} \dd y \dd x \\
&&{}+ C n^{-1} \int_{\pi/2}^{\pi-2/n} \int_{y+1/n}^{(\pi+y)/2}
(x-y)^{-1-\beta}
(\pi-y)^{\beta-2} \dd x \dd y\\
&\leq& C.
\end{eqnarray*}
Finally, by (\ref{e61.4}),
\begin{eqnarray*}
\lefteqn{\int_{A_5\cup A_{10}} |K_{n}^1(x,y)| \dd x\dd y } \n\\
&\leq& C \int_{1/n}^\pi \int_{y}^{y+1/n} y^{-2} \dd x\dd y+C
\int_{\pi/2}^{\pi-1/n} \int_{x-1/n}^{x}(\pi-x)^{-2}
\dd y\dd x\\
&\leq& C,
\end{eqnarray*}
which completes the proof of the theorem.
\end{proof*}

\subsubsection{Proof for $q=1$ in higher dimensions ($d\geq 3$)}

We also need another representation of the kernel function $D_n^1$.
If we apply the inductive definition of the divided difference in
(\ref{e47}) to $D_n^1$, then in the denominator, we have to choose
the factors from the following table:
$$
\begin{array}{cccccccccc}
\cos x_1-\cos x_d       &                &           &         &          &   &             &     &   \\
\cos x_1-\cos x_{d-1}   & \cos x_2-\cos x_d        &           &         &          &   &             &     &   \\
\ldots        &                &           &         &          &   &             &     &   \\
\cos x_1-\cos x_{d-k+1} & \cos x_2-\cos x_{d-k+2}  & 
& \ldots  & \cos x_k-\cos x_d  &   &             &     &   \\
\ldots        &                &           &         &          &   &             &     &   \\
\cos x_1-\cos x_{2}   & \cos x_2-\cos x_{3}      & 
& \ldots  &          &   \cos x_{d-1}-\cos x_d. &     &  & \\
\end{array}
$$
Observe that the $k$th row contains $k$ terms and the differences of
the indices in the $k$th row is equal to $d-k$, more precisely, if
$\cos x_{i_k}-\cos x_{j_k}$ is in the $k$th row, then $j_k-i_k=d-k$.
We choose exactly one factor from each row. First, we choose $\cos
x_1-\cos x_d$ and then from the second row $\cos x_1-\cos x_{d-1}$
or $\cos x_2-\cos x_d$. If we have chosen the $(k-1)$th factor from
the $(k-1)$th row, say $\cos x_j-\cos x_{j+d-k+1}$, then we have to
choose the next one from the $k$th row as either the one below the
$(k-1)$th factor (it is equal to $\cos x_j-\cos x_{j+d-k}$) or its
right neighbor (it is equal to $\cos x_{j+1}-\cos x_{j+d-k+1}$).

\begin{dfn}\label{d62.1}
If the sequence of integer pairs $((i_n,j_n):n=1,\ldots,d-1)$ has
the following properties, then we say that it is in $\cI$. Let
$i_1=1$, $j_1=d$, $(i_n)$ is non-decreasing and $(j_n)$ is
non-increasing. If $(i_n,j_n)$ is given, then let $i_{n+1}=i_n$ and
$j_{n+1}=j_n-1$ or $i_{n+1}=i_n+1$ and $j_{n+1}=j_n$.
\end{dfn}

Observe that the difference $\cos x_{i_k}-\cos x_{j_k}$ is in the
$k$th row of the table ($k=1,\ldots,d-1$). So the factors we have
just chosen can be written as $\prod_{l=1}^{d-1}(\cos x_{i_l}-\cos
x_{j_l})$. In other words,
\begin{eqnarray}\label{e62.7}
\lefteqn{D_n^1(x) } \n\\&=& \sum_{(i_l,j_l)\in \cI} (-1)^{i_{d-1}-1}
\prod_{l=1}^{d-2}(\cos x_{i_l}-\cos x_{j_l})^{-1}
[\cos x_{i_{d-1}},\cos x_{j_{d-1}}]G_n\n\\
&=& \sum_{(i_l,j_l)\in \cI}(-1)^{i_{d-1}-1} \nonumber \\
&&{}\prod_{l=1}^{d-1}(\cos x_{i_l}-\cos x_{j_l})^{-1}
(G_n(\cos x_{i_{d-1}})-G_n(\cos x_{j_{d-1}}))\\
&=:& \sum_{(i_l,j_l)\in \cI}
D^1_{n,(i_l,j_l)}(x).\n\index{\file-1}{$D^1_{n,(i_l,j_l)}$}
\end{eqnarray}
This proves

\begin{lem}\label{l62.10}
We have
\begin{eqnarray*}
\lefteqn{K_n^1(x) } \n\\
&=& \sum_{(i_l,j_l)\in \cI} \frac{(-1)^{i_{d-1}-1}}{n}
\prod_{l=1}^{d-1}(\cos x_{i_l}-\cos x_{j_l})^{-1}
\sum_{k=0}^{n-1} (G_k(\cos x_{i_{d-1}})-G_k(\cos x_{j_{d-1}}))\n\\
&=:& \sum_{(i_l,j_l)\in \cI}
K_{n,(i_l,j_l)}^1(x).\index{\file-1}{$K_{n,(i_l,j_l)}^1$}
\end{eqnarray*}
\end{lem}

We may suppose that $\pi>x_1>x_2>\cdots>x_d>0$. We will need the
following sharp estimations of the kernel functions.

\begin{lem}\label{l62.1}
For all $0<\beta<\frac{2}{d-1}$,
\begin{eqnarray}\label{e62.8}
|K_{n,(i_l,j_l)}^1(x)| &\leq& \frac{C}{n}
\prod_{l=1}^{d-1}(x_{i_l}-x_{j_l})^{-1-\beta}
x_{j_{d-1}}^{\beta(d-1)-2} 1_{\{x_{j_{d-1}}\leq \pi/2\}} \n\\
&&{}+\frac{C}{n} \prod_{l=1}^{d-1}(x_{i_l}-x_{j_l})^{-1-\beta}
(\pi-x_{i_{d-1}})^{\beta(d-1)-2} 1_{\{x_{j_{d-1}}> \pi/2\}}.
\end{eqnarray}
\end{lem}

\begin{proof}
Using the formulas in (\ref{e61.5}), we conclude
$$
|K^1_{n,(i_l,j_l)}(x)| \leq \prod_{l=1}^{d-1} \frac{(\sin
x_{i_{d-1}})^{d-2} (\sin (x_{i_{d-1}}/2))^{-1} + (\sin
x_{j_{d-1}})^{d-2} (\sin (x_{j_{d-1}}/2))^{-1}}
{n\sin((x_{i_l}-x_{j_l})/2)\sin((x_{i_l}+x_{j_l})/2)}.
$$
If $x_{j_{d-1}}\leq \pi/2$, then $(x_{i_l}+x_{j_l})/2\leq 3\pi/4$
and so
$$
|K^1_{n,(i_l,j_l)}(x)| \leq \frac{C}{n}
\prod_{l=1}^{d-1}(x_{i_l}-x_{j_l})^{-1}(x_{i_l}+x_{j_l})^{-1}
(x_{i_{d-1}}^{d-3} + x_{j_{d-1}}^{d-3}).
$$
Since $x_{i_l}+x_{j_l}>x_{i_l}-x_{j_l}$ and
$x_{i_l}+x_{j_l}>x_{i_{d-1}}>x_{j_{d-1}}$, we can see that
\begin{eqnarray*}
|K^1_{n,(i_l,j_l)}(x)| &\leq& \frac{C}{n}
\prod_{l=1}^{d-1}(x_{i_l}-x_{j_l})^{-1-\beta}
(x_{i_{d-1}}^{d-3+(\beta-1)(d-1)} + x_{j_{d-1}}^{d-3+(\beta-1)(d-1)}) \n\\
&\leq & \frac{C}{n} \prod_{l=1}^{d-1}(x_{i_l}-x_{j_l})^{-1-\beta}
x_{j_{d-1}}^{\beta(d-1)-2} \n
\end{eqnarray*}
for all $0<\beta<\frac{2}{d-1}$.

If $x_{j_{d-1}}>\pi/2$, then $(x_{i_l}+x_{j_l})/2> \pi/4$ and
$$
|K^1_{n,(i_l,j_l)}(x)| \leq \frac{C}{n}
\prod_{l=1}^{d-1}(x_{i_l}-x_{j_l})^{-1}(2\pi-x_{i_l}-x_{j_l})^{-1}
((\pi-x_{i_{d-1}})^{d-3} + (\pi-x_{j_{d-1}})^{d-3}).
$$
Observe that $2\pi-x_{i_l}-x_{j_l}>x_{i_l}-x_{j_l}$ and
$2\pi-x_{i_l}-x_{j_l}>\pi-x_{j_l}
>\pi-x_{j_{d-1}}>\pi-x_{i_{d-1}}$. Thus
\begin{eqnarray*}
|K^1_{n,(i_l,j_l)}(x)| &\leq& \frac{C}{n} \prod_{l=1}^{d-1}(x_{i_l}-x_{j_l})^{-1-\beta}\\
&&{}
((\pi-x_{i_{d-1}})^{d-3+(\beta-1)(d-1)} + (\pi-x_{j_{d-1}})^{d-3+(\beta-1)(d-1)}) \n\\
&\leq & \frac{C}{n} \prod_{l=1}^{d-1}(x_{i_l}-x_{j_l})^{-1-\beta}
(\pi-x_{i_{d-1}})^{\beta(d-1)-2}\n
\end{eqnarray*}
if $0<\beta<\frac{2}{d-1}$.
\end{proof}

\begin{lem}\label{l62.2}
For all $0<\beta<\frac{2}{d-2}$,
\begin{eqnarray}\label{e62.12}
|K^1_{n,(i_l,j_l)}(x)| &\leq& C
\prod_{l=1}^{d-2}(x_{i_l}-x_{j_l})^{-1-\beta}
x_{j_{d-1}}^{\beta(d-2)-2}1_{\{x_{j_{d-1}}\leq \pi/2\}} \n\\
&&{}+C \prod_{l=1}^{d-2}(x_{i_l}-x_{j_l})^{-1-\beta}
(\pi-x_{i_{d-1}})^{\beta(d-2)-2} 1_{\{x_{j_{d-1}}> \pi/2\}}.
\end{eqnarray}
\end{lem}

\begin{proof}
Lagrange's mean value theorem and (\ref{e62.7}) imply that there
exists $x_{i_{d-1}}>\xi>x_{j_{d-1}}$, such that
$$
D^1_{k,(i_l,j_l)}(x) = (-1)^{i_{d-1}-1} \prod_{l=1}^{d-1}(\cos
x_{i_l}-\cos x_{j_l})^{-1} H_k'(\xi) (x_{i_{d-1}}-x_{j_{d-1}}),
$$
where
$$
H_k(t):=(-1)^{[(d-1)/2]}2 \cos (t/2) (\sin t)^{d-2} \soc (k+1/2)t.
$$
Then
\begin{eqnarray*}
|K^1_{n,(i_l,j_l)}(x)| &\leq& C \prod_{l=1}^{d-1} \frac{(\sin
\xi)^{d-2} + n(\sin \xi)^{d-2}}
{n\sin((x_{i_l}-x_{j_l})/2)\sin((x_{i_l}+x_{j_l})/2)\sin (\xi/2)}(x_{i_{d-1}}-x_{j_{d-1}})\\
&&{}+ C \prod_{l=1}^{d-1} \frac{(\sin \xi)^{d-3}}
{\sin((x_{i_l}-x_{j_l})/2)\sin((x_{i_l}+x_{j_l})/2)}(x_{i_{d-1}}-x_{j_{d-1}}).
\end{eqnarray*}
Beside (\ref{e61.5}), we have used that $|\sum_{k=0}^{n-1}
\soc(k+1/2)t|\leq n$. In the case $x_{j_{d-1}}\leq \pi/2$,
\begin{eqnarray*}
|K^1_{n,(i_l,j_l)}(x)| &\leq& C
\prod_{l=1}^{d-1}(x_{i_l}-x_{j_l})^{-1}(x_{i_l}+x_{j_l})^{-1}
(x_{i_{d-1}}-x_{j_{d-1}}) \xi^{d-3}\\
&\leq& C \prod_{l=1}^{d-2}(x_{i_l}-x_{j_l})^{-1-\beta}
\xi^{d-4+(\beta-1)(d-2)} \n\\
&\leq & C \prod_{l=1}^{d-2}(x_{i_l}-x_{j_l})^{-1-\beta}
x_{j_{d-1}}^{\beta(d-2)-2}\n
\end{eqnarray*}
for all $0<\beta<\frac{2}{d-2}$.

Similarly, if $x_{j_{d-1}}>\pi/2$, then $(x_{i_l}+x_{j_l})/2> \pi/4$
and
\begin{eqnarray*}
|K^1_{n,(i_l,j_l)}(x)| &\leq& C
\prod_{l=1}^{d-1}(x_{i_l}-x_{j_l})^{-1}(2\pi-x_{i_l}-x_{j_l})^{-1}
(x_{i_{d-1}}-x_{j_{d-1}}) (\pi-\xi)^{d-3}\\
&\leq& C \prod_{l=1}^{d-2}(x_{i_l}-x_{j_l})^{-1-\beta}
(\pi-\xi)^{d-4+(\beta-1)(d-2)} \n\\
&\leq & C \prod_{l=1}^{d-2}(x_{i_l}-x_{j_l})^{-1-\beta}
(\pi-x_{i_{d-1}})^{\beta(d-2)-2}\n
\end{eqnarray*}
if $0<\beta<\frac{2}{d-2}$.
\end{proof}

In the next lemma, we estimate the partial derivatives of the kernel
function.

\begin{lem}\label{l62.6}
If $0<\beta<\frac{2}{d-1}$, then for all $q=1,\ldots,d$,
\begin{eqnarray}\label{e62.16}
|\partial_q K^1_{n,(i_l,j_l)}(x)| &\leq& C
\prod_{l=1}^{d-1}(x_{i_l}-x_{j_l})^{-1-\beta}
x_{j_{d-1}}^{\beta(d-1)-2} 1_{\{x_{j_{d-1}}\leq \pi/2\}}\n\\
&&{}+C \prod_{l=1}^{d-1}(x_{i_l}-x_{j_l})^{-1-\beta}
(\pi-x_{i_{d-1}})^{\beta(d-1)-2} 1_{\{x_{j_{d-1}}> \pi/2\}}.
\end{eqnarray}
\end{lem}

\begin{proof}
Let $m_i=0,1$ and $\delta_{m_1,\ldots,m_{d}}=0,\pm1$ be suitable
numbers. (\ref{e62.7}) implies that the partial derivative of
$D^1_{k,(i_l,j_l)}$,
\begin{eqnarray*}
\partial_qD^1_{k,(i_l,j_l)}(x)
&=& (-1)^{i_{d-1}-1}
\sum_{m_1+\cdots+m_{d}=1} \delta_{m_1,\ldots,m_{d}}\\
&&{} \prod_{l=1}^{d-1}\partial_q^{m_l}((\cos x_{i_l}-\cos
x_{j_l})^{-1})
\partial_q^{m_d}(G_k(\cos x_{i_{d-1}})-G_k(\cos x_{j_{d-1}})).\n
\end{eqnarray*}
If we differentiate in the first $(d-1)$ factors, then
\begin{eqnarray*}
\lefteqn{\sum_{m_1+\cdots+m_{d-1}=1} \prod_{l=1}^{d-1}(\cos
x_{i_l}-\cos x_{j_l})^{-1-m_l} (-\sin y_l)^{m_l}
(G_k(\cos x_{i_{d-1}})-G_k(\cos x_{j_{d-1}}))}\\
&&{}\qquad =\sum_{m_1+\cdots+m_{d-1}=1} \prod_{l=1}^{d-1}(\cos
x_{i_l}-\cos x_{j_l})^{-1-m_l} (-\sin y_l)^{m_l} H_k'(\xi)
(x_{i_{d-1}}-x_{j_{d-1}}),
\end{eqnarray*}
where $y_l=x_{i_l}$ or $y_l=-x_{j_l}$. If $x_{j_{d-1}}\leq \pi/2$,
then $|\sin y_l|\leq |y_l|\leq x_{i_l}+x_{j_l}$ and, as in the proof
of Lemma \ref{l62.2},
\begin{eqnarray*}
\lefteqn{\Big|\frac{1}{n}\sum_{k=0}^{n-1}\sum_{m_1+\cdots+m_{d-1}=1}
\prod_{l=1}^{d-1}(\cos x_{i_l}-\cos x_{j_l})^{-1-m_l} (\sin
y_l)^{m_l}
(G_k(\cos x_{i_{d-1}})-G_k(\cos x_{j_{d-1}}))\Big|}\qquad \\
&\leq& C\sum_{m_1+\cdots+m_{d-1}=1}
\prod_{l=1}^{d-1}(x_{i_l}-x_{j_l})^{-1-m_l} (x_{i_l}+x_{j_l})^{-1}
\Big|\frac{1}{n}\sum_{k=0}^{n-1}H_k'(\xi) (x_{i_{d-1}}-x_{j_{d-1}})\Big|\\
&\leq& C \sum_{m_1+\cdots+m_{d-1}=1}
\prod_{l=1}^{d-1}(x_{i_l}-x_{j_l})^{-1} (x_{i_l}+x_{j_l})^{-1} \xi^{d-3}\\
&\leq& C \prod_{l=1}^{d-1}(x_{i_l}-x_{j_l})^{-1-\beta} \xi^{(\beta-1)(d-1)+d-3}\\
&\leq& C \prod_{l=1}^{d-1}(x_{i_l}-x_{j_l})^{-1-\beta}
x_{j_{d-1}}^{\beta(d-1)-2},
\end{eqnarray*}
whenever  $0<\beta<\frac{2}{d-1}$.

If $m_d=1$ and, say $i_{d-1}=q$, then
\begin{eqnarray*}
\lefteqn{\prod_{l=1}^{d-1}(\cos x_{i_l}-\cos x_{j_l})^{-1}
\partial_q(G_k(\cos x_{i_{d-1}})-G_k(\cos x_{j_{d-1}}))}\\
&&{}\qquad \qquad =\prod_{l=1}^{d-1}(\cos x_{i_l}-\cos x_{j_l})^{-1}
H_k'(\cos x_{i_{d-1}})
\end{eqnarray*}
and
\begin{eqnarray*}
\lefteqn{\Big|\frac{1}{n}\sum_{k=0}^{n-1}\prod_{l=1}^{d-1}(\cos
x_{i_l}-\cos x_{j_l})^{-1}
\partial_q(G_k(\cos x_{i_{d-1}})-G_k(\cos x_{j_{d-1}}))\Big|}\qquad \qquad \\
&\leq &C\prod_{l=1}^{d-1}(x_{i_l}-x_{j_l})^{-1}
(x_{i_l}+x_{j_l})^{-1}
\Big|\frac{1}{n}\sum_{k=0}^{n-1}H_k'(\cos x_{i_{d-1}})\Big|\\
&\leq& C \prod_{l=1}^{d-1}(x_{i_l}-x_{j_l})^{-1} (x_{i_l}+x_{j_l})^{-1} \xi^{d-3}\\
&\leq& C \prod_{l=1}^{d-1}(x_{i_l}-x_{j_l})^{-1-\beta}
x_{j_{d-1}}^{\beta(d-1)-2}.
\end{eqnarray*}

Consequently,
$$
|\partial_qK^1_{n,(i_l,j_l)}(x)| = \Big|\frac{1}{n}\sum_{k=0}^{n-1}
\partial_q D^1_{k,(i_l,j_l)}(x)\Big| \leq C
\prod_{l=1}^{d-1}(x_{i_l}-x_{j_l})^{-1-\beta}
x_{j_{d-1}}^{\beta(d-1)-2}
$$
if $x_{j_{d-1}}\leq \pi/2$ and $0<\beta<\frac{2}{d-1}$. The case
$x_{j_{d-1}}>\pi/2$ can be proved similarly.
\end{proof}

Now we show that the $L_1$-norm of the kernel functions are
uniformly bounded.

\begin{proof*}{Theorem \ref{t12} for $q=1$ and $d\geq 3$}
We may suppose again that $\pi>x_1>x_2>\cdots>x_d>0$. If $x_1\leq
16/n$ or $\pi-x_d\leq 16/n$, then (\ref{e61.1}) implies
$$
\int_{\{16/n\geq x_1>x_2>\cdots>x_d>0\}} |K^1_n(x)| \dd x +
\int_{\{\pi>x_1>x_2>\cdots>x_d\geq \pi-16/n\}} |K^1_n(x)| \dd x \leq
C.
$$
Hence it is enough to integrate over
$$
\cS:=\{x\in \T^d:\pi>x_1>x_2>\cdots>x_d>0, x_1>16/n, x_d<\pi-16/n\}.
$$
For a sequence $(i_l,j_l)\in \cI$, let us define the set
$\cS_{(i_l,j_l),k}$ by
$$
\cS_{(i_l,j_l),k}:=\left\{
                   \begin{array}{ll}
                     x\in \cS: x_{i_l}-x_{j_l}>4/n, l=1,\ldots,k-1,
x_{i_k}-x_{j_k}\leq 4/n , & \hbox{if $k<d$;} \\
                     x\in \cS: x_{i_l}-x_{j_l}>4/n, l=1,\ldots,d-1, & \hbox{if $k=d$}
                   \end{array}
                 \right.
$$
and
$$
\cS_{(i_l,j_l),k,1}:=\left\{
                   \begin{array}{ll}
                     x\in \cS_{(i_l,j_l),k}: x_{j_k}>4/n, x_{j_{d-1}}\leq \pi/2, & \hbox{if $k<d$;} \\
                     x\in \cS_{(i_l,j_l),k}: x_{j_{d-1}}>4/n, x_{j_{d-1}}\leq \pi/2, & \hbox{if $k=d$,}
                   \end{array}
                 \right.
$$
$$
\cS_{(i_l,j_l),k,2}:=\left\{
                   \begin{array}{ll}
                     x\in \cS_{(i_l,j_l),k}: x_{j_k}\leq 4/n, x_{j_{d-1}}\leq \pi/2, & \hbox{if $k<d$;} \\
                     x\in \cS_{(i_l,j_l),k}: x_{j_{d-1}}\leq 4/n, x_{j_{d-1}}\leq \pi/2, & \hbox{if $k=d$,}
                   \end{array}
                 \right.
$$
$$
\cS_{(i_l,j_l),k,3}:=\left\{
                   \begin{array}{ll}
                     x\in \cS_{(i_l,j_l),k}: \pi-x_{i_k}>4/n, x_{j_{d-1}}>\pi/2, & \hbox{if $k<d$;} \\
                     x\in \cS_{(i_l,j_l),k}: \pi-x_{i_{d-1}}>4/n, x_{j_{d-1}}>\pi/2, & \hbox{if $k=d$,}
                   \end{array}
                 \right.
$$
$$
\cS_{(i_l,j_l),k,4}:=\left\{
                   \begin{array}{ll}
                     x\in \cS_{(i_l,j_l),k}: \pi-x_{i_k}\leq 4/n, x_{j_{d-1}}>\pi/2, & \hbox{if $k<d$;} \\
                     x\in \cS_{(i_l,j_l),k}: \pi-x_{i_{d-1}}\leq 4/n, x_{j_{d-1}}>\pi/2, & \hbox{if $k=d$.}
                   \end{array}
                 \right.
$$
Then
$$
\int_{\T^d} |K^1_n(x)| 1_\cS(x) \dd x \leq
\sum_{k=1}^{d}\sum_{m=1}^{4} \int_{\T^d} |K^1_{n,(i_l,j_l)}(x)|
1_{\cS_{(i_l,j_l),k,m}(x)}\dd x.
$$
We estimate the right hand side in four steps.

\eword{Step 1.} First, we consider the set $\cS_{(i_l,j_l),k,1}$ and
let $1\leq k\leq d-1$. Since $x_{i_{d-1}}-x_{j_{d-1}}\leq
x_{i_l}-x_{j_l}$, (\ref{e62.12}) implies
\begin{eqnarray*}
\int_{\T^d} K^1_{n,(i_l,j_l)}(x) 1_{\cS_{(i_l,j_l),k,1}}(x)\dd x
&\leq & C \int_{\T^d} \prod_{l=1}^{d-2}(x_{i_l}-x_{j_l})^{-1-\beta}
x_{j_{d-1}}^{\beta(d-2)-2} 1_{\cS_{(i_l,j_l),k,1}}(x)\dd x\\
&\leq & C \int_{\T^d} \prod_{l=1}^{k-1}(x_{i_l}-x_{j_l})^{-1-\beta}
\prod_{l=k}^{d-2}(x_{i_l}-x_{j_l})^{-1-\beta+1/(d-k)} \\
&&{} (x_{i_{d-1}}-x_{j_{d-1}})^{1/(d-k)-1}
x_{j_{d-1}}^{\beta(d-2)-2} 1_{\cS_{(i_l,j_l),k,1}}(x)\dd x.
\end{eqnarray*}
First, we choose the indices ${j_{d-1}}$ $(=i'_{d})$, ${i_{d-1}}$
$(=i'_{d-1})$ and then ${i_{d-2}}$ if ${i_{d-2}}\neq {i_{d-1}}$ or
${j_{d-2}}$ if ${j_{d-2}}\neq {j_{d-1}}$. (Exactly one case of these
two cases is satisfied.) If we repeat this process, then we get an
injective sequence $(i_l',l=1,\ldots,d)$. We integrate the term
$x_{i_1}-x_{j_1}$ in $x_{i_1'}$, the term $x_{i_{2}}-x_{j_{2}}$ in
$x_{i_{2}'}$, \ldots, and finally the term $x_{i_{d-1}}-x_{j_{d-1}}$
in $x_{i_{d-1}'}$ and $x_{j_{d-1}}$ in $x_{i_{d}'}$. Since
$x_{i_l}-x_{j_l}> 4/n$ $(l=1,\ldots,k-1)$, $x_{i_l}-x_{j_l}\leq 4/n$
$(l=k,\ldots,d-1)$, $x_{j_{d-1}}\geq x_{j_k}> 4/n$ and we can choose
$\beta$ such that $\beta<1/(d-1)$, we have
\begin{eqnarray*}
\lefteqn{\int_{\T^d} K^1_{n,(i_l,j_l)}(x) 1_{\cS_{(i_l,j_l),k,1}}(x)\dd x } \n\\
&\leq & C \prod_{l=1}^{k-1}(1/n)^{-\beta}
\prod_{l=k}^{d-2}(1/n)^{-\beta+1/(d-k)}
(1/n)^{1/(d-k)} (1/n)^{\beta(d-2)-1} \\
&\leq & C.
\end{eqnarray*}

\eword{Step 2.} For $k=d$, we use (\ref{e62.8}) to obtain
\begin{eqnarray*}
\lefteqn{\int_{\T^d} |K^1_{n,(i_l,j_l)}(x)| 1_{\cS_{(i_l,j_l),d,1}}(x)\dd x } \n\\
&\leq & C n^{-1} \int_{\T^d}
\prod_{l=1}^{d-1}(x_{i_l}-x_{j_l})^{-1-\beta}
x_{j_{d-1}}^{\beta(d-1)-2} 1_{\cS_{(i_l,j_l),d,1}}(x)\dd x\\
&\leq & C n^{-1} \prod_{l=1}^{d-1}(1/n)^{-\beta} (1/n)^{\beta(d-1)-1} \\
&\leq & C
\end{eqnarray*}
if $\beta<1/(d-1)$.

\eword{Step 3.} Now let us consider the set $\cS_{(i_l,j_l),k,2}$
for $k=1,\ldots,d-1$. Then $x_{i_k}-x_{j_k}\leq 4/n$ and so
$x_{j_{d-1}}\leq x_{i_k} \leq 8/n$ and this holds also for $k=d$.
Observe that $k\neq 1$, because $i_1=1$ and $x_1>16/n$ in $\cS$. It
follows from (\ref{e62.12}) that
\begin{eqnarray*}
\lefteqn{\int_{\T^d} K^1_{n,(i_l,j_l)}(x) 1_{\cS_{(i_l,j_l),k,2}}(x)\dd x } \n\\
&\leq & C \int_{\T^d} \prod_{l=1}^{d-2}(x_{i_l}-x_{j_l})^{-1-\beta}
x_{j_{d-1}}^{\beta(d-2)-2} 1_{\cS_{(i_l,j_l),k,2}}(x)\dd x\\
&\leq & C \int_{\T^d} \prod_{l=1}^{k-1}(x_{i_l}-x_{j_l})^{-1-\beta}
\prod_{l=k}^{d-2}(x_{i_l}-x_{j_l})^{-1-\beta+(1-\epsilon)/(d-k-1)} \\
&&{} (x_{i_{d-1}}-x_{j_{d-1}})^{\epsilon-1}
x_{j_{d-1}}^{\beta(d-2)-2} 1_{\cS_{(i_l,j_l),k,2}}(x)\dd x\\
&\leq & C \prod_{l=1}^{k-1}(1/n)^{-\beta}
\prod_{l=k}^{d-2}(1/n)^{-\beta+(\alpha-\epsilon)/(d-k-1)}
(1/n)^{\epsilon}
(1/n)^{\beta(d-2)-\alpha} \\
&\leq & C,
\end{eqnarray*}
whenever $0<\epsilon<1$ and $1/(d-2)<\beta<(1-\epsilon)/(d-3)$,
which implies $\epsilon<1/(d-2)$.

\eword{Step 4.} For the set $\cS_{(i_l,j_l),d,2}$, we obtain
similarly to Step 2 that
\begin{eqnarray*}
\lefteqn{\int_{\T^d} |K_{n,(i_l,j_l)}^1(x)| 1_{\cS_{(i_l,j_l),d,2}}(x)\dd x } \n\\
&\leq & C n^{-1} \int_{\T^d}
\prod_{l=1}^{d-1}(x_{i_l}-x_{j_l})^{-1-\beta}
x_{j_{d-1}}^{\beta(d-1)-2} 1_{\cS_{(i_l,j_l),d,2}}(x)\dd x\\
&\leq & C n^{-1} \prod_{l=1}^{d-1}(1/n)^{-\beta} (1/n)^{\beta(d-1)-1} \\
&\leq & C
\end{eqnarray*}
if $\beta>1/(d-1)$. We can prove the corresponding inequalities for
the sets $\cS_{(i_l,j_l),k,3}$ and $\cS_{(i_l,j_l),k,4}$ in the same
way.
\end{proof*}

\subsubsection{Proof for $q=\infty$}\label{s6.2.3}

Recall that
$$
D_{n}^\infty(x) = \prod_{i=1}^{d}
\frac{\sin((n+1/2)x_i)}{\sin(x_i/2)}.
$$
Let $\epsilon:=(\epsilon_1,\ldots,\epsilon_d)$ with $\epsilon_1:=1$
and $\epsilon_j:=\pm1$, $j=2,\ldots,d$ and
$\epsilon':=(\epsilon_1,\ldots,\epsilon_{d-1})$. The sums
$\sum_{\epsilon}$ and $\sum_{\epsilon'}$ mean
$\sum_{\epsilon_j=\pm1,j=2,\ldots,d}$ and
$\sum_{\epsilon_j=\pm1,j=2,\ldots,d-1}$, respectively. We may
suppose that $x_1>x_2>\cdots>x_d>0$. Applying the identities
$$
\sin a \sin b = \frac{1}{2} (\cos (a-b)-\cos(a+b)), \qquad \cos a
\sin b = \frac{1}{2} (\sin (a+b)-\sin(a-b)),
$$
we obtain
\begin{eqnarray*}
\prod_{i=1}^{d} \sin((k+1/2)x_i) &=& 2^{-d+1} \sum_{\epsilon'} \pm
\Big(\soc \Big((k+1/2)\Big(\sum_{j=1}^{d-1} \epsilon_jx_j+x_d\Big)\Big) \\
&&{} - \soc \Big((k+1/2)\Big(\sum_{j=1}^{d-1}
\epsilon_jx_j-x_d\Big)\Big) \Big).
\end{eqnarray*}
Thus
\begin{eqnarray}\label{e63.1}
K_n^{\infty}(x) &=& \frac{1}{n} \sum_{k=0}^{n-1} \prod_{i=1}^{d}
\frac{\sin((k+1/2)x_i)}{\sin(x_i/2)}
\n\\
&=& \sum_{\epsilon'}\pm 2^{-d+1} \prod_{i=1}^{d} (\sin(x_i/2))^{-1}\frac{1}{n} \sum_{k=0}^{n-1} \n\\
&&{} \Big(\soc ((k+1/2)(\sum_{j=1}^{d-1} \epsilon_jx_j+x_d))-
\soc ((k+1/2)(\sum_{j=1}^{d-1} \epsilon_jx_j-x_d))\Big)\n\\
&=:&
\sum_{\epsilon'}K_{n,\epsilon'}^\infty(x).\index{\file-1}{$K_{n,\epsilon'}^\infty$}
\end{eqnarray}
It is easy to see that
\begin{equation}\label{e63.9}
|K_n^{\infty}(x)|\leq Cn^d \qquad \mbox{and} \qquad
|K_{n,\epsilon'}^{\infty}(x)|\leq C \prod_{i=1}^{d} x_i^{-1}.
\end{equation}
For a fixed $\epsilon=(\epsilon_1,\ldots,\epsilon_d)$, we consider
those $x\in \T^d$ for which $|\sum_{j=1}^{d} \epsilon_jx_j|\leq \pi$
and $|\sum_{j=1}^{d-1} \epsilon_jx_j|\leq \pi$. If this were not the
case, say if $\sum_{j=1}^{d} \epsilon_jx_j$ were in the interval
$(2k\pi,(2k+1)\pi)$ or in $((2k+1)\pi,(2k+2)\pi)$ for a fixed $k\in
\N$, then we should write in the definitions and theorems below
$\sum_{j=1}^{d} \epsilon_jx_j-2k\pi$ or $(2k+2)\pi-\sum_{j=1}^{d}
\epsilon_jx_j$ instead of $\sum_{j=1}^{d} \epsilon_jx_j$. The same
holds for $\sum_{j=1}^{d-1} \epsilon_jx_j$. So, for simplicity, we
will always suppose that $|\sum_{j=1}^{d} \epsilon_jx_j|\leq \pi$
and $|\sum_{j=1}^{d-1} \epsilon_jx_j|\leq \pi$.

\begin{lem}\label{l63.1}
We have
$$
|K_{n,\epsilon'}^{\infty}(x)|\leq C \sum_{\epsilon_d} n^{-1}
\Big(\prod_{i=1}^{d} x_i^{-1} \Big) \Big|\sum_{j=1}^{d}
\epsilon_jx_j \Big|^{-1}.
$$
\end{lem}

\begin{proof}
Use (\ref{e63.1}) and the trigonometric identities (\ref{e61.5}).
\end{proof}

Let us introduce the sets
$$
\cS:=\{x\in \T^d:\pi>x_1>x_2>\cdots>x_d>0, x_1>32/n\},
$$
$$
\cS_{\epsilon'}:=\{x\in \T^d: \Big|\sum_{j=1}^{d-1} \epsilon_jx_j
\Big| <d\cdot16/n\},
$$
$$
\cS':=\{x\in \T^d: \exists \ \epsilon, \Big|\sum_{j=1}^{d}
\epsilon_jx_j \Big| <d\cdot16/n\},
$$
$$
\cS_{\epsilon,1}:=\{x\in \T^d: |\sum_{j=1}^{d} \epsilon_jx_j
|<4x_1\},$$
$$
\cS_{\epsilon',d}:=\{x\in \T^d: \Big|\sum_{j=1}^{d-1} \epsilon_jx_j
\Big|< 4x_d\},
$$
$$
\cS_{k}:=\{x\in \cS:x_1>x_2>\cdots>x_k\geq
4/n>x_{k+1}>\cdots>x_d>0\},
$$
$k=1,\ldots,d$. Recall that $\epsilon_1=1$ and $\epsilon_j=\pm1$,
$j=2,\ldots,d$.

\begin{lem}\label{l63.2}
For all $x\in \cS_k\setminus\cS_{\epsilon'}$, $x\in
\cS_k\setminus\cS'$ $(k=1,\ldots,d-1)$ and $x\in
\cS_{\epsilon',d}^c$,
$$
|K_{n,\epsilon'}^{\infty}(x)|\leq C \sum_{\epsilon_d}
\Big(\prod_{i=1}^{d-1} x_i^{-1} \Big) \Big|\sum_{j=1}^{d}
\epsilon_jx_j \Big|^{-1}.
$$
\end{lem}

\begin{proof}
By Lagrange's mean value theorem,
\begin{eqnarray*}
\lefteqn{\soc \Big((k+1/2)\Big(\sum_{j=1}^{d-1}
\epsilon_jx_j+x_d\Big)\Big)-
\soc \Big((k+1/2)\Big(\sum_{j=1}^{d-1} \epsilon_jx_j-x_d\Big)\Big) } \n\\
&&{} \qquad \qquad \qquad \qquad \qquad \qquad \qquad \qquad =\soc'
\Big((k+1/2)u\Big)(2k+1)x_d,
\end{eqnarray*}
where $u\in (\sum_{j=1}^{d-1} \epsilon_jx_j-x_d, \sum_{j=1}^{d-1}
\epsilon_jx_j+x_d)$. If $x\in \cS_{\epsilon',d}^c$, then
$$
|\sum_{j=1}^{d-1} \epsilon_jx_j+\epsilon_dx_d|\geq 3x_d.
$$
In the case $\sum_{j=1}^{d-1} \epsilon_jx_j+x_d\geq 0$, we have
$$
\sum_{j=1}^{d-1} \epsilon_jx_j-x_d\geq x_d \qquad \mbox{and so}
\quad |u|>|\sum_{j=1}^{d-1} \epsilon_jx_j-x_d|.
$$
If $\sum_{j=1}^{d-1} \epsilon_jx_j+x_d<0$, then
$$
\sum_{j=1}^{d-1} \epsilon_jx_j-x_d<0 \quad \mbox{and} \quad
|u|>|\sum_{j=1}^{d-1} \epsilon_jx_j+x_d|.
$$
In both cases
$$
|u|^{-1} \leq |\sum_{j=1}^{d-1} \epsilon_jx_j+x_d|^{-1}+
|\sum_{j=1}^{d-1} \epsilon_jx_j-x_d|^{-1}.
$$
The lemma can be proved in the same way if $x\in
\cS_k\setminus\cS_{\epsilon'}$ or $x\in \cS_k\setminus\cS'$
$(k=1,\ldots,d-1)$.
\end{proof}

\begin{lem}\label{l63.6}
For all $l=1,\ldots,d$ and $x\in \cS$,
\begin{eqnarray*}
|\partial_lK_{n,\epsilon'}^{\infty}(x)| &\leq& C \sum_{\epsilon_d}
\Big(\prod_{i=1}^{d} x_i^{-1}
\Big) \Big|\sum_{j=1}^{d} \epsilon_jx_j \Big|^{-1} \n\\
&&{} + C \sum_{\epsilon_d} \Big(\prod_{i=1}^{d} x_i^{-1}\Big)
x_d^{-1} 1_{\cup_{k=1}^{d-1}(\cS_k\cap\cS_{\epsilon'})
\cup(\cS_d\cap\cS_{\epsilon',d})}(x).
\end{eqnarray*}
\end{lem}

\begin{proof}
Since
$$
\frac{\partial}{\partial x_l}\frac{\soc((k+1/2)x_l)}{\sin(x_l/2)}=
\frac{(k+1/2)\soc'((k+1/2)x_l)}{\sin(x_l/2)}
-\frac{\soc((k+1/2)x_l)\cos(x_l/2)}{2\sin^2(x_i/2)},
$$
we obtain by Lemmas \ref{l63.1} and \ref{l63.2} that
\begin{eqnarray*}
|\partial_lK_{n,\epsilon'}^{\infty}(x)| &\leq& C \sum_{\epsilon_d}
\Big(\prod_{i=1}^{d} x_i^{-1}
\Big) \Big|\sum_{j=1}^{d} \epsilon_jx_j \Big|^{-1}\\
&&{}+C \sum_{\epsilon_d} \Big(\prod_{i=1}^{d-1} x_i^{-1} \Big)
x_l^{-1} \Big|\sum_{j=1}^{d} \epsilon_jx_j \Big|^{-1}
(\sum_{k=1}^{d-1}1_{\cS_k\setminus\cS_{\epsilon'}}(x)+
1_{\cS_d\setminus\cS_{\epsilon',d}}(x))\\
&&{}+ C \sum_{\epsilon_d} \Big(\prod_{i=1}^{d} x_i^{-1}\Big)
x_l^{-1} (\sum_{k=1}^{d-1}1_{\cS_k\cap\cS_{\epsilon'}}(x)+
1_{\cS_d\cap\cS_{\epsilon',d}}(x)).
\end{eqnarray*}
Now $x_l>x_d$ finishes the proof.
\end{proof}

\begin{proof*}{Theorem \ref{t12} for $q=\infty$}
If $x_1\leq 32/n$, then (\ref{e63.9}) implies
$$
\int_{\{32/n\geq x_1>x_2>\cdots>x_d>0\}} |K_n^{\infty}(x)| \dd x
\leq C.
$$
It is enough to estimate the integrals
\begin{eqnarray}\label{e63.10}
\int_{\cS} |K_{n}^{\infty}(x)| \dd x &\leq& \sum_{k=1,d}\int_{\cS_k\cap \cS'} |K_{n}^{\infty}(x)| \dd x+ \sum_{\epsilon'}\sum_{k=2}^{d-1} \int_{\cS_k\cap \cS_{\epsilon'}} |K_{n,\epsilon'}^{\infty}(x)| \dd x\n\\
&&{}+ \sum_{k=1,d}\int_{\cS_k\setminus\cS'} |K_{n}^{\infty}(x)| \dd
x + \sum_{\epsilon'}\sum_{k=2}^{d-1}
\int_{\cS_k\setminus\cS_{\epsilon'}} |K_{n,\epsilon'}^{\infty}(x)|
\dd x.
\end{eqnarray}

\eword{Step 1.} It is easy to see in the first sum that if $x\in
\cS'$, i.e., $|\sum_{j=1}^{d} \epsilon_jx_j|<d\cdot16/n$, then $x_1$
must be in an interval of length $d\cdot32/n$. Since $x_k\geq
4/n>x_{k+1}$ on $\cS_k$, we have
$$
\int_{\cS_1\cap \cS'} |K_{n}^{\infty}(x)| \dd x \leq C n^{d}
\int_{\cS_1\cap \cS'} \dd x \leq C.
$$
If $k=d$, then (\ref{e63.9}) implies
$$
\int_{\cS_d\cap \cS'} |K_{n,\epsilon'}^{\infty}(x)| \dd x \leq C
\int_{\cS_d\cap \cS'} \prod_{i=1}^{d} x_i^{-1}\dd x \leq C
\int_{\cS_d\cap \cS'} \prod_{i=2}^{d} x_i^{-1-1/(d-1)}\dd x \leq C
n^{-1}n.
$$

\eword{Step 2.} Let us investigate the second sum in (\ref{e63.10}).
First, we multiply by $1_{\cS_{\epsilon',d}}(x)$ in the integrand.
If $x\in \cS_{\epsilon',d}$, then $x_1$ is in an interval of length
$8x_d$. By (\ref{e63.9}),
\begin{eqnarray*}
\int_{\cS_k\cap \cS_{\epsilon'}\cap \cS_{\epsilon',d}}
|K_{n,\epsilon'}^{\infty}(x)| \dd x &\leq& C
\int_{\cS_k\cap \cS_{\epsilon'}\cap \cS_{\epsilon',d}} \prod_{i=1}^{d} x_i^{-1} \dd x\\
&\leq & C \int_{\cS_k} \Big(\prod_{i=2}^{k} x_i^{-1-1/(k-1)} \Big) \Big(\prod_{i=k+1}^{d} x_i^{-1} \Big) x_d \dd x_2\cdots \dd x_d\\
&\leq & C \int_{\cS_k} \Big(\prod_{i=2}^{k} x_i^{-1-1/(k-1)} \Big) \Big(\prod_{i=k+1}^{d} x_i^{-1+1/(d-k)} \Big) \dd x_2\cdots \dd x_d\\
&\leq & C n/n.
\end{eqnarray*}
If $x\not\in \cS_{\epsilon',d}$, then $|\sum_{j=1}^{d}
\epsilon_jx_j|\geq 3x_d$. Since $|\sum_{j=1}^{d}
\epsilon_jx_j|<d\cdot20/n$ on $\cS_{\epsilon'}$, Lemma \ref{l63.2}
implies
\begin{eqnarray*}
\int_{\cS_k\cap \cS_{\epsilon'}\setminus \cS_{\epsilon',d}}
|K_{n,\epsilon'}^{\infty}(x)| \dd x &\leq& C \sum_{\epsilon_d}
\int_{\cS_k\cap \cS_{\epsilon'}\setminus \cS_{\epsilon',d}} \Big(\prod_{i=1}^{d-1} x_i^{-1}\Big) x_d^{-\delta}\Big|\sum_{j=1}^{d} \epsilon_jx_j \Big|^{-1+\delta} \dd x \\
&\leq & C \sum_{\epsilon_d}\int_{\cS_k\cap \cS_{\epsilon'}}  \Big(\prod_{i=2}^{k} x_i^{-1-1/(k-1)} \Big)\\
&&{}\Big(\prod_{i=k+1}^{d} x_i^{-1+(1-\delta)/(d-k)} \Big)
\Big|\sum_{j=1}^{d} \epsilon_jx_j \Big|^{-1+\delta}\dd x \\
&\leq& C n \Big(\frac{1}{n}\Big)^{1-\delta}
\Big(\frac{1}{n}\Big)^\delta,
\end{eqnarray*}
whenever $0<\delta<1$ and $k=2,\ldots,d-1$. This proves that the
second sum in (\ref{e63.10}) is finite.

\eword{Step 3.} To estimate the fourth sum, let us use Lemma
\ref{l63.2}:
$$
\int_{\cS_k\setminus\cS_{\epsilon'}} |K_{n,\epsilon'}^{\infty}(x)|
\dd x\leq C \sum_{\epsilon_d}
\Big(\int_{(\cS_k\setminus\cS_{\epsilon'})\cap \cS_{\epsilon,1}}
+\int_{(\cS_k\setminus\cS_{\epsilon'})\setminus
\cS_{\epsilon,1}}\Big) \Big(\prod_{i=1}^{d-1} x_i^{-1} \Big)
\Big|\sum_{j=1}^{d} \epsilon_jx_j \Big|^{-1}\dd x,
$$
$k=2,\ldots,d-1$. Since $|\sum_{j=1}^{d} \epsilon_jx_j|\geq
d\cdot12/n$ on $\cS_{\epsilon'}^c$, we have
\begin{eqnarray*}
\lefteqn{\sum_{\epsilon_d} \int_{(\cS_k\setminus\cS_{\epsilon'})\cap
\cS_{\epsilon,1}} \Big(\prod_{i=1}^{d-1} x_i^{-1}
\Big) \Big|\sum_{j=1}^{d} \epsilon_jx_j \Big|^{-1}\dd x } \n\\
&\leq & C \sum_{\epsilon_d}
\int_{(\cS_k\setminus\cS_{\epsilon'})\cap \cS_{\epsilon,1}}
x_1^\delta
\Big(\prod_{i=1}^{d} x_i^{-1} \Big) x_d\Big|\sum_{j=1}^{d} \epsilon_jx_j \Big|^{-1-\delta}\dd x \n\\
&\leq& C \sum_{\epsilon_d} \int_{\cS_k\setminus\cS_{\epsilon'}}
\Big(\prod_{i=2}^{k} x_i^{-1+(\delta-1)/(k-1)}\Big)
\Big(\prod_{i=k+1}^{d} x_i^{1/(d-k)-1}\Big)
\Big|\sum_{j=1}^{d} \epsilon_jx_j \Big|^{-1-\delta}\dd x\\
&\leq& C \sum_{\epsilon_d} \Big(\frac{1}{n}\Big)^{\delta-1}
\Big(\frac{1}{n}\Big)
\Big(\frac{1}{n}\Big)^{-\delta} \\
&\leq& C,\n
\end{eqnarray*}
whenever $0<\delta\leq 1$. Similarly,
\begin{eqnarray*}
\lefteqn{\sum_{\epsilon_d}
\int_{(\cS_k\setminus\cS_{\epsilon'})\setminus \cS_{\epsilon,1}}
\Big(\prod_{i=1}^{d-1} x_i^{-1}
\Big) \Big|\sum_{j=1}^{d} \epsilon_jx_j \Big|^{-1}\dd x } \n\\
&\leq& C\sum_{\epsilon_d} \int_{\cS_k}
x_1^{-1}\Big(\prod_{i=1}^{d} x_i^{-1} \Big) x_d \dd x\\
&\leq& C\sum_{\epsilon_d} \int_{\cS_k} \Big(\prod_{i=1}^{k}
x_i^{-1-1/k}\Big)
\Big(\prod_{i=k+1}^{d} x_i^{1/(d-k)-1}\Big)\dd x\\
&\leq& C.\n
\end{eqnarray*}

\eword{Step 4.} In the third sum of (\ref{e63.10}) the inequality
$$
\int_{\cS_1\setminus\cS'} |K_{n}^{\infty}(x)| \dd x \leq C
$$
can be computed as in Step 3 with $\delta=1$ . If $k=d$, then
instead of Lemma \ref{l63.2}, we use Lemma \ref{l63.1} to show that
$$
\int_{\cS_d\setminus\cS'} |K_{n,\epsilon'}^{\infty}(x)| \dd x\leq C,
$$
which finishes the proof.
\end{proof*}

\subsubsection{Proof for $q=2$}\label{s6.2.4}

Define
$$
\theta(s):=\cases{(1-|s|^2)^\alpha &if $|s|\leq 1$ \cr 0 &if $|s|>1$
\cr} \qquad (s\in \R)
$$
and
$$
\theta_0(x):=\theta(\|x\|_2) \qquad (x\in \R^d).
$$
In this section, we use another method. We will express the Riesz
means in terms of the Fourier transform of $\theta_0$.

To this end, we need the concept of Bessel functions. First, we
introduce the \idword{gamma function},
$$
\Gamma(x) := \int_{0}^{\infty} t^{x-1} \ee^{-t} \dd t.
$$
Integration by parts yields
$$
\Gamma(x) = \Big[ \frac{t^x \ee^{-t}}{x}\Big]_0^\infty + \frac{1}{x}
\int_{0}^{\infty} t^{x} \ee^{-t} \dd t = \frac{1}{x} \Gamma(x+1).
$$
Since $\Gamma(1)=1$, we have
\begin{equation}\label{e6.1}
\Gamma(x+1)=x\Gamma(x) \quad (x>0) \qquad \mbox{and} \qquad
\Gamma(n)=(n-1)!.
\end{equation}
It is easy to see that
$$
\Gamma\Big(\frac{1}{2}\Big) = \int_{0}^{\infty} t^{-1/2} \ee^{-t}
\dd t = 2\int_{0}^{\infty} \ee^{-u^2} \dd u=\sqrt{\pi}.
$$
The \idword{beta function} is defined by
$$
B(x,y):= \int_{0}^{1} s^{x-1} (1-s)^{y-1} \dd s=\int_{0}^{1} s^{y-1}
(1-s)^{x-1} \dd s.
$$
The relationship between the beta and gamma function reads as
follows:
\begin{equation}\label{e6.2}
\Gamma(x+y)B(x,y)=\Gamma(x)\Gamma(y).
\end{equation}
Indeed, substituting $s=u/(1+u)$, we obtain
\begin{eqnarray*}
\Gamma(x+y)B(x,y)&=&\Gamma(x+y) \int_{0}^{1} s^{y-1} (1-s)^{x-1} \dd s\\
&=&\Gamma(x+y) \int_{0}^{\infty} u^{y-1} \Big(\frac{1}{1+u}\Big)^{x+y} \dd u\\
&=&\int_{0}^{\infty} \int_{0}^{\infty} u^{y-1}
\Big(\frac{1}{1+u}\Big)^{x+y} v^{x+y-1} \ee^{-v} \dd vdu.
\end{eqnarray*}
The substitution $v=t(1+u)$ in the inner integral yields
\begin{eqnarray*}
\Gamma(x+y)B(x,y) &=&\int_{0}^{\infty} \int_{0}^{\infty} u^{y-1} t^{x+y-1} \ee^{-t(1+u)} \dd t\dd u \\
&=&\int_{0}^{\infty} t^{x} \ee^{-t} \int_{0}^{\infty} (ut)^{y-1} \ee^{-tu} \dd u\dd t \\
&=&\int_{0}^{\infty} t^{x-1} \ee^{-t} \Gamma(y)\dd t \\
&=&\Gamma(x)\Gamma(y),
\end{eqnarray*}
which shows (\ref{e6.2}).

For $k>-1/2$, the \idword{Bessel function} is defined by
$$
J_k(t):=\frac{(t/2)^k}{\Gamma(k+1/2)\Gamma(1/2)} \int_{-1}^{1}
\ee^{\ii ts} (1-s^2)^{k-1/2} \dd s \qquad (t\in
\R).\index{\file-1}{$J_k$}
$$
Note that the Bessel functions are real valued. We prove some basic
properties of the Bessel functions.

\begin{lem}\label{l6.1}
We have
$$
J_k'(t)=kt^{-1}J_k(t)-J_{k+1}(t)\qquad (t\neq 0).
$$
\end{lem}

\begin{proof}
By integrating by parts,
\begin{eqnarray*}
\frac{\dd }{\dd t} (t^{-k}J_{k}(t)) &=& \frac{\ii 2^{-k} }{\Gamma(k+1/2)\Gamma(1/2)} \int_{-1}^{1} \ee^{\ii ts} s (1-s^2)^{k-1/2} \dd s \\
&=& \frac{\ii 2^{-k} }{\Gamma(k+1/2)\Gamma(1/2)} \int_{-1}^{1} \frac{\ii t}{2k+1} \ee^{\ii ts} (1-s^2)^{k+1/2} \dd s \\
&=& \frac{-2^{-k-1} t}{(k+1/2)\Gamma(k+1/2)\Gamma(1/2)} \int_{-1}^{1} \ee^{\ii ts} (1-s^2)^{k+1/2} \dd s \\
&=& -t^{-k}J_{k+1}(t).
\end{eqnarray*}
In the last step, we used (\ref{e6.1}). The lemma follows
immediately.
\end{proof}

\begin{lem}\label{l6.2}
For $k>-1/2$ and $t>0$,
$$
J_{k}(t)\leq C_k t^{k} \qquad \mbox{and} \qquad J_{k}(t)\leq C_k
t^{-1/2},
$$
where $C$ is independent of $t$.
\end{lem}

\begin{proof}
The first estimate trivially follows from the definition of $J_k$.
The second one follows from the first one if $0<t\leq 1$. Assume
that $t>1$ and integrate the complex valued function $\ee^{\ii tz}
(1-z^2)^{k-1/2}$ $(z\in \C)$ over the boundary of the rectangle
whose lower side is $[-1,1]$ and whose height is $R>0$. By Cauchy's
theorem,
\begin{eqnarray*}
0 &=& \ii \int_R^{0} \ee^{\ii t(-1+\ii s)} (s^2+2\ii s)^{k-1/2} \dd
s +
\int_{-1}^{1} \ee^{\ii ts} (1-s^2)^{k-1/2} \dd s \\
&&{}+ \ii \int_{0}^{R} \ee^{\ii t(1+\ii s)} (s^2-2\ii s)^{k-1/2} \dd
s +\epsilon(R),
\end{eqnarray*}
where $\epsilon(R)\to 0$ as $R\to\infty$. Hence
\begin{eqnarray*}
\int_{-1}^{1} \ee^{\ii ts} (1-s^2)^{k-1/2} \dd s &=& \ii \ee^{-\ii t} \int_0^{\infty} \ee^{-ts} (s^2+2\ii s)^{k-1/2} \dd s \\
&&{}- \ii \ee^{\ii t} \int_{0}^{\infty} \ee^{-ts} (s^2-2\ii s)^{k-1/2} \dd s \\
&=:& I_1+I_2.
\end{eqnarray*}
Observe that
$$
(s^2+2\ii s)^{k-1/2} = (2\ii s)^{k-1/2} +\phi(s),
$$
where $|\phi(s)|\leq  C s^{k+1/2}$ if $0<s\leq 1$ or $s>1$ and
$k\leq 3/2$ and $|\phi(s)|\leq  C s^{2k-1}$ if $s>1$ and $k>3/2$.
Indeed, by Lagrange's mean value theorem
$$
|\phi(s)| = |(2\ii s)^{k-1/2}| |(\frac{s}{2\ii}+1)^{k-1/2} - 1| \leq
C_k s^{k+1/2} |\frac{\xi}{2\ii}+1|^{k-3/2},
$$
where $0<\xi<s$. Thus $|s^2+2\ii s|^{k-1/2}  \leq C_k s^{k-1/2} +
|\phi(s)|$ and
\begin{eqnarray*}
|I_1| &\leq& \int_0^{\infty} \ee^{-ts} (C_k s^{k-1/2} + |\phi(s)|) \dd s \\
&=& C_k t^{-1} \int_0^{\infty} \ee^{-u} (u/t)^{k-1/2} \dd u +
\int_0^{1} \ee^{-ts} |\phi(s)| \dd s + \int_1^{\infty} \ee^{-ts}
|\phi(s)| \dd s.
\end{eqnarray*}
The first term is $C_k \Gamma(k+1/2) t^{-k-1/2}$, the second term
can be estimated by
$$
\Gamma(k+3/2) t^{-k-3/2} \leq C_k t^{-k-1/2}
$$
and the third one can be estimated by $\Gamma(k+3/2) t^{-k-3/2}$ if
$k\leq 3/2$ or by $C_k\ee^{-t}$ if $k>3/2$, both are less than $C_k
t^{-k-1/2}$. The integral $I_2$ can be estimated in the same way.
\end{proof}

\begin{lem}\label{l6.3}
If $k>-1/2$, $l>-1$ and $t>0$, then
$$
J_{k+l+1}(t)=\frac{t^{l+1}}{2^l\Gamma(l+1)} \int_{0}^{1} J_{k}(ts)
s^{k+1} (1-s^2)^{l} \dd s.
$$
\end{lem}

\begin{proof}
Observe that
\begin{eqnarray*}
J_k(t)&=&\frac{2(t/2)^k}{\Gamma(k+1/2)\Gamma(1/2)} \int_{0}^{1} \cos(ts) (1-s^2)^{k-1/2} \dd s\\
&=& \sum_{j=0}^{\infty} (-1)^j \frac{2(t/2)^kt^{2j}}{(2j)! \Gamma(k+1/2)\Gamma(1/2)} \int_{0}^{1} s^{2j} (1-s^2)^{k-1/2} \dd s\\
&=& \sum_{j=0}^{\infty} (-1)^j \frac{(t/2)^kt^{2j}}{(2j)! \Gamma(k+1/2)\Gamma(1/2)} \int_{0}^{1} u^{j-1/2} (1-u)^{k-1/2} \dd u\\
&=& \sum_{j=0}^{\infty} (-1)^j \frac{(t/2)^kt^{2j}}{(2j)! \Gamma(k+1/2)\Gamma(1/2)} B(j+1/2,k+1/2)\\
&=& \frac{(t/2)^k}{\Gamma(1/2)} \sum_{j=0}^{\infty} (-1)^j
\frac{\Gamma(j+1/2)}{\Gamma(j+k+1)} \frac{t^{2j}}{(2j)!}.
\end{eqnarray*}
Thus
\begin{eqnarray*}
\lefteqn{\int_{0}^{1} J_{k}(ts) s^{k+1} (1-s^2)^{l} \dd s } \n\\
&=&\int_{0}^{1} \Big(\frac{(ts/2)^k}{\Gamma(1/2)} \sum_{j=0}^{\infty} (-1)^j \frac{\Gamma(j+1/2)}{\Gamma(j+k+1)} \frac{(ts)^{2j}}{(2j)!} \Big) s^{k+1} (1-s^2)^{l} \dd s \\
&=&\frac{(t/2)^k}{\Gamma(1/2)} \sum_{j=0}^{\infty} (-1)^j \frac{\Gamma(j+1/2)}{\Gamma(j+k+1)} \frac{t^{2j}}{(2j)!} \int_{0}^{1} s^{2k+2j+1} (1-s^2)^{l} \dd s \\
&=&\frac{(t/2)^k}{\Gamma(1/2)} \sum_{j=0}^{\infty} (-1)^j \frac{\Gamma(j+1/2)}{2\Gamma(j+k+1)} \frac{t^{2j}}{(2j)!} \int_{0}^{1} u^{k+j} (1-u)^{l} \dd u \\
&=&\frac{(t/2)^k}{\Gamma(1/2)} \sum_{j=0}^{\infty} (-1)^j \frac{\Gamma(j+1/2)}{2\Gamma(j+k+1)} \frac{t^{2j}}{(2j)!} B(k+j+1,l+1) \\
&=&\frac{2^l \Gamma(l+1)}{t^{l+1}}\frac{(t/2)^{k+l+1}}{\Gamma(1/2)} \sum_{j=0}^{\infty} (-1)^j \frac{\Gamma(j+1/2)}{\Gamma(k+l+j+2)} \frac{t^{2j}}{(2j)!} \\
&=&\frac{2^l \Gamma(l+1)}{t^{l+1}} J_{k+l+1}(t),
\end{eqnarray*}
which proves the lemma.
\end{proof}

If $\theta_0$ is radial as above, then its Fourier transform is also
radial and can be computed with the help of the Bessel functions.
Recall that the \idword{Fourier transform} of $f\in L_1(\R^d)$ is
defined by
$$
\widehat f(x) := \frac{1}{(2\pi)^d}\int_{\R^d} f(t) \ee^{-\ii x
\cdot t} \dd t \qquad (x \in \R^d).\index{\file-1}{$\widehat f$}
$$

\begin{thm}\label{t6.1}
For $x\in \R^d$ and $r=\|x\|_2$,
$$
\widehat {\theta}_0(x)= (2\pi)^{-d/2} \, r^{-d/2+1}
\int_{0}^{\infty} \theta(s) J_{d/2-1}(rs) s^{d/2}\dd s.
$$
\end{thm}

\begin{proof}
Obviously, $\theta_0\in L_1(\R^d)$ because $\int_{0}^{\infty}
|\theta(r)| r^{d-1} \dd r <\infty$. Let $r=\|x\|_2$, $x=rx'$,
$s=\|u\|_2$ and $u=su'$. Then
$$
\widehat {\theta}_0(x)= \frac{1}{(2\pi)^d}\int_{\R^d} \theta_0(u)
\ee^{-\ii x \cdot u} \dd u= \frac{1}{(2\pi)^d} \int_{0}^{\infty}
\theta(s) (\int_{\Sigma_{d-1}} \ee^{-\ii rs x' \cdot u'} \dd u')
s^{d-1}\dd s,
$$
where $\Sigma_{d-1}$ is the sphere. In the inner integral, we
integrate first over the parallel $P_\delta:=\{u'\in
\Sigma_{d-1}:x'\cdot u'=\cos \delta\}$ orthogonal to $x'$ obtaining
a function of $0\leq \delta\leq \pi$, which we then integrate over
$[0,\pi]$. If $\omega_{d-2}$ denotes the surface area of
$\Sigma_{d-2}$, then the measure of $P_\delta$ is
$$
\omega_{d-2}(\sin
\delta)^{d-2}=\frac{2\pi^{(d-1)/2}}{\Gamma((d-1)/2)}
(\sin\delta)^{d-2}.
$$
Hence
\begin{eqnarray*}
\int_{\Sigma_{d-1}} \ee^{-\ii rs x' \cdot u'} \dd u'&=& \int_{0}^{\pi} \ee^{-\ii rs \cos\delta} \omega_{d-2}(\sin \delta)^{d-2}\dd \delta\\
&=& \omega_{d-2} \int_{-1}^{1} \ee^{\ii rs \xi} (1-\xi^2)^{(d-3)/2} \dd \xi\\
&=& \frac{2\pi^{(d-1)/2}}{\Gamma((d-1)/2)} \frac{\Gamma(d/2-1/2)\Gamma(1/2)}{(rs/2)^{d/2-1}} J_{d/2-1}(rs)\\
&=& (2\pi)^{d/2}(rs)^{-d/2+1} J_{d/2-1}(rs),
\end{eqnarray*}
which finishes the proof of the theorem.
\end{proof}

Note that this theorem works for any radial function.

\begin{cor}\label{c6.1}
If $\alpha>0$, then
$$
\widehat {\theta}_0(x)= (2\pi)^{-d/2} 2^{\alpha}\Gamma(\alpha+1)
\|x\|_2^{-d/2-\alpha} J_{d/2+\alpha}(\|x\|_2).
$$
\end{cor}

\begin{proof}
By Theorem \ref{t6.1},
$$
\widehat {\theta}_0(x)= (2\pi)^{-d/2} \|x\|_2^{-d/2+1} \int_{0}^{1}
J_{d/2-1}(\|x\|_2s) s^{d/2}  (1-s^2)^\alpha\dd s.
$$
Applying Lemma \ref{l6.3} with $k=d/2-1$, $l=\alpha$, we see that
$$
\widehat {\theta}_0(x)= (2\pi)^{-d/2} \|x\|_2^{-d/2+1}
J_{d/2+\alpha}(\|x\|_2) \|x\|_2^{-\alpha-1}
2^{\alpha}\Gamma(\alpha+1),
$$
which shows the corollary.
\end{proof}

Corollary \ref{c6.1} and Lemma \ref{l6.1} imply that $\widehat
{\theta}_0(x)$ as well as all of its derivatives can be estimated by
$\|x\|_2^{-d/2-\alpha-1/2}$.

\begin{cor}\label{c6.2}
For all $i_1,\ldots,i_d\geq 0$ and $\alpha>0$,
$$
|\partial_1^{i_1}\cdots\partial_d^{i_d}\widehat \theta_0(x)| \leq C
\|x\|_2^{-d/2-\alpha-1/2} \qquad (x\neq 0).
$$
\end{cor}

The same result holds for
$$
\theta(s):=\cases{(1-|s|^\gamma)^\alpha &if $|s|\leq 1$ \cr 0 &if
$|s|>1$ \cr} \qquad (s\in \R)
$$
and $\theta_0(x):=\theta(\|x\|_2)$ $(x\in \R^d)$, whenever
$\gamma\in \N$ (see Lu \cite[p.~132]{lu}). From now on, we assume
that $\gamma\in \N$. Now we are ready to express the Riesz means
using the Fourier transform of $\theta_0$. Observe that $\widehat
\theta_0\in L_1(\R^d)$ if and only if $\alpha>(d-1)/2$.

\begin{thm}\label{t64.5}
If $n\in\N^d$, $f\in L_1(\T^d)$ and $\alpha>(d-1)/2$, then
$$
\sigma_n^{2,\alpha} f(x)= n^d \int_{\R^d} f(x-t) \widehat
\theta_0(nt) \dd t.
$$
\end{thm}

\begin{proof}
If $f(t)=\ee^{\ii k \cdot t}$ $(k\in \Z^d, t\in \T^d)$, then
$$
n^d \int_{\R^d} \ee^{\ii k \cdot (x-t)} \widehat \theta_0(nt) \dd t
= \ee^{\ii k \cdot x} \int_{\R^d} \ee^{-\ii k \cdot t/n} \widehat
\theta_0(t) \dd t = \theta_0\Big({-k \over n}\Big) \ee^{\ii k \cdot
x} = \sigma_n^{2,\alpha} f(x).
$$
The theorem holds also for trigonometric polynomials. Let $f$ be an
arbitrary element from $L_1(\T^d)$ and $(f_k)$ be a sequence of
trigonometric polynomials such that $f_k \to f$ in the
$L_1(\T^d)$-norm. It follows from the form (\ref{e49}) of
$\sigma_{n}^{2,\alpha}f$ and from the fact that $K_n^{2,\alpha} \in
L_1(\T^d)$ that $\sigma_n^{2,\alpha} f_k \to \sigma_n^{2,\alpha} f$
in the $L_1(\T^d)$ norm as $k\to \infty$.

On the other hand, since $\widehat \theta_0\in L_1(\R^d)$, we have
$$
\int_{\R^d} f_k(x-t) \widehat \theta_0(nt) \dd t \longrightarrow
\int_{\R^d} f(x-t) \widehat \theta_0(nt) \dd t
$$
in the $L_1(\T^d)$-norm as $k\to \infty$.
\end{proof}

\begin{proof*}{Theorem \ref{t12} for $q=2$}
Observe that Theorem \ref{t12} follows from Theorem \ref{t64.5}.
Indeed, since $f$ is periodic,
\begin{eqnarray*}
\sigma_n^{2,\alpha} f(x) &=& n^d \sum_{k\in \Z^d} \int_{2k\pi}^{2(k+1)\pi} f(x-t) \widehat \theta_0(nt) \dd t \\
&=& n^d \sum_{k\in \Z^d} \int_{0}^{2\pi} f(x-u) \widehat
\theta_0(n(u+2k\pi)) \dd u.
\end{eqnarray*}
By (\ref{e49}),
\begin{equation}\label{e6.3}
K_n^{q,\alpha}(u)=(2\pi)^d n^d \sum_{k\in \Z^d} \widehat
\theta_0(n(u+2k\pi))
\end{equation}
and
$$
\int_{0}^{2\pi} |K_n^{q,\alpha}(u)|\dd u \leq (2\pi)^d n^d
\sum_{k\in \Z^d} \int_{0}^{2\pi} |\widehat \theta_0(n(u+2k\pi))|\dd
u = (2\pi)^d \|\widehat \theta_0\|_1.
$$
\end{proof*}

\sect{$H_p^\Box(\T^d)$ Hardy spaces}\label{s7}

To prove almost everywhere convergence of the Riesz means, we will
need the concept of Hardy spaces and their atomic decomposition. A
distribution $f$ is in the \idword{Hardy space}
\inda{$H_p^\Box(\T^d)$} and in the \idword{weak Hardy space}
\inda{$H_{p,\infty}^\Box(\T^d)$} $(0<p\leq \infty)$ if
$$
\|f\|_{H_{p}^\Box}:= \|\sup_{0<t} |f * P_t^d|\|_{p}< \infty
$$
and
$$
\|f\|_{H_{p,\infty}^\Box}:= \|\sup_{0<t} |f * P_t^d|\|_{p,\infty}<
\infty,
$$
respectively, where
$$
P_{t}^d(x) := \sum_{k \in \Z^{d}} \ee^{-t \|k\|_2} \ee^{\ii k \cdot
x} \qquad (x \in \T^{d},t>0)
$$
is the $d$-dimensional periodic \idword{Poisson kernel}. Since
$P_{t}^d\in L_1(\T^d)$, the convolution in the definition of the
norms are well defined.\index{\file-1}{$P_t^d$} In the
one-dimensional case, we get back the usual Poisson kernel
$$
P_{t}(x) :=P_{t}^1(x) = \sum_{k=-\infty}^{\infty} r^{|k|} \ee^{\ii
kx} = {1-r^2 \over 1+r^2 - 2r \cos x} \qquad (x \in
\T),\index{\file-1}{$P_t$}
$$
where $r:=\ee^{-t}$. It is known (see e.g.~Stein \cite{st1} or Weisz
\cite{wk2}) that
$$
H_p^\Box(\T^d) \sim L_p(\T^d) \qquad (1<p \leq \infty)
$$
and $H_1^\Box(\T^d) \subset L_1(\T^d) \subset
H_{1,\infty}^\Box(\T^d)$. Moreover,
\begin{equation}\label{e4}
\|f\|_{H_{1,\infty}^\Box}=\sup_{\rho >0} \rho \, \lambda(\sup_{0<t}
|f * P_t^d| > \rho) \leq C \|f\|_1 \qquad (f \in L_1(\T^d)).
\end{equation}

The \ieword{atomic decomposition} provides a useful characterization
of Hardy spaces. A bounded function $a$ is an
\idword{$H_p^\Box$-atom} if there exists a cube $I \subset \T^d$
such that
\begin{enumerate}
\item []
\begin{enumerate}
\item [(i)] ${\rm supp} \ a \subset I$,
\item [(ii)] $\|a\|_\infty \leq |I|^{-1/p}$,
\item[(iii)] $\int_{I} a(x) x^{k}\dd x = 0$
for all multi-indices $k=(k_1,\ldots,k_{d})$ with $|k|\leq \lfloor
d(1/p-1) \rfloor $.
\end{enumerate}
\end{enumerate}

In the definition, the cubes can be replaced by balls and (ii) by
\begin{enumerate}
\item []
\begin{enumerate}
\item [(ii')] $\|a\|_q \leq |I|^{1/q-1/p}$ $(1<q\leq \infty)$.
\end{enumerate}
\end{enumerate}
We could suppose that the integral in (iii) is zero for all
multi-indices $k$ for which $|k|\leq N$, where $N \geq \lfloor
d(1/p-1) \rfloor $. The best possible choice of such numbers $N$ is
$\lfloor d(1/p-1) \rfloor$. Each Hardy space has an atomic
decomposition. In other words, every function from the Hardy space
can be decomposed into the sum of atoms (see e.g.~Latter \cite{la},
Lu \cite{lu}, Coifman and Weiss \cite{cowe}, Wilson \cite{wi2,wi1},
Stein \cite{st1} and Weisz \cite{wk2}).

\begin{thm}\label{t17}
A function $f$ is in $H_p^\Box(\T^d)$ $(0<p \leq 1)$ if and only if
there exist a sequence $(a^k,k \in {\N})$ of $H_p^\Box$-atoms and a
sequence $(\mu_k,k \in {\N})$ of real numbers such that
$$
\sum_{k=0}^\infty |\mu_k|^p < \infty \quad \mbox{and} \quad
\sum_{k=0}^{\infty} \mu_ka^k=f \quad \mbox{in the sense of
distributions}.
$$
Moreover,
$$
{\Vert f \Vert}_{H_p^\Box} \sim \inf \Big(\sum_{k=0}^\infty
|\mu_k|^p \Big)^{1/p}.
$$
\end{thm}

The ``only if'' part of the theorem holds also for $0<p<\infty$. The
following result gives a sufficient condition for an operator to be
bounded from $H_p^\Box(\T^d)$ to $L_p(\T^d)$ (see e.g.~Weisz
\cite{wk2}). Let $I^r$ be the interval having the same center as the
interval $I\subset \T$ and length $2^r|I|$ $(r\in\N)$. For a
rectangle
$$
R=I_1\times\cdots\times I_{d} \quad \mbox{let} \quad
R^r=I_1^r\times\cdots\times I_{d}^r.
$$

\begin{thm}\label{t18}
For each $n\in \N^d$, let $V_n:L_1(\T^d)\to L_1(\T^d)$ be a bounded
linear operator and let
$$
V_*f:=\sup_{n\in \N^d} |V_nf|.\index{\file-1}{$V_*f$}
$$
Suppose that
$$
\int_{\T^d \setminus I^r} |V_*a|^{p_0} \dd \lambda \leq C_{p_0}
$$
for all $H_{p_0}^\Box$-atoms $a$ and for some fixed $r\in\N$ and
$0<p_0\leq 1$, where the cube $I$ is the support of the atom. If
$V_*$ is bounded from $L_{p_1}(\T^d)$ to $L_{p_1}(\T^d)$ for some
$1<p_1 \leq \infty$, then
\begin{equation}\label{e5}
\|V_*f\|_p \leq C_p \|f\|_{H_p^\Box} \qquad (f\in H_p^\Box(\T^d)\cap
L_1(\T^d))
\end{equation}
for all $p_0\leq p\leq p_1$. If $\lim_{k\to\infty} f_k= f$ in the
$H_p^\Box$-norm implies that $\lim_{k\to\infty} V_nf_k=V_nf$ in the
sense of distributions $(n\in \N^d)$, then (\ref{e5}) holds for all
$f\in H_p^\Box(\T^d)$.
\end{thm}

\begin{proof}
Observe that, under the conditions of Theorem \ref{t18}, the
$L_{p_0}$-norms of $V_*a$ are uniformly bounded for all
$H_{p_0}^\Box$-atoms $a$. Indeed,
\begin{eqnarray*}
\int_{\T^d} |V_*a|^{p_0} \dd \lambda
&=& \int_{I^r} |V_*a|^{p_0} \dd \lambda + \int_{\T^d\setminus I^r} |V_*a|^{p_0} \dd \lambda \\
&\leq& \Big(\int_{I^r} |V_*a|^{p_1} \dd \lambda \Big)^{p_0/p_1} |I^r|^{1-p_0/p_1} + C_{p_0} \\
&\leq& C_{p_0} \Big(\int_{I^r} |a|^{p_1} \dd \lambda \Big)^{p_0/p_1} |I|^{1-p_0/p_1} + C_{p_0} \\
&\leq& C_{p_0} \Big( |I|^{-p_1/p_0}|I^r| \Big)^{p_0/p_1} |I|^{1-p_0/p_1} + C_{p_0}\\
&=& C_{p_0}.
\end{eqnarray*}

There is an atomic decomposition such that
$$
f=\sum_{k=0}^{\infty} \mu_ka_k \quad \mbox{in the
$H_{p_0}^\Box$-norm} \quad \mbox{and} \quad \Big(\sum_{k=0}^\infty
|\mu_k|^{p_0}\Big)^{1/p_0} \leq C_{p_0}\|f\|_{H_{p_0}^\Box},
$$
where the convergence holds also in the $H_{1}^\Box$-norm and in the
$L_1$-norm if $f\in H_{1}^\Box(\T^d)$. Then
$$
V_nf= \sum_{k=0}^\infty \mu_k V_na_k
$$
and
$$
|V_*f| \leq \sum_{k=0}^\infty |\mu_k| |V_*a_k|
$$
for $f\in H_{1}^\Box(\T^d)$. Thus
\begin{equation}\label{e31}
\|V_*f\|_{p_0}^{p_0} \leq \sum_{k=0}^\infty |\mu_k|^{p_0}
\|V_*a_k\|_{p_0}^{p_0} \leq C_{p_0} \|f\|_{H_{p_0}^\Box}^{p_0}
\qquad (f\in H_{1}^\Box(\T^d)).
\end{equation}
Obviously, the same inequality holds for the operators $V_n$. This
and interpolation proves the theorem if $p_0=1$. Assume that
$p_0<1$. Since $H_{1}^\Box(\T^d)$ is dense in $L_1(\T^d)$ as well as
in $H_{p_0}^\Box(\T^d)$, we can extend uniquely the operators $V_n$
and $V_*$ such that (\ref{e31}) holds for all $f\in
H_{p_0}^\Box(\T^d)$. Let us denote these extended operators by
$V'_n$ and $V'_*$. Then $V_nf=V'_nf$ and $V_*f= V'_*f$ for all $f\in
H_{1}^\Box(\T^d)$. It is enough to show that these equalities hold
for all $f\in H_{p_0}^\Box(\T^d)\cap L_1(\T^d)$. We get by
interpolation from (\ref{e31}) that the operator
\begin{equation}\label{e34}
V_*'  \quad \mbox{is bounded from} \quad H_{p,\infty}^\Box(\T^d)
\quad \mbox{to} \quad L_{p,\infty}(\T^d)
\end{equation}
when $p_0<p<p_1$. For the basic definitions and theorems on
\ind{interpolation theory}, see Bergh and L{\"o}fstr{\"o}m
\cite{belo}, Bennett and Sharpley \cite{besh} or Weisz \cite{wk2}.
Since $p_0<1$, the boundedness in (\ref{e34}) holds especially for
$p=1$, and so (\ref{e4}) implies that $V_*'$ is of weak type
$(1,1)$:
\begin{equation}\label{e32}
\sup_{\rho >0} \rho \, \lambda(|V_*'f| > \rho) =
\|V_*'f\|_{1,\infty} \leq C \|f\|_{H_{1,\infty}^\Box} \leq C \|f\|_1
\qquad (f \in L_1(\R^d)).
\end{equation}
Obviously, the same holds for $V'_n$. Since $V_n$ is bounded on
$L_1(\T^d)$ if $f_k\in H_1^\Box(\T^d)$ such that
$\lim_{k\to\infty}f_k=f$ in the $L_1$-norm, then
$\lim_{k\to\infty}V_nf_k=V_nf$ in the $L_1$-norm. Inequality
(\ref{e32}) implies that $\lim_{k\to\infty}V_nf_k=V'_nf$ in the
$L_{1,\infty}$-norm, hence $V_nf=V'_nf$ for all $f\in L_{1}(\T^d)$.
Similarly, for a fixed $N\in \N$, the operator
$$
V_{N,*}f:=\sup_{|n|\leq N} |V_nf|
$$
satisfies (\ref{e32}) for all $f\in H_1^\Box(\T^d)$ and its
extension $V'_{N,*}$ for all $f\in L_{1}(\T^d)$. The inequality
\begin{eqnarray*}
\sup_{\rho>0} \rho \, \lambda(|V'_{N,*}f-V_{N,*}f | > \rho) &\leq& \sup_{\rho>0} \rho \, \lambda(|V'_{N,*}f-V'_{N,*}f_k | > \rho/2) \\
&&{} + \sup_{\rho>0} \rho \, \lambda(|V_{N,*}f_k - V_{N,*}f| >
\rho/2)\to 0
\end{eqnarray*}
as $k\to\infty$, shows that $V'_{N,*}f=V_{N,*}f$ for all $f\in
L_{1}(\T^d)$. Moreover, for a fixed $\rho$,
\begin{eqnarray*}
\lefteqn{\lambda(|V'_{*}f-V_{N,*}f| > \rho) } \n\\ &\leq& \lambda(|V'_*f-V'_*f_k | > \rho/3) + \lambda(|V_*f_k-V_{N,*}f_k| > \rho/3)+ \lambda(|V_{N,*}f_k - V_{N,*}f| > \rho/3)\\
&\leq& \lambda(V'_*(f-f_k) > \rho/3) + \lambda(V_*f_k-V_{N,*}f_k > \rho/3)+ \lambda(V_{N,*}(f_k-f)> \rho/3)\\
&\leq& \frac{C}{\rho}\|f-f_k\|_1 + \lambda(V_*f_k-V_{N,*}f_k > \rho/3) \\
&<&\epsilon
\end{eqnarray*}
if $k$ and $N$ are large enough. Hence $\lim_{N\to\infty}
V_{N,*}f=V'_{*}f$ in measure for all $f\in L_{1}(\T^d)$. On the
other hand, $\lim_{N\to\infty} V_{N,*}f=V_{*}f$ a.e., which implies
that
$$
V_*f= V'_*f \quad \mbox{for all} \quad f\in L_{1}(\T^d).
$$
Consequently, (\ref{e32}) holds also for $V_*$ and (\ref{e31}) for
all $f\in H_{p_0}^\Box(\T^d)\cap L_1(\T^d)$.

Assume that $V_n$ is defined also for distributions and that
$\lim_{k\to\infty} f_k= f$ in the $H_p^\Box$-norm implies
$\lim_{k\to\infty} V_nf_k=V_nf$ in the sense of distributions $(n\in
\N^d)$. Suppose that $p<1$ and $f_k\in H_p^\Box(\T^d)\cap L_1(\T^d)$
$(k\in \N)$. Since by (\ref{e5}), $V_nf_k$ is convergent in the
$L_p$-norm as $k\to \infty$, we can identify the distribution $V_nf$
with the $L_p$-limit $\lim_{k\to\infty} V_nf_k$. Hence the same
holds for $V_{N,*}f$: $V_{N,*}f=\lim_{k\to\infty} V_{N,*}f_k$ in the
$L_p$-norm. Moreover,
\begin{eqnarray*}
\|V'_{*}f-V_{N,*}f\|_p &\leq & \|V'_*f-V'_*f_k \|_p + \|V_*f_k-V_{N,*}f_k\|_p + \|V_{N,*}f_k - V_{N,*}f\|_p \\
&\leq& C_p \|f-f_k\|_{H_p^\Box} + \|V_*f_k-V_{N,*}f_k\|_p + \|V_{N,*}f_k - V_{N,*}f\|_p  \\
&<&\epsilon
\end{eqnarray*}
if $k$ and $N$ are large enough. Thus $\lim_{N\to\infty}
V_{N,*}f=V'_{*}f$ in the $L_p$-norm and, on the other hand,
$\lim_{N\to\infty} V_{N,*}f=V_{*}f$ a.e.,  which implies that $V_*f=
V'_*f$ for all $f\in H_{p}^\Box(\T^d)$. Consequently, (\ref{e5})
holds for all $f\in H_p^\Box(\T^d)$.
\end{proof}

Unfortunately, for a linear operator $V$, the uniform boundedness of
the $L_{p_0}$-norms of $Va$ is not enough for the boundedness
$V:H_{p_0}^\Box(\T^d)\to L_{p_0}(\T^d)$ (see
\cite{{Bownik-2005},{Meda-2008},{Meda-2009},{Bownik-2010},{Ricci-2011}}).
The next weak version of Theorem \ref{t18} can be proved similarly
(see also the proof in Weisz \cite{wk2}).

\begin{thm}\label{t19}
For each $n\in \N^d$, let $V_n:L_1(\T^d)\to L_1(\T^d)$ be a bounded
linear operator and let
$$
V_*f:=\sup_{n\in \N^d} |V_nf|.
$$
Suppose that
$$
\sup_{\rho>0} \rho^{p} \lambda \Big(\{|V_*a|>\rho\}\cap \{\T^d
\setminus I^r\} \Big) \leq C_{p}
$$
for all $H_{p}^\Box$-atoms $a$ and for some fixed $r\in\N$ and $0<p<
1$. If $V_*$ is bounded from $L_{p_1}(\T^d)$ to $L_{p_1}(\T^d)$
$(1<p_1 \leq \infty)$, then
\begin{equation}\label{e30}
\|V_*f\|_{p,\infty} \leq C_p \|f\|_{H_p^\Box} \qquad (f\in
H_p^\Box(\T^d)\cap L_1(\T^d)).
\end{equation}
If $\lim_{k\to\infty} f_k= f$ in the $H_p^\Box$-norm implies that
$\lim_{k\to\infty} V_nf_k=V_nf$ in the sense of distributions $(n\in
\N^d)$, then (\ref{e30}) holds for all $f\in H_p^\Box(\T^d)$.
\end{thm}

The next result follows from inequality (\ref{e32}).

\begin{cor}\label{c20}
If $p_0<1$ in Theorem \ref{t18}, then for all $f \in L_1(\T^d)$,
$$
\sup_{\rho>0} \rho \, \lambda(|V_*f| > \rho) \leq C \|f\|_{1}.
$$
\end{cor}

Theorem \ref{t18} and Corollary \ref{c20} can be regarded also as an
alternative tool to the \ind{Calderon-Zygmund decomposition} lemma
for proving weak type $(1,1)$ inequalities. In many cases, this
method can be applied better and more simply than the
Calderon-Zygmund decomposition lemma.

\sect{Almost everywhere convergence of the $\ell_q$-summability
means}

We define the \idword{maximal Riesz operator} by
$$
\sigma_*^{q,\alpha}f := \sup_{n \in \N} |\sigma_{n}^{q,\alpha}
f|.\index{\file-1}{$\sigma_*^{q,\alpha}f$}
$$
If $\alpha = 1$, we obtain the \idword{maximal Fej\'er operator} and
write it simply as $\sigma_*^qf$.\index{\file-1}{$\sigma_*^qf$}

Using Theorems \ref{t18} and \ref{t19} and some estimations of the
kernel functions, we can show that the maximal Riesz operator is
bounded from $H_{p}^\Box(\T^d)$ to $L_{p}(\T^d)$. This was proved by
Stein, Taibleson and Weiss \cite{stta} and Lu \cite{lu} for $q=2$,
by Oswald \cite{os3} for Fourier transforms and for $q=\infty$,
$\gamma=2$, by Weisz \cite{wk2,wamalg-hardy,wel1-fs2,wel1-ft2,wmar6}
for $q=1,2,\infty$ and more general summability methods.

\begin{thm}\label{t21}
If $q=1,\infty$, $\alpha\geq 1$ and $d/(d+1)<p \leq \infty$, then
\begin{equation}\label{e6}
\|\sigma_*^{q,\alpha} f\|_{p} \leq C_{p} \|f\|_{H_{p}^\Box} \qquad
(f\in H_{p}^\Box(\T^d))
\end{equation}
and for $f\in H_{d/(d+1)}^\Box(\T^d)$,
\begin{equation}\label{e7}
\|\sigma_*^{q,\alpha} f\|_{d/(d+1), \infty}= \sup_{\rho >0} \rho
\lambda(\sigma_*^{q,\alpha} f>\rho)^{(d+1)/d} \leq C
\|f\|_{H_{d/(d+1)}^\Box}.
\end{equation}
If $q=2$ and $\alpha>(d-1)/2$, then the same holds with the critical
index $d/(d/2+\alpha+1/2)$ instead of $d/(d+1)$.
\end{thm}

This theorem will be proved in Subsection \ref{s8.2}. Recall that
$H_p^\Box(\T^d) \sim L_p(\T^d)$ for $1<p \leq \infty$ and so
(\ref{e6}) yields
$$
\|\sigma_*^{q,\alpha} f\|_{p} \leq C_{p} \|f\|_{p} \qquad (f\in
L_{p}(\T^d),1<p \leq \infty).
$$
If $p$ is smaller than or equal to the critical index, then this
theorem is not true (Oswald \cite{os3}, Stein, Taibleson and Weiss
\cite{stta}).

\begin{thm}\label{t22}
If $q=\infty$ and $\alpha=1$ (resp. $q=2$ and $\alpha>(d-1)/2$),
then the operator $\sigma_*^{q,\alpha}$ is not bounded from
$H_p^\Box(\T^d)$ to $L_p(\T^d)$ if $p$ is smaller than or equal to
the critical index $d/(d+1)$ (resp. $d/(d/2+\alpha+1/2)$).
\end{thm}

Of course, (\ref{e7}) cannot be true for $p<d/(d+1)$, i.e.,
$\sigma_*^{q,\alpha}$ is not bounded from $H_p^\Box(\T^d)$ to the
weak $L_{p,\infty}(\T^d)$ space for $p<d/(d+1)$. If the operator was
bounded, then by interpolation (\ref{e6}) would hold for
$p=d/(d+1)$, which contradicts Theorem \ref{t22}.

Marcinkiewicz \cite{ma} verified for two-dimensional Fourier series
that the cubic (i.e., $q=\infty$) Fej{\'e}r means of a function
$f\in L \log L(\T^2)$ converge almost everywhere to $f$ as
$n\to\infty$. Later Zhizhiashvili \cite{zi2,zi1} extended this
result to all $f \in L_1(\T^2)$ and to Ces{\`a}ro means, Oswald
\cite{os3} to Fourier transform and $\gamma=2$ and the author
\cite{wmar6} to higher dimensions. The same result for $q=2$ can be
found in Stein and Weiss \cite{stwe}, Lu \cite{lu} and Weisz
\cite{wamalg-hardy}, for $q=1$ in Berens, Li and Xu \cite{bexu2} and
Weisz \cite{wel1-fs2,wel1-ft2}.

\begin{cor}\label{c23}
Suppose that  $q=1,\infty$ and $\alpha\geq 1$ or $q=2$ and
$\alpha>(d-1)/2$. If $f \in L_1(\T^d)$, then
$$
\sup_{\rho >0} \rho \, \lambda(\sigma_*^{q,\alpha} f > \rho) \leq C
\|f\|_{1}.
$$
\end{cor}

The density argument of Marcinkiewicz and Zygmund implies

\begin{cor}\label{c24}
Suppose that  $q=1,\infty$ and $\alpha\geq 1$ or $q=2$ and
$\alpha>(d-1)/2$. If $f \in L_1(\T^d)$, then
$$
\lim_{n\to\infty}\sigma_{n}^{q,\alpha}f=f \qquad \mbox{ a.e.}
$$
\end{cor}

\begin{proof}
Since the trigonometric polynomials are dense in $L_1(\T^2)$, the
corollary follows from Theorem \ref{t4} and Corollary \ref{c23}.
\end{proof}

\subsection{Further results for the Bochner-Riesz means}

The boundedness of the operator $\sigma_*^{2,\alpha}$ is complicated
and not completely solved if $q=2$ and $\alpha\leq (d-1)/2$. All
results of this section can be found in Grafakos \cite[Chapter
10]{gra}, so we omit the proofs. The following result is due to Tao
\cite{tao}.

\begin{thm}\label{t25}
Suppose that $d\geq 2$ and $q=\gamma=2$. If $0<\alpha\leq (d-1)/2$
and
$$
1<p<\frac{2d-1}{d+2\alpha} \qquad \mbox{or} \qquad
p>\frac{2d}{d-1-2\alpha},
$$
then the maximal operator $\sigma_*^{2,\alpha}$ is not bounded from
$L_p(\T^d)$ to $L_{p,\infty}(\T^d)$ (see Figure \ref{f5}).
\end{thm}

By Theorems \ref{t25} and \ref{t11}, Figure \ref{f5} shows the
region where $\sigma_*^{2,\alpha}$ is unbounded from $L_p(\T^d)$ to
$L_{p,\infty}(\T^d)$. Obviously, the operator $\sigma_*^{2,\alpha}$
is unbounded from $L_p(\T^d)$ to $L_{p}(\T^d)$ on the same region.
Note that the exact region of the boundedness or unboundedness of
$\sigma_*^{2,\alpha}$ is still unknown (see Figure \ref{f400}).

\begin{figure}[ht] 
   \centering
   \includegraphics[width=1\textwidth]{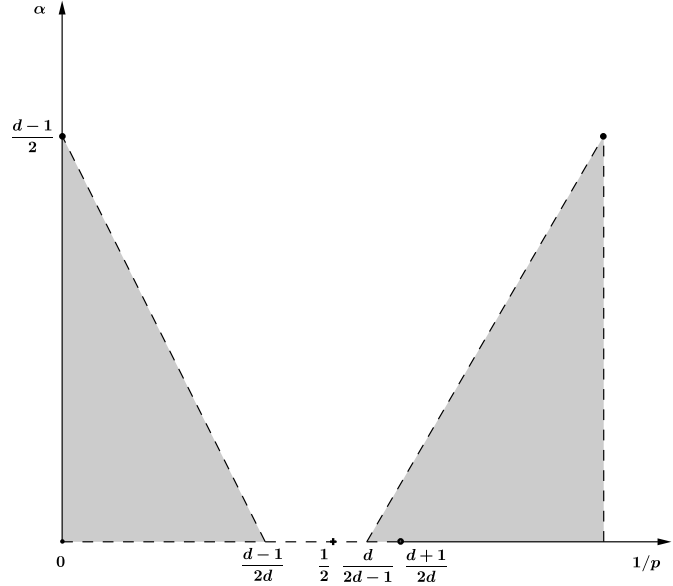}
   \caption{Unboundedness of $\sigma_*^{2,\alpha}$ from $L_p(\T^d)$ to $L_{p,\infty}(\T^d)$.} \label{f5}
\end{figure}

Stein \cite[p.~276]{stwe} proved that Theorem \ref{t260} holds also
for the maximal operator $\sigma_*^{2,\alpha}$.

\begin{thm}\label{t26}
Suppose that $d\geq 2$ and $q=\gamma=2$. If $0<\alpha\leq
\frac{d-1}{2}$ and
$$
\frac{2(d-1)}{d-1+2\alpha}<p<\frac{2(d-1)}{d-1-2\alpha},
$$
then the maximal Bochner-Riesz operator $\sigma_*^{2,\alpha}$ is
bounded on $L_p(\T^d)$ (see Figure \ref{f41}).
\end{thm}

Carbery improved this theorem in \cite{car} for $d=2$.

\begin{thm}\label{t27}
Suppose that $d=2$ and $q=\gamma=2$. If $0<\alpha\leq 1/2$ and
$$
\frac{2}{1+2\alpha}<p<\frac{4}{1-2\alpha},
$$
then the maximal Bochner-Riesz operator $\sigma_*^{2,\alpha}$ is
bounded on $L_p(\T^d)$ (see Figure \ref{f6}).
\end{thm}

\begin{figure}[ht] 
   \centering \includegraphics[width=1\textwidth]{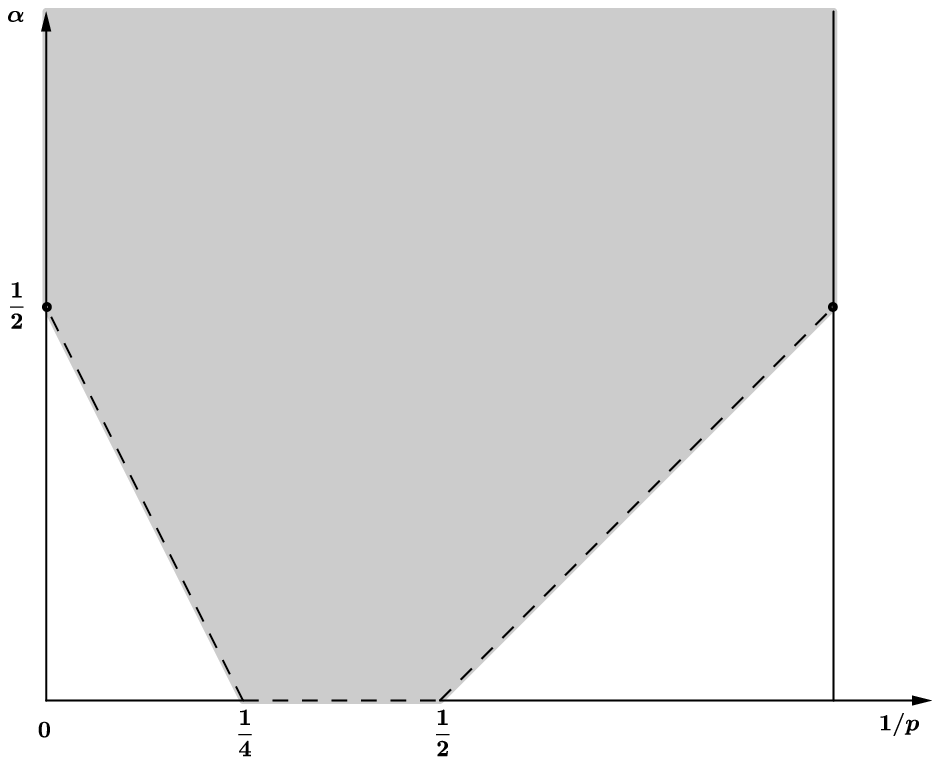}
   \caption{Boundedness of $\sigma_*^{2,\alpha}$ from $L_p(\T^d)$ to $L_{p,\infty}(\T^d)$ when $d=2$.}
   \label{f6}
\end{figure}

Christ \cite{chr1} generalized this result to higher dimensions.

\begin{thm}\label{t28}
Suppose that $d\geq 3$ and $q=\gamma=2$. If $\frac{d-1}{2(d+1)}\leq
\alpha\leq \frac{d-1}{2}$ and
$$
\frac{2(d-1)}{d-1+2\alpha}<p<\frac{2d}{d-1-2\alpha},
$$
then the maximal Bochner-Riesz operator $\sigma_*^{2,\alpha}$ is
bounded on $L_p(\T^d)$ (see Figure \ref{f7}).
\end{thm}

\begin{figure}[ht] 
   \centering
   \includegraphics[width=1\textwidth]{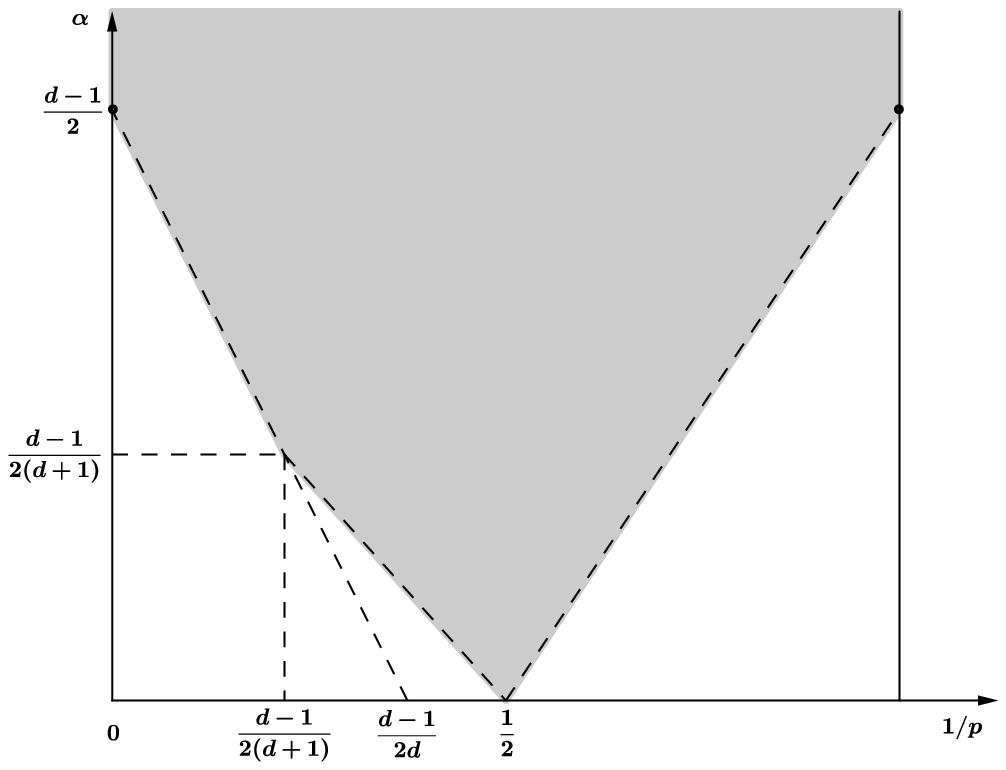}
   \caption{Boundedness of $\sigma_*^{2,\alpha}$ from $L_p(\T^d)$ to $L_{p,\infty}(\T^d)$ when $d\geq 3$.}
   \label{f7}
\end{figure}

The following result follows from analytic interpolation and from
Theorems \ref{t26} and \ref{t28} (see e.g.~ Stein and Weiss
\cite[p.~276, p.~205]{stwe}).

\begin{thm}\label{t1600}
Suppose that $d\geq 3$ and $q=\gamma=2$. If
$0<\alpha<\frac{d-1}{2(d+1)}$ and
$$
\frac{2(d-1)}{d-1+2\alpha}<p<\frac{2(d-1)}{d-1-4\alpha},
$$
then the maximal Bochner-Riesz operator $\sigma_*^{2,\alpha}$ is
bounded on $L_p(\T^d)$ (see Figure \ref{f7}).
\end{thm}

It is still an open question as to whether $\sigma_*^{2,\alpha}$ is
bounded or unbounded in the region of Figure \ref{f400}. If $d=2$,
then the question is open on the right hand side of the region of
Figure \ref{f400} only, i.e., for $1/p\geq 1/2$.

\begin{figure}[ht] 
   \centering
   \includegraphics[width=1\textwidth]{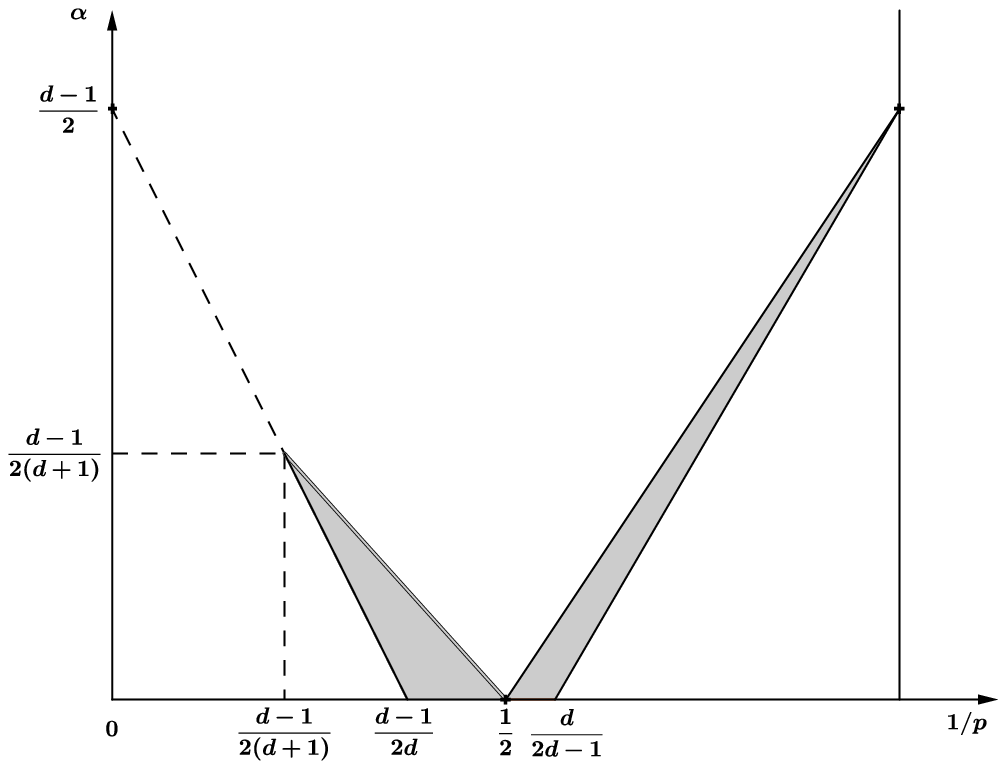}
   \caption{Open question of the boundedness of $\sigma_*^{2,\alpha}$ when $d\geq 3$.}
   \label{f400}
\end{figure}

Of course, in Theorems \ref{t26}--\ref{t1600}, $\sigma_*^{2,\alpha}$
is also bounded from $L_p(\T^d)$ to $L_{p,\infty}(\T^d)$. This
implies that
$$
\lim_{n\to\infty}\sigma_{n}^{q,\alpha}f=f \qquad \mbox{ a.e.}
$$
Figures \ref{f6} and \ref{f7} show the region where
$\sigma_*^{2,\alpha}$ is bounded from $L_p(\T^d)$ to
$L_{p,\infty}(\T^d)$ and almost everywhere convergence hold.

By Theorem \ref{t4}, if $\sigma_*^{q,\alpha}$ is of \idword{weak
type $(p,p)$}, i.e., it is bounded from $L_p(\T^d)$ to
$L_{p,\infty}(\T^d)$, then $\lim_{n\to\infty}
\sigma_n^{q,\alpha}f=f$ almost everywhere for all $f\in L_p(\T^d)$.
The converse is also true when $1\leq p\leq 2$: almost everywhere
convergence implies that the corresponding maximal operator is of
weak type $(p,p)$. More exactly, if $X=L_p$, $1\leq p\leq 2$ and
$T_n$ commutes with translation, then the converse of Theorem
\ref{t4} holds (see Stein \cite{st2}). However, this result is no
longer true for $p>2$. The preceding theorems concerning the almost
everywhere convergence were generalized by Carbery, Rubio de Francia
and Vega \cite{caru} (see Figure \ref{f8}).

\begin{thm}\label{t29}
Suppose that $d\geq 2$ and $q=\gamma=2$. If $0<\alpha\leq (d-1)/2$
and
$$
\frac{2(d-1)}{d-1+2\alpha}<p<\frac{2d}{d-1-2\alpha},
$$
then for all $f\in L_p(\T^d)$
$$
\lim_{n\to\infty}\sigma_{n}^{q,\alpha}f=f \qquad \mbox{ a.e.{}}
$$
(see Figure \ref{f8}).
\end{thm}

\begin{figure}[ht] 
   \centering
   \includegraphics[width=1\textwidth]{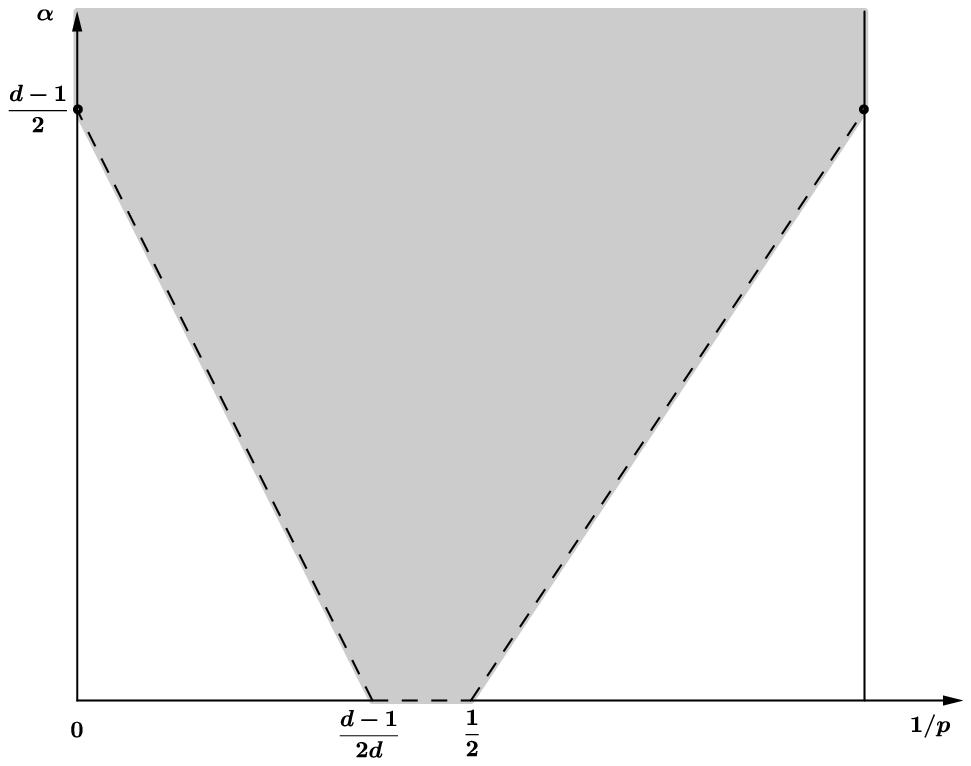}
   \caption{Almost everywhere convergence of $\sigma_{n}^{q,\alpha}f$, $f\in L_p(\T^d)$.}
   \label{f8}
\end{figure}

\subsection{Proof of Theorem \ref{t21}}\label{s8.2}

In this section, we will prove Theorem \ref{t21} in four
subsections. If $q=1,\infty$, then we suppose here that
$\alpha=\gamma=1$. For other parameters, see Subsection \ref{s10.1}.

\subsubsection{Proof for $q=1$ in the two-dimensional case}

\begin{proof*}{Theorem \ref{t21}  for $q=1$ and $d=2$}
First, we will show that
\begin{equation}\label{e61.16}
\int_{\T^2} |\sigma_*^1 a(x,y)|^p \dd x\dd y = \int_{\T^2}
\sup_{n\geq 1} \Big| \int_{I} a(u,v) K_n^1(x-u,y-v) \dd u\dd v
\Big|^p \dd x\dd y\leq C_p
\end{equation}
for every $H_p^\Box$-atom $a$, where $2/3<p<1$ and $I$ is the
support of the atom. By Theorems \ref{t13} and \ref{t18}, this will
imply (\ref{e6}). Without loss of generality, we can suppose that
$a$ is a $H_p^\Box$-atom with support $I=I_1\times I_2$ and
$$
[-2^{-K-2}, 2^{-K-2}] \subset I_j \subset [-2^{-K-1},
2^{-K-1}]\qquad (j=1,2)
$$
for some $K\in\N$. By symmetry, we can assume that $\pi>x-u>y-v>0$,
and so, instead of (\ref{e61.16}), it is enough to show that
$$
\int_{\T^2} \sup_{n\geq 1} \Big| \int_{I} a(u,v) K_n^1(x-u,y-v)
1_{A_i}(x-u,y-v) \dd u\dd v \Big|^p \dd x\dd y\leq C_p
$$
for all $i=1,\ldots,10$, where
\begin{eqnarray*}
A_1&:=&\{(x,y): 0<x\leq 2^{-K+5}, 0<y<x<\pi,y\leq \pi/2\}, \\
A_2&:=&\{(x,y): 2^{-K+5}<x<\pi, 0<y\leq 2^{-K+2},y\leq \pi/2\}, \\
A_3&:=&\{(x,y): 2^{-K+5}<x<\pi, 2^{-K+2}<y\leq x/2,y\leq \pi/2\}, \\
A_4&:=&\{(x,y): 2^{-K+5}<x<\pi, x/2<y\leq x-2^{-K+2},y\leq \pi/2\}, \\
A_5&:=&\{(x,y): 2^{-K+5}<x<\pi, x-2^{-K+2}<y<x,y\leq \pi/2\}\\
A_6&:=&\{(x,y): y>\pi/2,\pi-2^{-K+5}\leq y<\pi, 0<y<x<\pi\}, \\
A_7&:=&\{(x,y): \pi/2<y<\pi-2^{-K+5}, \pi-2^{-K+2}<x<\pi\}, \\
A_8&:=&\{(x,y): \pi/2<y<\pi-2^{-K+5}, (\pi+y)/2<x\leq \pi-2^{-K+2}\}, \\
A_9&:=&\{(x,y): \pi/2<y<\pi-2^{-K+5}, y+2^{-K+2}<x\leq (\pi+y)/2\}, \\
A_{10}&:=&\{(x,y): \pi/2<y<\pi-2^{-K+5}, y<x\leq y+2^{-K+2}\}.
\end{eqnarray*}
These sets are similar to those in Theorem \ref{t12} (see Figure
\ref{f11}).

First of all, if $0<x-u\leq 2^{-K+5}$, then $-2^{-K-1}<x\leq
2^{-K+6}$ and the same holds for $y$. If $\pi-2^{-K+5}\leq y-v<\pi$,
then $\pi-2^{-K+6}<y\leq \pi+2^{-K-1}$ and the same is true for $x$.
By the definition of the atom and by Theorem \ref{t12},
$$
\int_{\T^2} \sup_{n\geq 1} \Big| \int_{I} a(u,v) K_n^1(x-u,y-v)
1_{A_1\cup A_6}(x-u,y-v)\dd u\dd v \Big|^p \dd x\dd y \leq  C_p
2^{2K} 2^{-2K}.
$$

Considering the set $A_2$, we have $2^{-K+5}<x-u<\pi$ and $0<y-v\leq
2^{-K+2}$, thus $2^{-K+4}<x<\pi+2^{-K-1}$ and $-2^{-K-1}<y\leq
2^{-K+3}$. Using (\ref{e61.2}), we conclude
\begin{eqnarray}\label{e62.30}
\lefteqn{\Big| \int_{I} a(u,v) K_n^1(x-u,y-v) 1_{A_2}(x-u,y-v)\dd u\dd v \Big| } \n\\
&\leq& C_p 2^{2K/p} \int_{I} (x-u-y+v)^{-3/2} (y-v)^{-1/2} 1_{A_2}(x-u,y-v)\dd u\dd v \n\\
&\leq& C_p 2^{2K/p} 1_{\{2^{-K+4}<x<\pi+2^{-K-1}\}}1_{\{-2^{-K-1}<y\leq 2^{-K+3}\}}\n\\
&&{}
\int_{I} (x-2^{-K+3})^{-3/2} (y-v)^{-1/2} \dd u\dd v \n\\
&\leq& C_p 2^{2K/p-3K/2}
1_{\{2^{-K+4}<x<\pi+2^{-K-1}\}}1_{\{-2^{-K-1}<y\leq 2^{-K+3}\}}
(x-2^{-K+3})^{-3/2}.
\end{eqnarray}
Similarly, on $A_7$, $\pi/2<y-v<\pi-2^{-K+5}$ and
$\pi-2^{-K+2}<x-u<\pi$, thus $\pi/2-2^{-K-1}<y<\pi-2^{-K+4}$ and
$\pi-2^{-K+3}<x<\pi+ 2^{-K-1}$. By (\ref{e61.2}),
\begin{eqnarray*}
\lefteqn{\Big| \int_{I} a(u,v) K_n^1(x-u,y-v) 1_{A_7}(x-u,y-v)\dd u\dd v \Big| } \n\\
&\leq& C_p 2^{2K/p} \int_{I} (x-u-y+v)^{-3/2} (\pi-x+u)^{-1/2} 1_{A_7}(x-u,y-v)\dd u\dd v\\
&\leq& C_p 2^{2K/p} 1_{\pi/2-2^{-K-1}<y<\pi-2^{-K+4}}1_{\pi-2^{-K+3}<x<\pi+ 2^{-K-1}}\\
&&{}
\int_{I} (\pi-2^{-K+3}-y)^{-3/2} (\pi-x+u)^{-1/2} \dd u\dd v\\
&\leq& C_p 2^{2K/p-3K/2}
1_{\pi/2-2^{-K-1}<y<\pi-2^{-K+4}}1_{\pi-2^{-K+3}<x<\pi+ 2^{-K-1}}
(\pi-2^{-K+3}-y)^{-3/2}.
\end{eqnarray*}
If $p>2/3$, then
\begin{eqnarray*}
\lefteqn{\int_{\T^2} \sup_{n\geq 1} \Big|
\int_{I} a(u,v) K_n^1(x-u,y-v) 1_{A_2\cup A_7}(x-u,y-v)\dd u\dd v \Big|^p \dd x\dd y } \quad \n\\
&\leq& C_p 2^{2K-3Kp/2} \int_{2^{-K+4}}^{\pi+2^{-K-1}}
\int_{-2^{-K-1}}^{2^{-K+3}}
(x-2^{-K+3})^{-3p/2}\dd x\dd y \\
&&{}+C_p 2^{2K-3Kp/2} \int_{\pi/2-2^{-K-1}}^{\pi-2^{-K+4}}
\int_{\pi-2^{-K+3}}^{\pi+ 2^{-K-1}}
(\pi-2^{-K+3}-y)^{-3p/2}\dd y\dd x \\
&\leq& C_p.
\end{eqnarray*}

We may suppose that the center of $I$ is zero, in other words $I:=
(-\nu,\nu)\times (-\nu,\nu)$. Let
$$
A_1(u,v):= \int_{-\nu}^{u} a(t,v) \dd t \qquad \mbox{and} \qquad
A_{2}(u,v):= \int_{-\nu}^{v} A_{1}(u,t) \dd t.
$$
Observe that
$$
|A_k(u,v)| \leq C_p 2^{K(2/p-k)}.
$$
Integrating by parts, we can see that
\begin{eqnarray*}
\lefteqn{\int_{I_1} a(u,v) K_n^1(x-u,y-v) 1_{A_3\cup A_8}(x-u,y-v)\dd u } \n\\
&=& A_1(\nu,v) K_n^1(x-\nu,y-v)1_{A_3\cup A_8}(x-\nu,y-v) \\
&&{}+ \int_{-\nu}^\nu A_1(u,v) \partial_1 K_n^1(x-u,y-v) 1_{A_3\cup
A_8}(x-u,y-v)\dd u,
\end{eqnarray*}
because $A_1(-\nu,v)=0$. As $A_2(\nu,\nu)=\int_{I}a=0$, integrating
the first term again by parts, we obtain
\begin{eqnarray*}
\lefteqn{\int_{I_1} \int_{I_2} a(u,v) K_n^1(x-u,y-v) 1_{A_3\cup A_8}(x-u,y-v)\dd u\dd v } \n\\
&=& \int_{-\nu}^\nu A_2(\nu,v) \partial_2 K_n^1(x-\nu,y-v)1_{A_3\cup A_8}(x-\nu,y-v) \dd v\\
&&{}+ \int_{I_1}\int_{I_2} A_1(u,v) \partial_1 K_n^1(x-u,y-v)
1_{A_3\cup A_8}(x-u,y-v)\dd u \dd v.
\end{eqnarray*}
Since $x-u-y+v>(x-u)/2$ on the set $A_3$ and $x-u-y+v>(\pi-y+v)/2$
on the set $A_8$, we get from (\ref{e61.6}) that
\begin{eqnarray}\label{e62.31}
\lefteqn{\Big| \int_{I} a(u,v) K_n^1(x-u,y-v) 1_{A_3}(x-u,y-v)\dd u\dd v \Big| } \n\\
&\leq& C_p 2^{2K/p-2K} \int_{I_2}  (x-\nu)^{-1-\beta}
(y-v)^{\beta-2}
1_{A_3}(x-\nu,y-v)\dd v\n\\
&&{} +C_p 2^{2K/p-K} \int_{I}  (x-u)^{-1-\beta} (y-v)^{\beta-2}
1_{A_3}(x-u,y-v)\dd u\dd v\n\\
&\leq& C_p 2^{2K/p-3K} 1_{\{2^{-K+4}<x<\pi+2^{-K-1}\}}
1_{\{2^{-K+1}<y\leq x/2+2^{-K}\}} \nonumber \\
&&{} (x-2^{-K-1})^{-1-\beta} (y-2^{-K-1})^{\beta-2}
\end{eqnarray}
and
\begin{eqnarray*}
\lefteqn{\Big| \int_{I} a(u,v) K_n^1(x-u,y-v) 1_{A_8}(x-u,y-v)\dd u\dd v \Big| } \n\\
&\leq& C_p 2^{2K/p-2K} \int_{I_2}  (\pi-y+v)^{-1-\beta}
(\pi-x+\nu)^{\beta-2}
1_{A_8}(x-\nu,y-v)\dd v\\
&&{} +C_p 2^{2K/p-K} \int_{I} (\pi-y+v)^{-1-\beta}
(\pi-x+u)^{\beta-2}
1_{A_8}(x-u,y-v)\dd u\dd v\\
&\leq& C_p 2^{2K/p-3K} 1_{\pi/2-2^{-K-1}<y<\pi-2^{-K+4}}
1_{(\pi+y)/2-2^{-K}<x<\pi-2^{-K+1}}\\
&&{} (\pi-y-2^{-K-1})^{-1-\beta} (\pi-x-2^{-K-1})^{\beta-2}.
\end{eqnarray*}
Choosing $\beta=1/2$, we conclude
\begin{eqnarray*}
\lefteqn{\int_{\T^2} \sup_{n\geq 1} \Big| \int_{I} a(u,v)
K_n^1(x-u,y-v)
1_{A_3\cup A_8}(x-u,y-v)\dd u\dd v \Big|^p \dd x\dd y } \n\\
&\leq& C_p 2^{2K-3Kp} \Big(\int_{2^{-K+4}}^{\pi+2^{-K-1}}
\int_{2^{-K+1}}^{x/2+2^{-K}}
(x-2^{-K-1})^{-3p/2} (y-2^{-K-1})^{-3p/2} \dd x\dd y\n\\
&&{}+\int_{\pi/2-2^{-K-1}}^{\pi-2^{-K+4}}
\int_{(\pi+y)/2-2^{-K}}^{\pi-2^{-K+1}}
(\pi-y-2^{-K-1})^{-3p/2} (\pi-x-2^{-K-1})^{-3p/2} \dd x\dd y \Big)\n\\
&\leq& C_p 2^{2K-3Kp} 2^{-K(1-3p/2)} 2^{-K(1-3p/2)} \\
&\leq& C_p.
\end{eqnarray*}

Since $y-v>(x-u)/2$ on $A_4$ and $\pi-x+u>(\pi-y+v)/2$ on the set
$A_9$, (\ref{e61.6}) implies
\begin{eqnarray}\label{e62.32}
\lefteqn{\Big| \int_{I} a(u,v) K_n^1(x-u,y-v) 1_{A_4}(x-u,y-v)\dd u\dd v \Big| } \n\\
&\leq& C_p 2^{2K/p-2K} \int_{I_2}  (x-\nu-y+v)^{-1-\beta}
(x-u)^{\beta-2}
1_{A_4}(x-\nu,y-v)\dd v\n\\
&&{} +C_p 2^{2K/p-K} \int_{I}  (x-u-y+v)^{-1-\beta} (x-u)^{\beta-2}
1_{A_4}(x-u,y-v)\dd u\dd v \n\\
&\leq& C_p 2^{2K/p-3K} 1_{\{2^{-K+4}<x<\pi+2^{-K-1}\}}
1_{\{x/2-2^{-K}<y\leq x-2^{-K+1}\}} \n\\
&&{} (x-y-2^{-K})^{-1-\beta} (x-2^{-K-1})^{\beta-2}
\end{eqnarray}
and
\begin{eqnarray*}
\lefteqn{\Big| \int_{I} a(u,v) K_n^1(x-u,y-v) 1_{A_9}(x-u,y-v)\dd u\dd v \Big| } \n\\
&\leq& C_p 2^{2K/p-2K} \int_{I_2}  (x-\nu-y+v)^{-1-\beta}
(\pi-y+v)^{\beta-2}
1_{A_8}(x-\nu,y-v)\dd v\\
&&{} +C_p 2^{2K/p-K} \int_{I} (x-u-y+v)^{-1-\beta}
(\pi-y+v)^{\beta-2}
1_{A_9}(x-u,y-v)\dd u\dd v\\
&\leq& C_p 2^{2K/p-3K} 1_{\pi/2-2^{-K-1}<y<\pi-2^{-K+4}}
1_{y+2^{-K+1}<x<(\pi+y)/2+2^{-K}}\\
&&{} (x-y-2^{-K})^{-1-\beta} (\pi-y-2^{-K-1})^{\beta-2}
\end{eqnarray*}
as before. Choosing again $\beta=1/2$, we obtain
\begin{eqnarray*}
\lefteqn{\int_{\T^2} \sup_{n\geq 1} \Big| \int_{I} a(u,v)
K_n^1(x-u,y-v)
1_{A_4\cup A_9}(x-u,y-v)\dd u\dd v \Big|^p \dd x\dd y } \n\\
&\leq& C_p 2^{2K-3Kp} \Big(\int_{2^{-K+4}}^{\pi+2^{-K-1}}
\int_{x/2-2^{-K}}^{x-2^{-K+1}}
(x-y-2^{-K})^{-3p/2} (x-2^{-K-1})^{-3p/2} \dd x\dd y\n\\
&&{}+\int_{\pi/2-2^{-K-1}}^{\pi-2^{-K+4}}
\int_{y+2^{-K+1}}^{(\pi+y)/2+2^{-K}}
(x-y-2^{-K})^{-3p/2} (\pi-y-2^{-K-1})^{-3p/2} \dd x\dd y \Big)\n\\
&\leq& C_p 2^{2K-3Kp} 2^{-K(1-3p/2)} 2^{-K(1-3p/2)}\\
&\leq& C_p.
\end{eqnarray*}

Finally, inequality (\ref{e61.4}) implies
\begin{eqnarray*}
\lefteqn{\Big| \int_{I} a(u,v) K_n^1(x-u,y-v) 1_{A_5}(x-u,y-v)\dd u\dd v \Big| } \n\\
&\leq& C_p 2^{2K/p} \int_{I} (y-2^{-K-1})^{-2} 1_{A_5}(x-u,y-v)\dd u\dd v\\
&\leq& C_p 2^{2K/p-2K} 1_{\{2^{-K+4}<x<\pi+2^{-K-1}\}}
1_{\{x-2^{-K+3}<y\leq x+2^{-K}\}} (y-2^{-K-1})^{-2}
\end{eqnarray*}
and
\begin{eqnarray*}
\lefteqn{\Big| \int_{I} a(u,v) K_n^1(x-u,y-v) 1_{A_{10}}(x-u,y-v)\dd u\dd v \Big| } \n\\
&\leq& C_p 2^{2K/p} \int_{I} (\pi-x-2^{-K-1})^{-2} 1_{A_{10}}(x-u,y-v)\dd u\dd v\\
&\leq& C_p 2^{2K/p-2K} 1_{\pi/2-2^{-K-1}<y<\pi-2^{-K+4}}
1_{y-2^{-K}<x<y+2^{-K+3}}(\pi-x-2^{-K-1})^{-2},
\end{eqnarray*}
hence
\begin{eqnarray*}
\lefteqn{\int_{\T^2} \sup_{n\geq 1} \Big| \int_{I} a(u,v)
K_n^1(x-u,y-v)
1_{A_5\cup A_{10}}(x-u,y-v)\dd u\dd v \Big|^p \dd x\dd y } \n\\
&\leq& C_p 2^{2K-2Kp} \int_{2^{-K+4}}^{\pi+2^{-K-1}}
\int_{x-2^{-K+3}}^{x+2^{-K}}
(y-2^{-K-1})^{-2p}\dd x\dd y\\
&&{}+C_p 2^{2K-2Kp} \int_{\pi/2-2^{-K-1}}^{\pi-2^{-K+4}}
\int_{y-2^{-K}}^{y+2^{-K+3}}
(\pi-x-2^{-K-1})^{-2p}\dd x\dd y\\
&\leq& C_p 2^{2K-2Kp} \int_{2^{-K+3}}^{\pi+2^{-K+5}}
\int_{y-2^{-K}}^{y+2^{-K+3}}
(y-2^{-K-1})^{-2p}\dd x\dd y\\
&&{}+C_p 2^{2K-2Kp} \int_{\pi/2-2^{-K+1}}^{\pi-2^{-K+3}}
\int_{x+2^{-K}}^{x-2^{-K+3}}
(\pi-x-2^{-K-1})^{-2p}\dd y\dd x\\
&\leq& C_p,
\end{eqnarray*}
which finishes the proof of (\ref{e6}).

Next, we will verify the weak inequality (\ref{e7}). To this end, we
use Theorem \ref{t19} and prove that
$$
\sup_{\rho>0}\rho^{2/3}\lambda(\sigma_*^1 a > \rho) \leq C
$$
for all $H_{2/3}^\Box$-atom $a$. Since
$$
\rho^{2/3}\lambda(|g| > \rho) \leq \int |g|^{2/3},
$$
taking into account the above inequalities, we have to show that
$$
\rho^{2/3}\lambda\Big(\sup_{n\geq 1} \Big| \int_{I} a(u,v)
K_n^1(x-u,y-v) 1_{A_i}(x-u,y-v)\dd u\dd v \Big| > \rho\Big) \leq C
$$
for $i=2,3,4,7,8,9$ and $\rho>0$. We will prove the inequality for
the first three sets. For the last three, the proof is similar.

If the expression in (\ref{e62.30}) with $p=2/3$ is greater than
$\rho$, then
$$
1_{\{2^{-K+4}<x<\pi+2^{-K-1}\}} (x-2^{-K+3}) \leq C \rho^{-2/3}
2^{2K-K} 1_{\{-2^{-K-1}<y\leq 2^{-K+3}\}}
$$
and
\begin{eqnarray*}
\lefteqn{\lambda\Big(\sup_{n\geq 1} \Big| \int_{I} a(u,v)
K_n^1(x-u,y-v)
1_{A_2}(x-u,y-v)\dd u\dd v \Big| > \rho\Big) } \qquad \qquad \qquad \qquad \qquad \qquad \qquad \n\\
&\leq& C \rho^{-2/3} 2^{K} \int 1_{\{-2^{-K-1}<y \leq 2^{-K+3}\}} \dd y \\
&=&C \rho^{-2/3}.
\end{eqnarray*}
If (\ref{e62.31}) is greater than $\rho$, then
$$
1_{\{2^{-K+1}<y\leq x/2+2^{-K}\}}(y-2^{-K-1}) \leq C
\rho^{-\frac{1}{2-\beta}} 1_{\{2^{-K+4}<x<\pi+2^{-K-1}\}}
(x-2^{-K-1})^{-\frac{1+\beta}{2-\beta}}.
$$
Choosing $\beta$ such that $-\frac{1+\beta}{2-\beta}+1<0$, i.e.,
$1/2<\beta\leq 1$, we obtain
\begin{eqnarray*}
\lefteqn{\lambda\Big(\sup_{n\geq 1} \Big| \int_{I} a(u,v)
K_n^1(x-u,y-v)
1_{A_3}(x-u,y-v)\dd u\dd v \Big| > \rho\Big) } \n\\
&&{} \leq \int_{2^{-K+4}}^{\rho^{-1/3}+2^{-K-1}} x \dd x + C \rho^{-\frac{1}{2-\beta}} \int_{\rho^{-1/3}+2^{-K-1}}^{2\pi} (x-2^{-K-1})^{-\frac{1+\beta}{2-\beta}}  \dd x \n\\
&&{} \leq C \rho^{-2/3} + C \rho^{-\frac{1}{2-\beta}} \rho^{\frac{-1}{3}(-\frac{1+\beta}{2-\beta}+1)} \\
&=&C \rho^{-2/3}.
\end{eqnarray*}

For $A_4$, we get from (\ref{e62.32}) that
\begin{eqnarray*}
\lefteqn{1_{\{x/2-2^{-K}<y\leq x-2^{-K+1}\}} (x-y-2^{-K}) } \qquad
\qquad \n\\ &\leq& C \rho^{-\frac{1}{1+\beta}}
1_{\{2^{-K+4}<x<\pi+2^{-K-1}\}}
(x-2^{-K-1})^{\frac{\beta-2}{1+\beta}}.
\end{eqnarray*}
Hence
\begin{eqnarray*}
\lefteqn{\lambda\Big(\sup_{n\geq 1} \Big| \int_{I} a(u,v)
K_n^1(x-u,y-v)
1_{A_4}(x-u,y-v)\dd u\dd v \Big| > \rho\Big) } \qquad \quad \n\\
&\leq& \int_{2^{-K+4}}^{\rho^{-1/3}+2^{-K-1}} x \dd x + C \rho^{-\frac{1}{1+\beta}} \int_{\rho^{-1/3}+2^{-K-1}}^{2\pi} (x-2^{-K-1})^{\frac{\beta-2}{1+\beta}}  \dd x \n\\
&\leq& C \rho^{-2/3} + C \rho^{-\frac{1}{1+\beta}} \rho^{\frac{-1}{3}(\frac{\beta-2}{1+\beta}+1)} \\
&=&C \rho^{-2/3}.
\end{eqnarray*}
Here, we have chosen $\beta$ such that
$\frac{\beta-2}{1+\beta}+1<0$, i.e., $0<\beta\leq 1$. The proof of
the theorem is complete.
\end{proof*}

\subsubsection{Proof for $q=1$ in higher dimensions ($d\geq 3$)}

\begin{proof*}{Theorem \ref{t21}  for $q=1$ and $d\geq 3$}
To prove (\ref{e6}), we will show that
\begin{equation}\label{e62.11}
\int_{\T^d} |\sigma^1_* a(x)|^p \dd x = \int_{\T^d} \sup_{n\geq 1}
\Big| \int_{I} a(u) K_n^1(x-u) \dd u \Big|^p \dd x\leq C_p
\end{equation}
for every $H_p^\Box$-atom $a$, where $\frac{d}{d+1}<p<1$ and $I$ is
the support of the atom. We may suppose that $a$ is a
$H_p^\Box$-atom with support $I=I_1\times\cdots\times I_d$ and
\begin{equation}\label{e62.1}
[-2^{-K-2}, 2^{-K-2}] \subset I_j \subset [-2^{-K-1},
2^{-K-1}]\qquad (j=1,\ldots,d)
\end{equation}
for some $K\in\N$. By symmetry, we can assume that
$$
\pi>x_1-u_1>x_2-u_2>\cdots>x_d-u_d>0.
$$
If $0<x_1-u_1\leq 2^{-K+4}$, then $-2^{-K-1}<x_1\leq 2^{-K+5}$ and
if $\pi-2^{-K+4}\leq x_d-u_d<\pi$, then $\pi-2^{-K+5}\leq
x_d<\pi+2^{-K-1}$. By the definition of the atom and by Theorem
\ref{t2},
$$
\int_{\T^d} \sup_{n\geq 1} \Big| \int_{I} a(u) K^1_n(x-u)
1_{\{2^{-K+4}\geq x_1-u_1>x_2-u_2>\cdots>x_d-u_d>0\}}\dd u \Big|^p
\dd x \leq  C_p 2^{Kd} 2^{-Kd}.
$$
We get the same inequality if we integrate over the set
$$
\{\pi>x_1-u_1>x_2-u_2>\cdots>x_d-u_d\geq \pi-2^{-K+4}\}.
$$
Hence, instead of (\ref{e62.11}), it is enough to prove that
$$
\int_{\T^d} \sup_{n\geq 1} \Big| \int_{I} a(u) K^1_n(x-u) 1_\cS(x-u)
\dd u \Big|^p \dd x\leq C_p,
$$
where
$$
\cS:=\{x\in \T^d:\pi>x_1>x_2>\cdots>x_d>0, x_1>2^{-K+4},
x_d<\pi-2^{-K+4}\}.
$$
Let
$$
\cS_{(i_l,j_l),k}:=\left\{
                   \begin{array}{ll}
                     x\in \cS: x_{i_l}-x_{j_l}>2^{-K+2}, l=1,\ldots,k-1,
x_{i_k}-x_{j_k}\leq 2^{-K+2} , & \hbox{if $k<d$;} \\
                     x\in \cS: x_{i_l}-x_{j_l}>2^{-K+2}, l=1,\ldots,d-1, & \hbox{if $k=d$}
                   \end{array}
                 \right.
$$
and
$$
\cS_{(i_l,j_l),k,1}:=\left\{
                   \begin{array}{ll}
                     x\in \cS_{(i_l,j_l),k}: x_{j_k}>2^{-K+2}, x_{j_{d-1}}\leq \pi/2, & \hbox{if $k<d$;} \\
                     x\in \cS_{(i_l,j_l),k}: x_{j_{d-1}}>2^{-K+2}, x_{j_{d-1}}\leq \pi/2, & \hbox{if $k=d$,}
                   \end{array}
                 \right.
$$
$$
\cS_{(i_l,j_l),k,2}:=\left\{
                   \begin{array}{ll}
                     x\in \cS_{(i_l,j_l),k}: x_{j_k}\leq 2^{-K+2}, x_{j_{d-1}}\leq \pi/2, & \hbox{if $k<d$;} \\
                     x\in \cS_{(i_l,j_l),k}: x_{j_{d-1}}\leq 2^{-K+2}, x_{j_{d-1}}\leq \pi/2, & \hbox{if $k=d$,}
                   \end{array}
                 \right.
$$

$$
\cS_{(i_l,j_l),k,3}:=\left\{
                   \begin{array}{ll}
                     x\in \cS_{(i_l,j_l),k}: \pi-x_{i_k}>2^{-K+2}, x_{j_{d-1}}>\pi/2, & \hbox{if $k<d$;} \\
                     x\in \cS_{(i_l,j_l),k}: \pi-x_{i_{d-1}}>2^{-K+2}, x_{j_{d-1}}>\pi/2, & \hbox{if $k=d$,}
                   \end{array}
                 \right.
$$
$$
\cS_{(i_l,j_l),k,4}:=\left\{
                   \begin{array}{ll}
                     x\in \cS_{(i_l,j_l),k}: \pi-x_{i_k}\leq 2^{-K+2}, x_{j_{d-1}}>\pi/2, & \hbox{if $k<d$;} \\
                     x\in \cS_{(i_l,j_l),k}: \pi-x_{i_{d-1}}\leq 2^{-K+2}, x_{j_{d-1}}>\pi/2, & \hbox{if $k=d$.}
                   \end{array}
                 \right.
$$
Then
\begin{eqnarray}\label{e62.14} \lefteqn{\int_{\T^d} \sup_{n\geq 1} \Big|
\int_{I} a(u) K^1_n(x-u) 1_\cS(x-u) \dd u \Big|^p \dd x } \n\\
&\leq & \sum_{(i_l,j_l)\in \cI} \sum_{k=1}^{d} \sum_{m=1}^{4}
\int_{\T^d} \sup_{n\geq 1} \Big| \int_{I} a(u)K^1_{n,(i_l,j_l)}(x-u)
1_{\cS_{(i_l,j_l),k,m}}(x-u)\dd u \Big|^p \dd x.
\end{eqnarray}

\eword{Step 1.} In this step, we estimate the first $d-1$ summands
in (\ref{e62.14}) on the set $\cS_{(i_l,j_l),k,1}$ by
$$
C_p 2^{Kd} \sum_{(i_l,j_l)\in \cI} \sum_{k=1}^{d-1} \int_{\T^d}
\sup_{n\geq 1} \Big( \int_{I} |K^1_{n,(i_l,j_l)}(x-u)|
1_{\cS_{(i_l,j_l),k,1}}(x-u)\dd u \Big)^p \dd x.
$$
Since $x_{i_{d-1}}-u_{i_{d-1}}-(x_{j_{d-1}}-u_{j_{d-1}})\leq
x_{i_l}-u_{i_l}-(x_{j_l}-u_{j_l})$, (\ref{e62.12}) implies
\begin{eqnarray*}
\lefteqn{\int_{I} K^1_{n,(i_l,j_l)}(x-u) 1_{\cS_{(i_l,j_l),k,1}}(x-u)\dd u } \n\\
&\leq & C \int_{I}
\prod_{l=1}^{d-2}(x_{i_l}-u_{i_l}-(x_{j_l}-u_{j_l}))^{-1-\beta}
(x_{j_{d-1}}-u_{j_{d-1}})^{\beta(d-2)-2} 1_{\cS_{(i_l,j_l),k,1}}(x-u)\dd u\\
&\leq & C \int_{I}
\prod_{l=1}^{k-1}(x_{i_l}-u_{i_l}-(x_{j_l}-u_{j_l}))^{-1-\beta}
\prod_{l=k}^{d-2}(x_{i_l}-u_{i_l}-(x_{j_l}-u_{j_l}))^{-1-\beta+1/(d-k)} \\
&&{}
(x_{i_{d-1}}-u_{i_{d-1}}-(x_{j_{d-1}}-u_{j_{d-1}}))^{1/(d-k)-1} \nonumber \\
&&{} (x_{j_{d-1}}-u_{j_{d-1}})^{\beta(d-2)-2}
1_{\cS_{(i_l,j_l),k,1}}(x-u)\dd u.
\end{eqnarray*}
In the first product, we estimate the factors and in the second one,
we integrate. More exactly,
$$
x_{i_l}-u_{i_l}-(x_{j_l}-u_{j_l})>x_{i_l}-x_{j_l}-2^{-K}, \qquad
l=1,\ldots,k-1,
$$
and
$$
x_{j_{d-1}}-u_{j_{d-1}}>x_{j_{d-1}}-2^{-K-1}.
$$
For the integration, we first choose the index ${i_{d-1}}$
$(=i'_{d-1})$ and then ${i_{d-2}}$ if ${i_{d-2}}\neq {i_{d-1}}$ or
${j_{d-2}}$ if ${j_{d-2}}\neq {j_{d-1}}$. Repeating this process, we
get a sequence $(i_l',l=k,\ldots,d-1)$. Note that $i_l'\not=i_m'$ if
$l\not=m$, $l,m=k,\ldots,d-1$. We integrate the term
$$
(x_{i_k}-u_{i_k}-(x_{j_k}-u_{j_k}))^{-1-\beta+1/(d-k)} \quad
\mbox{in} \quad u_{i_k'},
$$
the term
$$
(x_{i_{k+1}}-u_{i_{k+1}}-(x_{j_{k+1}}-u_{j_{k+1}}))^{-1-\beta+1/(d-k)}
\quad \mbox{in} \quad u_{i_{k+1}'},
$$
\ldots, and finally the term
$$
(x_{i_{d-1}}-u_{i_{d-1}}-(x_{j_{d-1}}-u_{j_{d-1}}))^{-1-\beta+1/(d-k)}
\quad \mbox{in} \quad u_{i_{d-1}'}.
$$
Since $x_{i_l}-u_{i_l}-(x_{j_l}-u_{j_l})\leq 2^{-K+2}$
$(l=k,\ldots,d-1)$ and we can choose $\beta$ such that
$\beta<1/(d-1)$, we have
$$
\int_{I_l} (x_{i_l}-u_{i_l}-(x_{j_l}-u_{j_l}))^{-1-\beta+1/(d-k)}
1_{\cS_{(i_l,j_l),k,1}}(x-u)\dd u_{i_l} (\mbox{or} \dd u_{j_l})\leq
C 2^{-K(1/(d-k)-\beta)},
$$
$(l=k,\ldots,d-1)$. If $x-u\in \cS_{(i_l,j_l),k,1}$, then
$$
x_{i_l}-x_{j_l}> 2^{-K+2} +u_{i_l}-u_{j_l}>2^{-K+2}-2^{-K}>2^{-K+1},
\qquad l=1,\ldots,k-1,
$$
and
$$
x_{j_{d-1}}>2^{-K+2}+u_{j_{d-1}}>2^{-K+2}-2^{-K-1}>2^{-K+1}.
$$
Moreover,
$$
x_{i_l}-x_{j_l}\leq 2^{-K+2} +u_{i_l}-u_{j_l} <2^{-K+3}, \qquad
l=k,\ldots,d-1,
$$
and
$$
x_{i_l}-x_{j_l}>u_{i_l}-u_{j_l} >-2^{-K}, \qquad l=k,\ldots,d-1.
$$
Hence
\begin{eqnarray*}
\lefteqn{\int_{I} K^1_{n,(i_l,j_l)}(x-u) 1_{\cS_{(i_l,j_l),k,1}}(x-u)\dd u } \n\\
&\leq & C 2^{-Kk} 2^{-K(1/(d-k)-\beta)(d-k-1)}2^{-K/(d-k)}
\prod_{l=1}^{k-1}(x_{i_l}-x_{j_l}-2^{-K})^{-1-\beta} 1_{\{x_{i_l}-x_{j_l}> 2^{-K+1}\}}\\
&&{} \prod_{l=k}^{d-1} 1_{\{-2^{-K}<x_{i_l}-x_{j_l}<2^{-K+3}\}}
(x_{j_{d-1}}-2^{-K-1})^{\beta(d-2)-2} 1_{\{x_{j_{d-1}}>2^{-K+1}\}}
\end{eqnarray*}
and
\begin{eqnarray*}
\lefteqn{C_p 2^{Kd} \sum_{(i_l,j_l)\in \cI}
\sum_{k=1}^{d-1}\int_{\T^d} \Big( \int_{I} K^1_{n,(i_l,j_l)}(x-u)
1_{\cS_{(i_l,j_l),k,1}}(x-u)
\dd u \Big)^p \dd x }\n\\
&\leq &C_p 2^{Kd} 2^{-Kkp} 2^{-K(1-\beta(d-k-1))p} \sum_{(i_l,j_l)\in \cI} \sum_{k=1}^{d-1}\\
&&{} \int_{\T^d}
\prod_{l=1}^{k-1}(x_{i_l}-x_{j_l}-2^{-K})^{-(1+\beta)p} 1_{\{x_{i_l}-x_{j_l}> 2^{-K+1}\}}\\
&&{} \prod_{l=k}^{d-1} 1_{\{-2^{-K}<x_{i_l}-x_{j_l}<2^{-K+3}\}}
(x_{j_{d-1}}-2^{-K-1})^{(\beta(d-2)-2)p} 1_{\{x_{j_d}>2^{-K+1}\}} \dd x \n\\
&\leq &C_p 2^{Kd} 2^{-Kkp} 2^{-K(1-\beta(d-k-1))p} \\
&&{}\sum_{(i_l,j_l)\in \cI} \sum_{k=1}^{d-1}
2^{-K(1-(1+\beta)p)(k-1)} 2^{-K(d-k)}  2^{-K(1-(2-\beta(d-2))p)}\\
&\leq &C_p,
\end{eqnarray*}
whenever $1-(1+\beta)p<0$, $1-(2-\beta(d-2))p<0$ and
$\beta<1/(d-1)$. Since $\beta$ can be arbitrarily near to $1/(d-1)$,
we obtain $p>\frac{d-1}{d}$.

\eword{Step 2.} For $k=d$, we have $x_{i_l}-x_{j_l}>2^{-K+1}$ for
all $l=1,\ldots,d-1$ and $x_{j_{d-1}}>2^{-K+1}$. Suppose again that
the center of $I$ is zero, in other words $I:=\prod_{j=1}^{d}
(-\nu,\nu)$. If we introduce
$$
A_1(u):= \int_{-\nu}^{u_1} a(t_1,u_2,\ldots,u_d) \dd t_1
$$
and
\begin{equation}\label{e62.2}
A_{k}(u):= \int_{-\nu}^{u_k}
A_{k-1}(u_1,\ldots,u_{k-1},t_k,u_{k+1},\ldots,u_d) \dd t_k \qquad
(2\leq k\leq d),
\end{equation}
then
$$
|A_k(u)| \leq C_p 2^{K(d/p-k)}.
$$
Integrating by parts, we can see that
\begin{eqnarray*}
\lefteqn{\int_{I_1} a(u)K^1_{n,(i_l,j_l)}(x-u)
1_{\cS_{(i_l,j_l),d,1}}(x-u) \dd u_1 } \n\\ &=&
A_1(\nu,u_2,\ldots,u_d) (K^1_{n,(i_l,j_l)}1_{\cS_{(i_l,j_l),d,1}})
(x_1-\nu,x_2-u_2,\ldots,x_d-u_d) \\
&&{}+ \int_{-\nu}^\nu A_1(u) \partial_1K^1_{n,(i_l,j_l)}(x-u)
1_{\cS_{(i_l,j_l),d,1}}(x-u) \dd u_1,
\end{eqnarray*}
because $A_1(-\nu,u_2,\ldots,u_d)=0$. Integrating the first term
again by parts, we obtain
\begin{eqnarray*}
\lefteqn{\int_{I_1} \int_{I_2} a(u)K^1_{n,(i_l,j_l)}(x-u)
1_{\cS_{(i_l,j_l),d,1}}(x-u)
\dd u_1 du_2 } \n\\
&=& A_2(\nu,\nu,u_3,\ldots,u_d)
(K^1_{n,(i_l,j_l)}1_{\cS_{(i_l,j_l),d,1}})
(x_1-\nu,x_2-\nu,x_3-u_3,\ldots,x_d-u_d) \\
&&{}+ \int_{-\nu}^\nu A_2(\nu,u_2,\ldots,u_d) (\partial_2
K^1_{n,(i_l,j_l)}1_{\cS_{(i_l,j_l),d,1}})
(x_1-\nu,x_2-u_2,\ldots,x_d-u_d) \dd u_2\\
&&{}+ \int_{I_1}\int_{I_2} A_1(u)
(\partial_1K^1_{n,(i_l,j_l)}1_{\cS_{(i_l,j_l),d,1}})(x-u) \dd u_1
du_2.
\end{eqnarray*}
Since $A_d(\nu,\ldots,\nu)=\int_{I}a=0$, repeating this process, we
get that
\begin{eqnarray}\label{e62.3}
\lefteqn{\int_{I} a(u)K^1_{n,(i_l,j_l)}(x-u)
1_{\cS_{(i_l,j_l),d,1}}(x-u)
\dd u } \\
&=& \sum_{k=1}^{d} \int_{I_k} \cdots \int_{I_d} A_k(\nu,\ldots\nu,u_k,\ldots,u_d) \n\\
&&{}(\partial_k K^1_{n,(i_l,j_l)}1_{\cS_{(i_l,j_l),d,1}})
(x_1-\nu,\ldots,x_1-\nu,x_k-u_k,\ldots,x_d-u_d) \dd u_k\cdots
du_d.\n
\end{eqnarray}
Inequality (\ref{e62.16}) implies
\begin{eqnarray}\label{e62.19}
\lefteqn{
\Big|\int_{I} a(u)K^1_{n,(i_l,j_l)}(x-u) 1_{\cS_{(i_l,j_l),d,1}}(x-u)\dd u\Big| } \n\\
&\leq & C_p \sum_{k=1}^{d} 2^{K(d/p-k)}
\int_{I_k} \cdots \int_{I_d} \prod_{l=1}^{k-1}(x_{i_l}-\nu-(x_{j_l}-\nu))^{-1-\beta} \n\\
&&{} \prod_{l=k}^{d-1}(x_{i_l}-u_{i_l}-(x_{j_l}-u_{j_l}))^{-1-\beta}
(x_{j_{d-1}}-u_{j_{d-1}})^{\beta(d-1)-2} \n\\
&&{}1_{\cS_{(i_l,j_l),d,1}}(x-u)\dd u_k\cdots du_d \\
&\leq & C_p 2^{K(d/p-k)} 2^{-K(d-k+1)}
\prod_{l=1}^{d-1}(x_{i_l}-x_{j_l}-2^{-K})^{-1-\beta}
1_{\{x_{i_l}-x_{j_l}> 2^{-K+1}\}}\n\\
&&{} (x_{j_{d-1}}-2^{-K-1})^{\beta(d-1)-2}
1_{\{x_{j_{d-1}}>2^{-K+1}\}}\n
\end{eqnarray}
and
\begin{eqnarray*}
\lefteqn{\sum_{(i_l,j_l)\in \cI} \int_{\T^d} \sup_{n\geq 1} \Big|
\int_{I} a(u)K^1_{n,(i_l,j_l)}(x-u) 1_{\cS_{(i_l,j_l),d,1}}(x-u)
\dd u \Big|^p \dd x }\n\\
&\leq &C_p 2^{Kd} 2^{-Kdp-Kp} \sum_{(i_l,j_l)\in \cI} \int_{\T^d}
\prod_{l=1}^{d-1}(x_{i_l}-x_{j_l}-2^{-K})^{-(1+\beta)p} 1_{\{x_{i_l}-x_{j_l}> 2^{-K+1}\}}\\
&&{}
(x_{j_{d-1}}-2^{-K-1})^{(\beta(d-1)-2)p} 1_{\{x_{j_{d-1}}>2^{-K+1}\}} \dd x \n\\
&\leq &C_p 2^{Kd} 2^{-Kdp-Kp} \sum_{(i_l,j_l)\in \cI}
2^{-K(1-(1+\beta)p)(d-1)} 2^{-K(1-(2-\beta(d-1))p)} \\
&\leq &C_p,
\end{eqnarray*}
whenever $1-(1+\beta)p<0$, $1-(2-\beta(d-1))p<0$ and
$\beta<2/(d-1)$. In other words
$$
p>\frac{d}{d+1}.
$$

\eword{Step 3.} Now we investigate the first $d-1$ summands and the
set $\cS_{(i_l,j_l),k,2}$ in (\ref{e62.14}). In this case,
$$
x_{i_k}-u_{i_k}-(x_{j_k}-u_{j_k})\leq 2^{-K+2}
$$
and $x_{j_{d-1}}<x_{i_k} <2^{-K+4}$. Note that $x_{j_{d-1}}\leq
2^{-K+2}$ for $k=d$. Observe that $k=1$ can be excluded, because
$i_1=1$, $j_1=d$ and this contradicts the definition of $\cS$, where
$x_1>2^{-K+4}$. Using the method of Step 1, we get that

\begin{eqnarray*}
\lefteqn{\int_{I} K^1_{n,(i_l,j_l)}(x-u) 1_{\cS_{(i_l,j_l),k,2}}(x-u)\dd u } \n\\
&\leq & C \int_{I}
\prod_{l=1}^{d-2}(x_{i_l}-u_{i_l}-(x_{j_l}-u_{j_l}))^{-1-\beta}
(x_{j_{d-1}}-u_{j_{d-1}})^{\beta(d-2)-2} 1_{\cS_{(i_l,j_l),k,1}}(x-u)\dd u\\
&\leq & C \int_{I}
\prod_{l=1}^{k-1}(x_{i_l}-u_{i_l}-(x_{j_l}-u_{j_l}))^{-1-\beta}
\prod_{l=k}^{d-2}(x_{i_l}-u_{i_l}-(x_{j_l}-u_{j_l}))^{-1-\beta+(1-\epsilon)/(d-k-1)} \\
&&{}
(x_{i_{d-1}}-u_{i_{d-1}}-(x_{j_{d-1}}-u_{j_{d-1}}))^{\epsilon-1}
(x_{j_{d-1}}-u_{j_{d-1}})^{\beta(d-2)-2} 1_{\cS_{(i_l,j_l),k,2}}(x-u)\dd u\\
&\leq & C 2^{-K(k-1)} 2^{-K((1-\epsilon)/(d-k-1)-\beta)(d-k-1)}2^{-K\epsilon}\\
&&{}
\prod_{l=1}^{k-1}(x_{i_l}-x_{j_l}-2^{-K})^{-1-\beta} 1_{\{x_{i_l}-x_{j_l}> 2^{-K+1}\}}\\
&&{} \prod_{l=k}^{d-1} 1_{\{-2^{-K}<x_{i_l}-x_{j_l}<2^{-K+3}\}}
2^{-K(\beta(d-2)-1)} 1_{\{x_{j_{d-1}}\leq 2^{-K+4}\}}
\end{eqnarray*}
and so
\begin{eqnarray*}
\lefteqn{\sum_{(i_l,j_l)\in \cI} \sum_{k=1}^{d-1} \int_{\T^d}
\sup_{n\geq 1} \Big| \int_{I} a(u)K^1_{n,(i_l,j_l)}(x-u)
1_{\cS_{(i_l,j_l),k,2}}(x-u)
\dd u \Big|^p \dd x }\n\\
&\leq &C_p 2^{Kd} 2^{-K(k-1)p} 2^{-K(1-\beta(d-k-1))p} 2^{-K(\beta(d-2)-1)p} \\
&&{} \sum_{(i_l,j_l)\in \cI} \sum_{k=2}^{d-1} \int_{\T^d}
\prod_{l=1}^{k-1}(x_{i_l}-x_{j_l}-2^{-K})^{-(1+\beta)p} 1_{\{x_{i_l}-x_{j_l}> 2^{-K+1}\}}\\
&&{} \prod_{l=k}^{d-1} 1_{\{-2^{-K}<x_{i_l}-x_{j_l}<2^{-K+3}\}}
1_{\{x_{j_d}\leq 2^{-K+4}\}} \dd x \n\\
&\leq &C_p 2^{Kd} 2^{-K(k-1)p} 2^{-K(1-\beta(d-k-1))p} 2^{-K(\beta(d-2)-1)p}\\
&&{}\sum_{(i_l,j_l)\in \cI} \sum_{k=2}^{d-1}
2^{-K(1-(1+\beta)p)(k-1)} 2^{-K(d-k+1)} \\
&\leq &C_p.
\end{eqnarray*}
We have used that $0<\epsilon<1$, $1/(d-2)<\beta<(1-\epsilon)/(d-3)$
and $1-(1+\beta)p<0$. (Since the term for $k=1$ is zero, we can
suppose here that $d>3$.) This implies that $\epsilon<1/(d-2)$.
Since $\beta$ can be chosen arbitrarily near to $(1-\epsilon)/(d-3)$
and $\epsilon$ to 0, we obtain
$$
p>\frac{d-3}{d-2}.
$$

\eword{Step 4.} Here, we consider the $d$th summand and the set
$\cS_{(i_l,j_l),d,2}$ in (\ref{e62.14}). Similarly to
(\ref{e62.19}),
\begin{eqnarray}\label{e62.23}
\lefteqn{
\Big|\int_{I} a(u)K^1_{n,(i_l,j_l)}(x-u) 1_{\cS_{(i_l,j_l),d,2}}(x-u)\dd u\Big| } \n\\
&\leq & C_p \sum_{k=1}^{d} 2^{K(d/p-k)} \int_{I_k} \cdots \int_{I_d}
\prod_{l=1}^{k-1}(x_{i_l}-\nu-(x_{j_l}-\nu))^{-1-\beta}
\n\\
&&{}
\prod_{l=k}^{d-1}(x_{i_l}-u_{i_l}-(x_{j_l}-u_{j_l}))^{-1-\beta}\n\\
&&{} (x_{j_{d-1}}-u_{j_{d-1}})^{\beta(d-1)-2}
1_{\cS_{(i_l,j_l),d,2}}(x-u)\dd u_k\cdots \dd u_d \n\\
&\leq & C_p 2^{K(d/p-k)} 2^{-K(d-k)}
\prod_{l=1}^{d-1}(x_{i_l}-x_{j_l}-2^{-K})^{-1-\beta}
1_{\{x_{i_l}-x_{j_l}> 2^{-K+1}\}} \n\\
&&{} 2^{-K(\beta(d-1)-1)} 1_{\{x_{j_{d-1}}\leq 2^{-K+4}\}},
\end{eqnarray}
thus
\begin{eqnarray*}
\lefteqn{\sum_{(i_l,j_l)\in \cI} \int_{\T^d} \sup_{n\geq 1} \Big|
\int_{I} a(u)K^1_{n,(i_l,j_l)}(x-u) 1_{\cS_{(i_l,j_l),d,2}}(x-u)
\dd u \Big|^p \dd x }\n\\
&\leq &C_p 2^{Kd} 2^{-Kdp} \sum_{(i_l,j_l)\in \cI}
\int_{\T^d} \prod_{l=1}^{d-1}(x_{i_l}-x_{j_l}-2^{-K})^{-(1+\beta)p} 1_{\{x_{i_l}-x_{j_l}> 2^{-K+1}\}}\\
&&{}
2^{-K(\beta(d-1)-1)p} 1_{\{x_{j_{d-1}}\leq 2^{-K+4}\}} \dd x \n\\
&\leq &C_p 2^{Kd} 2^{-Kdp} \sum_{(i_l,j_l)\in \cI}
2^{-K(1-(1+\beta)p)(d-1)} 2^{-K(\beta(d-1)-1)p} 2^{-K} \\
&\leq &C_p,
\end{eqnarray*}
whenever $p>\frac{d}{d+1}$. The corresponding inequalities for the
sets $\cS_{(i_l,j_l),k,3}$ and $\cS_{(i_l,j_l),k,4}$ can be proved
similarly. This proves (\ref{e62.11}) and (\ref{e6}).

\medskip

By Theorem \ref{t19}, to prove the weak inequality (\ref{e7}), it is
enough to show that
$$
\sup_{\rho>0}\rho^{d/(d+1)}\lambda(\sigma_*^1 a > \rho) \leq C
$$
for all $H_{d/(d+1)}^\Box$-atoms $a$. Observe that
\begin{equation}\label{e62.25}
\rho^{d/(d+1)}\lambda(|g| > \rho) \leq \int |g|^{d/(d+1)}
\end{equation}
implies that we have to show only that
\begin{equation}\label{e62.22}
\rho^{d/(d+1)}\lambda\Big(\sup_{n\geq 1} \Big| \int_{I}
a(u)K_{n,(i_l,j_l)}^{1}(x-u) 1_{\cS_{(i_l,j_l),k,l}}(x-u)\dd u \Big|
> \rho\Big) \leq C
\end{equation}
for $k=d$, $l=1$ or $k=d$, $l=2$ and all $\rho>0$.

\eword{Step 5.} Suppose that $k=d$ and $l=1$. We can see as in Lemma
\ref{l62.6} that
\begin{eqnarray}\label{e62.20}
|\partial_q K^1_{n,(i_l,j_l)}(x)| &\leq& C
\prod_{l=1}^{d-1}(x_{i_l}-x_{j_l})^{-1-\beta_l}
x_{j_{d-1}}^{d-3+\sum_{l=1}^{d-1}(\beta_l-1)} 1_{\{x_{j_{d-1}}\leq \pi/2\}}\n\\
&&+C \prod_{l=1}^{d-1}(x_{i_l}-x_{j_l})^{-1-\beta_l} \nonumber \\
&&\qquad (\pi-x_{i_{d-1}})^{d-3+\sum_{l=1}^{d-1}(\beta_l-1)}
1_{\{x_{j_{d-1}}> \pi/2\}},
\end{eqnarray}
whenever $q=1,\ldots,d$, $0<\beta_l<1$,
$d-3+\sum_{l=1}^{d-1}(\beta_l-1)<0$ and
$\beta_l+\frac{\beta_{d-1}}{d-2}<1$ for all $l=1,\ldots ,d-2$. We
get similarly to (\ref{e62.19}) in Step 2 with $p=d/(d+1)$ that
\begin{eqnarray*}
\lefteqn{
\Big|\int_{I} a(u)K^1_{n,(i_l,j_l)}(x-u) 1_{\cS_{(i_l,j_l),d,1}}(x-u)\dd u\Big| } \n\\
&\leq & C \sum_{k=1}^{d} 2^{K(d+1-k)}
\int_{I_k} \cdots \int_{I_d} \prod_{l=1}^{k-1}(x_{i_l}-\nu-(x_{j_l}-\nu))^{-1-\beta_l} \\
&&{}
\prod_{l=k}^{d-1}(x_{i_l}-u_{i_l}-(x_{j_l}-u_{j_l}))^{-1-\beta_l}
(x_{j_{d-1}}-u_{j_{d-1}})^{d-3+\sum_{l=1}^{d-1}(\beta_l-1)} \n\\
&&{}1_{\cS_{(i_l,j_l),d,1}}(x-u)\dd u_k\cdots \dd u_d\n \\
&\leq & C 2^{K(d+1-k)} 2^{-K(d-k+1)}
\prod_{l=1}^{d-1}(x_{i_l}-x_{j_l}-2^{-K})^{-1-\beta_l}
1_{\{x_{i_l}-x_{j_l}> 2^{-K+1}\}}\n\\
&&{} (x_{j_{d-1}}-2^{-K-1})^{d-3+\sum_{l=1}^{d-1}(\beta_l-1)}
1_{\{x_{j_{d-1}}>2^{-K+1}\}}.\n
\end{eqnarray*}
If this is greater than $\rho$, then
\begin{eqnarray*}
\lefteqn{(x_{i_1}-x_{j_1}-2^{-K}) 1_{\{x_{i_1}-x_{j_1}> 2^{-K+1}\}}
} \qquad \n\\ &\leq & C \rho^{\frac{-1}{1+\beta_1}}
\prod_{l=2}^{d-1}(x_{i_l}-x_{j_l}-2^{-K})^{\frac{-1-\beta_l}{1+\beta_1}}
1_{\{x_{i_l}-x_{j_l}> 2^{-K+1}\}}\n\\
&&{}
(x_{j_{d-1}}-2^{-K-1})^{\frac{d-3+\sum_{l=1}^{d-1}(\beta_l-1)}{1+\beta_1}}
1_{\{x_{j_{d-1}}>2^{-K+1}\}}.\n
\end{eqnarray*}
We know that either $i_1>i_2$ or $j_1<j_2$. Suppose that the first
inequality is satisfied and let $x':=(x_1,\ldots
,x_{i_1-1},x_{i_1+1},\ldots ,x_d)$. Let
$$
\cR_k:=\left\{\begin{array}{ll}
x': x_{i_l}-x_{j_l}-2^{-K}>\rho^{\frac{-1}{d+1}}, l=2,\ldots,k-1; \\
x': x_{i_l}-x_{j_l}-2^{-K}\leq \rho^{\frac{-1}{d+1}}, l=k,\ldots,d-1
\end{array} \right.
$$
and
\begin{eqnarray*}
\cR_{k,1}&:=&\{x'\in \cR_k: x_{j_{d-1}}-2^{-K+1}>\rho^{\frac{-1}{d+1}}\},\\
\cR_{k,2}&:=&\{x'\in \cR_k: x_{j_{d-1}}-2^{-K+1}\leq
\rho^{\frac{-1}{d+1}}\}
\end{eqnarray*}
for some $k=2,\ldots ,d-1$. We may assume that $x'\in \cR_{k,1}$ or
$x'\in \cR_{k,2}$. In both cases, let $\beta_2=\beta_3=\cdots =
\beta_{k-1}$ and $\beta_k=\beta_{k+1}=\cdots = \beta_{d-1}$. Then
\begin{eqnarray}\label{e62.21}
\lambda\Big(\sup_{n\geq 1} \Big|
\lefteqn{\int_{I} a(u)K_{n,(i_l,j_l)}^{1}(x-u) 1_{\cS_{(i_l,j_l),d,1}}(x-u)\dd u \Big| 1_{\cR_{k,1}}(x') > \rho\Big) } \n\\
&\leq & C \rho^{\frac{-1}{1+\beta_1}} \int
\prod_{l=2}^{k-1}(x_{i_l}-x_{j_l}-2^{-K})^{\frac{-1-\beta_2}{1+\beta_1}}
1_{\{x_{i_l}-x_{j_l}> 2^{-K+1}\}} \n\\
&&{}
\prod_{l=k}^{d-1}(x_{i_l}-x_{j_l}-2^{-K})^{\frac{-1-\beta_k}{1+\beta_1}}
1_{\{x_{i_l}-x_{j_l}> 2^{-K+1}\}}\n\\
&&{}
(x_{j_{d-1}}-2^{-K-1})^{\frac{d-3+(\beta_1-1)+(\beta_2-1)(k-2)+(\beta_k-1)(d-k)}{1+\beta_1}} 1_{\{x_{j_{d-1}}>2^{-K+1}\}} 1_{\cR_{k,1}}(x') \dd x' \n\\
&\leq & C \rho^{\frac{-1}{1+\beta_1}}
\rho^{\frac{-1}{d+1}(\frac{-1-\beta_2}{1+\beta_1}+1)(k-2)}
\rho^{\frac{-1}{d+1}(\frac{-1-\beta_k}{1+\beta_1}+1)(d-k)}
\rho^{\frac{-1}{d+1}(\frac{-2+\beta_1+\beta_2(k-2)+\beta_k(d-k)}{1+\beta_1}+1)} \n\\
&\leq & C \rho^{\frac{-d}{d+1}},
\end{eqnarray}
whenever
$$
1+\beta_1<1+\beta_2, \ 1+\beta_1>1+\beta_k, \
1+\beta_1<2-\beta_1-\beta_2(k-2)-\beta_k(d-k).
$$
Substituting $\beta_i=1/d+\mu_i$ with small $\mu_i$ in these
inequalities, we obtain
$$
\mu_1<\mu_2, \ \mu_1>\mu_k,\ 0<-2\mu_1-\mu_2(k-2)-\mu_k(d-k).
$$
If $\mu_1<0$, $\mu_2>0$ and $\mu_k<0$ are small enough, then these
inequalities are satisfied for a fixed $k$.

For the set $\cR_{k,2}$, we get the same inequality as in
(\ref{e62.21}) if
$$
1+\beta_1<1+\beta_2, \ 1+\beta_1>1+\beta_k, \
0<2-\beta_1-\beta_2(k-2)-\beta_k(d-k)<1+\beta_1,
$$
or, after the substitution,
$$
\mu_1<\mu_2, \ \mu_1>\mu_k,\
-\frac{1}{d}-1-\mu_1<-2\mu_1-\mu_2(k-2)-\mu_k(d-k)<0.
$$
These inequalities are again satisfied if $\mu_1>0$, $\mu_2>0$ and
$\mu_k<0$ are small enough, which shows (\ref{e62.22}) for $k=d$,
$l=1$.

\eword{Step 6.} Let $k=d$ and $l=2$. Setting $\beta_1=1/d+\epsilon$,
$\epsilon>0$ and $\beta_l=\beta$, $l=2,\ldots ,d-1$ in
(\ref{e62.20}), we obtain
\begin{eqnarray*}
\lefteqn{|\partial_q K^1_{n,(i_l,j_l)}(x)| } \n\\ &\leq& C
(x_{i_1}-x_{j_1})^{-\frac{d+1}{d}-\epsilon}
\prod_{l=2}^{d-1}(x_{i_l}-x_{j_l})^{-1-\beta}
x_{j_{d-1}}^{\beta(d-2)+\frac{1}{d}+\epsilon-2} 1_{\{x_{j_{d-1}}\leq \pi/2\}}\n\\
&&+C (x_{i_1}-x_{j_1})^{-\frac{d+1}{d}-\epsilon}
\prod_{l=2}^{d-1}(x_{i_l}-x_{j_l})^{-1-\beta}
(\pi-x_{i_{d-1}})^{\beta(d-2)+\frac{1}{d}+\epsilon-2} 1_{\{x_{j_{d-1}}> \pi/2\}}\\
&\leq& C (x_{i_1}-x_{j_1})^{-\frac{d+1}{d}}
\prod_{l=2}^{d-1}(x_{i_l}-x_{j_l})^{-1-\beta-\frac{\epsilon}{d-2}}
x_{j_{d-1}}^{\beta(d-2)+\frac{1}{d}+\epsilon-2} 1_{\{x_{j_{d-1}}\leq \pi/2\}}\n\\
&&+C (x_{i_1}-x_{j_1})^{-\frac{d+1}{d}}
\prod_{l=2}^{d-1}(x_{i_l}-x_{j_l})^{-1-\beta-\frac{\epsilon}{d-2}}
(\pi-x_{i_{d-1}})^{\beta(d-2)+\frac{1}{d}+\epsilon-2}
1_{\{x_{j_{d-1}}> \pi/2\}},
\end{eqnarray*}
if $\beta(d-2)+\frac{1}{d}+\epsilon-2<0$. Similarly to
(\ref{e62.23}),
\begin{eqnarray*}
\lefteqn{
\Big|\int_{I} a(u)K^1_{n,(i_l,j_l)}(x-u) 1_{\cS_{(i_l,j_l),d,2}}(x-u)\dd u\Big| } \n\\
&\leq & C \sum_{k=1}^{d} 2^{K(d+1-k)}
\int_{I_k} \cdots \int_{I_d} (x_{i_1}-u_{i_1}-(x_{j_1}-u_{j_1}))^{-\frac{d+1}{d}} \\
&&{}\prod_{l=2}^{k-1}(x_{i_l}-\nu-(x_{j_l}-\nu))^{-1-\beta-\frac{\epsilon}{d-2}}
\prod_{l=k}^{d-1}(x_{i_l}-u_{i_l}-(x_{j_l}-u_{j_l}))^{-1-\beta-\frac{\epsilon}{d-2}}\\
&&{}
(x_{j_{d-1}}-u_{j_{d-1}})^{\beta(d-2)+\frac{1}{d}+\epsilon-2} 1_{\cS_{(i_l,j_l),d,2}}(x-u)\dd u_k\cdots du_d \n\\
&\leq & C 2^{K(d+1-k)} 2^{-K(d-k)} (x_{i_1}-x_{j_1}-2^{-K})^{-\frac{d+1}{d}} 1_{\{x_{i_1}-x_{j_1}> 2^{-K+1}\}} \\
&&{}
\prod_{l=2}^{d-1}(x_{i_l}-x_{j_l}-2^{-K})^{-1-\beta-\frac{\epsilon}{d-2}}
1_{\{x_{i_l}-x_{j_l}> 2^{-K+1}\}}
2^{-K(\beta(d-2)+\frac{1}{d}+\epsilon-1)} 1_{\{x_{j_{d-1}}\leq
2^{-K+4}\}}.
\end{eqnarray*}
Here $\beta(d-2)+\frac{1}{d}+\epsilon-1>0$ and $u_{i_1}=u_{j_1}=\nu$
if $k\geq 2$. If the last expression is greater than $\rho$, then
\begin{eqnarray*}
(x_{i_1}-x_{j_1}-2^{-K}) 1_{\{x_{i_1}-x_{j_1}> 2^{-K+1}\}} &\leq & C \rho^{-\frac{d}{d+1}} 2^{-\frac{Kd}{d+1}(\beta(d-2)+\frac{1}{d}+\epsilon-2)}  \\
&&{} \prod_{l=2}^{d-1}(x_{i_l}-x_{j_l}-2^{-K})^{(-1-\beta-\frac{\epsilon}{d-2})\frac{d}{d+1}}\\
&&{} 1_{\{x_{i_l}-x_{j_l}> 2^{-K+1}\}}
 1_{\{x_{j_{d-1}}\leq 2^{-K+4}\}}
\end{eqnarray*}
and
\begin{eqnarray*}
\lambda\Big(\sup_{n\geq 1} \Big|
\lefteqn{\int_{I} a(u)K_{n,(i_l,j_l)}^{1}(x-u) 1_{\cS_{(i_l,j_l),d,2}}(x-u)\dd u \Big|> \rho\Big) } \n\\
&\leq & C \rho^{-\frac{d}{d+1}} 2^{-\frac{Kd}{d+1}(\beta(d-2)+\frac{1}{d}+\epsilon-2)} \int \\
&&{}
\prod_{l=2}^{d-1}(x_{i_l}-x_{j_l}-2^{-K})^{(-1-\beta-\frac{\epsilon}{d-2})\frac{d}{d+1}}
1_{\{x_{i_l}-x_{j_l}> 2^{-K+1}\}}
 1_{\{x_{j_{d-1}}\leq 2^{-K+4}\}} \dd x' \n\\
&\leq & C \rho^{-\frac{d}{d+1}} 2^{-\frac{Kd}{d+1}(\beta(d-2)+\frac{1}{d}+\epsilon-2)} 2^{-K((-1-\beta-\frac{\epsilon}{d-2})\frac{d}{d+1}+1)(d-2)} 2^{-K}\n\\
&\leq & C \rho^{\frac{-d}{d+1}},
\end{eqnarray*}
whenever $(-1-\beta-\frac{\epsilon}{d-2})\frac{d}{d+1}+1<0$. Recall
that $x':=(x_1,\ldots ,x_{i_1-1},x_{i_1+1},\ldots ,x_d)$. It is easy
to see that we can choose a $\beta$ that satisfies all conditions
mentioned above. This implies inequality (\ref{e62.22}) for $k=d$,
$l=2$. The cases $k=d$, $l=3$ and $k=d$, $l=4$ can be handled
similarly. The proof of the theorem is complete.
\end{proof*}

\subsubsection{Proof for $q=\infty$}

\begin{proof*}{Theorem \ref{t21}  for $q=\infty$}
We will show again that
\begin{equation}\label{e63.11}
\int_{\T^d \setminus 2^7I} |\sigma_*^\infty a(x)|^p \dd x =
\int_{\T^d \setminus 2^7I} \sup_{\nn} \Big| \int_{I} a(u)
K_n^{\infty}(x-u) \dd u \Big|^p \dd x\leq C_p
\end{equation}
for every $H_p^\Box$-atom $a$, where $\frac{d}{d+1}<p<1$. Assume
again that $I$, the support of $a$, satisfies (\ref{e62.1}).

Here, we modify slightly the definition of the sets $\cS$,
$\cS_{\epsilon'}$, $\cS'$ and $\cS_{k}$. Let
$$
\cS:=\{x\in \T^d:x_1>x_2>\cdots>x_d>0, x_1>2^{-K+5}\},
$$
$$
\cS_{\epsilon'}:=\{x\in \T^d: \Big|\sum_{j=1}^{d-1} \epsilon_jx_j
\Big| <d 2^{-K+4}\},
$$
$$
\cS':=\{x\in \T^d: \exists \epsilon,\Big|\sum_{j=1}^{d}
\epsilon_jx_j \Big| <d 2^{-K+4}\},
$$
$$
\cS_{k}:=\{x\in \cS:x_1>x_2>\cdots>x_k\geq
2^{-K+2}>x_{k+1}>\cdots>x_d>0\},
$$
$k=1,\ldots,d$. The sets $\cS_{\epsilon,1}$ and $\cS_{\epsilon',d}$
are defined as before in Subsection \ref{s6.2.3}.
$$
\cS_{\epsilon,1}:=\{x\in \T^d: |\sum_{j=1}^{d} \epsilon_jx_j
|<4x_1\},
$$
$$
\cS_{\epsilon',d}:=\{x\in \T^d: \Big|\sum_{j=1}^{d-1} \epsilon_jx_j
\Big|< 4x_d\}.
$$
We may suppose again that $|\sum_{j=1}^{d} \epsilon_jx_j|\leq \pi$
and $|\sum_{j=1}^{d-1} \epsilon_jx_j|\leq \pi$. It is easy to see
that the Lemmas \ref{l63.1}, \ref{l63.2} and \ref{l63.6} also hold
for
these sets.

Instead of (\ref{e63.11}) it is enough to prove by symmetry that
\begin{eqnarray}\label{e63.7}
\lefteqn{\int_{\T^d \setminus 2^7I} \sup_{\nn} \Big|
\int_{I} a(u) K_n^{\infty}(x-u) 1_\cS(x-u) \dd u \Big|^p \dd x } \n\\
&\leq& \sum_{k=1,d}\int_{\T^d \setminus 2^7I} \sup_{\nn} \Big|
\int_{I} a(u) K_n^{\infty}(x-u) 1_{\cS_k\cap \cS'}(x-u)
\dd u \Big|^p \dd x \n \\
&&{}+ \sum_{\epsilon'}\sum_{k=2}^{d-1}\int_{\T^d \setminus 2^7I} \sup_{\nn} \Big| \int_{I} a(u) K_{n,\epsilon'}^{\infty}(x-u) 1_{\cS_k\cap \cS_{\epsilon'}}(x-u) \dd u \Big|^p \dd x \n \\
&&+ \sum_{k=1,d} \int_{\T^d \setminus 2^7I} \sup_{\nn} \Big|
\int_{I} a(u) K_n^{\infty}(x-u) 1_{\cS_k\setminus\cS'}(x-u)
\dd u \Big|^p \dd x\n\\
&&+ \sum_{\epsilon'}\sum_{k=2}^{d-1}\int_{\T^d \setminus 2^7I}
\sup_{\nn} \Big| \int_{I} a(u) K_{n,\epsilon'}^{\infty}(x-u)
1_{\cS_k\setminus\cS_{\epsilon'}}(x-u)
\dd u \Big|^p \dd x \n\\
&\leq& C_p.
\end{eqnarray}

\eword{Step 1.} Let us consider the first sum of (\ref{e63.7}).
Since $u\in I$, $x-u\in \cS_{\epsilon'}$ or $x-u\in \cS'$ implies
that $x_1$ must be in an interval of length $C2^{-K}$. If $x-u\in
\cS_k$ and $u\in I$, then $x_i-u_i\geq 2^{-K+2}$ and so $x_i\geq
2^{-K+1}$ $(i=1,\ldots,k)$. Moreover, $x_i-u_i<2^{-K+2}$ and so
$x_i<2^{-K+3}$ $(i=k+1,\ldots,d)$. By Theorem \ref{t12}, the
integral of $K_n^\infty$ can be estimated by a constant, thus
$$
\int_{\T^d \setminus 2^7I} \sup_{\nn} \Big| \int_{I} a(u)
K_n^{\infty}(x-u) 1_{\cS_1\cap\cS'}(x-u) \dd u \Big|^p \dd x \leq
C_p 2^{Kd}2^{-Kd}.
$$
If $k=d$, then (\ref{e63.9}) implies
\begin{eqnarray*}
\lefteqn{\sup_{\nn} \Big| \int_{I} a(u)
K_{n,\epsilon'}^{\infty}(x-u) 1_{\cS_d\cap\cS'}(x-u) \dd u \Big| }
\n\\ &\leq & C 2^{Kd/p}
\int_{I} \prod_{i=1}^{d} (x_i-u_i)^{-1}1_{\cS_d\cap\cS'}(x-u)\dd u\\
&\leq & C 2^{Kd/p} \int_{I} \prod_{i=2}^{d}
(x_i-2^{-K-1})^{-1-1/(d-1)}1_{\cS_d\cap\cS'}(x-u)
\dd u\\
&\leq & C 2^{Kd/p-Kd} 1_{I_1'}(x_1)\Big(\prod_{i=2}^{d}
(x_i-2^{-K-1})^{-d/(d-1)} 1_{\{x_i\geq 2^{-K+1}\}} \Big),
\end{eqnarray*}
where the length of $I_1'$ is $c2^{-K}$. Then
\begin{eqnarray*}
\lefteqn{\int_{\T^d \setminus 2^7I} \sup_{\nn}  \Big| \int_{I} a(u)
K_{n,\epsilon'}^{\infty}(x-u) 1_{\cS_d\cap\cS'}(x-u)
\dd u \Big|^p \dd x } \n\\
&\leq& C_p 2^{Kd-Kdp} 2^{-K}\int_{\T^{d-1}} \Big(\prod_{i=2}^{d}
(x_i-2^{-K-1})^{-dp/(d-1)}
1_{\{x_i\geq 2^{-K+1}\}} \Big) \dd x_{2}\cdots \dd x_d\\
&\leq &C_p,
\end{eqnarray*}
when $p>(d-1)/d$.

\eword{Step 2.} In the second sum let us investigate first the term
multiplied by $1_{\cS_{\epsilon',d}}(x-u)$ in the integrand for all
$k=2,\ldots,d-1$. If $x-u\in \cS_{\epsilon',d}$, then $u_1$ is in an
interval of length $8(x_d-u_d)$. By (\ref{e63.9}),
\begin{eqnarray*}
\lefteqn{\sup_{\nn} \Big| \int_{I} a(u)
K_{n,\epsilon'}^{\infty}(x-u) 1_{\cS_k\cap\cS_{\epsilon'} \cap
\cS_{\epsilon',d}}(x-u) \dd u \Big| } \n\\ &\leq & C 2^{Kd/p}
\int_{I} \prod_{i=1}^{d} (x_i-u_i)^{-1}1_{\cS_k\cap\cS_{\epsilon'} \cap \cS_{\epsilon',d}}(x-u) \dd u\\
&\leq & C 2^{Kd/p} 1_{I_1'}(x_1)
\int_{I_2\times\cdots \times I_d} \Big(\prod_{i=2}^{k} (x_i-u_i)^{-1-1/(k-1)} \Big) \\
&&{}\Big(\prod_{i=k+1}^{d} (x_i-u_i)^{-1} \Big) (x_d-u_d)
1_{\cS_k}(x-u) \dd u_2\cdots du_d\\
&\leq & C 2^{Kd/p}1_{I_1'}(x_1)
\int_{I_2\times\cdots \times I_d} \Big(\prod_{i=2}^{k} (x_i-2^{-K-1})^{-1-1/(k-1)} \Big) \\
&&{}\Big(\prod_{i=k+1}^{d} (x_i-u_i)^{-1+1/(d-k)} \Big)
1_{\cS_k}(x-u) \dd u_2\cdots du_d\\
&\leq & C 2^{Kd/p-K(k-1)-K} 1_{I_1'}(x_1) \nonumber \\
&&{} \Big(\prod_{i=2}^{k} (x_i-2^{-K-1})^{-k/(k-1)} 1_{\{x_i\geq
2^{-K+1}\}} \Big) \Big(\prod_{i=k+1}^{d} 1_{\{x_i< 2^{-K+3}\}}
\Big),
\end{eqnarray*}
where the length of $I_1'$ is $c2^{-K}$. Hence
\begin{eqnarray*}
\lefteqn{\sum_{k=2}^{d-1}\int_{\T^d \setminus 2^7I} \sup_{\nn}
\Big|
\int_{I} a(u) K_{n,\epsilon'}^{\infty}(x-u) 1_{\cS_k\cap\cS_{\epsilon'} \cap \cS_{\epsilon',d}}(x-u) \dd u \Big|^p \dd x } \qquad \qquad \qquad\n\\
&\leq& C_p 2^{Kd-Kkp} 2^{-K} 2^{-K(-kp/(k-1)+1)(k-1)} 2^{-K(d-k)}\\
&\leq& C_p,
\end{eqnarray*}
 whenever $p>(d-1)/d$.

If $x-u\not\in \cS_{\epsilon',d}$, then $|\sum_{j=1}^{d}
\epsilon_j(x_j-u_j)|\geq 3(x_d-u_d)$. Applying this and Lemma
\ref{l63.2}, we can see that
\begin{eqnarray*}
\lefteqn{\sup_{\nn} \Big|
\int_{I} a(u) K_{n,\epsilon'}^{\infty}(x-u) 1_{\cS_k\cap\cS_{\epsilon'} \setminus\cS_{\epsilon',d}}(x-u) \dd u \Big| } \n\\
&\leq & C \sum_{\epsilon_d} 2^{Kd/p}
\int_{I} \Big(\prod_{i=1}^{d-1} (x_i-u_i)^{-1}\Big) (x_d-u_d)^{-\delta}\\
&&{}\Big|\sum_{j=1}^{d} \epsilon_j(x_j-u_j) \Big|^{-1+\delta} 1_{\cS_k\cap\cS_{\epsilon'}}(x-u) \dd u\\
&\leq & C \sum_{\epsilon_d} 2^{Kd/p}
\int_{I} \Big(\prod_{i=2}^{k} (x_i-2^{-K-1})^{-1-1/(k-1)} \Big)\Big(\prod_{i=k+1}^{d} (x_i-u_i)^{-1+(1-\delta)/(d-k)} \Big) \\
&&{} \Big|\sum_{j=1}^{d} \epsilon_j(x_j-u_j) \Big|^{-1+\delta}
1_{\cS_k\cap\cS_{\epsilon'}}(x-u) \dd u\\
&\leq & C 2^{Kd/p-K(k-1)-K(1-\delta)-K(-1+\delta+1)} 1_{I_1'}(x_1) \\
&&{}\Big(\prod_{i=2}^{k} (x_i-2^{-K-1})^{-k/(k-1)} 1_{\{x_i\geq
2^{-K+1}\}} \Big) \Big(\prod_{i=k+1}^{d} 1_{\{x_i< 2^{-K+3}\}} \Big)
\end{eqnarray*}
and
$$
\sum_{k=2}^{d-1}\int_{\T^d \setminus 2^7I} \sup_{\nn}  \Big|
\int_{I} a(u) K_{n,\epsilon'}^{\infty}(x-u)
1_{\cS_k\cap\cS_{\epsilon'} \setminus\cS_{\epsilon',d}}(x-u) \dd u
\Big|^p \dd x \leq C_p,
$$
as before, whenever $0<\delta<1$ and $p>(d-1)/d$. This proves that
the second sum in (\ref{e63.7}) can be estimated by a constant.

\eword{Step 3.} Now let us consider the fourth sum of (\ref{e63.7}):
\begin{eqnarray*}
\lefteqn{\sup_{\nn} \Big| \int_{I} a(u)
K_{n,\epsilon'}^{\infty}(x-u) 1_{\cS_k\setminus\cS_{\epsilon'}}(x-u)
\dd u \Big| } \n\\
&\leq & C \sum_{\epsilon_d} 2^{Kd/p}
\int_{I} \Big(\prod_{i=1}^{d-1} (x_i-u_i)^{-1}\Big) \\
&&{}\Big|\sum_{j=1}^{d} \epsilon_j(x_j-u_j) \Big|^{-1}
(1_{(\cS_k\setminus\cS_{\epsilon'})\cap \cS_{\epsilon,1}}(x-u) +
1_{(\cS_k\setminus\cS_{\epsilon'})\setminus\cS_{\epsilon,1}}(x-u))
\dd u,
\end{eqnarray*}
$k=2,\ldots,d-1$. For the first sum, we get that
\begin{eqnarray*}
\lefteqn{\sum_{\epsilon_d} 2^{Kd/p} \int_{I} \Big(\prod_{i=1}^{d-1}
(x_i-u_i)^{-1}\Big)
\Big|\sum_{j=1}^{d} \epsilon_j(x_j-u_j) \Big|^{-1} 1_{(\cS_k\setminus\cS_{\epsilon'})\cap \cS_{\epsilon,1}}(x-u)\dd u } \n\\
&\leq & \sum_{\epsilon_d} 2^{Kd/p}
\int_{I} \Big(\prod_{i=1}^{d} (x_i-u_i)^{-1}\Big) (x_d-u_d) \nonumber \\
&&{}
\Big|\sum_{j=1}^{d} \epsilon_j(x_j-u_j) \Big|^{-1-\delta+\delta} 1_{(\cS_k\setminus\cS_{\epsilon'})\cap \cS_{\epsilon,1}}(x-u)\dd u \n\\
&\leq& C \sum_{\epsilon_d} 2^{Kd/p} \int_{I} \Big(\prod_{i=2}^{k}
(x_i-u_i)^{-1+(\delta-1)/(k-1)}\Big)
\Big(\prod_{i=k+1}^{d} (x_i-u_i)^{1/(d-k)-1}\Big) \\
&&{}
\Big|\sum_{j=1}^{d} \epsilon_j(x_j-u_j) \Big|^{-1-\delta}1_{\cS_k\setminus\cS_{\epsilon'}}(x-u)\dd u\\
&\leq & C \sum_{\epsilon_d} 2^{Kd/p-Kk-K}
\Big(\prod_{i=2}^{k} (x_i-2^{-K-1})^{-(k-\delta)/(k-1)} 1_{\{x_i\geq 2^{-K+1}\}} \Big)\\
&&{} \Big(\prod_{i=k+1}^{d} 1_{\{x_i< 2^{-K+3}\}} \Big)
\Big|\sum_{j=1}^{d} \epsilon_j(x_j-2^{-K-1})
\Big|^{-1-\delta}1_{\{|\sum_{j=1}^{d} \epsilon_j(x_j-2^{-K-1}) |\geq
2^{-K+3}\}},
\end{eqnarray*}
because
$$
|\sum_{j=1}^{d} \epsilon_j(x_j-2^{-K-1})|\geq |\sum_{j=1}^{d}
\epsilon_j(x_j-u_j) |-|\sum_{j=1}^{d} \epsilon_j(u_j-2^{-K-1}) |\geq
d\cdot 2^{-K+2}
$$
and so
\begin{eqnarray*}
|\sum_{j=1}^{d} \epsilon_j(x_j-u_j)| &\geq& |\sum_{j=1}^{d} \epsilon_j(x_j-2^{-K-1}) |-|\sum_{j=1}^{d} \epsilon_j(u_j-2^{-K-1}) | \\
&\geq& \frac{1}{2} |\sum_{j=1}^{d} \epsilon_j(x_j-2^{-K-1})|.
\end{eqnarray*}
Hence,
\begin{eqnarray}\label{e63.3}
\lefteqn{\sum_{\epsilon_d} 2^{Kd}\int_{\T^d \setminus 2^7I} \Big|
\int_{I} \Big(\prod_{i=1}^{d-1} (x_i-u_i)^{-1}\Big) \Big|\sum_{j=1}^{d} \epsilon_j(x_j-u_j) \Big|^{-1} 1_{(\cS_k\setminus\cS_{\epsilon'})\cap \cS_{\epsilon,1}}(x-u)\dd u \Big|^p \dd x }\qquad \qquad \qquad \n\\
&\leq& C_p 2^{Kd-Kkp-Kp} 2^{-K(-(k-\delta)p/(k-1)+1)(k-1)} 2^{-K(d-k)} 2^{-K((-1-\delta)p+1)} \n\\
&\leq& C_p,
\end{eqnarray}
whenever $\delta=1/k$ and $p>k/(k+1)$, thus $p>(d-1)/d$.

Similarly,
\begin{eqnarray*}
\lefteqn{\sum_{\epsilon_d} 2^{Kd/p} \int_{I} \Big(\prod_{i=1}^{d-1}
(x_i-u_i)^{-1}\Big)
\Big|\sum_{j=1}^{d} \epsilon_j(x_j-u_j) \Big|^{-1} 1_{(\cS_k\setminus\cS_{\epsilon'})\setminus \cS_{\epsilon,1}}(x-u)\dd u } \n\\
&\leq & \sum_{\epsilon_d} 2^{Kd/p}
\int_{I} (x_i-u_i)^{-1} \Big(\prod_{i=1}^{d} (x_i-u_i)^{-1}\Big) (x_d-u_d) 1_{\cS_k}(x-u)\dd u \\
&\leq& C \sum_{\epsilon_d} 2^{Kd/p} \int_{I} \Big(\prod_{i=1}^{k}
(x_i-u_i)^{-1-1/k}\Big)
\Big(\prod_{i=k+1}^{d} (x_i-u_i)^{1/(d-k)-1}\Big) 1_{\cS_k}(x-u)\dd u\\
&\leq & C \sum_{\epsilon_d} 2^{Kd/p-Kk-K} \Big(\prod_{i=1}^{k}
(x_i-2^{-K-1})^{-(k+1)/k} 1_{\{x_i\geq 2^{-K+1}\}} \Big)
\Big(\prod_{i=k+1}^{d} 1_{\{x_i< 2^{-K+3}\}} \Big)
\end{eqnarray*}
and
\begin{eqnarray}\label{e63.4}
\lefteqn{\sum_{\epsilon_d} 2^{Kd}\int_{\T^d \setminus 2^7I} \Big|
\int_{I} \Big(\prod_{i=1}^{d-1} (x_i-u_i)^{-1}\Big)
\Big|\sum_{j=1}^{d} \epsilon_j(x_j-u_j) \Big|^{-1}
1_{(\cS_k\setminus\cS_{\epsilon'})\setminus\cS_{\epsilon,1}}(x-u)\dd u \Big|^p \dd x } \qquad \qquad \qquad \qquad \qquad \qquad \qquad \qquad \n\\
&\leq& C_p 2^{Kd-Kkp-Kp} 2^{-K(-(k+1)p/k+1)k} 2^{-K(d-k)} \n\\
&\leq& C_p,
\end{eqnarray}
if $p>(d-1)/d$. This yields that
$$
\sum_{k=2}^{d-1}\int_{\T^d \setminus 2^7I} \sup_{\nn} \Big| \int_{I}
a(u) K_{n,\epsilon'}^{\infty}(x-u)
1_{\cS_k\setminus\cS_{\epsilon'}}(x-u) \dd u \Big|^p \dd x \leq C_p.
$$

\eword{Step 4.} The inequality
$$
\int_{\T^d \setminus 2^7I} \sup_{\nn} \Big| \int_{I} a(u)
K_n^{\infty}(x-u) 1_{\cS_1\setminus\cS'}(x-u) \dd u \Big|^p \dd x
\leq C_p
$$
can be proved in the same way. Now let $\delta=1$. If $k=d$, then
using the notation in (\ref{e62.2}), we get by integrations by
parts, that
\begin{eqnarray*}
\lefteqn{\int_{I} a(u)K_{n,\epsilon'}^{\infty}(x-u)
1_{\cS_d\setminus\cS'}(x-u)
\dd u } \\
&=& \sum_{l=1}^{d} \int_{I_l} \cdots \int_{I_d} A_l(u^{(l)})
(\partial_l K_{n,\epsilon'}^{\infty} 1_{\cS_d\setminus\cS'})
(x-u^{(l)}) \dd u_l \cdots \dd u_d\n
\end{eqnarray*}
as in (\ref{e62.3}), where
$u^{(l)}:=(\nu,\ldots,\nu,u_l,\ldots,u_d)$. We remark that
$$
A_d(\nu,\ldots,\nu)=\int_{I}a=0.
$$
By Lemma \ref{l63.6},
\begin{eqnarray}\label{e63.14}
\lefteqn{\sup_{\nn} \Big| \int_{I} a(u)
K_{n,\epsilon'}^{\infty}(x-u) 1_{\cS_d\setminus\cS'}(x-u)
\dd u \Big| } \n\\
&\leq & C \sum_{\epsilon_d} \sum_{l=1}^{d} 2^{Kd/p-Kl}
\int_{I_l} \cdots \int_{I_d} \Big(\prod_{i=1}^{d} (x_i-u_i^{(l)})^{-1}\Big) \n\\
&&{}\Big|\sum_{j=1}^{d} \epsilon_j(x_j-u_j^{(l)}) \Big|^{-1}
1_{\cS_d\setminus\cS'}(x-u^{(l)})
\dd u_l \ldots \dd u_d\n\\
&&{}+ C \sum_{l=1}^{d} 2^{Kd/p-Kl}
\int_{I_l} \cdots \int_{I_d} \Big(\prod_{i=1}^{d} (x_i-u_i^{(l)})^{-1}\Big) \n\\
&&{}\qquad (x_d-u_d^{(l)})^{-1} 1_{(\cS_d\setminus\cS')\cap \cS_{\epsilon',d}}(x-u^{(l)}) \dd u_l \cdots \dd u_d \n\\
&=:& A(x)+B(x).
\end{eqnarray}
For the first sum, we obtain
\begin{eqnarray*}
A(x)&=& C \sum_{\epsilon_d} \sum_{l=1}^{d} 2^{Kd/p-Kl}
\int_{I_l} \cdots \int_{I_d} \Big(\prod_{i=1}^{d} (x_i-u_i^{(l)})^{-1}\Big) \Big|\sum_{j=1}^{d} \epsilon_j(x_j-u_j^{(l)}) \Big|^{-1} \n\\
&& (1_{(\cS_d\setminus\cS')\cap \cS_{\epsilon,1}}(x-u^{(l)}) +
1_{(\cS_d\setminus\cS')\setminus\cS_{\epsilon,1}}(x-u^{(l)}))
\dd u_l \cdots \dd u_d\n\\
&=:& A_1(x)+A_2(x).
\end{eqnarray*}
Then
\begin{eqnarray}\label{e63.22}
A_1(x) &=& C \sum_{\epsilon_d} 2^{Kd/p-Kl}
\int_{I_l} \cdots \int_{I_d} \Big(\prod_{i=1}^{d} (x_i-u_i^{(l)})^{-1}\Big) \n\\
&&{}\Big|\sum_{j=1}^{d} \epsilon_j(x_j-u_j^{(l)}) \Big|^{-1} 1_{(\cS_d\setminus\cS')\cap \cS_{\epsilon,1}}(x-u^{(l)}) \dd u_l \cdots \dd u_d\n\\
&\leq& C \sum_{\epsilon_d} 2^{Kd/p-Kl}
\int_{I_l} \cdots \int_{I_d} \Big(\prod_{i=2}^{d} (x_i-u_i^{(l)})^{-1+(\delta-1)/(d-1)}\Big)\n\\
&&{} \Big|\sum_{j=1}^{d} \epsilon_j(x_j-u_j^{(l)})
\Big|^{-1-\delta}1_{\cS_d\setminus\cS'}(x-u^{(l)})\dd u_l \cdots \dd
u_d.
\end{eqnarray}
On the other hand,
\begin{eqnarray}\label{e63.23}
A_2(x) &=& C \sum_{\epsilon_d} 2^{Kd/p-Kl}
\int_{I_l} \cdots \int_{I_d} \Big(\prod_{i=1}^{d} (x_i-u_i^{(l)})^{-1}\Big) \n\\
&&{}\Big|\sum_{j=1}^{d} \epsilon_j(x_j-u_j^{(l)}) \Big|^{-1}
1_{(\cS_d\setminus\cS')\setminus\cS_{\epsilon,1}}(x-u^{(l)}) \dd u_l \cdots \dd u_d \\
&\leq & C \sum_{\epsilon_d} 2^{Kd/p-Kl} \int_{I_l} \cdots \int_{I_d}
\Big(\prod_{i=1}^{d} (x_i-u_i^{(l)})^{-1-1/d}\Big)
1_{\cS_d}(x-u^{(l)}) \dd u_l \cdots \dd u_d \n
\end{eqnarray}
and the inequality
$$
\int_{\T^d \setminus 2^7I} |A(x)|^p \dd x \leq C_p
$$
can be seen as in (\ref{e63.3}) and (\ref{e63.4}) for $p>d/(d+1)$.

Next, we investigate the second sum of (\ref{e63.14}). Since $u\in
I$, $x-u\in \cS_d\cap \cS_{\epsilon',d}$ implies that $x_i\geq
2^{-K+1}$ $(i=1,\ldots,d)$ and $x_1$ is in an interval $I_1'$ with
length $4x_d+C2^{-K-1}\leq Cx_d$. Then
\begin{eqnarray}\label{e63.26}
B(x) &=& \sum_{l=1}^{d} 2^{Kd/p-Kl}
\int_{I_l} \cdots \int_{I_d} \Big(\prod_{i=1}^{d} (x_i-u_i^{(l)})^{-1}\Big) \n\\
&&{}\qquad (x_d-u_d^{(l)})^{-1} 1_{(\cS_d\setminus\cS')\cap \cS_{\epsilon',d}}(x-u^{(l)}) \dd u_l \cdots \dd u_d \n\\
&\leq & \sum_{l=1}^{d} 2^{Kd/p-Kl}
\int_{I_l} \cdots \int_{I_d} \Big(\prod_{i=1}^{d} (x_i-2^{-K-1})^{-1}\Big) \n\\
&&{}(x_d-2^{-K-1})^{-1} 1_{\cS_d\cap \cS_{\epsilon',d}}(x-u^{(l)})\dd u_l \cdots \dd u_d \n\\
&\leq & C 2^{Kd/p-Kd-K} \Big(\prod_{i=2}^{d-1} (x_i-2^{-K-1})^{-1-1/d}1_{\{x_i\geq 2^{-K+1}\}}\Big)\n \\
&&{}(x_d-2^{-K-1})^{-2-2/d} 1_{\{x_d\geq 2^{-K+1}\}} 1_{I_1'}(x_1),
\end{eqnarray}
consequently, integrating first in $x_1$,
\begin{eqnarray*}
\int_{\T^d \setminus 2^7I} |B(x)|^p \dd x &\leq& C_p \int_{\T^{d-1}} 2^{Kd-Kdp-Kp} \Big(\prod_{i=2}^{d-1} (x_i-2^{-K-1})^{-(d+1)p/d}1_{\{x_i\geq 2^{-K+1}\}}\Big) \\
&&{}(x_d-2^{-K-1})^{-(2d+2)p/d+1} 1_{\{x_d\geq 2^{-K+1}\}}
\dd x_2 \ldots \dd x_d \\
&\leq& C_p 2^{Kd-Kdp-Kp} 2^{-K(-(d+1)p/d+1)(d-2)}2^{-K(-(2d+2)p/d+2)} \\
&\leq& C_p,
\end{eqnarray*}
whenever $p>\frac{d}{d+1}$. This finishes the proof of
(\ref{e63.11}) as well as (\ref{e6}).

\medskip

\eword{Step 5.} For the weak inequality (\ref{e7}), it is enough to
show that
$$
\sup_{\rho>0}\rho^{d/(d+1)}\lambda \left(\{\sigma_*^\infty a >
\rho\} \cap (\T^d \setminus 2^7I)\right) \leq C
$$
for all $H_{d/(d+1)}^\Box$-atoms $a$. Observe that (\ref{e62.25})
implies that we have to show only that
$$
\rho^{d/(d+1)}\lambda\Big(\sup_{\nn} | \int_{I} a(u)
K_n^{\infty}(x-u) 1_{\cS_d\setminus\cS'}(x-u) \dd u| > \rho,\T^d
\setminus 2^7I\Big) \leq C.
$$
Similarly to (\ref{e63.23}) with $p=d/(d+1)$,
\begin{eqnarray*}
A_2(x) &\leq & C \sum_{\epsilon_d} 2^{K(d+1)-Kl}
\int_{I_l} \cdots \int_{I_d} (x_1-u_1^{(l)})^{-1} \Big(\prod_{i=1}^{d} (x_i-u_i^{(l)})^{-1}\Big) \n\\
&&{} 1_{(\cS_d\setminus\cS')\setminus\cS_{\epsilon,1}}(x-u^{(l)}) \dd u_l \cdots \dd u_d \\
&\leq & C (x_1-2^{-K-1})^{-2}1_{\{x_1\geq 2^{-K+1}\}}
\Big(\prod_{i=2}^{d} (x_i-2^{-K-1})^{-1} 1_{\{x_i\geq 2^{-K+1}\}}
\Big).
\end{eqnarray*}
If this is greater than $C\rho$, then by translation, we may suppose
that
\begin{equation}\label{e63.25}
x_d \leq \rho^{-1} x_1^{-2} \Big(\prod_{i=2}^{d-1} x_i^{-1} \Big)
\end{equation}
and each $x_i$ is positive. We assume that
\begin{equation}\label{e63.15}
0\leq x_d<\cdots<x_{k+1}<\rho^{-1/(d+1)}<x_k<\cdots<x_1
\end{equation}
for some $k=0,1,\ldots,d$. The case $k=d$ contradicts
(\ref{e63.25}). For another $k$ and for some $0\leq \epsilon\leq 1$,
$$
x_d=x_d^\epsilon x_d^{1-\epsilon}\leq \Big(\prod_{i=k+1}^{d-1}
x_i^{\epsilon/(d-k-1)} \Big) \Big(\rho^{-1} x_1^{-2}
\prod_{i=2}^{d-1} x_i^{-1} \Big)^{1-\epsilon}.
$$
Then
\begin{eqnarray*}
\lefteqn{\int 1_{\{x_1^{-2}\prod_{i=2}^{d} x_i^{-1}\geq \rho\}} \dd
x } \n\\ &\leq &
\int \rho^{\epsilon-1}\Big(\prod_{i=k+1}^{d-1} x_i^{\epsilon/(d-k-1)+\epsilon-1} \Big) \Big(  \prod_{i=1}^{k} x_i^{-1-1/k} \Big)^{1-\epsilon}\dd x_1\cdots \dd x_{d-1}\\
&\leq &
\rho^{\epsilon-1}\rho^{-\frac{1}{d+1}(\frac{\epsilon}{d-1-k}+\epsilon)(d-1-k)}
\rho^{-\frac{1}{d+1}(-\frac{k+1}{k}(1-\epsilon)+1)k} \\
&=& \rho^{-\frac{d}{d+1}},
\end{eqnarray*}
whenever we choose $\epsilon$ such that
$-\frac{k+1}{k}(1-\epsilon)+1<0$. If $k=0$, then let $\epsilon=1$
and if $k=d-1$, then $\epsilon=0$.

On the other hand, by (\ref{e63.22}),
\begin{eqnarray*}
A_1(x) &\leq & C \sum_{\epsilon_d} 2^{K(d+1)-Kl}
\int_{I_l} \cdots \int_{I_d} (x_1-u_1^{(l)})^{\delta} \Big(\prod_{i=1}^{d} (x_i-u_i^{(l)})^{-1}\Big) \n\\
&&{}\Big|\sum_{j=1}^{d} \epsilon_j(x_j-u_j^{(l)}) \Big|^{-1-\delta} 1_{(\cS_d\setminus\cS')\cap \cS_{\epsilon,1}}(x-u^{(l)}) \dd u_l \cdots \dd u_d\n\\
&\leq & C \sum_{\epsilon_d}(x_1-2^{-K-1})^{-1+\delta}1_{\{x_1\geq 2^{-K+1}\}} \Big(\prod_{i=2}^{d} (x_i-2^{-K-1})^{-1} 1_{\{x_i\geq 2^{-K+1}\}} \Big) \\
&&{} \Big|\sum_{j=1}^{d} \epsilon_j(x_j-2^{-K-1})
\Big|^{-1-\delta}1_{\{|\sum_{j=1}^{d} \epsilon_j(x_j-2^{-K-1})
|^{-1-\delta}\geq 2^{-K+3}\}}.
\end{eqnarray*}
We may suppose again that
$$
x_1^{-1+\delta}\Big(\prod_{i=2}^{d} x_i^{-1} \Big)
\Big|\sum_{j=1}^{d} \epsilon_jx_j \Big|^{-1-\delta}\geq \rho
$$
and that (\ref{e63.15}) holds. Then
$$
\Big(\prod_{i=2}^{k}
x_i^{-1+(\delta-1)/(k-1)}\Big)\Big(\prod_{i=k+1}^{d} x_i^{-1} \Big)
\Big|\sum_{j=1}^{d} \epsilon_jx_j \Big|^{-1-\delta}\geq \rho.
$$
By a transformation,
\begin{eqnarray*}
\lefteqn{\int 1_{\{x_1^{-1+\delta}(\prod_{i=2}^{d} x_i^{-1}) |\sum_{j=1}^{d} \epsilon_jx_j|^{-1-\delta}\geq \rho\}} \dd x } \n\\
&\leq& \int 1_{\{(\prod_{i=2}^{k}
t_i^{-1+(\delta-1)/(k-1)})(\prod_{i=k+1}^{d} t_i^{-1})
|t_1|^{-1-\delta}\geq \rho\}} \dd t.
\end{eqnarray*}
Assume that $\rho^{-1/(d+1)}<|t_1|$. The case $k=d$ is again
impossible. In other cases,
\begin{eqnarray*}
\lefteqn{\int 1_{\{(\prod_{i=2}^{k}
t_i^{-1+(\delta-1)/(k-1)})(\prod_{i=k+1}^{d} t_i^{-1})
|t_1|^{-1-\delta}\geq \rho\}} \dd t  } \n\\&\leq &
\int \rho^{\epsilon-1}(\prod_{i=k+1}^{d-1} t_i^{\epsilon/(d-k-1)+\epsilon-1}) (\prod_{i=2}^{k} t_i^{-1+(\delta-1)/(k-1)})^{1-\epsilon} |t_1|^{(-1-\delta)(1-\epsilon)} \dd t_1\cdots \dd t_{d-1}\\
&\leq &
\rho^{\epsilon-1}\rho^{-\frac{1}{d+1}(\frac{\epsilon}{d-1-k}+\epsilon)(d-1-k)}
\rho^{-\frac{1}{d+1}(
(-\frac{k-\delta}{k-1}(1-\epsilon)+1)(k-1)+(-1-\delta)(1-\epsilon)+1)} \\
&=& \rho^{-\frac{d}{d+1}},
\end{eqnarray*}
if we choose $\epsilon$ and $\delta$ such that
$-\frac{k-\delta}{k-1}(1-\epsilon)+1<0$ and
$(-1-\delta)(1-\epsilon)+1<0$. The cases $k=0$ ($\epsilon=1$), $k=1$
($\delta=1$) and $k=d-1$ ($\epsilon=0$) are included again.

If $x_d<|t_1|\leq \rho^{-1/(d+1)}$, then $k<d$ and
\begin{eqnarray*}
\lefteqn{\int 1_{\{(\prod_{i=2}^{k} t_i^{-1+(\delta-1)/(k-1)})(\prod_{i=k+1}^{d} t_i^{-1}) |t_1|^{-1-\delta}\geq \rho\}} \dd t } \n\\
&\leq & \int \rho^{\epsilon-1}(\prod_{i=k+1}^{d-1}
t_i^{\epsilon/(d-k)+\epsilon-1})
|t_1|^{\epsilon/(d-k)-(1+\delta)(1-\epsilon)}(\prod_{i=2}^{k} t_i^{-1+(\delta-1)/(k-1)})^{1-\epsilon}  \dd t_1\cdots \dd t_{d-1}\\
&\leq & \rho^{\epsilon-1}\rho^{-\frac{1}{d+1}
((\frac{\epsilon}{d-k}+\epsilon)(d-1-k)+\frac{\epsilon}{d-k}-
(1+\delta)(1-\epsilon)+1)} \rho^{-\frac{1}{d+1}
(-\frac{k-\delta}{k-1}(1-\epsilon)+1)(k-1)} \\
&=& \rho^{-\frac{d}{d+1}},
\end{eqnarray*}
assuming that $\frac{\epsilon}{d-k}- (1+\delta)(1-\epsilon)+1>0$ and
$-\frac{k-\delta}{k-1}(1-\epsilon)+1<0$.

If $|t_1|<x_d$ and $|t_1|\leq \rho^{-1/(d+1)}$, then
\begin{eqnarray*}
\lefteqn{\int 1_{\{(\prod_{i=2}^{k} t_i^{-1+(\delta-1)/(k-1)})(\prod_{i=k+1}^{d} t_i^{-1}) |t_1|^{-1-\delta}\geq \rho\}} \dd t } \n\\
&\leq & \int \rho^{-\frac{1-\epsilon}{1+\delta}}(\prod_{i=k+1}^{d}
t_i^{\frac{\epsilon}{d-k}-\frac{1-\epsilon}{1+\delta}})
(\prod_{i=2}^{k} t_i^{-1+\frac{\delta-1}{k-1}})^\frac{1-\epsilon}{1+\delta}  \dd t_2\cdots \dd t_{d}\\
&\leq & \rho^{-\frac{1-\epsilon}{1+\delta}}\rho^{-\frac{1}{d+1}
(\frac{\epsilon}{d-k}-\frac{1-\epsilon}{1+\delta}+1)(d-k)}
\rho^{-\frac{1}{d+1}
(-\frac{k-\delta}{k-1}\frac{1-\epsilon}{1+\delta}+1)(k-1)} \\
&=& \rho^{-\frac{d}{d+1}},
\end{eqnarray*}
if $\frac{\epsilon}{d-k}-\frac{1-\epsilon}{1+\delta}+1>0$ and
$-\frac{k-\delta}{k-1}\frac{1-\epsilon}{1+\delta}+1<0$.

Finally, we have to investigate $B(x)$. Similarly to (\ref{e63.26}),
if
$$
\Big(\prod_{i=1}^{d-1} x_i^{-1}\Big) x_d^{-2}1_{I_1'}(x_1)\geq \rho,
$$
then
$$
x_d \leq \rho^{-1/2} 1_{I_1'}(x_1) \Big(\prod_{i=2}^{k}
x_i^{-1-1(k-1)} \Big)^{1/2} \Big(\prod_{i=k+1}^{d-1} x_i^{-1}
\Big)^{1/2}.
$$
Then
\begin{eqnarray*}
\int 1_{\{(\prod_{i=1}^{d-1} x_i^{-1}) x_d^{-2}1_{I_1'}(x_1)\geq \rho\}} \dd x &\leq & \int x_d 1_{\{(\prod_{i=1}^{d-1} x_i^{-1}) x_d^{-2}\geq \rho\}} \dd x_2\cdots \dd x_{d}\\
&\leq &
\int \rho^{\epsilon-1}\Big(\prod_{i=k+1}^{d-1} x_i^{\epsilon/(d-k-1)+\epsilon/2-1/2} \Big)^2 \\
&&{}\Big(  \prod_{i=2}^{k} x_i^{-1/2-1/2(k-1)} \Big)^{2(1-\epsilon)}\dd x_2\ldots \dd x_{d-1}\\
&\leq &
\rho^{\epsilon-1}\rho^{-\frac{1}{d+1}(\frac{2\epsilon}{d-1-k}+\epsilon)(d-1-k)
+(-\frac{k}{k-1}(1-\epsilon)+1)(k-1)} \\
&=& \rho^{-\frac{d}{d+1}},
\end{eqnarray*}
whenever we choose $\epsilon$ such that
$-\frac{k+1}{k}(1-\epsilon)+1<0$. This finishes the proof of
(\ref{e7}).
\end{proof*}

\subsubsection{Proof for $q=2$}

\begin{proof*}{Theorem \ref{t21}  for $q=2$}
Assume that $\alpha>(d-1)/2$, $\gamma\in \N$ and let us choose $N\in
\N$ such that $N<\alpha-(d-1)/2\leq N+1$. As we mentioned in Section
\ref{s7}, we may suppose that the support of an atom $a$ is a ball
$B$ with radius $\beta$, $2^{-K-1} < \beta \leq 2^{-K}$ $(K\in\N)$.
Moreover, we may suppose that the center of $B$ is zero, i.e.,
$B=B(0,\beta)$. Obviously,
\begin{eqnarray*}
\lefteqn{\int_{\T^d\setminus (sB)} |\sigma_*^{2,\alpha} a(x)|^{p}
\dd x } \n\\ &\leq& \sum_{i=4 \lfloor d^{1/2} \rfloor-1}^{\lfloor
d^{1/2}2^{K}\pi \rfloor-1} \int_{B(0,(i+2)2^{-K})\setminus
B(0,(i+1)2^{-K})\cap\T^d}
\sup_{n\geq d^{1/2} 2^{K+1}} |\sigma_n^{2,\alpha} a(x)|^{p} \dd x \n\\
&&{} + \sum_{i=4 \lfloor d^{1/2} \rfloor -1}^{ \lfloor
d^{1/2}2^{K}\pi \rfloor -1} \int_{B(0,(i+2)2^{-K})\setminus
B(0,(i+1)2^{-K})\cap\T^d}
\sup_{n<d^{1/2} 2^{K+1}} |\sigma_n^{2,\alpha} a(x)|^{p} \dd x \n\\
&=:& (A)+ (B),
\end{eqnarray*}
where $s=8 d^{1/2}$. Note that if $K\leq 3$, then the integral is
equal to $0$.

Using Theorem \ref{t64.5}, the definition of the atom and Taylor's
formulae, we obtain
\begin{eqnarray*}
\sigma_{n}^{2,\alpha} a(x) &=& \sum_{k\in \Z^d} n^d \int_{B+2k\pi}
a(t)
\widehat \theta_0(n(x - t)) \dd t\\
&=& \sum_{k\in \Z^d} n^d \sum_{i_1+\cdots+i_d=N} (-1)^N
\int_{B+2k\pi} a(t) \\
&&{} \Big( \partial_1^{i_1}\cdots\partial_d^{i_d} \widehat
\theta_0(n(x - 2k\pi)-nv(t-2k\pi)) \Big) n^N \prod_{j=1}^{d}
\frac{(t_j-k_j)^{i_j}}{i_j!}\dd t,
\end{eqnarray*}
where $0<v<1$. Then, by Corollary \ref{c6.2},
\begin{eqnarray}\label{e64.24}
|\sigma_{n}^{2,\alpha} a(x)| &\leq& C_{p} \sum_{k\in \Z^d} n^{(d-1)/2+N-\alpha} 2^{Kd/p}2^{-KN} \nonumber \\
&&{} \int_{B+2k\pi} \|x-2k\pi-v(t-2k\pi)\|_2^{-d/2-\alpha-1/2} \dd
t.
\end{eqnarray}
Moreover,
\begin{eqnarray*}
\sup_{n\geq d^{1/2} 2^{K+1}} |\sigma_{n}^{2,\alpha} a(x)| &\leq & C_{p} \sum_{k\in \Z^d} 2^{K((d-1)/2-\alpha)} 2^{Kd/p} \\
&&{}\int_{B+2k\pi} \|x-2k\pi-v(t-2k\pi)\|_2^{-d/2-\alpha-1/2} \dd t \nonumber \\
&=:& A_1(x)+A_2(x),
\end{eqnarray*}
where $A_1$ denotes the term $k=0$ and $A_2$ the remaining sum. If
$k=0$, $u\in B$ and $x\in B(0,(i+2)2^{-K})\setminus
B(0,(i+1)2^{-K})\cap\T^d$ for some $i=4 \lfloor d^{1/2} \rfloor
-1,\ldots, \lfloor d^{1/2}2^{K}\pi \rfloor -1$, then
$$
\|x-u\|_2\geq \|x\|_2-\|u\|_2 \geq i 2^{-K}.
$$
In the case $k\neq 0$, $u\in B+2k\pi$ and $x\in
B(0,(i+2)2^{-K})\setminus B(0,(i+1)2^{-K})\cap\T^d$, one can see
that $\|x-u\|_2\geq \|k\|_2/4$. Then
$$
A_1(x) \leq C_{p} 2^{K((d-1)/2-\alpha)} 2^{Kd/p} \int_{B}  \|i
2^{-K}\|_2^{-d/2-\alpha-1/2} \dd t \leq C_p 2^{Kd/p}
i^{-d/2-\alpha-1/2}
$$
and
\begin{eqnarray*}
A_2(x) &\leq& C_{p} \sum_{k\in \Z^d,k\neq 0} 2^{K((d-1)/2-\alpha)}
2^{Kd/p} \int_{B+2k\pi}
\|k\|_2^{-d/2-\alpha-1/2} \dd t \\
&\leq& C_{p} \sum_{k\in \Z^d,k\neq 0} 2^{K(-d/2-1/2-\alpha)} 2^{Kd/p} \|k\|_2^{-d/2-1/2-\alpha} \\
&\leq& C_{p} \sum_{j=1}^\infty 2^{K(-d/2-1/2-\alpha)} 2^{Kd/p} j^{(-d/2-1/2-\alpha)} j^{d-1}\\
&\leq& C_{p},
\end{eqnarray*}
whenever $x\in B(0,(i+2)2^{-K})\setminus B(0,(i+1)2^{-K})\cap\T^d$,
$p\geq d/(d/2+\alpha+1/2)$ and $\alpha>(d-1)/2$. Hence,
$$
(A) \leq C_{p} \sum_{i=4 \lfloor d^{1/2} \rfloor -1}^{ \lfloor
d^{1/2}2^{K}\pi \rfloor -1} 2^{-Kd} i^{d-1} 2^{Kd}
i^{p(-d/2-\alpha-1/2)} + C_{p} \leq C_{p},
$$
if $p>d/(d/2+\alpha+1/2)$.

Similarly to (\ref{e64.24}), we obtain
\begin{eqnarray*}
|\sigma_{n}^{2,\alpha} a(x)| &\leq& C_{p} \sum_{k\in \Z^d} n^{(d-1)/2+(N+1)-\alpha} 2^{Kd/p}2^{-K(N+1)} \nonumber \\
&&{} \int_{B+2k\pi} \|x-2k\pi-v(t-2k\pi)\|_2^{-d/2-\alpha-1/2} \dd
t.
\end{eqnarray*}
It is easy to see that $\sup_{n<d^{1/2} 2^{K+1}}
|\sigma_{n}^{2,\alpha} a(x)|$ can be estimated in the same way as
\newline $\sup_{n\geq d^{1/2} 2^{K+1}} |\sigma_{n}^{2,\alpha} a(x)|$
above, which proves
$$
\int_{\T^d\setminus (sB)} |\sigma_*^{2,\alpha} a(x)|^{p} \dd x \leq
C_p
$$
as well as (\ref{e6}).

Let us introduce the set
$$
E_{\rho}:=\left\{i\geq 4 \lfloor d^{1/2} \rfloor -1:
i^{-d/2-\alpha-1/2} > C^{-1}\rho 2^{-Kd/p}\right\},
$$
where $p=d/(d/2+\alpha+1/2)$. To prove (\ref{e7}), observe that
$$
\rho^{p} \, \lambda \Big(\{A_1>\rho\}\cap \{\T^d\setminus (sB)\}
\Big) \leq C \rho^{p} \sum_{i\in E_\rho} i^{d-1}2^{-Kd}.
$$
If $k$ is the largest integer for which $k^{-d/2-\alpha-1/2} >
C^{-1}\rho 2^{-Kd/p}$, then
$$
\rho^{p} \, \lambda \Big(\{A_1>\rho\}\cap \{\T^d\setminus (sB)\}
\Big) \leq \rho^{p}2^{-Kd} k^d \leq C.
$$
The same inequality for $(A_2)$ is trivial. We can estimate
$\sup_{n<d^{1/2} 2^{K+1}} |\sigma_{n}^{2,\alpha} a(x)|$ similarly,
which shows (\ref{e7}).
\end{proof*}

\sect{Conjugate functions}\label{s9}

For a distribution
$$
f \sim \sum_{n\in\Z^d} \widehat f(n) \ee^{\ii n \cdot x},
$$
the \idword{conjugate distributions} or  \idword{Riesz transforms}
are defined by
$$
\tilde f^{(i)} \sim \sum_{n \in\Z^d} -\ii\ {n_{i} \over \|n\|_2}
\widehat f(n) \ee^{\ii n \cdot x} \qquad
(i=1,\ldots,d)\index{\file-1}{$\tilde f^{(i)}$}
$$
(see e.g.~Stein \cite{st1} or Weisz \cite{wk2}). In the
one-dimensional case,
$$
\tilde f := \tilde f^{(1)} \sim \sum_{n=-\infty}^{\infty} (-\ii\
{\rm sign} \ n) \widehat f(n) \ee^{\ii nx}
$$
is called the \idword{Hilbert transform}. As is well known, if $f$
is an integrable function, then
$$
\tilde f(x) = {\rm p.v.} \ {1 \over \pi} \int_\T {f(x-t) \over 2
\tan(t/2)} \dd t := \lim_{\epsilon\to 0} {1 \over \pi}
\int_{\epsilon<|t|<\pi} {f(x-t) \over 2 \tan(t/2)} \dd t \qquad
\mbox{a.e.}
$$
Note that p.v.\ is the abbreviation of the \ind{principal value}.
More generally, in the higher dimensional case, if $f \in
L_1(\T^d)$, then there also exist a principal value form of the
conjugate functions $\tilde f^{(i)}$. They do exist almost
everywhere, but they are not integrable in general (see e.g.~Shapiro
\cite{Shapiro-1964}, Stein and Weiss \cite{st,st1,stwe} or Weisz
\cite{wk2}). The following inequalities can be found in Stein
\cite{st1} or Weisz \cite{wk2}.

\begin{thm}\label{t51}
For all $f\in H_p^\Box(\T^d)$,
\begin{equation}\label{e12}
\|\tilde f^{(i)}\|_{H_p^\Box} \leq C_p \|f\|_{H_p^\Box} \qquad
(0<p<\infty,i=0,\ldots,d)
\end{equation}
and
\begin{equation}\label{e13}
\|f\|_{H_p^\Box} \sim \sum_{i=0}^d \|\tilde f^{(i)}\|_p \qquad
((d-1)/d<p<\infty).
\end{equation}
\end{thm}

Since $H_p^\Box(\T^d)\sim L_p(\T^d)$ for $1<p<\infty$, Theorem
\ref{t51} holds also for $L_p(\T^d)$ spaces $(1<p<\infty)$.

The \dword{conjugate Fej{\'e}r and Riesz
means}\index{\file}{conjugate Fej{\'e}r
means}\index{\file}{conjugate Riesz means} and \idword{conjugate
maximal operators} of a distribution $f$ are introduced by
$$
\tilde \sigma_{n}^{(i);q,\alpha} f(x):= \tilde f^{(i)} *
K_{n}^{q,\alpha} \qquad
(i=1,\ldots,d)\index{\file-1}{$\tilde\sigma_n^{(i);q,\alpha}f$}
$$
and
$$
\tilde \sigma_*^{(i);q,\alpha}f := \sup_{n\geq 1} |\tilde
\sigma_{n}^{(i);q,\alpha}
f|\index{\file-1}{$\tilde\sigma_*^{(i);q,\alpha}f$},
$$
respectively. We use the notations
$$
\tilde f^{(0)}=f,\quad \tilde \sigma_{n}^{(0);q,\alpha} f
=\sigma_{n}^{q,\alpha}f \quad \mbox{and} \quad \tilde
\sigma_*^{(0);q,\alpha}f=\sigma_*^{q,\alpha}f.
$$
By (\ref{e12}) and (\ref{e13}), Theorem \ref{t21} can be generalized
for the conjugate means.

\begin{thm}\label{t30}
If $q=1,\infty$, $\alpha\geq 1$ and $d/(d+1)<p<\infty$, then for all
$i=0,1,\ldots,d$,
\begin{equation}\label{e9.1}
\|\tilde \sigma_*^{(i);q,\alpha}f\|_{p} \leq C_{p}
\|f\|_{H_{p}^\Box} \qquad (f\in H_{p}^\Box(\T^d)).
\end{equation}
If $q=2$ and $\alpha>(d-1)/2$, then the same holds with
$p>d/(d/2+\alpha+1/2)$. In particular, if $f \in L_1(\T^d)$, then
\begin{equation}\label{e9.2}
\sup_{\rho>0}\rho\, \lambda(\tilde \sigma_*^{(i);q,\alpha}f > \rho)
\leq C \|f\|_{1}.
\end{equation}
If, in addition to the conditions just mentioned, we assume that
$p>(d-1)/d$, then
$$
\|\tilde \sigma_{n}^{(i);q,\alpha} f\|_{H_{p}^\Box} \leq C_{p}
\|f\|_{H_{p}^\Box} \qquad (f\in H_{p}^\Box(\T^d)).
$$
\end{thm}

\begin{proof}
Theorem \ref{t21} and (\ref{e12}) imply
$$
\|\tilde \sigma_*^{(i);q,\alpha}f \|_{p} =
\|\sigma_*^{q,\alpha}\tilde f^{(i)}\|_{p} \leq C_{p} \|\tilde
f^{(i)}\|_{H_{p}^\Box} \leq C_{p} \|f\|_{H_{p}^\Box} \qquad (f\in
H_{p}^\Box(\T^d)),
$$
which is exactly (\ref{e9.1}). Inequality (\ref{e9.2}) follows by
interpolation as above. Since
$$
(\sigma_{n}^{q,\alpha} f)^{\sim (i)} = \tilde
\sigma_{n}^{(i);q,\alpha} f,
$$
we have by (\ref{e9.1}) that
$$
\|(\sigma_{n}^{q,\alpha} f)^{\sim (i)}\|_{p} \leq C_{p}
\|f\|_{H_{p}^\Box} \qquad (f\in H_{p}^\Box(\T^d)).
$$
Inequality (\ref{e13}) implies that
$$
\|\tilde \sigma_{n}^{(i);q,\alpha} f\|_{H_p^\Box} \leq C_{p}
\|f\|_{H_{p}^\Box} \qquad (n\geq 1,f\in H_{p}^\Box(\T^d))
$$
if $p>(d-1)/d$.
\end{proof}

The following result, which is a consequence of the density theorem
(Theorem \ref{t4}) due to Marcinkiewicz and Zygmund, is really a
generalization of Corollary \ref{c24} because the conjugate function
$\tilde f^{(i)}$ is not necessarily integrable.

\begin{cor}\label{c31}
Suppose that $q=1,\infty$ and $\alpha\geq 1$ or $q=2$ and
$\alpha>(d-1)/2$. If $i=0,1,\ldots,d$ and $f\in L_1(\T^d)$, then
$$
\lim_{n\to \infty}\tilde \sigma_{n}^{(i);q,\alpha} f = \tilde
f^{(i)} \quad \mbox{ a.e.}
$$
Moreover, if $f \in H_{p}^\Box(\T^d)$ with $d/(d+1)< p <\infty$,
then this convergence also holds in the $H_{p}^\Box(\T^d)$-norm. If
$q=2$ and $\alpha>(d-1)/2$, then the same holds with
$p>d/(d/2+\alpha+1/2)$ and $p>(d-1)/d$.
\end{cor}

\sect{$\ell_q$-summation defined by $\theta$}\label{s10}

\subsection{Summation generated by a function}\label{s10.1}

Now, we introduce a general summability method, the so-called
$\theta$-summability generated by a single function. A natural
question arises, under which conditions on $\theta$ can we prove the
preceding results for the $\theta$-means.
\ieword{$\theta$-summation} was considered in many papers and books,
such as Butzer and Nessel \cite{bune}, Trigub and Belinsky
\cite{trbe}, Natanson and Zuk \cite{nazsu}, Bokor, Schipp, Szili and
V\'ertesi \cite{schbo,bosch,schsz1,szive,szi1}, and Feichtinger and
Weisz \cite{feiwe1,feiwe2,wk2,wel1-fs2,wel1-ft2,wmar6}.

We suppose that $\theta:\R\to\R$ and
\begin{equation}\label{e8}
\left\{
  \begin{array}{ll}
    \mbox{the support of $\theta$ is $[-c,c]$ $(0<c\leq \infty)$}, \\
    \mbox{$\theta$ is even and continuous, } \theta(0)=1, \\
    \sum_{k=0}^{\infty}k^d\Big|\Delta_1\theta(\frac{k}{n}) \Big|<\infty, \\ \lim_{t\to \infty} t^d\theta(t) =0,
  \end{array}
\right.
\end{equation}
where
$$
\Delta_1 \theta \Big(\frac{k}{n} \Big) := \theta \Big(\frac{k}{n}
\Big)-\theta\Big(\frac{k+1}{n}\Big)\index{\file-1}{$\Delta_1\theta$}
$$
is the \idword{first difference}. For $q=2$, we suppose furthermore
that $\theta$ is non-increasing on $(0,\infty)$ or it has compact
support. For $q=1,\infty$, \ind{Abel rearrangement} implies that
$$
\sum_{j\in \Z^d} \Big|\theta\Big(\frac{\|j\|_q}{n} \Big)\Big| \leq C
\sum_{k=0}^\infty k^{d-1} \Big|\theta\Big(\frac{k}{n} \Big)\Big|
\leq C \sum_{k=0}^\infty k^{d} \Big|\Delta_1\theta\Big(\frac{k}{n}
\Big)\Big|<\infty.
$$
The same holds for $q=2$ when $\theta$ is non-increasing on
$(0,\infty)$. If $\theta$ has compact support, then obviously
$\sum_{j\in \Z^d} \Big|\theta\Big(\frac{\|j\|_q}{n}
\Big)\Big|<\infty$. The \idword{$\ell_q$-$\theta$-means} of $f\in
L_1(\T^d)$ are given by
$$
\sigma_n^{q,\theta}f(x):= \sum_{j\in \Z^d}
\theta\Big(\frac{\|j\|_q}{n} \Big) \widehat f(j) \ee^{\ii j \cdot
x}=\int_{\T^d} f(x-u) K_n^{q,\theta}(u) \dd
u,\index{\file-1}{$\sigma_n^{q,\theta}f$}
$$
where
$$
K_{n}^{q,\theta}(u) := \sum_{j\in \Z^d} \theta\Big(\frac{\|j\|_q}{n}
\Big) \ee^{\ii j \cdot u}.\index{\file-1}{$K_n^{q,\theta}$}
$$
Let
$$
\sigma_*^{q,\theta}f := \sup_{n\geq 1} |\sigma_{n}^{q,\theta}
f|\index{\file-1}{$\sigma_*^{q,\theta}f$}
$$
be the \idword{maximal $\theta$-operator}. If $q=1,\infty$, we have
$$
K_{n}^{q,\theta}(u) = \sum_{j\in \Z^d} \sum_{k\geq \|j\|_q}
\Delta_1\theta\Big(\frac{k}{n} \Big) \ee^{\ii j \cdot u} =
\sum_{k=0}^\infty \Delta_1\theta\Big(\frac{k}{n} \Big) D_k^q(u)
$$
and
$$
\sigma_n^{q,\theta}f(x)= \sum_{k=0}^\infty
\Delta_1\theta\Big(\frac{k}{n} \Big) s_{k}^q f(x).
$$
If $\theta(t)=\max((1-|t|^\gamma)^\alpha,0)$, then we get back the
Riesz (or in special case $\alpha=\gamma=1$, the Fej{\'e}r) means.

Let first $q=1$ or $\infty$ and suppose in addition that
\begin{equation}\label{e9}
\left\{
  \begin{array}{ll}
    \theta \mbox{ is twice continuously differentiable on $(0,c)$, } \\
    \theta''\neq 0 \mbox{ except at finitely many points and finitely many intervals,} \\
    \lim_{t\to 0+0} t\theta'(t) \mbox{ is finite}, \\
\lim_{t\to c-0} t\theta'(t) \mbox{ is finite}, \\
\lim_{t\to \infty}t\theta'(t) =0.
  \end{array}
\right.
\end{equation}

Let $\bX$ and $\bY$ be two complete quasi-normed spaces of
measurable functions, $L_ \infty(\T^d)$ be continuously embedded
into $\bX$, and $L_ \infty(\T^d)$ be dense in $\bX$. Suppose that if
$0\leq f\leq g$, $f,g\in \bY$, then $\|f\|_\bY\leq \|g\|_\bY$. If
$f_n,f \in \bY$, $f_n\geq 0$ $(\nn)$ and $f_n \nearrow f$ a.e.~as $n
\to \infty$, then assume that $\|f-f_n\|_\bY \to 0$. Note that the
spaces $L_p(\T^d)$ and $L_{p, \infty}(\T^d)$ $(0<p\leq \infty)$
satisfy these properties. Recall that $\sigma_*^q$ denotes the
maximal Fej{\'e}r operator.

\begin{thm}\label{t10.1}
Assume that $q=1,\infty$ and (\ref{e8}) and (\ref{e9}) are
satisfied. If $\sigma_*^q: \bX \to \bY$ is bounded, i.e.,
$$
\|\sigma^q_*f\|_\bY \leq C \|f\|_\bX \qquad (f\in \bX \cap L_
\infty(\T^d)),
$$
then $\sigma_*^{q,\theta}$ is also bounded,
$$
\|\sigma_*^{q,\theta} f\|_\bY \leq C \|f\|_\bX \qquad (f\in \bX).
$$
\end{thm}

\begin{proof}
By Abel rearrangement,
$$
\sum_{k=0}^m \Delta_1 \theta \Big(\frac{k}{n} \Big) D_{k}^q(x) =
\sum_{k=0}^{m-1} \Delta_2 \theta \Big(\frac{k}{n} \Big) k K_k^q(x) +
\Delta_1 \theta \Big(\frac{m}{n} \Big) m K_{m}^q(x),
$$
where
$$
\Delta_2 \theta \Big(\frac{k}{n} \Big) := \Delta_1 \theta
\Big(\frac{k}{n} \Big)- \Delta_1 \theta \Big(\frac{k+1}{n}
\Big)\index{\file-1}{$\Delta_2\theta$}
$$
is the \idword{second difference}. Observe that, for a fixed $x$, we
have that $K_{m}^q(x)$ is uniformly bounded in $m$. By Lagrange's
mean value theorem, there exists $m<\xi(m)<m+1$, such that
$$
m \Delta_1 \theta \Big(\frac{m}{n} \Big) = - \frac{m}{n}
\theta'\Big(\frac{\xi(m)}{n} \Big)
$$
and this converges to zero if $m\to \infty$. Thus,
$$
K_n^{q,\theta}(x)= \sum_{k=0}^{ \infty} k \, \Delta_2 \theta
\Big(\frac{k}{n} \Big) K_k^q(x).
$$
Now we prove that
\begin{equation}\label{e10.8}
\sup_{n\geq 1} \sum_{k=0}^{ \infty} k\, \Big|\Delta_2 \theta
\Big(\frac{k}{n} \Big) \Big| \leq C < \infty.
\end{equation}
If $\theta''\geq 0$ on the interval $(i/n, (j+2)/n)$, then $\theta$
is convex on this interval and this yields that
$$
\Delta_2 \theta \Big(\frac{k}{n} \Big)\geq 0 \quad \mbox{for} \quad
i\leq k \leq j.
$$
Hence
\begin{eqnarray*}
\sum_{k=i}^{j} k\, \Big|\Delta_2 \theta \Big(\frac{k}{n} \Big) \Big|
&=&
\sum_{k=i}^{j} k\, \Delta_2 \theta \Big(\frac{k}{n} \Big) \\
&=& \theta \Big(\frac{i}{n} \Big) + (i-1) \, \Delta_1 \theta
\Big(\frac{i}{n} \Big) - j \, \Delta_1 \theta \Big(\frac{j+1}{n}
\Big) - \theta \Big(\frac{j+1}{n} \Big).
\end{eqnarray*}
Applying again Lagrange's mean value theorem, we have
$$
(i-1) \Big| \Delta_1 \theta \Big(\frac{i}{n} \Big) \Big| =
\frac{i-1}{n} \Big|\theta'\Big(\frac{\xi(i)}{n} \Big) \Big|
=\frac{i-1}{\xi(i)} \Big| \frac{\xi(i)}{n}
\theta'\Big(\frac{\xi(i)}{n} \Big) \Big| \leq C,
$$
where $i<\xi(i)<i+1$. Here, we used the fact that the function $x
\mapsto |x \theta'(x)|$ is bounded, which follows from (\ref{e9}).
If $\theta''= 0$ at an isolated point $u$ or if $\theta''$ is not
twice continuously differentiable at $u$, $u\in (k/n, (k+1)/n)$,
then the boundedness of $k\, \Big|\Delta_2 \theta \Big(\frac{k}{n}
\Big) \Big|$ can be seen in the same way. Since there are only
finitely many intervals and isolated points satisfying the above
properties, we have shown (\ref{e10.8}).

Hence
$$
\sigma_{n}^{q,\theta}f(x)= \int_{\T^d} f(t) K_{n}^{q,\theta}(x-t)
\dd t =\sum_{k=0}^{ \infty} \int_{\T^d} k \, \Delta_2 \theta
\Big(\frac{k}{n} \Big) f(t) K_k^q(x-t)\dd t
$$
for all $f\in L_ \infty(\T^d)$. Thus
$$
\sigma_*^{q,\theta} f \leq C \sigma_*^q f \qquad (f\in L_
\infty(\T^d))
$$
and so
$$
\|\sigma_*^{q,\theta} f\|_\bY \leq C \|f\|_\bX \qquad (f\in \bX\cap
L_ \infty(\T^d)).
$$
By a usual density argument, we finish the proof of the theorem.
\end{proof}

This theorem implies that the analogues of Theorems \ref{t12},
\ref{t13}, \ref{t21}, \ref{t30} and Corollary \ref{c23}, \ref{c24}
and \ref{c31} hold.

Now we give some examples for the $\theta$-summation. It is easy to
see that all the next examples satisfy (\ref{e8}) and (\ref{e9}).

\begin{exa}[Fej{\'e}r and Riesz summation\index{\file}{Fej{\'e}r summation}\index{\file}{Riesz summation}]\label{x11}\rm
Let
$$
\theta(t)=\cases{(1-|t|^\gamma)^\alpha &if $|t|\leq 1$ \cr 0 &if
$|t|>1$ \cr}
$$
for some $1\leq \alpha,\gamma< \infty$.
\end{exa}

\begin{exa}[\ind{de La Vall\'ee-Poussin summation}]\label{x14}\rm
Let
$$
\theta(t)=\cases{1 &if $|t|\leq 1/2$ \cr -2|t|+2 &if $1/2<|t|\leq 1$
\cr 0 &if $|t|>1$. \cr}
$$
\end{exa}

\begin{exa}[\ind{Jackson-de La Vall\'ee-Poussin summation}]\label{x15}\rm
Let
$$
\theta(t)=\cases{1-3t^2/2+3|t|^3/4 &if $|t|\leq 1$ \cr (2-|t|)^3/4
&if $1<|t|\leq 2$ \cr 0 &if $|t|>2$. \cr}
$$
\end{exa}

\begin{exa}\label{x17}\rm
Let $0=\alpha_0<\alpha_1< \cdots < \alpha_m$ and
$\beta_0,\ldots,\beta_m$ $(m\in \N)$ be real numbers, $\beta_0=1$,
$\beta_m=0$. Suppose that $\theta$ is even,
$\theta(\alpha_j)=\beta_j$ $(j=0,1,\ldots,m)$, $\theta(t)=0$ for $t
\geq \alpha_m$, $\theta$ is a polynomial on the interval
$[\alpha_{j-1},\alpha_j]$ $(j=1,\ldots,m)$.
\end{exa}

\begin{exa}[\ind{Rogosinski summation}]\label{x18}\rm
Let
$$
\theta(t)=\cases{\cos \pi t/2 &if $|t|\leq 1+2j$ \cr 0 &if
$|t|>1+2j$ \cr} \qquad \mbox{for some }j\in\N.
$$
\end{exa}

\begin{exa}[\ind{Weierstrass summation}]\label{x1}\rm
Let
$$
\theta(t)=\ee^{-|t|^\gamma} \quad \mbox{for some} \quad 1\leq
\gamma< \infty.
$$
Note that if $\gamma=1$, then we obtain the Abel means.
\end{exa}

\begin{exa}\label{x2}\rm Let
$$
\theta(t)=\ee^{-(1+|t|^q)^\gamma} \quad \mbox{for some} \quad 1\leq
q< \infty, 0< \gamma< \infty.
$$
\end{exa}

\begin{exa}[Picard and Bessel summations\index{\file}{Picard summation}\index{\file}{Bessel summation}]\label{x8}\rm Let
$$
\theta(t)=(1+|t|^\gamma)^{-\alpha} \quad \mbox{for some} \quad
0<\alpha< \infty, 1\leq \gamma< \infty,\alpha\gamma>d.
$$
\end{exa}

If $q=2$, then we have to assume other additional conditions instead
of (\ref{e9}). Let $\theta_0(x):=\theta(\|x\|_2)$ satisfy
\begin{equation}\label{e10}
\theta_0\in L_1(\R^d) \qquad \mbox{and} \qquad \widehat \theta_0\in
L_1(\R^d).
\end{equation}
Assume that $\widehat \theta_0$ is $N+1$-times differentiable
$(N\geq 0)$ and there exists $d+N-1<\beta \leq d+N$ such that
\begin{equation}\label{e11}
|\partial_1^{i_1}\cdots\partial_d^{i_d} \widehat \theta_0(x)|\leq C
\|x\|_2^{-\beta-1}\qquad (x\neq 0),
\end{equation}
whenever $i_1+\cdots+i_d=N$ or $i_1+\cdots+i_d=N+1$. If $\beta=d+N$,
then it is enough to suppose (\ref{e11}) for $i_1+\cdots+i_d=N+1$.
Under the conditions (\ref{e8}), (\ref{e10}) and (\ref{e11}), the
analogues of Theorems \ref{t12}, \ref{t13}, \ref{t21}, \ref{t30} and
Corollary \ref{c23}, \ref{c24} and \ref{c31} hold with the critical
index $d/(\beta+1)$. One can show (\cite{wamalg-hardy,wel2}) that
Example \ref{x11} with $\alpha>(d-1)/2$, $\gamma\in \N$ and
$\beta=(d-1)/2+\alpha$, Example \ref{x1} with $0<\gamma<\infty$ and
$\beta=d+N$, Example \ref{x2} with $0<\gamma,q<\infty$ and
$\beta=d+N$ and Example \ref{x8} with $\beta=d+N$ satisfy
(\ref{e8}), (\ref{e10}) and (\ref{e11}).

\subsection{Ces{\`a}ro summability}

The well known \ieword{Ces{\`a}ro summation} is not generated by a
function. For $k\in \N,\alpha\neq -1,-2,\ldots$, let
$$
A_{k}^{\alpha} := {k+\alpha \choose k}=
{(\alpha+1)(\alpha+2)\cdots(\alpha+k) \over
k!}.\index{\file-1}{$A_k^\alpha$}
$$
It is known (see Zygmund \cite[p.~77]{zy}) that
$$
A_k^\alpha = \sum_{i=0}^k A_{k-i}^{\alpha-1}, \qquad
A_k^\alpha-A_{k-1}^\alpha= A_k^{\alpha-1}
$$
and
\begin{equation}\label{e10.2}
A_{k}^{\alpha}= O(k^\alpha) \qquad (k\in\N).
\end{equation}
Here, we assume that $q=1$ or $q=\infty$. For $n\geq 1$, the
\dword{Ces{\`a}ro (or $(C,\alpha)$)-means}\index{\file}{Ces{\`a}ro
means}\index{\file}{$(C,\alpha)$-means} of a function $f\in
L_1(\T^d)$ are defined by
\begin{eqnarray*}
\sigma_n^{q,(c,\alpha)} f(x) &:=& \frac{1}{A_{n-1}^{\alpha}}
\sum_{k\in \Z^d, \, \|k\|_q\leq n} A_{n-1-\|k\|_q}^{\alpha}
\widehat f(k) \ee^{\ii k\cdot x} \\
&=& \frac{1}{(2\pi)^d}\int_{\T^d} f(x-u) K_n^{q,(c,\alpha)}(u) \dd
u,\index{\file-1}{$\sigma_n^{q,(c,\alpha)}f$}
\end{eqnarray*}
where the \idword{Ces{\`a}ro kernel} is given by
\begin{eqnarray*}
K_n^{q,(c,\alpha)}(u) &:=& \frac{1}{A_{n-1}^{\alpha}}
\sum_{k\in \Z^d, \, \|k\|_q\leq n} A_{n-1-\|k\|_q}^{\alpha} \ee^{\ii k \cdot u}\n\\
&=& \frac{1}{A_{n-1}^{\alpha}} \sum_{\|k\|_q\leq n}
\sum_{j=\|k\|_q}^{n-1}
A_{n-1-j}^{\alpha-1} \ee^{\ii k\cdot u} \\
&=& \frac{1}{A_{n-1}^{\alpha}} \sum_{j=0}^{n-1} A_{n-1-j}^{\alpha-1}
D_j^q(u).\index{\file-1}{$K_n^{q,(c,\alpha)}$}
\end{eqnarray*}
Hence
$$
\sigma_n^{q,(c,\alpha)} f(x) = \frac{1}{A_{n-1}^{\alpha}}
\sum_{k=0}^{n-1} A_{n-1-k}^{\alpha-1} s_{k}^q f(x)
$$
and if $\alpha=1$, we get back the Fej{\'e}r means.

The \idword{conjugate Ces{\`a}ro means} and \idword{conjugate
maximal operators} of a distribution $f$ are introduced by
$$
\tilde \sigma_{n}^{(i);q,(c,\alpha)} f(x):= \tilde f^{(i)} *
K_{n}^{q,(c,\alpha)}, \qquad \tilde \sigma_*^{(i);q,(c,\alpha)}f :=
\sup_{n\geq 1} |\tilde \sigma_{n}^{(i);q,(c,\alpha)} f|,
\index{\file-1}{$\tilde\sigma_n^{(i);q,(c,\alpha)}f$}\index{\file-1}{$\tilde\sigma_*^{(i);q,(c,\alpha)}f$}
$$
where $i=0,1,\ldots,d$. We proved in \cite{wel1-fs1,wmar6} that
Theorems \ref{t12}, \ref{t13}, \ref{t21}, \ref{t30} and Corollary
\ref{c23}, \ref{c24} and \ref{c31} hold for the Ces{\`a}ro
summability with the critical index $\frac{d}{d+\alpha\wedge 1}$. We
use the notations
$$
a\vee b:=\max\{a,b\} \quad \mbox{and} \quad a\wedge
b:=\min\{a,b\}\index{\file-1}{$a\vee b$}\index{\file-1}{$a\wedge b$}
$$
for two real numbers $a$ and $b$. Here, we give only some hints for
the proofs.

Instead of the inequalities (\ref{e61.5}), we will use the
inequality
\begin{equation}\label{e10.1}
|\sum_{k=0}^{n-1} A_{n-1-k}^{\alpha-1} \soc((k+1/2)u)| \leq
\frac{C}{(\sin(u/2))^\alpha} + \frac{Cn^{\alpha-1}}{\sin(u/2)}
\end{equation}
for $0<\alpha\leq 1$. Indeed,
$$
\sum_{k=0}^{n-1} A_{n-1-k}^{\alpha-1} \ee^{\ii (k+1/2)u}= \ee^{\ii
(n-1/2)u} \sum_{k=0}^{n-1} A_{k}^{\alpha-1} \ee^{-\ii ku}.
$$
Since
$$
\sum_{k=0}^{\infty} A_{k}^{\alpha-1} z^k=(1-z)^{-\alpha} \qquad
(|z|<1),
$$
we have
$$
|\sum_{k=0}^{n-1} A_{n-1-k}^{\alpha-1} \ee^{\ii (k+1/2)u}|=
|(1-\ee^{-\ii u})^{-\alpha} - \sum_{k=n}^{\infty} A_{k}^{\alpha-1}
\ee^{-\ii ku}|.
$$
As $(A_{k}^{\alpha-1})_{k\in \N}$ decreases monotonically to $0$,
the last series converges and the absolute value of its sum can be
estimated by $2A_{n}^{\alpha-1}(1-\ee^{-\ii u})^{-1}$. Thus
$$
|\sum_{k=0}^{n-1} A_{n-1-k}^{\alpha-1} \ee^{\ii (k+1/2)u}| \leq
C|(1-\ee^{-\ii u})^{-\alpha}| + C A_{n}^{\alpha-1}|(1-\ee^{-\ii
u})^{-1}|,
$$
which together with (\ref{e10.2}) proves (\ref{e10.1}).

For $q=1$, we get similarly to Lemma \ref{l62.10} that
\begin{eqnarray*}
K_n^{1,(c,\alpha)}(x) &=& \sum_{(i_l,j_l)\in \cI} \frac{(-1)^{i_{d-1}-1}}{A_{n-1}^{\alpha}} \prod_{l=1}^{d-1}(\cos x_{i_l}-\cos x_{j_l})^{-1} \\
&&{} \sum_{k=0}^{n-1} A_{n-1-k}^{\alpha-1}(G_k(\cos x_{i_{d-1}})-G_k(\cos x_{j_{d-1}}))\n\\
&=:& \sum_{(i_l,j_l)\in \cI} K^{1,(c,\alpha)}_{n,(i_l,j_l)}(x).
\end{eqnarray*}
Then the following three lemmas can be proved as are Lemmas
\ref{l62.1}, \ref{l62.2} and \ref{l62.6} (for details, see Weisz
\cite{wel1-cesaro}).

\begin{lem}\label{l62.100}
For all  $0<\alpha\leq 1$ and $0<\beta<\frac{\alpha+1}{d-1}$,
\begin{eqnarray*}
|K^{1,(c,\alpha)}_{n,(i_l,j_l)}(x)| &\leq& \frac{C}{n^\alpha}
\prod_{l=1}^{d-1}(x_{i_l}-x_{j_l})^{-1-\beta}
x_{j_{d-1}}^{\beta(d-1)-\alpha-1} 1_{\{x_{j_{d-1}}\leq \pi/2\}} \n\\
&&{}+ \frac{C}{n} \prod_{l=1}^{d-1}(x_{i_l}-x_{j_l})^{-1-\beta}
x_{j_{d-1}}^{\beta(d-1)-2} 1_{\{x_{j_{d-1}}\leq \pi/2\}} \n\\
&&{}+\frac{C}{n^\alpha}
\prod_{l=1}^{d-1}(x_{i_l}-x_{j_l})^{-1-\beta}
(\pi-x_{i_{d-1}})^{\beta(d-1)-\alpha-1} 1_{\{x_{j_{d-1}}> \pi/2\}}\n\\
&&{}+\frac{C}{n} \prod_{l=1}^{d-1}(x_{i_l}-x_{j_l})^{-1-\beta}
(\pi-x_{i_{d-1}})^{\beta(d-1)-2} 1_{\{x_{j_{d-1}}> \pi/2\}}.
\end{eqnarray*}
\end{lem}

\begin{lem}\label{l62.20}
For all $0<\alpha\leq 1$ and $0<\beta<\frac{\alpha+1}{d-2}$,
\begin{eqnarray*}
|K^{1,(c,\alpha)}_{n,(i_l,j_l)}(x)| &\leq& C n^{1-\alpha}
\prod_{l=1}^{d-2}(x_{i_l}-x_{j_l})^{-1-\beta}
x_{j_{d-1}}^{\beta(d-2)-\alpha-1} 1_{\{x_{j_{d-1}}\leq \pi/2\}} \n\\
&&{}+C \prod_{l=1}^{d-2}(x_{i_l}-x_{j_l})^{-1-\beta}
x_{j_{d-1}}^{\beta(d-2)-2} 1_{\{x_{j_{d-1}}\leq \pi/2\}} \n\\
&&{}+C n^{1-\alpha} \prod_{l=1}^{d-2}(x_{i_l}-x_{j_l})^{-1-\beta}
(\pi-x_{i_{d-1}})^{\beta(d-2)-\alpha-1} 1_{\{x_{j_{d-1}}> \pi/2\}}\n\\
&&{}+C \prod_{l=1}^{d-2}(x_{i_l}-x_{j_l})^{-1-\beta}
(\pi-x_{i_{d-1}})^{\beta(d-2)-2} 1_{\{x_{j_{d-1}}> \pi/2\}}.
\end{eqnarray*}
\end{lem}

\begin{lem}\label{l62.60}
If $0<\alpha\leq 1$ and $0<\beta<\frac{\alpha+1}{d-1}\wedge
\frac{d-2}{d-1}$, then for all $q=1,\ldots,d$,
\begin{eqnarray*}
|\partial_q K^{1,(c,\alpha)}_{n,(i_l,j_l)}(x)| &\leq& C n^{1-\alpha}
\prod_{l=1}^{d-1}(x_{i_l}-x_{j_l})^{-1-\beta}
x_{j_{d-1}}^{\beta(d-1)-\alpha-1} 1_{\{x_{j_{d-1}}\leq \pi/2\}}\n\\
&&{}+ C \prod_{l=1}^{d-1}(x_{i_l}-x_{j_l})^{-1-\beta}
x_{j_{d-1}}^{\beta(d-1)-2} 1_{\{x_{j_{d-1}}\leq \pi/2\}}\n\\
&&{}+C n^{1-\alpha} \prod_{l=1}^{d-1}(x_{i_l}-x_{j_l})^{-1-\beta}
(\pi-x_{i_{d-1}})^{\beta(d-1)-\alpha-1} 1_{\{x_{j_{d-1}}> \pi/2\}}\n\\
&&{}+C \prod_{l=1}^{d-1}(x_{i_l}-x_{j_l})^{-1-\beta}
(\pi-x_{i_{d-1}})^{\beta(d-1)-2} 1_{\{x_{j_{d-1}}> \pi/2\}}.
\end{eqnarray*}
\end{lem}

For $q=\infty$, we obtain as in (\ref{e63.1})
\begin{eqnarray*}
K_n^{\infty,(c,\alpha)}(x) &=& \frac{1}{A_{n-1}^{\alpha}}
\sum_{k=0}^{n-1} A_{n-1-k}^{\alpha-1} \prod_{i=1}^{d}
\frac{\sin((k+1/2)x_i)}{\sin(x_i/2)}
\n\\
&=& \sum_{\epsilon'}\pm 2^{-d+1} \prod_{i=1}^{d} (\sin(x_i/2))^{-1}\frac{1}{A_{n-1}^{\alpha}} \sum_{k=0}^{n-1} A_{n-1-k}^{\alpha-1} \n\\
&&{} \Big(\soc ((k+1/2)(\sum_{j=1}^{d-1} \epsilon_jx_j+x_d))-
\soc ((k+1/2)(\sum_{j=1}^{d-1} \epsilon_jx_j-x_d))\Big)\n\\
&=:& \sum_{\epsilon'} K^{\infty,(c,\alpha)}_{n,\epsilon'}(x).
\end{eqnarray*}
Then
$$
|K^{\infty,(c,\alpha)}_{n,\epsilon'}(x)|\leq Cn^d \qquad \mbox{and}
\qquad  |K^{\infty,(c,\alpha)}_{n,\epsilon'}(x)|\leq C
\prod_{i=1}^{d} x_i^{-1}.
$$
Similarly to Lemmas \ref{l63.1}, \ref{l63.2} and \ref{l63.6}, we can
show the next three lemmas (see Weisz \cite{wmar6}).

\begin{lem}\label{l63.100}
If $0<\alpha\leq 1$, then
$$
|K^{\infty,(c,\alpha)}_{n,\epsilon'}(x)|\leq C \sum_{\epsilon_d}
n^{-\alpha} \Big(\prod_{i=1}^{d} x_i^{-1} \Big) \Big|\sum_{j=1}^{d}
\epsilon_jx_j \Big|^{-\alpha} + C \sum_{\epsilon_d} n^{-1}
\Big(\prod_{i=1}^{d} x_i^{-1} \Big) \Big|\sum_{j=1}^{d}
\epsilon_jx_j \Big|^{-1}.
$$
\end{lem}

\begin{lem}\label{l63.200}
If $0<\alpha\leq 1$, then for all $x\in
\cS_k\setminus\cS_{\epsilon'}$, $x\in \cS_k\setminus\cS'$
$(k=1,\ldots,d-1)$ and $x\in \cS_{\epsilon',d}^c$,
$$
|K^{\infty,(c,\alpha)}_{n,\epsilon'}(x)| |\leq C \sum_{\epsilon_d}
n^{1-\alpha} \Big(\prod_{i=1}^{d-1} x_i^{-1} \Big)
\Big|\sum_{j=1}^{d} \epsilon_jx_j \Big|^{-\alpha} + C
\sum_{\epsilon_d} \Big(\prod_{i=1}^{d-1} x_i^{-1} \Big)
\Big|\sum_{j=1}^{d} \epsilon_jx_j \Big|^{-1}.
$$
\end{lem}

\begin{lem}\label{l63.600}
If $0<\alpha\leq 1$, then for all $l=1,\ldots,d$ and $x\in \cS$,
\begin{eqnarray*}
|\partial_l K^{\infty,(c,\alpha)}_{n,\epsilon'}(x)| &\leq& C \sum_{\epsilon_d} n^{1-\alpha} \Big(\prod_{i=1}^{d} x_i^{-1} \Big) \Big|\sum_{j=1}^{d} \epsilon_jx_j \Big|^{-\alpha} \n\\
&&{} + C \sum_{\epsilon_d} \Big(\prod_{i=1}^{d} x_i^{-1}
\Big) \Big|\sum_{j=1}^{d} \epsilon_jx_j \Big|^{-1} \n\\
&&{} + C \sum_{\epsilon_d} \Big(\prod_{i=1}^{d} x_i^{-1}\Big)
x_d^{-1} 1_{\cup_{k=1}^{d-1}(\cS_k\cap\cS_{\epsilon'})
\cup(\cS_d\cap\cS_{\epsilon',d})}(x).
\end{eqnarray*}
\end{lem}

\sect{$\ell_q$-summability of Fourier transforms}\label{s11}

The above results hold also for summability means of Fourier
transforms. First suppose that $f\in L_p(\R^d)$ for some $1\leq p
\leq 2$. It is known that if $\widehat f \in L_1(\R^d)$, then
$$
f(x) = \int_{\R^d} \widehat f(v) \ee^{\ii x \cdot v} \dd v \qquad (x
\in \R^d)
$$
(see e.g.~Butzer and Nessel \cite{bune}). This motivates the
definition of the \idword{Dirichlet integral} \inda{$s_t^qf$},
$$
s_{t}^q f(x) := \int_{\R^d} 1_{\{\|v\|_q\leq t\}} \widehat f(v)
\ee^{\ii x \cdot v} \dd v = \frac{1}{(2\pi)^d}\int_{\R^d} f(x-u)
D_t^q(u) \dd u \qquad (t>0),
$$
where the \idword{Dirichlet kernel} is defined by
$$
D_{t}^q(u) := \int_{\R^d} 1_{\{\|v\|_q\leq t\}} \ee^{\ii u \cdot v}
\dd v.\index{\file-1}{$D_t^q$}
$$
Carleson's theorem also holds in this case. More exactly, replacing
$\T$ by $\R$, Theorems \ref{t1}, \ref{t2}, \ref{t8}, \ref{t9},
\ref{t10} and \ref{t11} remain true (see Grafakos \cite{gra}).

Let $\theta_0^q(u):=\theta(\|u\|_q)$ and suppose that $\theta_0^q\in
L_1(\R^d)\cap C_0(\R^d)$. For $T>0$, the
\idword{$\ell_q$-$\theta$-means} of a function $f\in L_p(\R^d)$
$(1\leq p \leq 2)$ are defined by
$$
\sigma_T^{q,\theta}f(x):= \int_{\R^d} \theta\Big(\frac{\|v\|_q}{T}
\Big)\widehat f(v)  \ee^{\ii x \cdot v} \dd
v.\index{\file-1}{$\sigma_T^{q,\theta}f$}
$$
It is easy to see that
$$
\sigma_T^{q,\theta}f(x)=\frac{1}{(2\pi)^d}\int_{\R^d} f(x-u)
K_T^{q,\theta}(u) \dd u
$$
where the \idword{$\ell_q$-$\theta$-kernel} is given by
\begin{equation}\label{e14}
K_{T}^{q,\theta}(u) = \int_{\R^d} \theta\Big(\frac{\|v\|_q}{T}
\Big)\ee^{\ii u\cdot v} \dd v = (2\pi)^dT^d
\widehat\theta_0^q(Tu).\index{\file-1}{$K_T^{q,\theta}$}
\end{equation}
On the other hand, by the first equality of (\ref{e14}),
$$
K_T^{q,\theta}(u) = {-1 \over T} \int_{\R^d} \int_{\|v\|_q}^{\infty}
\theta' \Big(\frac{t}{T} \Big) \dd t\, \ee^{\ii u\cdot v} \dd v= {-1
\over T} \int_{0}^{\infty} \theta' \Big(\frac{t}{T} \Big)
D_t^q(u)\dd t.
$$
Hence
$$
\sigma_T^{q,\theta} f(x) = \frac{-1}{T}\int_{0}^{ \infty} \theta'
\Big(\frac{t}{T} \Big) s_{t}^q f(x) \dd t.
$$
Note that, for the Fej{\'e}r means (i.e., for
$\theta(t)=\max((1-|t|),0)$), we get the usual definition
$$
\sigma_T^{q} f(x) = \frac{1}{T}\int_{0}^{T} s_{t}^q f(x) \dd
t.\index{\file-1}{$\sigma_T^qf$}
$$

On $\R^d$, we will consider tempered distributions rather than
distributions. In this case the \idword{Schwartz class}
\inda{$\cS(\R^d)$} consists of all $f\in C^ \infty(\R^d)$ for which
\begin{equation}\label{e27}
\|f\|_{\alpha,\beta} := \sup_{x\in\R^d} |x^\alpha \partial^\beta
f(x)| < \infty,
\end{equation}
where for the multiindices
$\alpha=(\alpha_1,\ldots,\alpha_d),\beta=(\beta_1,\ldots,\beta_d)\in\N^d$,
we use the conventional notations
$$
x^\alpha= x_1^{\alpha_1} \cdots x_d^{\alpha_d}, \qquad
\partial^\beta = \partial_1^{\beta_1} \cdots \partial_d^{\beta_d}.
$$
A \idword{tempered distribution} $u$ (briefly $u \in
\cS'(\R^d)$)\index{\file-1}{$\cS'(\R^d)$} is a linear functional on
$\cS(\R^d)$ that is continuous in the topology on $\cS(\R^d)$
induced by the family of seminorms (\ref{e27}). Namely, $f_n\to f$
in $\cS(\R^d)$ if
$$
\lim_{n\to\infty} \sup_{x\in\R^d} |x^\alpha \partial^\beta
(f_n-f)(x)|= 0  \qquad \mbox{for all multiindices $\alpha,\beta$}.
$$
Moreover, if $u$ is a tempered distribution, then it is linear and
$u(f_n)\to u(f)$ if $f_n\to f$ in $\cS(\R^d)$. The functions from
$L_p(\R^d)$ $(1\leq p\leq \infty)$ can be identified with tempered
distributions $u \in \cS'(\R^d)$ as in Section \ref{s2}. We say that
the distributions $u_j$ tend to the distribution $u$ \idword{in the
sense of tempered distributions} or in $\cS'(\R^d)$ if $u_j(f)\to
u(f)$ for all $f\in \cS(\R^d)$ as $j\to\infty$. For more about
tempered distributions, see e.g.~Stein and Weiss \cite{stwe} or
Grafakos \cite{gra}.

Now the definition of the \idword{$\ell_q$-$\theta$-means} can be
extended to tempered distributions by
$$
\sigma_T^{q,\theta} f := f * K_T^{q,\theta} \qquad (T>0).
$$
Indeed, the convolution $f*g$ is again well defined for all $g\in
L_1(\R^d)$ and $f\in L_p(\R^d)$ $(1 \leq p\leq \infty)$ or $f\in B$,
where $B$ is a homogeneous Banach space on $\R^d$. For a tempered
distribution $u\in \cS'(\T^d)$ and $g\in \cS(\T^d)$, we define the
tempered distribution $u*g$ by (\ref{e25}) as in Section \ref{s6}.
We can show again that $u*g$ is equal to the function $x\mapsto
u(T_x\tilde g)$, which is a $C^\infty$ function. We say that a
tempered distribution $u$ is \idword{$L_r$-bounded} if $u*g\in
L_r(\R^d)$ for all $g\in \cS(\R^d)$. Then we can define $u*g$ as in
(\ref{e26}) for all $L_{r}$-bounded tempered distributions $u$ and
$g\in L_{r'}(\R^d)$, where $1/r+1/r'=1$.

Since $\theta_0^q\in L_1(\R^d)\cap L_\infty(\R^d)$, we have
$\widehat {\theta}_0^q \in L_{r'}(\R^d)$ $(2\leq r'\leq \infty)$ and
the same holds for $K_T^{q,\theta}$ by (\ref{e14}). Since all
tempered distributions $f\in H_p^\Box(\R^d)$ $(0<p\leq \infty)$ are
$L_r$-bounded for all $p\leq r\leq \infty$, $\sigma_T^{q,\theta} f$
is well defined as a tempered distribution for all $f\in
H_p^\Box(\R^d)$.

In the definition of the \idword{homogeneous Banach space} $B$
containing Lebesgue measurable functions on $\R^d$, we have to
replace (iii) at the beginning of Section \ref{s6} by
\begin{enumerate}
\item []
\begin{enumerate}
\item [(iii')] for every compact set $K\subset \R^d$ there exists a
constant $C_K$ such that
$$
\int_{K} |f|\dd \lambda \leq C_K \|f\|_B \qquad (f\in B).
$$
\end{enumerate}
\end{enumerate}

The \idword{Hardy space} \inda{$H_p^\Box(\R^d)$}, which is defined
with the help of the \idword{non-periodic Poisson kernel}
$$
P_t^d(x):= {c_d t \over (t^2 + |x|^2)^{(d+1)/2}} \qquad (t>0, x \in
\R^d),\index{\file-1}{$P_t^d$}
$$
has the same properties as the periodic space $H_p^\Box(\T^d)$. We
get the same Hardy space with equivalent norms if we use the kernel
$\phi_t$ instead of the Poisson kernel, where $\phi\in \cS(\R^d)$,
$\int_{\R^d} \phi\dd \lambda\neq 0$ and
$$
\phi_t(x):=t^{-d}\phi(x/t)  \qquad (t>0, x \in
\R^d).\index{\file-1}{$\phi_t$}
$$
Starting from this definition, we can verify that a tempered
distribution from $H_p^\Box(\R^d)$ is $L_r$-bounded for all $p\leq
r\leq \infty$ (see Stein \cite[p.~100]{st1}).

For a tempered distribution $f$, the \idword{conjugate
distributions} or \idword{Riesz transforms} are defined by
$$
(\tilde f^{(i)})^\land(t) := -\ii\ {t_i \over \|t\|_2} \widehat f(t)
\qquad (t\in\R^d,i=1,\ldots,d).\index{\file-1}{$\tilde f^{(i)}$}
$$
One can show that for an integrable function $f$,
$$
\tilde f^{(i)}(x) := {\rm p.v.} \ \int_\T f(x-t) \Phi_i(t) \dd t :=
\lim_{\epsilon\to 0} \int_{\epsilon<\|t\|_2} f(x-t)\Phi_i(t) \dd t
\qquad \mbox{a.e.},
$$
where
$$
\widehat \Phi_i(t)= -\ii\ {t_i \over \|t\|_2}, \qquad \Phi_i(t)=
{c_d t_i \over \|t\|_2^{d+1}} \qquad (t\in\R^d).
$$
Here, $c_d$ is the constant appearing in the definition of the
Poisson kernel. In the one-dimensional case the transform is called
again the Hilbert transform and
$$
\widehat \Phi(t)= -\ii\ {\rm sign} \ t, \qquad \Phi(t)= {1 \over \pi
t} \qquad (t\in\R).
$$

We have proved in \cite{wel1-ft1,wamalg-hardy,wel1-ft2,wmar6} that
the same results hold for the operator $\sigma_T^{q,\theta}$ and for
the \idword{maximal operator}
$$
\sigma_*^{q,\theta}f := \sup_{T \in \R_+} |\sigma_{T}^{q,\theta}
f|\index{\file-1}{$\sigma_*^{q,\theta}f$}
$$
as in Sections \ref{s6}--\ref{s10} with the difference that we can
allow $0<\alpha<\infty$ for $q=1,\infty$ with the critical index
$\frac{d}{d+\alpha\wedge 1}$. It is easy to see that the operators
$\sigma_T^{q,\theta}$ are uniformly bounded on $L_p(\R^d)$ if and
only if $\sigma_1^{q,\theta}$ is bounded on $L_p(\R^d)$. We point
out that the space \inda{$C_u(\R^d)$} of uniformly continuous
bounded functions endowed with the supremum norm is also a
homogeneous Banach space:

\begin{cor}\label{c40}
Assume that (\ref{e8})--(\ref{e11}) are satisfied. If $f$ is a
uniformly continuous and bounded function, then
$$
\lim_{T\to\infty} \sigma_T^{q,\theta} f \to f \qquad
\mbox{uniformly.}
$$
\end{cor}


\sect{Rectangular summability}

We now investigate the other type of summability method, the
\ieword{rectangular summability}. The \dword{rectangular Fej{\'e}r
and Riesz means}\index{\file}{rectangular Fej{\'e}r
means}\index{\file}{rectangular Riesz means} of $f\in L_1(\T^d)$ are
defined by
\begin{eqnarray*}
\sigma_nf(x) &=& \sum_{|k_1|\leq n_1} \cdots \sum_{|k_d|\leq n_d} \prod_{i=1}^{d} \Big(1-\frac{|k_i|}{n_i} \Big) \widehat f(k) \ee^{\ii k \cdot x} \\
&=&\frac{1}{(2\pi)^d}\int_{\T^d} f(x-u) K_n(u) \dd
u\index{\file-1}{$\sigma_nf$}
\end{eqnarray*}
and
\begin{eqnarray*}
\sigma_n^{\alpha}f(x)&=& \sum_{|k_1|\leq n_1} \cdots \sum_{|k_d|\leq
n_d} \prod_{i=1}^{d} \Big(1-\Big(\frac{|k_i|}{n_i}\Big)^\gamma
\Big)^\alpha \widehat f(k) \ee^{\ii k \cdot x} \\
&=&\frac{1}{(2\pi)^d}\int_{\T^d} f(x-u) K_n^{\alpha}(u) \dd
u,\index{\file-1}{$\sigma_n^\alpha f$}
\end{eqnarray*}
respectively, where the \dword{rectangular Fej{\'e}r and Riesz
kernels}\index{\file}{rectangular Fej{\'e}r
kernels}\index{\file}{rectangular Riesz kernels} are given by
$$
K_{n}(u) := \sum_{|k_1|\leq n_1} \cdots \sum_{|k_d|\leq n_d}
\prod_{i=1}^{d} \Big(1-\frac{|k_i|}{n_i} \Big) \ee^{\ii k \cdot u} =
{1 \over \prod_{i=1}^d n_i} \sum_{k_1=1}^{n_1-1} \cdots
\sum_{k_d=1}^{n_d-1} D_k(u)\index{\file-1}{$K_n$}
$$
and
$$
K_{n}^{\alpha}(u) := \sum_{|k_1|\leq n_1} \cdots \sum_{|k_d|\leq
n_d} \prod_{i=1}^{d} \Big(1-\Big(\frac{|k_i|}{n_i}\Big)^\gamma
\Big)^\alpha \ee^{\ii k \cdot u},\index{\file-1}{$K_n^\alpha$}
$$
respectively (see Figure \ref{f24}).
\begin{figure}[ht] 
   \centering
   \includegraphics[width=0.8\textwidth]{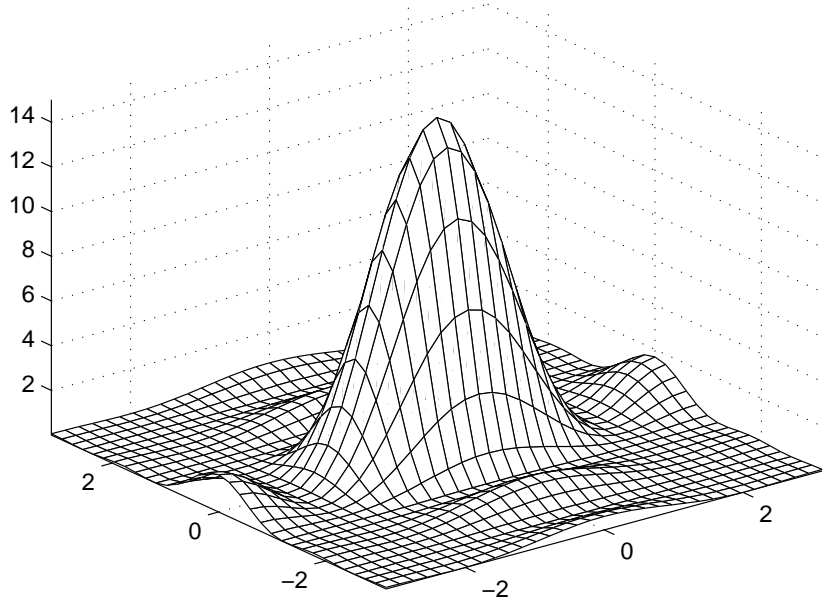}
   \caption{The rectangular Fej{\'e}r kernel $K_n$ with $d=2$, $n_1=3$, $n_2=5$, $\alpha=1$, $\gamma=1$.}
   \label{f24}
\end{figure}

We could choose also different exponents $\alpha_i$ and $\gamma_i$
in the product. Again, the Fej{\'e}r means are the arithmetic means
of the partial sums:
$$
\sigma_n f(x) = {1 \over \prod_{i=1}^d n_i} \sum_{k_1=1}^{n_1-1}
\cdots \sum_{k_d=1}^{n_d-1} s_{k}f(x).
$$

\sect{Norm convergence of rectangular summability means}\label{s13}

We extend the definition of the Fej{\'e}r and Riesz means to
distributions by
$$
\sigma_{n}^{\alpha} f(x):= f * K_{n}^{\alpha} \qquad (\nn).
$$
This is well defined for all $f\in H_p(\T^d)$ $(0<p\leq \infty)$
(see Section \ref{s15}), for all $f\in L_p(\T^d)$ $(1 \leq p\leq
\infty)$ and for all $f\in B$, where $B$ is a homogeneous Banach
space.

The norm convergence of the rectangular means follows immediately
from the one-dimensional result by iteration. Since the
$d$-dimensional Riesz kernel is the Kronecker product of the
one-dimensional kernels,
$$
K_{n}^{\alpha}(u)= (K_{n_1}^{\alpha} \otimes \cdots \otimes
K_{n_d}^{\alpha})(u) :=K_{n_1}^{\alpha}(u_1) \cdots
K_{n_d}^{\alpha}(u_d)\qquad (n\in \N^d),
$$
we obtain easily

\begin{thm}\label{t41}
If $\alpha>0$, then
$$
\int_{\T^d} |K_n^{\alpha}(x)| \dd x \leq C \qquad (n\in \N^d).
$$
\end{thm}

\begin{thm}\label{t42}
If $\alpha>0$ and $B$ is a homogeneous Banach space on $\T^d$, then
$$
\|\sigma_n^{\alpha} f\|_B \leq C \|f\|_B \qquad (n\in \N^d)
$$
and
$$
\lim_{n\to\infty} \sigma_n^{\alpha} f=f \qquad \mbox{in the $B$-norm
for all $f\in B$}.
$$
\end{thm}

Here, the convergence is understood in Pringsheim's sense as in
Theorem \ref{t7}. For the almost everywhere convergence, we
investigate two types of convergence, the restricted convergence (or
the convergence taken in a cone) and the unrestricted convergence
(or the convergence taken in Pringsheim's sense).

\sect{Restricted summability}\label{s14}

\subsection{Summability over a cone}\label{s14.1}

For a given $\tau\geq 1$, we define a cone by
$$
\R_\tau^d:=\{x\in \R_+^d: \tau^{-1} \leq x_i/x_j \leq
\tau,i,j=1,\ldots,d\}.\index{\file-1}{$\R_\tau^d$}
$$
The choice $\tau=1$ obviously yields the diagonal. The
\idword{restricted maximal operator} is defined by
$$
\sigma_\Box^\alpha f := \sup_{n \in \R_\tau^d} |\sigma_{n}^\alpha
f|.\index{\file-1}{$\sigma_\Box^\alpha f$}
$$
As we can see on Figure \ref{f9}, in the restricted maximal operator
the supremum is taken on a cone only.

\begin{figure}[ht] 
   \centering
   \includegraphics[width=1\textwidth]{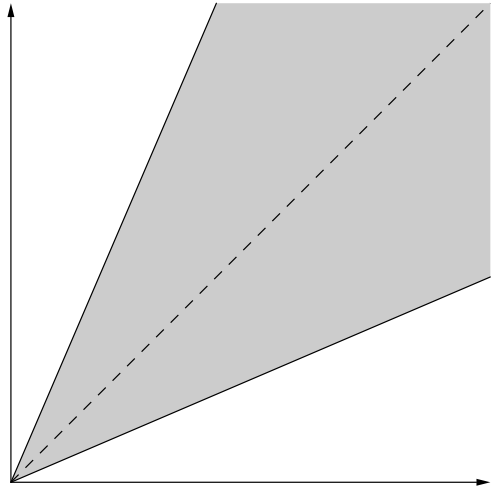}
   \caption{The cone for $d=2$.} \label{f9}
\end{figure}

Marcinkiewicz and Zygmund \cite{mazy} were the first who considered
the restricted convergence. Similarly to Theorem \ref{t21}, the
restricted maximal operator is bounded from $H_{p}^\Box(\T^d)$ to
$L_{p}(\T^d)$ (Weisz \cite{wcesf,wk2}).

\begin{thm}\label{t43} If $\alpha>0$ and $\max\{d/(d+1),1/(\alpha\wedge 1+1)\}< p\leq \infty$, then
$$
\|\sigma^\alpha_\Box f\|_{p} \leq C_{p} \|f\|_{H_{p}^\Box} \qquad
(f\in H_{p}^\Box(\T^d)).
$$
\end{thm}

\begin{proof}
Let
$$
\theta(s):=\cases{(1-|s|^\gamma)^\alpha &if $|s|\leq 1$ \cr 0 &if
$|s|>1$ \cr} \qquad (s\in\R).
$$
By the one-dimensional version of Corollary \ref{c6.2},
$$
|\widehat \theta(x)|,|(\widehat \theta )'(x)| \leq C |x|^{-\alpha-1}
\qquad (x\neq 0).
$$
Taking into account (\ref{e6.3}), we conclude that
\begin{equation}\label{e4.22}
|K_{n_j}^{\alpha}(x)| \leq {C \over n_j^{\alpha} |x|^{\alpha+1}}
\qquad (x\neq 0)
\end{equation}
and
\begin{equation}\label{e4.23}
|(K_{n_j}^{\alpha})'(x)| \leq {C \over n_j^{\alpha-1}
|x|^{\alpha+1}} \qquad (x\neq 0).
\end{equation}
We will prove the result for $d=2$, only. For $d>2$, the
verification is very similar. Instead of $n_1,n_2$ and $I_1,I_2$ we
will write $n,m$ and $I,J$, respectively. Let $a$ be an arbitrary
$H_p^\Box$-atom with support $I\times J$ and
$$
2^{-K-1} < |I|/\pi= |J|/\pi \leq 2^{-K} \qquad (K\in\N).
$$
We can suppose again that the center of $I\times J$ is zero. In this
case,
$$
[- \pi2^{-K-2},  \pi2^{-K-2}] \subset I, J \subset [- \pi2^{-K-1},
\pi2^{-K-1}].
$$

Choose $s\in\N$ such that $2^{s-1}< \tau \leq 2^s$. It is easy to
see that if $n \geq k$ or $m \geq k$, then we have $n,m \geq
k2^{-s}$. Indeed, since $(n,m)$ is in a cone, $n \geq k$ implies $m
\geq \tau^{-1} n \geq k 2^{-s}$. By Theorem \ref{t18}, it is enough
to prove that
\begin{equation}\label{e14.1}
\int_{\T^2\setminus 4(I\times J)} |\sigma_\Box^\alpha a(x,y)|^p \dd
x\dd y \leq C_p.
\end{equation}

First suppose that $\alpha\leq 1$ and let us integrate over
$(\T\setminus 4I) \times 4J$. Obviously,
\begin{eqnarray*}
\int_{\T\setminus 4I} \int_{4J} |\sigma_\Box^{\alpha} a(x,y)|^p \dd
x \dd y &\leq& \sum_{|i|=1}^{2^K-1} \int_{ \pi i 2^{-K}}^{ \pi(i+1)
2^{-K}} \int_{4J}
\sup_{n,m \geq 2^{K-s}} |\sigma_{n,m}^{\alpha} a(x,y)|^p \dd x \dd y \\
&&{} + \sum_{|i|=1}^{2^K-1} \int_{\pi i 2^{-K}}^{\pi (i+1) 2^{-K}}
\int_{4J} \sup_{n,m <2^K} |\sigma_{n,m}^{\alpha} a(x,y)|^p \dd x \dd y \\
&=:& (A)+ (B).
\end{eqnarray*}
We can suppose that $i>0$. Using that
$$
\int_\T |K_m^{\alpha}| \dd \lambda \leq C \qquad (m\in \N),
$$
(\ref{e4.22}) and the definition of the atom, we conclude
\begin{eqnarray*}
|\sigma_{n,m}^{\alpha} a(x,y)| &=& \Big|\int_I \int_J a(t,u)
K_n^{\alpha}(x - t) K_m^{\alpha}(y - u)
\dd t \dd u \Big| \\
&\leq& C_p 2^{2K/p} \int_I {1 \over n^\alpha |x-t|^{\alpha+1}} \dd
t.
\end{eqnarray*}
For $x\in [{\pi i 2^{-K}},{\pi (i+1) 2^{-K}})$ $(i \geq 1)$ and
$t\in I$, we have
\begin{equation}\label{e4.24}
{1 \over |x-t|^\nu} \leq {1 \over (\pi i 2^{-K} -  \pi
2^{-K-1})^\nu} \leq {C 2^{K\nu} \over i^\nu} \qquad (\nu>0).
\end{equation}
From this, it follows that
$$
|\sigma_{n,m}^{\alpha} a(x,y)| \leq C_p 2^{2K/p+K\alpha} {1 \over
n^{\alpha} i^{\alpha+1}}.
$$
Since $n \geq 2^K 2^{-s}$, we obtain
$$
(A) \leq C_p \sum_{i=1}^{2^K-1} 2^{-2K} 2^{2K+K \alpha p} {1 \over
2^{K\alpha p} i^{(\alpha+1)p}} \leq C_p \sum_{i=1}^{2^K-1} {1 \over
i^{(\alpha+1)p}},
$$
which is a convergent series if $p>1/(\alpha+1)$.

To consider $(B)$, let $I=J= (-\mu,\mu)$ and
\begin{equation}\label{e14.2}
A_1(x,v):= \int_{-\pi}^{x} a(t,v) \dd t \qquad \mbox{and} \qquad
A_{2}(x,y):= \int_{-\pi}^{y} A_{1}(x,t) \dd t.
\end{equation}
Then
\begin{equation}\label{e4.25}
|A_k(x,y)| \leq C_p 2^{K(2/p-k)}.
\end{equation}
Integrating by parts, we get that
\begin{equation}\label{e4.26}
\int_I a(t,u) K_n^{\alpha}(x - t) \dd t = A_1(\mu,u) K_n^{\alpha}(x
- \mu) - \int_I A_1(t,u) (K_n^{\alpha})'(x - t) \dd t.
\end{equation}
Recall that the one-dimensional kernel $K_m^\alpha$ satisfies
$$
|K_m^\alpha|\leq Cm \qquad (m\in \N).
$$
For $x\in [{\pi i 2^{-K}},{\pi(i+1) 2^{-K}})$, the inequalities
(\ref{e4.22}), (\ref{e4.24}) and (\ref{e4.25}) imply
\begin{eqnarray*}
\Big|\int_J A_1(\mu,u) K_n^{\alpha}(x - \mu) K_m^{\alpha}(y - u) \dd
u \Big|
&\leq& C_p 2^{2K/p-K} 2^{-K} {1 \over n^\alpha |x-\mu|^{\alpha+1}} m \\
&\leq& C_p 2^{2K/p+K\alpha - K} n^{1-\alpha} {1 \over i^{\alpha+1}}.
\end{eqnarray*}
Moreover, by (\ref{e4.23}), (\ref{e4.24}) and (\ref{e4.25}),
\begin{eqnarray*}
\Big|\int_J \int_I A_1(t,u) (K_n^{\alpha})'(x - t) K_m^{\alpha}(y -
u)
\dd u \dd t \Big| &\leq& C_p 2^{2K/p-K} \int_I {1 \over n^{\alpha-1} |x-t|^{\alpha+1}} \dd t \\
&\leq& C_p 2^{2K/p+K\alpha-K} n^{1-\alpha} {1 \over i^{\alpha+1}}.
\end{eqnarray*}
Consequently,
$$
(B) \leq C_p \sum_{i=1}^{2^K-1} 2^{-2K} 2^{2K+K \alpha p-Kp}
2^{K(1-\alpha)p} {1 \over i^{(\alpha+1)p}} \leq C_p
\sum_{i=1}^{2^K-1} {1 \over i^{(\alpha+1)p}} < \infty,
$$
because $p>1/(\alpha+1)$. Hence, we have proved that in this case
$$
\int_{\T\setminus 4I} \int_{4J} |\sigma_\Box^{\alpha} a(x,y)|^p \dd
x\dd y \leq C_p.
$$

Next, we integrate over $(\T\setminus 4I) \times (\T\setminus 4J)$,
\begin{eqnarray*}
\lefteqn{\int_{\T\setminus 4I} \int_{\T\setminus 4J}
|\sigma_\Box^{\alpha}
a(x,y)|^p \dd x \dd y}\\
&\leq& \sum_{|i|=1}^{\infty} \sum_{|j|=1}^{\infty} \int_{\pi i
2^{-K}}^{\pi(i+1) 2^{-K}} \int_{\pi j 2^{-K}}^{\pi (j+1) 2^{-K}}
\sup_{n,m \geq 2^{K-s}} |\sigma_{n,m}^{\alpha} a(x,y)|^p \dd x \dd y \\
&&{} + \sum_{|i|=1}^{\infty} \sum_{|j|=1}^{\infty} \int_{\pi i
2^{-K}}^{\pi (i+1) 2^{-K}} \int_{\pi j 2^{-K}}^{\pi (j+1) 2^{-K}}
\sup_{n,m <2^K} |\sigma_{n,m}^{\alpha} a(x,y)|^p \dd x \dd y \\
&=:& (C)+ (D).
\end{eqnarray*}
We may suppose again that $i,j>0$.

For $x\in [{\pi i 2^{-K}},{\pi(i+1) 2^{-K}})$ and $y\in [{\pi j
2^{-K}},{\pi (j+1) 2^{-K}})$, we have by (\ref{e4.22}) and
(\ref{e4.24}) that
\begin{eqnarray*}
|\sigma_{n,m}^{\alpha} a(x,y)| &\leq& C_p 2^{2K/p} \int_I {1 \over
n^\alpha |x-t|^{\alpha+1}} \dd t
\int_J {1 \over m^{\alpha} |y-u|^{\alpha+1}} \dd u \\
&\leq& C_p {2^{2K/p+K\alpha+K\alpha} \over n^\alpha m^\alpha
i^{\alpha+1} j^{\alpha+1}}.
\end{eqnarray*}
This implies that
\begin{eqnarray*}
(C) &\leq& C_p \sum_{i=1}^{2^K-1} \sum_{j=1}^{2^K-1} 2^{-2K}
{2^{2K+K\alpha p+K\alpha p} \over 2^{K\alpha p +K\alpha p} i^{(\alpha+1)p} j^{(\alpha+1)p}} \\
&\leq& C_p \sum_{i=1}^{\infty} \sum_{j=1}^{\infty} {1 \over
i^{(\alpha+1)p} j^{(\alpha+1)p}}< \infty.
\end{eqnarray*}

Using (\ref{e4.26}) and integrating by parts in both variables, we
get that
\begin{eqnarray}\label{e14.3}
\lefteqn{\int_I \int_J a(t,u) K_n^{\alpha}(x - t) K_m^{\alpha}(y - u) \dd t \dd u}\n\\
&&{}= - \int_J A_2(\mu,u) K_n^{\alpha}(x - \mu) (K_m^{\alpha})'(y - u) \dd u \n\\
&&{}\quad+ \int_I A_2(t,\mu) (K_n^{\alpha})'(x - t) K_m^{\alpha}(y - \mu) \dd t\n \\
&&{}\quad- \int_I \int_J A_2(t,u) (K_n^{\alpha})'(x - t)
(K_m^{\alpha})'(y - u)
\dd t \dd u \n\\
&&{}=: D_{n,m}^1(x,y)+ D_{n,m}^2(x,y)+ D_{n,m}^3(x,y).
\end{eqnarray}
Note that $A(\mu,-\mu)=A(\mu,\mu)=0$. Since $|K_n^\alpha|\leq Cn$
and (\ref{e4.22}) holds as well, we obtain
$$
|K_n^\alpha(x)| \leq C {n^{\eta + \alpha(\eta-1)} \over
|x|^{(\alpha+1)(1-\eta)}}
$$
for all $0\leq \eta\leq 1$. It is easy to see that
$$
|(K_m^\alpha)'|\leq Cm^2 \qquad (m\in \N).
$$
Then
\begin{equation}\label{e14.4}
|(K_m^\alpha)'(y)| \leq C {m^{2\zeta + (\alpha-1)(\zeta-1)} \over
|y|^{(\alpha+1)(1-\zeta)}} = C {m^{\zeta +1 + \alpha(\zeta-1)} \over
|y|^{(\alpha+1)(1-\zeta)}}
\end{equation}
follows from (\ref{e4.23}) for all $0\leq \zeta\leq 1$. Inequalities
(\ref{e4.24}) and (\ref{e4.25}) imply
\begin{eqnarray}\label{e4.28}
|D_{n,m}^1(x,y)| &\leq& C_p 2^{2K/p-2K} {n^{\eta + \alpha(\eta-1)}
\over |x-\mu|^{(\alpha+1)(1-\eta)}} \int_J {m^{\zeta +1 +
\alpha(\zeta-1)} \over
|y-u|^{(\alpha+1)(1-\zeta)}} \dd u\n \\
&\leq& C_p 2^{2K/p-3K}
n^{\eta + \alpha(\eta-1)} \Bigl({2^{K} \over i}\Bigr)^{(\alpha+1)(1-\eta)} \n \\
&&{} m^{\zeta +1 + \alpha(\zeta-1)} \Bigl({2^{K} \over
j}\Bigr)^{(\alpha+1)(1-\zeta)},
\end{eqnarray}
whenever $x\in [{\pi i 2^{-K}},{\pi(i+1) 2^{-K}})$, $y\in [{\pi j
2^{-K}}, {\pi(j+1) 2^{-K}})$ and $0\leq \eta,\zeta\leq 1$. If
$$
\eta + \alpha(\eta-1)+\zeta +1 + \alpha(\zeta-1) \geq 0,
$$
then
$$
\sup_{n,m <2^K} |D_{n,m}^1(x,y)| \leq C_p 2^{2K/p} {1 \over
i^{(\alpha+1)(1-\eta)}} {1 \over j^{(\alpha+1)(1-\mu)}}
$$
because $(n,m)$ is in a cone. Choosing
$$
\eta:= \zeta:= {2\alpha-1 \over 2(\alpha+1)}\vee 0,
$$
we can see that
\begin{eqnarray*}
\lefteqn{\int_{\T\setminus 4I} \int_{\T\setminus 4J}
\sup_{n,m <2^K} |D^1_{n,m}(x,y)|^p \dd x \dd y }\\
&\leq& C_p \sum_{i=1}^{\infty} \sum_{j=1}^{\infty} 2^{-2K} 2^{2K} {1
\over i^{3p/2\wedge (\alpha+1)p}} {1 \over
j^{3p/2\wedge(\alpha+1)p}},
\end{eqnarray*}
which is a convergent series. The analogous estimate for
$|D^2_{n,m}(x,y)|$ can be similarly proved.

For $x\in [{\pi i 2^{-K}},{\pi (i+1) 2^{-K}})$ and $y\in [{\pi j
2^{-K}},{\pi (j+1) 2^{-K}})$, we conclude that
\begin{eqnarray*}
|D^3_{n,m}(x,y)| &\leq& C_p 2^{2K/p-2K} \int_I {1 \over n^{\alpha-1}
|x-t|^{\alpha+1}} \dd t
\int_J {1 \over m^{\alpha-1} |y-u|^{\alpha+1}} \dd u \\
&\leq& C_p {2^{2K/p-2K+K\alpha+K\alpha} n^{1-\alpha} m^{1-\alpha}
\over i^{\alpha+1} j^{\alpha+1}}.
\end{eqnarray*}
So
\begin{eqnarray*}
\lefteqn{\int_{\T\setminus 4I} \int_{\T\setminus 4J}
\sup_{n,m <2^K} |D^3_{n,m}(x,y)|^p \dd x \dd y}\\
&&{}\leq C_p \sum_{i=1}^{2^K-1} \sum_{j=1}^{2^K-1} 2^{-2K}
{2^{2K-2Kp+K\alpha p+K\alpha p} 2^{K(2-\alpha-\alpha)p} \over i^{(\alpha+1)p} j^{(\alpha+1)p}}\\
&&{}\leq C_p \sum_{i=1}^{\infty} \sum_{j=1}^{\infty} {1 \over
i^{(\alpha+1)p}} {1 \over j^{(\alpha+1)p}} < \infty
\end{eqnarray*}
by the hypothesis. The integration over $4I \times (\T\setminus 4J)$
can be done as above. This finishes the proof of (\ref{e14.1}).

Now let $\alpha>1$. Since $|\widehat {\theta}|\leq C$ and
$|(\widehat \theta )'(x)|\leq C$ trivially and since
$|x|^{-\alpha-1}\leq |x|^{-2}$ if $|x|\geq 1$, we conclude that
$$
|\widehat \theta(x)|,|(\widehat \theta )'(x)| \leq C |x|^{-2} \qquad
(x\neq 0).
$$
Hence
$$
|K_{n_j}^{\alpha}(x)| \leq {C \over n_j |x|^{2}}, \qquad
|(K_{n_j}^{\alpha})'(x)| \leq {C \over|x|^{2}} \qquad (x\neq 0)
$$
and (\ref{e14.1}) can be proved as above. The theorem follows from
Theorem \ref{t18}.
\end{proof}

\begin{rem}\label{r14.3} \rm
In the $d$-dimensional case, the constant $d/(d+1)$ appears if we
investigate the corresponding term to $D_{n,m}^1$. More exactly, if
we integrate the term
$$
\int_{I_d} A(\mu,\cdots,\mu,u) K_{n_1}^{\alpha}(x_1 - \mu) \cdots
K_{n_{d-1}}^{\alpha}(x_{d-1} - \mu) (K_{n_d}^{\alpha})'(x_d - u) \dd
u
$$
over $(\T\setminus 4I_1) \times \cdots \times (\T\setminus 4I_d)$
similarly to (\ref{e4.28}), then we get that $p>d/(d+1)$.
\end{rem}

Since $H_p^\Box(\T^d) \sim L_p(\T^d)$ for $1<p \leq \infty$, we have
$$
\|\sigma_\Box^{\alpha} f\|_{p} \leq C_{p} \|f\|_{p} \qquad (f\in
L_{p}(\T^d),1<p \leq \infty).
$$
As we have seen in Theorem \ref{t22}, in the one-dimensional case,
the operator $\sigma_\Box^{\alpha}$ is not bounded from
$H_p^\Box(\T)$ to $L_p(\T)$ if $0<p\leq 1/(\alpha+1)$ and
$\alpha=1$. Using interpolation and Theorem \ref{t43}, we obtain the
weak type $(1,1)$ inequality.

\begin{cor}\label{c44}
If $\alpha>0$ and $f \in L_1(\T^d)$, then
$$
\sup_{\rho >0} \rho \lambda(\sigma^\alpha_\Box f > \rho) \leq C
\|f\|_{1}.
$$
\end{cor}

The density argument of Marcinkiewicz and Zygmund (Theorem \ref{t4})
implies

\begin{cor}\label{c45}
If $\alpha>0$ and $f \in L_1(\T^d)$, then
$$
\lim_{n\to \infty, \, n\in \R_\tau^d} \sigma^\alpha_{n}f = f \qquad
\mbox{a.e.}
$$
\end{cor}

This result was proved by Marcinkiewicz and Zygmund \cite{mazy} for
the two-dimensional Fej{\'e}r means. The general version of
Corollary \ref{c45} is due to the author \cite{wcesf,wk2}.

Similarly to Theorem \ref{t30} and Corollary \ref{c31}, we can prove
for the \idword{conjugate operators}
$$
\tilde \sigma_{n}^{(i);\alpha} f(x):= \tilde f^{(i)} *
K_{n}^{\alpha} \qquad
(i=0,1,\ldots,d)\index{\file-1}{$\tilde\sigma_n^{(i);\alpha}f$}
$$
and
$$
\tilde \sigma_\Box^{(i);\alpha}f := \sup_{n \in \R_\tau^d} |\tilde
\sigma_{n}^{(i);\alpha}
f|\index{\file-1}{$\tilde\sigma_\Box^{(i);\alpha}f$}
$$
the following results.

\begin{thm}\label{t46}
If $\alpha>0$ and $\max\{d/(d+1),1/(\alpha\wedge 1+1)\}< p<\infty$,
then for all $i=0,1,\ldots,d$,
$$
\|\tilde \sigma_\Box^{(i);\alpha}f\|_{p} \leq C_{p}
\|f\|_{H_{p}^\Box} \qquad (f\in H_{p}^\Box(\T^d))
$$
and
$$
\|\tilde \sigma_{n}^{(i);\alpha} f\|_{H_{p}^\Box} \leq C_{p}
\|f\|_{H_{p}^\Box} \qquad (n\in \N^d, f\in H_{p}^\Box(\T^d)).
$$
In particular, if $f \in L_1(\T^d)$, then
$$
\sup_{\rho>0}\rho\,\lambda(\tilde \sigma_\Box^{(i);\alpha}f > \rho)
\leq C \|f\|_{1}. $$
\end{thm}

The proof of this result is similar to that of Theorem \ref{t30}.

\begin{cor}\label{c47}
If $\alpha>0$, $i=0,1,\ldots,d$ and $f\in L_1(\T^d)$, then
$$
\lim_{n\to \infty, \, n\in \R_\tau^d} \tilde \sigma_{n}^{(i);\alpha}
f = \tilde f^{(i)} \qquad \mbox{a.e.}
$$
Moreover, if $f \in H_{p}^\Box(\T^d)$ with
$\max\{d/(d+1),1/(\alpha\wedge 1+1)\}< p <\infty$, then this
convergence also holds in the $H_{p}^\Box(\T^d)$-norm.
\end{cor}

\subsection{Summability over a cone-like set}

G{\'a}t \cite{ga11} generalized Corollary \ref{c45} to so called
cone-like sets. Suppose that for all $j=2,\ldots,d$,
$\gamma_j:\R_+\to \R_+$ are strictly increasing and continuous
functions such that $\lim_{j\to\infty}\gamma_j=\infty$ and
$\lim_{j\to +0}\gamma_j=0$. Moreover, suppose that there exist
$c_{j,1},c_{j,2},\xi>1$ such that
\begin{equation}\label{e21}
c_{j,1} \gamma_j(x) \leq \gamma_j(\xi x) \leq c_{j,2}\gamma_j(x)
\qquad (x>0).
\end{equation}
Note that this is satisfied if $\gamma_j$ is a power function. Let
us define the numbers $\omega_{j,1}$ and $\omega_{j,2}$ via the
formula
\begin{equation}\label{e14.6}
c_{j,1}=\xi^{\omega_{j,1}} \quad \mbox{and} \quad
c_{j,2}=\xi^{\omega_{j,2}} \qquad (j=2,\ldots,d).
\end{equation}
For convenience, we extend the notations for $j=1$ by
$\gamma_1:=\cI$, $c_{1,1}=c_{1,2}=\xi$. Here $\cI$ denotes the
identity function $\cI(x)=x$. Let
$\gamma=(\gamma_1,\ldots,\gamma_d)$ and
$\tau=(\tau_1,\ldots,\tau_d)$ with $\tau_1=1$ and fixed $\tau_j\geq
1$ $(j=2,\ldots,d)$. We define the \idword{cone-like set} (with
respect to the first dimension) by
\begin{equation}\label{e29}
\R_{\tau,\gamma}^d:=\{x\in \R_+^d: \tau_j^{-1}\gamma_j(n_1)\leq
n_j\leq \tau_j\gamma_j(n_1), j=2,\ldots,d\}.
\index{\file-1}{$\R_{\tau,\gamma}^d$}
\end{equation}
Figure \ref{f12} shows a cone-like set for $d=2$.

\begin{figure}[ht] 
   \centering
   \includegraphics[width=0.7\textwidth]{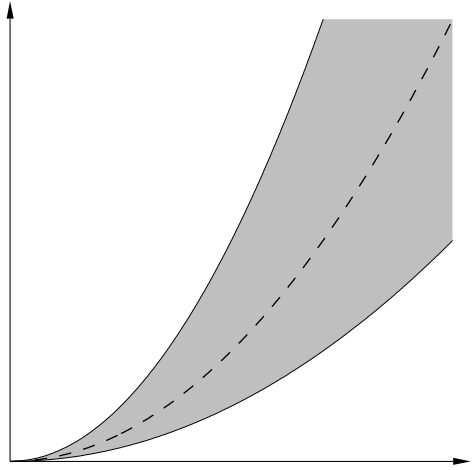}
   \caption{Cone-like set for $d=2$.} \label{f12}
\end{figure}

The condition on $\gamma_j$ seems to be natural, because G{\'a}t
\cite{ga11} proved in the two-dimensional case that to each
cone-like set with respect to the first dimension there exists a
larger cone-like set with respect to the second dimension and
conversely, if and only if (\ref{e21}) holds.

For given $\gamma,\tau$ satisfying the above conditions, the
\idword{restricted maximal operator} is defined by
$$
\sigma_\gamma^{\alpha}f := \sup_{n\in \R_{\tau,\gamma}^d}
|\sigma_{n}^\alpha f|.\index{\file-1}{$\sigma_\gamma^\alpha f$}
$$
If $\gamma_j=\cI$ for all $j=2,\ldots,d$, then we get a cone
investigated above. Replacing the definition of the \idword{Hardy
space} $H_p^\Box(\T^d)$ by
$$
\|f\|_{H_{p}^\gamma}:= \|\sup_{t>0} |f * (P_{t} \otimes
P_{\gamma_2(t)} \otimes \cdots \otimes P_{\gamma_d(t)})| \|_{p},
$$
we can prove all the theorems of Subsection \ref{s14.1} for
\inda{$H_p^\gamma(\T^d)$} and $\sigma_\gamma^{\alpha}$ (see Weisz
\cite{wcone1}). Here
$$
P_{t}(x) := \sum_{k \in \Z} \ee^{-t |k|} \ee^{\ii kx}= {1-r^2 \over
1+r^2 - 2r \cos x} \qquad (x \in
\T,t>0,r=\ee^{-t})\index{\file-1}{$P_t$}
$$
is the \dword{one-dimensional Poisson kernel}.\index{\file}{Poisson
kernel} If each $\gamma_j=\cI$, we get back the Hardy spaces
$H_p^\Box(\T^d)$. We have to modify slightly the definition of
atoms. A bounded function $a$ is an \idword{$H_p^\gamma$-atom} if
there exists a rectangle $I:=I_1\times\cdots\times I_d\subset \T^d$
with $|I_j|=\gamma_j(|I_1|^{-1})^{-1}$ such that
\begin{enumerate}
\item []
\begin{enumerate}
\item [(i)] ${\rm supp} \ a \subset I$,
\item [(ii)] $\|a\|_\infty \leq |I|^{-1/p}$,
\item[(iii)] $\int_{I} a(x) x^{k}\dd x = 0$
for all multi-indices $k$ with $|k|\leq \lfloor d(1/p-1) \rfloor$.
\end{enumerate}
\end{enumerate}
Theorems \ref{t17} and \ref{t18} are valid in this case as well.

Let $H$ be an arbitrary subset of $\{1,\ldots,d\}$, $H\neq
\emptyset$, $H\neq\{1,\ldots,d\}$ and $H^c:=\{1,\ldots,d\}\setminus
H$. Define
$$
p_1:= \sup_{H\subset \{1,\ldots,d\}}\frac{\sum_{j\in H}
\omega_{j,2}+\sum_{j\in H^c} \omega_{j,1}}{\sum_{j\in H}
\omega_{j,2} +2\sum_{j\in H^c} \omega_{j,1}},
$$
where the numbers $\omega_{j,1}$ and $\omega_{j,2}$ are defined in
(\ref{e14.6}).

\begin{thm}\label{t14.1} If $\alpha>0$ and $\max\{p_1,1/(\alpha\wedge 1+1)\}< p\leq \infty$, then
$$
\|\sigma^\alpha_\gamma f\|_{p} \leq C_{p} \|f\|_{H_{p}^\gamma}
\qquad (f\in H_{p}^\gamma(\T^d)).
$$
\end{thm}

\begin{proof}
We will again prove the result only for $d=2$. For $d>2$, the
verification is very similar. To simplify notation, instead of
$n_1,n_2$, $c_{2,1},c_{2,2}$ and $\omega_{2,1},\omega_{2,2}$, we
will write $n,m$, $c_{1},c_{2}$ and $\omega_{1},\omega_{2}$,
respectively. Let $a$ be an arbitrary $H_p^\gamma$-atom with support
$I\times J$, $|J|^{-1}=\gamma(|I|^{-1})$ and
$$
2^{-K-1} < |I|/\pi \leq 2^{-K}, \qquad \gamma(2^{K+1})^{-1} <
|J|/\pi \leq \gamma(2^{K})^{-1}
$$
for some $K\in\N$. We can suppose that the center of $R$ is zero. In
this case
$$
[-\pi2^{-K-2}, \pi2^{-K-2}] \subset I \subset [-\pi2^{-K-1},
\pi2^{-K-1}],
$$
$$
[-\pi\gamma(2^{K+1})^{-1}/2, \pi\gamma(2^{K+1})^{-1}/2] \subset J
\subset [-\pi\gamma(2^{K})^{-1}/2, \pi\gamma(2^{K})^{-1}/2].
$$

Assume that $\alpha\leq 1$. To prove (\ref{e14.1}), first we
integrate over $\T^2 \setminus 4(I\times J)$,
\begin{eqnarray*}
\int_{\T\setminus 4I} \int_{4J} |\sigma_\gamma^{\alpha} a(x,y)|^p
\dd x \dd y &\leq& \int_{\T\setminus
4I} \int_{4J} \sup_{n\geq 2^K,(n,m)\in \R_{\tau,\gamma}^d} |\sigma_{n,m}^{\alpha} a(x,y)|^p \dd x \dd y \n\\
&&{} + \int_{\T\setminus 4I} \int_{4J}
\sup_{n<2^K,(n,m)\in \R_{\tau,\gamma}^d} |\sigma_{n,m}^{\alpha} a(x,y)|^p \dd x \dd y \n\\
&=:& (A)+ (B).
\end{eqnarray*}

If $n\geq 2^K$ and $x\in [{\pi i 2^{-K}},{\pi(i+1) 2^{-K}})$ $(i
\geq 1)$, then by (\ref{e4.24}),
\begin{eqnarray*}
|\sigma_{n,m}^{\alpha} a(x,y)| &=& \Big|\int_I \int_J a(t,u)
K_n^{\alpha}(x - t) K_m^{\alpha}(y - u)
\dd t \dd u \Big| \\
&\leq& C_p 2^{K/p} \gamma(2^{K})^{1/p} \int_I {1 \over n^\alpha |x-t|^{\alpha+1}} \dd t\\
&\leq& C_p 2^{K/p+K\alpha} \gamma(2^{K})^{1/p} {1 \over n^{\alpha} i^{\alpha+1}} \\
&\leq& C_p 2^{K/p} \gamma(2^{K})^{1/p} {1 \over i^{\alpha+1}}.
\end{eqnarray*}
Then
\begin{eqnarray*}
(A) &\leq& \sum_{i=1}^{2^{K}-1} \int_{\pi i 2^{-K}}^{\pi (i+1)
2^{-K}} \int_{4J} \sup_{n\geq 2^K}
|\sigma_{n,m}^{\alpha} a(x,y)|^p \dd x \dd y \\
&\leq& C_p \sum_{i=1}^{2^{K}-1} 2^{-K} \gamma(2^{K})^{-1} 2^{K}
\gamma(2^{K}) {1 \over
i^{(\alpha+1)p}} \\
&=& C_p \sum_{i=1}^{2^K-1} {1 \over i^{(\alpha+1)p}},
\end{eqnarray*}
which is a convergent series if $p>1/(1+\alpha)$.

We estimate $(B)$ by
\begin{eqnarray*} (B) &\leq & \sum_{k=0}^{\infty} \int_{\T\setminus 4I}
\int_{4J} \sup_{\frac{2^K}{\xi^{k+1}}\leq n<\frac{2^K}{\xi^k},(n,m)\in \R_{\tau,\gamma}^d} |\sigma_{n,m}^{\alpha} a(x,y)|^p \dd x \dd y \\
&\leq & \sum_{k=0}^{\infty} \Big(\int_{\T\setminus
[-\frac{\pi\xi^k}{2^K},\frac{\pi\xi^k}{2^K}]}
\int_{4J} + \int_{[-\frac{\pi\xi^k}{2^K},\frac{\pi\xi^k}{2^K}]} \int_{4J} \Big) \\
&&{} \sup_{\frac{2^K}{\xi^{k+1}}\leq n<\frac{2^K}{\xi^k},(n,m)\in
\R_{\tau,\gamma}^d} |\sigma_{n,m}^{\alpha} a(x,y)|^p
\dd x \dd y \\
&=:& (B_1)+(B_2).
\end{eqnarray*}
If $(n,m)\in \R_{\tau,\gamma}^d$ and $\frac{2^K}{\xi^{k+1}}\leq
n<\frac{2^K}{\xi^k}$, then $m<\tau\gamma(\frac{2^K}{\xi^k})$. The
inequality $|K_m^{\alpha}|\leq Cm$ and (\ref{e4.22}) imply
\begin{eqnarray*}
|\sigma_{n,m}^{\alpha} a(x,y)| &\leq &
C_p 2^{K/p} \gamma(2^{K})^{1/p-1} m \int_I {1 \over n^\alpha |x-t|^{\alpha+1}} \dd t \\
&\leq& C_p 2^{K/p-K} \gamma(2^{K})^{1/p-1} \gamma(\frac{2^K}{\xi^k})
(\frac{2^K}{\xi^{k+1}})^{-\alpha} |x-\pi2^{-K-1}|^{-\alpha-1}.
\end{eqnarray*}
Hence
\begin{eqnarray*}
(B_1) &\leq & C_p \sum_{k=0}^{\infty} 2^{K(1-p-\alpha p)}
\gamma(2^{K})^{-p} \gamma(\frac{2^K}{\xi^k})^p \xi^{k\alpha p}
\int_{\T\setminus [-\frac{\pi\xi^k}{2^K},\frac{\pi\xi^k}{2^K}]}
|x-\pi2^{-K-1}|^{-(\alpha+1)p
}\dd x\\
&\leq & C_p \sum_{k=0}^{\infty} 2^{K(1-p-\alpha p)}
\gamma(2^{K})^{-p} \gamma(\frac{2^K}{\xi^k})^p \xi^{k\alpha p}
(\xi^k 2^{-K})^{-(\alpha+1)p+1}.
\end{eqnarray*}
Since $\gamma(x) \leq c_1^{-1}\gamma(\xi x)$ by (\ref{e21}), we
conclude
$$
(B_1) \leq C_p \sum_{k=0}^{\infty} \gamma(2^{K})^{-p} \gamma(2^K)^p
c_1^{-kp} \xi^{k(1-p)} = C_p \sum_{k=0}^{\infty}
\xi^{k(1-p-\omega_1p)},
$$
which is convergent if $p>1/(1+\omega_1)$. Note that
$1/(1+\omega_1)<(1+\omega_1)/(1+2\omega_1)\leq p_1$.

For $(B_2)$, we obtain similarly that
\begin{equation}\label{e14.7}
|\sigma_{n,m}^{\alpha} a(x,y)| \leq C_p 2^{K/p-K}
\gamma(2^{K})^{1/p-1} n m \leq C_p 2^{K/p-K} \gamma(2^{K})^{1/p-1}
\frac{2^K}{\xi^{k}} \gamma(\frac{2^K}{\xi^k})
\end{equation}
and
$$
(B_2) \leq C_p \sum_{k=0}^{\infty}\frac{\xi^k}{2^K} \gamma(2^K)^{-1}
2^{K} \gamma(2^{K})^{1-p} \xi^{-kp} \gamma(\frac{2^K}{\xi^k})^p \leq
C_p \sum_{k=0}^{\infty} \xi^{k(1-p)} c_1^{-kp},
$$
which was just considered. Hence, we have proved that
$$
\int_{\T\setminus 4I} \int_{4J} |\sigma_\gamma^{\alpha} a(x,y)|^p
\dd x\dd y \leq C_p \qquad  (p_1<p\leq 1).
$$

The integral over $4I \times (\T\setminus 4J)$ can be handled with a
similar idea. Indeed, let us denote the terms corresponding to $(A),
(B), (B_1), (B_2)$ by $(A'), (B'), (B_1'), (B_2')$. If we take the
integrals in $(A')$ over $4I\times [\pi j \gamma(2^{K})^{-1},
\pi(j+1) \gamma(2^{K})^{-1}]$ $(j=1,\ldots,\gamma(2^{K})/2-1)$, then
we get in the same way that $(A')$ is bounded if $p>1/(1+\alpha)$.
For $(B_1')$, we can see that
\begin{eqnarray*}
(B_1') &=& \sum_{k=0}^{\infty} \int_{4I} \int_{\T\setminus
[-\pi\gamma(\frac{2^K}{\xi^k})^{-1},\pi
\gamma(\frac{2^K}{\xi^k})^{-1}]}
\sup_{\frac{2^K}{\xi^{k+1}}\leq n<\frac{2^K}{\xi^k},(n,m)\in \R_{\tau,\gamma}^d} |\sigma_{n,m}^{\alpha} a(x,y)|^p \dd x \dd y\\
&\leq & C_p \sum_{k=0}^{\infty} 2^{K} \gamma(2^{K}) 2^{-K}2^{-Kp}
\int_{\T\setminus [-\pi\gamma(\frac{2^K}{\xi^k})^{-1}, \pi\gamma(\frac{2^K}{\xi^k})^{-1}]} \\
&&{} \sup_{\frac{2^K}{\xi^{k+1}}\leq n<\frac{2^K}{\xi^k},(n,m)\in
\R_{\tau,\gamma}^d} \Big(n
\int_J {1 \over m^\alpha |y-u|^{\alpha+1}} \dd u \Big)^p \dd y\\
&\leq &
C_p \sum_{k=0}^{\infty} \xi^{-kp} \gamma(2^{K})^{1-p} \gamma(\frac{2^K}{\xi^{k+1}})^{-\alpha p} \\
&&{} \int_{\T\setminus [-\pi\gamma(\frac{2^K}{\xi^k})^{-1},
\pi\gamma(\frac{2^K}{\xi^k})^{-1}]}
|y-\pi\gamma(2^K)^{-1}/2|^{-(\alpha+1)p} \dd y\\
&\leq & C_p \sum_{k=0}^{\infty} \xi^{-kp} \gamma(2^{K})^{1-p}
\gamma(\frac{2^K}{\xi^{k}})^{p-1}\\
&\leq &
C_p \sum_{k=0}^{\infty} \xi^{-kp} c_2^{k(1-p)} \\
&=& C_p \sum_{k=0}^{\infty} \xi^{k(\omega_2-\omega_2 p-p)}
\end{eqnarray*}
and this converges if $p>\omega_2/(1+\omega_2)$, which is less than
$(1+\omega_2)/(2+\omega_2)\leq p_1$. Using (\ref{e14.7}), we
establish that
\begin{eqnarray*}
(B_2') &=& \sum_{k=0}^{\infty} \int_{4I}
\int_{[-\gamma(\frac{2^K}{\xi^k})^{-1},
\gamma(\frac{2^K}{\xi^k})^{-1}]} \sup_{\frac{2^K}{\xi^{k+1}}\leq
n<\frac{2^K}{\xi^k},(n,m)\in \R_{\tau,\gamma}^d}
|\sigma_{n,m}^{\alpha} a(x,y)|^p
\dd x \dd y\\
&\leq& C_p \sum_{k=0}^{\infty} 2^{-K} \gamma(\frac{2^K}{\xi^k})^{-1}
2^{K} \gamma(2^{K})^{1-p}
\xi^{-kp} \gamma(\frac{2^K}{\xi^k})^p \\
&\leq& C_p \sum_{k=0}^{\infty} \xi^{-kp} c_2^{k(1-p)}.
\end{eqnarray*}
Hence
$$
\int_{4I} \int_{\T\setminus 4J} |\sigma_\gamma^{\alpha} a(x,y)|^p
\dd x\dd y \leq C_p \qquad  (p_1<p\leq 1).
$$

Integrating over $(\T\setminus 4I) \times (\T\setminus 4J)$, we
decompose the integral as
\begin{eqnarray*}
\int_{\T\setminus 4I} \int_{\T\setminus 4J} |\sigma_\gamma^{\alpha}
a(x,y)|^p \dd x \dd y &\leq& \int_{\T\setminus 4I} \int_{\T\setminus
4J} \sup_{n\geq 2^K,(n,m)\in \R_{\tau,\gamma}^d}
|\sigma_{n,m}^{\alpha} a(x,y)|^p \dd x \dd y \n\\
&&{} + \int_{\T\setminus 4I} \int_{\T\setminus 4J}
\sup_{n<2^K,(n,m)\in \R_{\tau,\gamma}^d} |\sigma_{n,m}^{\alpha} a(x,y)|^p \dd x \dd y \n\\
&=:& (C)+ (D)
\end{eqnarray*}
and
$$
(C) \leq \sum_{i=1}^{2^{K}-1} \sum_{j=1}^{\gamma(2^K)/2-1} \int_{\pi
i 2^{-K}}^{\pi(i+1) 2^{-K}} \int_{\pi j
\gamma(2^{K})^{-1}}^{\pi(j+1) \gamma(2^{K})^{-1}} \sup_{n\geq 2^K}
|\sigma_{n,m}^{\alpha} a(x,y)|^p \dd x \dd y.
$$
For $x\in [{\pi i 2^{-K}},{\pi(i+1) 2^{-K}})$ and $y\in [\pi j
\gamma(2^{K})^{-1},\pi(j+1) \gamma(2^{K})^{-1})$, we have by
(\ref{e4.22}) and (\ref{e4.24}) that
\begin{eqnarray}\label{e14.5}
|\sigma_{n,m}^{\alpha} a(x,y)| &\leq& C_p 2^{K/p} \gamma(2^{K})^{1/p} \int_I {1 \over n^\alpha |x-t|^{\alpha+1}} \dd t \int_J {1 \over m^{\alpha} |y-u|^{\alpha+1}} \dd u \n \\
&\leq& C_p {2^{K/p+K\alpha} \gamma(2^{K})^{1/p+\alpha} \over n^\alpha m^\alpha i^{\alpha+1} j^{\alpha+1}}\n\\
&\leq& C_p {2^{K/p} \gamma(2^{K})^{1/p} \over i^{\alpha+1}
j^{\alpha+1}}.
\end{eqnarray}
Then
$$
(C) \leq C_p \sum_{i=1}^{2^{K}-1} \sum_{j=1}^{\gamma(2^K)/2-1} {1
\over i^{(\alpha+1)p} j^{(\alpha+1)p}}< \infty
$$
if $p>1/(1+\alpha)$.

To consider $(D)$ let us define $A_1(x,y)$, $A_2(x,y)$,
$D_{n,m}^1(x,y)$, $D_{n,m}^2(x,y)$ and $D_{n,m}^3(x,y)$ as in
(\ref{e14.2}) and (\ref{e14.3}), respectively, and let
$I=[-\mu,\mu]$, $J=[-\nu,\nu]$. Then
\begin{equation}\label{e14.25}
|A_1(x,u)| \leq 2^{K/p-K} \gamma(2^{K})^{1/p}, \qquad |A_2(x,y)|
\leq 2^{K/p-K} \gamma(2^{K})^{1/p-1}.
\end{equation}
Obviously,
\begin{eqnarray*}
\lefteqn{\int_{\T\setminus 4I} \int_{\T\setminus 4J} \sup_{n<2^K,(n,m)\in \R_{\tau,\gamma}^d} |D_{n,m}^1(x,y)|^p \dd x \dd y} \\
&\leq & \sum_{k=0}^{\infty} \int_{\T\setminus 4I} \int_{\T\setminus
4J} \sup_{\frac{2^K}{\xi^{k+1}}\leq
n<\frac{2^K}{\xi^k},(n,m)\in \R_{\tau,\gamma}^d} |D_{n,m}^1(x,y)|^p \dd x \dd y\\
&\leq & \sum_{k=0}^{\infty} \sum_{j=1}^{\gamma(2^K)/2-1}
\Big(\int_{\T\setminus [-\frac{\pi\xi^k}{2^K},\frac{\pi\xi^k}{2^K}]}
\int_{\pi j \gamma(2^{K})^{-1}}^{\pi(j+1) \gamma(2^{K})^{-1}} +
\int_{[-\frac{\pi\xi^k}{2^K},\frac{\pi\xi^k}{2^K}]} \int_{\pi j
\gamma(2^{K})^{-1}}^{\pi(j+1)
\gamma(2^{K})^{-1}} \Big) \\
&&{}\sup_{n<2^K,(n,m)\in \R_{\tau,\gamma}^d} |D_{n,m}^1(x,y)|^p \dd x \dd y\\
&=:& (D_1)+(D_2).
\end{eqnarray*}
It follows from (\ref{e4.24}), (\ref{e14.4}) and (\ref{e14.25}) that
\begin{eqnarray*}
|D_{n,m}^1(x,y)| &\leq & C_p 2^{K/p-K} \gamma(2^{K})^{1/p-2}
\frac{1}{n^{\alpha}|x-\mu|^{\alpha+1}}
\frac{m^{\zeta +1 + \alpha(\zeta-1)} }{|y-\nu|^{(\alpha+1)(1-\zeta)}} \\
&\leq& C_p 2^{K/p-K} \gamma(2^{K})^{1/p-2+(\alpha+1)(1-\zeta)}
\frac{(\frac{2^K}{\xi^{k+1}})^{-\alpha}}{|x-\mu|^{\alpha+1}}
\frac{\gamma(\frac{2^K}{\xi^k})^{\zeta+1+\alpha(\zeta-1)}}{j^{(\alpha+1)(1-\zeta)}},
\end{eqnarray*}
where $0\leq \zeta \leq 1$. This leads to
\begin{eqnarray*}
(D_1) &\leq & C_p \sum_{k=0}^{\infty} \sum_{j=1}^{\gamma(2^K)/2-1}
\int_{\T\setminus
[-\frac{\pi\xi^k}{2^K},\frac{\pi\xi^k}{2^K}]} \\
&&{}2^{K(1-p-\alpha p)} \gamma(2^{K})^{p(-2+(\alpha+1)(1-\zeta))}
\xi^{k\alpha p} |x-\mu|^{-(\alpha+1)p}
\frac{\gamma(\frac{2^K}{\xi^k})^{p(2+(\alpha+1)(\zeta-1))}}{j^{p(\alpha+1)(1-\zeta)}} \dd x\\
&\leq & C_p \sum_{k=0}^{\infty} \sum_{j=1}^{\gamma(2^K)/2-1}
2^{K(1-p-\alpha p)} \xi^{k\alpha p}
(\xi^k 2^{-K})^{-(\alpha+1)p+1} \frac{c_1^{-kp(2+(\alpha+1)(\zeta-1))}}{j^{p(\alpha+1)(1-\zeta)}} \\
&\leq & C_p \sum_{k=0}^{\infty} \sum_{j=1}^{\gamma(2^K)/2-1}
\frac{\xi^{k(1-p
-\omega_1p(2+(\alpha+1)(\zeta-1)))}}{j^{p(\alpha+1)(1-\zeta)}},
\end{eqnarray*}
which is convergent if $p>1/(1+\omega_1(2+(\alpha+1)(\zeta-1)))$ and
$p>1/(\alpha+1)(1-\zeta)$. After some computation, we can see that
the optimal bound is reached if
$$
\zeta=\frac{\alpha-\omega_1+\alpha\omega_1}{1+\alpha+\omega_1+\alpha\omega_1},
$$
which means that $p>(1+\omega_1)/(1+2\omega_1)$.

Considering $(D_2)$, we estimate as follows:
\begin{eqnarray*}
|D_{n,m}^1(x,y)| &\leq & C_p 2^{K/p-K} \gamma(2^{K})^{1/p-2} n
\frac{m^{\zeta +1 + \alpha(\zeta-1)}}{|y-\nu|^{(\alpha+1)(1-\zeta)}} \\
&\leq& C_p 2^{K/p} \gamma(2^{K})^{1/p-2+(\alpha+1)(1-\zeta)}
\xi^{-k}
\frac{\gamma(\frac{2^K}{\xi^k})^{\zeta+1+\alpha(\zeta-1)}}{j^{(\alpha+1)(1-\zeta)}}
\end{eqnarray*}
and
\begin{eqnarray*}
(D_2) &\leq & C_p \sum_{k=0}^{\infty} \sum_{j=1}^{\gamma(2^K)/2-1} \\
&&{}\int_{ [-\frac{\pi\xi^k}{2^K},\frac{\pi\xi^k}{2^K}]} 2^{K}
\gamma(2^{K})^{p(-2+(\alpha+1)(1-\zeta))} \xi^{-kp}
\frac{\gamma(\frac{2^K}{\xi^k})^{p(2+(\alpha+1)(\zeta-1))}}{j^{p(\alpha+1)(1-\zeta)}} \dd x \\
&\leq & C_p \sum_{k=0}^{\infty} \sum_{j=1}^{\gamma(2^K)/2-1}
\frac{\xi^{k(1-p
-\omega_1p(2+(\alpha+1)(\zeta-1)))}}{j^{p(\alpha+1)(1-\zeta)}} \\
&\leq& C_p
\end{eqnarray*}
as above.

The term $D_{n,m}^2$ can be handled similarly. We obtain
$$
\int_{\T\setminus 4I} \int_{\T\setminus 4J} \sup_{n<2^K,(n,m)\in
\R_{\tau,\gamma}^d} |D_{n,m}^2(x,y)|^p \dd x \dd y \leq C_p
$$
if $p>(1+\omega_2)/(2+\omega_2)$.

Using (\ref{e4.23}), we estimate $D_{n,m}^3$ in the same way as
$(C)$ in (\ref{e14.5}). Now the exponents of $n$ and $m$ are
non-negative and so they can be estimated by $2^K$ and $\gamma(2^K)$
as in (\ref{e14.5}). This proves that
$$
\int_{\T\setminus 4I} \int_{\T\setminus 4J} |\sigma_\gamma^{\alpha}
a(x,y)|^p \dd x \dd y \leq C_p.
$$
This completes the proof for $0<\alpha\leq 1$. For $\alpha>1$, the
proof can be finished as in Theorem \ref{t43}.
\end{proof}

\begin{rem}\label{r14.1}\rm
In the $d$-dimensional case, the constant $p_1$ appears if we
investigate the terms corresponding to $D_{n,m}^1$ and $D_{n,m}^2$.
Indeed, let $\prod_{j=1}^{d}I_j$ be centered at $0$ and the support
of the atom $a$, $A$ be the integral of $a$, $I_j=:[-\mu_j,\mu_j]$
and
$$
\overline{t}_j:=\left\{
  \begin{array}{ll}
    \mu_j, & \hbox{$j\in H$;} \\
    t_j, & \hbox{$j\in H^c$,}
  \end{array}
\right.
$$
$H\subset\{1,\ldots,d\}$, $H\neq \emptyset$, $H\neq\{1,\ldots,d\}$.
If we integrate the term
$$
\int_{\prod_{j\in H^c}I_j} A(\overline{t}_1,\ldots,\overline{t}_d)
\prod_{j\in H} K_{n_j}^{\alpha}(x_j - \mu_j) \prod_{i\in H^c}
(K_{n_i}^{\alpha})'(x_i - t_i) \dd t
$$
over $\prod_{j=1}^{d}(\T\setminus 4I_j)$, then we get that
$$
p>\frac{\sum_{j\in H} \omega_{j,2}+\sum_{j\in H^c}
\omega_{j,1}}{\sum_{j\in H} \omega_{j,2} +2\sum_{j\in H^c}
\omega_{j,1}}.
$$
Moreover, considering the integral
$$
\int_{\prod_{j\in H} (\T\setminus 4I_j)} \int_{\prod_{j\in H^c}
4J_j} |\sigma_\gamma^{\alpha} a(x)|^p \dd x,
$$
we obtain
$$
p>\frac{\sum_{j\in H} \omega_{j,2}}{\sum_{j\in H} \omega_{j,2}
+\sum_{j\in H^c} \omega_{j,1}}.
$$
However, this bound is less than $p_1$.
\end{rem}

\begin{rem}\label{r14.2} \rm
If $\omega_{j,1}=\omega_{j,2}=1$ for all $j=1,\ldots,d$, then we
obtain in Theorem \ref{t14.1} the bound
$$
\max \{d/(d+1),1/(\alpha+1)\}.
$$
In particular, this holds if $\gamma_j=\cI$ for all $j=1,\ldots,d$,
i.e., if we consider a cone. This bound was obtained for cones in
Theorem \ref{t43}.
\end{rem}

\begin{cor}\label{c14.44}
If $\alpha>0$ and $f \in L_1(\T^d)$, then
$$
\sup_{\rho >0} \rho \lambda(\sigma^\alpha_\gamma f > \rho) \leq C
\|f\|_{1}.
$$
\end{cor}

\begin{cor}\label{c14.45}
If $\alpha>0$ and $f \in L_1(\T^d)$, then
$$
\lim_{n\to \infty,n\in \R_{\tau,\gamma}^d} \sigma^\alpha_{n}f = f
\qquad \mbox{a.e.}
$$
\end{cor}

In the two-dimensional case, Corollaries \ref{c14.44} and
\ref{c14.45} were proved by G{\'a}t \cite{ga11} for Fej{\'e}r
summability. In this case, he verified also that if the cone-like
set $\R_{\tau,\gamma}^d$ is defined by $\tau_j(n_1)$ instead of
$\tau_j$ and if $\tau_j(n_1)$ is not bounded, then Corollary
\ref{c14.45} does not hold and the largest space for the elements of
which we have a.e.\ convergence is $L\log L$. This means that under
these conditions Theorem \ref{t14.1} cannot be true for any $p<1$.

\sect{$H_p(\T^d)$ Hardy spaces}\label{s15}

For the investigation of the unrestricted almost everywhere
convergence, we need a new type of Hardy space, the so called
product Hardy spaces. A function $f$ is in the \idword{product Hardy
space} \inda{$H_p(\T^d)$}, in the \idword{product weak Hardy space}
\inda{$H_{p,\infty}(\T^d)$} and in the \idword{hybrid Hardy space}
\inda{$H_p^i(\T^d)$} if
$$
\|f\|_{H_{p}}:= \|\sup_{t_k>0, k=1,\ldots,d} |f *
(P_{t_1}\otimes\cdots\otimes P_{t_d})|\|_{p}< \infty,
$$
$$
\|f\|_{H_{p,\infty}}:= \|\sup_{t_k>0, k=1,\ldots,d} |f *
(P_{t_1}\otimes\cdots\otimes P_{t_d})|\|_{p,\infty}< \infty
$$
and
$$
\|f\|_{H_{p}^i}:= \|\sup_{t_k>0, k=1,\ldots,d;k\neq i} |f *
(P_{t_1}\otimes\cdots\otimes P_{t_{i-1}} \otimes
P_{t_{i+1}}\otimes\cdots\otimes P_{t_{d}})|\|_{p}< \infty,
$$
respectively, where $0<p\leq \infty$ and $P_{t}$ is the
one-dimensional Poisson kernel. It is known (see Chang and Fefferman
\cite{chfe,chfe1,chfe2}, Gundy and Stein \cite{gust} or Weisz
\cite{wk2}) that
$$
H_p(\T^d) \sim H_p^i(\T^d) \sim L_p(\T^d) \qquad (1<p \leq \infty)
$$
and $H_1(\T^d) \subset H_1^i(\T^d) \subset H_{1,\infty}(\T^d)\cap
L_1(\T^d)$. Moreover,
\begin{equation}\label{e15}
\|f\|_{H_{1,\infty}} \leq C \|f\|_{H_1^i} \qquad (f \in
H_1^i(\T^d),i=1,\ldots,d).
\end{equation}
Let the set \inda{$L(\log L)^{d-1}(\T^d)$} contain all measurable
functions for which
$$
\|\, |f| (\log^+|f|)^{d-1}\|_1 < \infty.
$$
Then $H_1^i(\T^d) \supset L(\log L)^{d-1}(\T^d)$ for all
$i=1,\ldots,d$ and
\begin{equation}\label{e16}
\|f\|_{H_1^{i}} \leq C + C \| |f| (\log^+|f|)^{d-1}\|_1 \qquad (f\in
L(\log L)^{d-1}(\T^d)).
\end{equation}

Considering the product Hardy spaces, we have to introduce new
\idword{conjugate distributions} defined by
$$
\tilde f^{(j_1,\ldots,j_d)} \sim \sum_{n \in\Z^d} \Big(\prod_{i=1}^d
(-\ii\ {\rm sign} \ n_{i})^{j_i} \Big) \widehat f(n) \ee^{\ii n
\cdot x} \qquad (j_i=0,1).\index{\file-1}{$\tilde
f^{(j_1,\ldots,j_d)}$}
$$
In the case $f$ is an integrable function,
\begin{eqnarray*}
\tilde f^{(j_1,\ldots,j_d)}(x) &=& {\rm p.v.}
\ \int_\T \cdots \int_\T {f(x_1-t_1^{j_1},\ldots,x_d-t_d^{j_d}) \over \prod_{i=1}^{d} (2\pi \tan(t_i/2))^{j_i}} \dd t^{j_1}\cdots \dd t^{j_d}\\
&:=& \lim_{\epsilon\to 0} \int_{\epsilon_1<|t_1|<\pi} \cdots
\int_{\epsilon_d<|t_d|<\pi} {f(x_1-t_1^{j_1},\ldots,x_d-t_d^{j_d})
\over \prod_{i=1}^{d} (2\pi \tan(t_i/2))^{j_i}} \dd t^{j_1}\cdots
\dd t^{j_d} \qquad \mbox{a.e.},
\end{eqnarray*}
where $\dd t^{j_i}$ means the ordinary integral if $j_i=1$ while the
integral with respect to the $i$th variable is omitted if $j_i=0$.

If $j_i=0$ for all $i=1,\ldots,d$, let $\tilde
f^{(j_1,\ldots,j_d)}=f$. The following result can be found in Gundy
and Stein \cite{gust,gu2}, Chang and Fefferman \cite{chfe} and Weisz
\cite{wsing,wk2}.

\begin{thm}\label{t54}
For $f\in H_p(\T^d)$,
$$
\|\tilde f^{(j_1,\ldots,j_d)}\|_{H_p} = C_p \|f\|_{H_p} \qquad
(j_i=0,1)
$$
and
$$
\|f\|_{H_p} \sim \sum_{j_1=0}^{1} \cdots \sum_{j_d=0}^{1} \|\tilde
f^{(j_1,\ldots,j_d)}\|_p \qquad (0<p<\infty).
$$
\end{thm}

We note again that this theorem holds also for $L_p(\T^d)$ spaces
when $1<p<\infty$.

The atomic decomposition for $H_p(\T^d)$ is much more complicated
than for $H_p^\Box(\T^d)$. One reason for this is that the support
of an atom is not a rectangle but an open set. Moreover, here we
have to choose the atoms from $L_2(\T^d)$ instead of $L_
\infty(\T^d)$. This atomic decomposition was proved by Chang and
Fefferman \cite{chfe,fe1} and Weisz \cite{wca4,wk2}. For an open set
$F \subset (\T^d)$, denote by $\cM(F)$ the set of the maximal dyadic
subrectangles of $F$.

A function $a\in L_2(\T^d)$ is a \idword{$H_p$-atom} if
\begin{enumerate}
\item []
\begin{enumerate}
\item ${\rm supp} \ a \subset F$ for some open set $F\subset \T^d$,
\item ${\Vert a \Vert }_2 \leq |F|^{1/2-1/p}$,
\item $a$ can be further decomposed into the sum of ``elementary particles'' $a_R \in L_2(\T^d)$, $a=\sum_{R\in \cM(F)} a_R$ in $L_2(\T^d)$, satisfying
\begin{enumerate}
\item [(a)] ${\rm supp} \ a_R \subset 2R \subset F$,
\item [(b)] for $i=1,\ldots,d$, $k\leq \lfloor 2/p-3/2 \rfloor $ and
$R\in \cM(F)$, we have
$$
\int_\T a_R(x) x_i^k \dd x_i  = 0,
$$
\item [(c)] for every disjoint partition $\cP_l$ $(l=1,2,\ldots)$
of $\cM(F)$,
$$
\Big(\sum_{l} \|\sum_{R\in \cP_l} a_R \|_2^2 \Big)^{1/2} \leq
|F|^{1/2-1/p}.
$$
\end{enumerate}
\end{enumerate}
\end{enumerate}
 The basic result about
the \ieword{atomic decomposition} was proved by Chang and Fefferman
\cite{chfe1,chfe2,chfe} (see also Weisz \cite{wk2}).

\begin{thm}\label{t48}
A function $f$ is in $H_p(\T^d)$ $(0<p \leq 1)$ if and only if there
exist a sequence $(a^k,k \in {\N})$ of $H_p$-atoms and a sequence
$(\mu_k,k \in {\N})$ of real numbers such that
$$
\sum_{k=0}^\infty |\mu_k|^p < \infty \quad \mbox{and} \quad
\sum_{k=0}^{\infty} \mu_ka^k=f \quad \mbox{in the sense of
distributions}.
$$
Moreover,
$$
{\Vert f \Vert}_{H_p} \sim \inf \Big(\sum_{k=0}^\infty |\mu_k|^p
\Big)^{1/p}.
$$
\end{thm}

The result corresponding to Theorem \ref{t18} for the $H_p(\T^d)$
space is much more complicated. Since the definition of the
$H_p$-atom is very complex, to obtain a usable condition about the
boundedness of the operators, we have to introduce simpler atoms.

If $d=2$, a function $a$ is a \idword{simple $H_p$-atom} if
\begin{enumerate}
\item []
\begin{enumerate}
\item ${\rm supp} \ a \subset R$ for some rectangle $R\subset \T^2$,
\item ${\Vert a \Vert }_2 \leq |R|^{1/2-1/p}$,
\item $\int_{\T} a(x) x_i^k \dd x_i  = 0$ for all $i=1,2$ and
$k\leq \lfloor 2/p-3/2 \rfloor$.
\end{enumerate}
\end{enumerate}

Note that not every $f\in H_p(\T^2)$ can be decomposed into simple
$H_p$-atoms. A counterexample can be found in Weisz \cite{wk}.
However, the following result says that for an operator $V$ to be
bounded from $H_p(\T^2)$ to $L_p(\T^2)$ $(0<p\leq 1)$, it is enough
to check $V$ on simple $H_p$-atoms and the boundedness of $V$ on
$L_2(\T^2)$. It can be proved with the help of an idea due to
Fefferman \cite{fe1} (see the proof of Theorem \ref{t18} and
\cite{wk2}).

\begin{thm}\label{t49}
For each $n\in \N^2$, let $V_n:L_1(\T^2)\to L_1(\T^2)$ be a bounded
linear operator and
$$
V_*f:=\sup_{n\in \N^2} |V_nf|.\index{\file-1}{$V_*f$}
$$
Let $d=2$ and $0<p_0\leq 1$. Suppose that there exists $\eta>0$ such
that for every simple $H_{p_0}$-atom $a$ and for every $r \geq 1$
$$
\int_{\T^2 \setminus R^r} |V_*a|^{p_0} \dd \lambda \leq C_p 2^{-
\eta r},
$$
where $R$ is the support of $a$. If $V_*$ is bounded from
$L_2(\T^2)$ to $L_2(\T^2)$, then
\begin{equation}\label{e19}
\|V_*f\|_{p} \leq C_{p_0} \|f\|_{H_{p}} \qquad (f\in H_{p}(\T^2)\cap
H_1^i(\T^2))
\end{equation}
for all $p_0\leq p\leq 2$ and $i=1,\ldots ,d$. If $\lim_{k\to\infty}
f_k= f$ in the $H_p$-norm implies that $\lim_{k\to\infty}
V_nf_k=V_nf$ in the sense of distributions $(n\in \N^d)$, then
(\ref{e19}) holds for all $f\in H_p(\T^2)$.
\end{thm}

Unfortunately, the preceding theorem is not true for higher
dimensions (Journ\'e \cite{jo2}). So, there are fundamental
differences between the theory in the two-parameter and three- or
multi-parameter cases. Fefferman asked in \cite{fe2} whether one can
find sufficient conditions for an operator to be bounded from
$H_p(T^d)$ to $L_p(\T^d)$ in higher dimensions. The following
theorem answers this problem.

If $d\geq 3$, $a$ is called a \idword{simple $H_p$-atom} if there
exist intervals $I_i\subset \T$, $i=1,\ldots,j$ for some $1\leq
j\leq d-1$, such that
\begin{enumerate}
\item []
\begin{enumerate}
\item ${\rm supp} \ a \subset I_1\times \cdots \times I_j \times A$ for some measurable set $A \subset \T^{d-j}$,
\item ${\Vert a \Vert }_2 \leq (|I_1| \cdots |I_j| |A|)^{1/2-1/p}$,
\item $\int_{\T} a(x) x_i^k \dd x_i = \int_A a \dd \lambda =0$ for all
$i=1,\ldots,j$ and $k\leq \lfloor 2/p-3/2 \rfloor$.
\end{enumerate}
\end{enumerate}

If $j=d-1$, we may suppose that $A=I_d$ is also an interval. Of
course if $a\in L_2(\T^d)$ satisfies these conditions for another
subset of $\{1,\ldots,d\}$ than $\{1,\ldots,j\}$, then it is also
called a simple $H_p$-atom. As for the $H_p^\Box(\T^d)$ spaces, we
could suppose that the integrals in (iii) of all definitions of
atoms or simple atoms are zero for all $k$ for which $k\leq N$,
where $N \geq \lfloor 2/p-3/2 \rfloor $.

As in the two-parameter case, not every $f\in H_p(\T^d)$ can be
decomposed into simple $H_p$-atoms. The next theorem is due to the
author \cite{wca4,wk2}. Let $H^c$ denote the complement of the set
$H$.

\begin{thm}\label{t50}
For each $n\in \N^d$, let $V_n:L_1(\T^d)\to L_1(\T^d)$ be a bounded
linear operator and
$$
V_*f:=\sup_{n\in \N^d} |V_nf|.\index{\file-1}{$V_*f$}
$$
Let $d\geq 3$ and $0<p_0\leq 1$. Suppose that there exist
$\eta_1,\ldots,\eta_d>0$ such that for every simple $H_{p_0}$-atom
$a$ and for every $r_1\ldots,r_d \geq 1$
$$
\int_{(I_1^{r_1})^c \times\cdots\times (I_j^{r_j})^c} \int_A
|V_*a|^{p_0} \dd \lambda \leq C_{p_0} 2^{-\eta_1 r_1} \cdots
2^{-\eta_j r_j},
$$
where $I_1\times \cdots \times I_j \times A$ is the support of $a$.
If $j=d-1$ and $A=I_d$ is an interval, then we also assume that
$$
\int_{(I_1^{r_1})^c \times\cdots\times (I_{d-1}^{r_{d-1}})^c}
\int_{(I_d)^c} |V_*a|^{p_0} \dd \lambda \leq C_{p_0} 2^{-\eta_1 r_1}
\cdots 2^{-\eta_{d-1} r_{d-1}}.
$$
If $V_*$ is bounded from $L_2(\T^d)$ to $L_2(\T^d)$, then
\begin{equation}\label{e20}
\|V_*f\|_{p} \leq C_{p} \|f\|_{H_{p}} \qquad (f\in H_{p}(\T^d)\cap
H_1^i(\T^d))
\end{equation}
for all $p_0\leq p\leq 2$ and $i=1,\ldots ,d$. If $\lim_{k\to\infty}
f_k= f$ in the $H_p$-norm implies that $\lim_{k\to\infty}
V_nf_k=V_nf$ in the sense of distributions $(n\in \N^d)$, then
(\ref{e20}) holds for all $f\in H_p(\T^d)$.
\end{thm}

Inequalities (\ref{e19}) or (\ref{e20}) imply by interpolation that
the operator
\begin{equation}\label{e33}
V_*  \quad \mbox{is bounded from} \quad H_{p,\infty}(\T^d) \quad
\mbox{to} \quad L_{p,\infty}(\T^d)
\end{equation}
when $p_0<p<2$. If $p_0<1$ in Theorems \ref{t49} or \ref{t50}, then
(\ref{e33}) holds also for $p=1$. Thus $V_*$ is of weak type
$(H_1^i,L_1)$ by (\ref{e15}):
$$
\sup_{\rho >0} \rho \, \lambda(|V_*f| > \rho) = \|V_*f\|_{1,\infty}
\leq C \|f\|_{H_{1,\infty}} \leq C \|f\|_{H_1^i}
$$
for all $f \in H_1^i(\T^d)$, $i=1,\ldots,d$.

\begin{cor}\label{c51}
If $p_0<1$ in Theorems \ref{t49} or \ref{t50}, then for all $f \in
H_1^i(\T^d)$ and $i=1,\ldots,d$
$$
\sup_{\rho>0} \rho \, \lambda(|V_*f| > \rho) \leq C \|f\|_{H_1^i}.
$$
\end{cor}

\sect{Unrestricted summability}\label{s16}

Let us define the \idword{unrestricted maximal operator} by
$$
\sigma_*^\alpha f := \sup_{n \in \N^d} |\sigma_{n}^\alpha
f|.\index{\file-1}{$\sigma_*^\alpha f$}
$$
We will first prove that the operator $\sigma_*^\alpha$ is bounded
from $L_p(\T^d)$ to $L_p(\T^d)$ $(1<p\leq \infty)$ and then that it
is bounded from $H_p(\T^d)$ to $L_p(\T^d)$ $(1/(\alpha+1)<p\leq 1)$.
To this end, we introduce the one-dimensional operators
$$
\tau_n^\alpha f(x) :=  \frac{1}{2\pi}\int_{\T} f(x-u)
|K_n^\alpha(u)| \dd u = f* |K_n^\alpha|(x)
$$
and
$$
\tau_*^{\alpha}f := \sup_{n \in \N} |\tau_{n}^{\alpha} f|.
$$
Obviously,
\begin{equation}\label{e16.4}
|\sigma_{n}^{\alpha} f| \leq \tau_{n}^{\alpha} |f| \quad (\nn)
\qquad  \mbox{and} \qquad \sigma_{*}^{\alpha} f \leq
\tau_{*}^{\alpha} |f|.
\end{equation}
The next result can be proved as was Theorem \ref{t43}.

\begin{thm}\label{t16.1} If $\alpha>0$ and $1/(\alpha\wedge 1+1)< p\leq \infty$, then
$$
\|\tau^\alpha_* f\|_{p} \leq C_{p} \|f\|_{H_{p}} \qquad (f\in
H_{p}(\T)).
$$
\end{thm}

\begin{proof}
It is easy to see that
$$
\|\tau_*^{\alpha} f\|_\infty \leq C \|f\|_\infty \qquad (f\in
L_\infty(\T)).
$$
Let $\alpha\leq 1$ and $a$ be an arbitrary $H_p^\Box$-atom with
support $I\subset \T$ and
$$
[- \pi2^{-K-2},  \pi2^{-K-2}] \subset I \subset [- \pi2^{-K-1},
\pi2^{-K-1}].
$$
Then
\begin{eqnarray*}
\int_{\T\setminus 4I} |\tau_*^{\alpha} a(x)|^p \dd x &\leq&
\sum_{|i|=1}^{2^K-1} \int_{ \pi i 2^{-K}}^{ \pi(i+1) 2^{-K}}
\sup_{n \geq 2^{K}} |\tau_{n}^{\alpha} a(x)|^p \dd x \\
&&{} + \sum_{|i|=1}^{2^K-1} \int_{\pi i 2^{-K}}^{\pi (i+1) 2^{-K}} \sup_{n<2^K} |\tau_{n}^{\alpha} a(x)|^p \dd x \\
&=:& (A)+ (B).
\end{eqnarray*}
Using (\ref{e4.22}), (\ref{e4.23}) and (\ref{e4.24}), we can see
that
$$
|\tau_{n}^{\alpha} a(x)| =\Big|\int_I a(t) |K_n^{\alpha}(x - t)|\dd
t \Big| \leq C_p 2^{K/p} \int_I {1 \over n^\alpha |x-t|^{\alpha+1}}
\dd t \leq C_p 2^{K/p} {1 \over i^{\alpha+1}}
$$
and
$$
(A) \leq C_p \sum_{i=1}^{2^K-1} 2^{-K} 2^{K}{1 \over
i^{(\alpha+1)p}} \leq C_p
$$
as in Theorem \ref{t43}.

To estimate $(B)$, observe that by (iii) of the definition of the
atom,
$$
\tau_{n}^{\alpha} a(x) = \int_I a(t) |K_n^{\alpha}(x - t)| \dd t =
\int_I a(t) (|K_n^{\alpha}(x - t)|-|K_n^{\alpha}(x)|) \dd t.
$$
Thus,
$$
|\tau_{n}^{\alpha} a(x)| \leq \int_I |a(t)| |K_n^{\alpha}(x -
t)-K_n^{\alpha}(x)| \dd t.
$$
Using Lagrange's mean value theorem and (\ref{e4.23}), we conclude
$$
|K_n^{\alpha}(x - t)-K_n^{\alpha}(x)|=|(K_n^{\alpha})'(x -\xi)| |t|
\leq {C_p 2^{-K} \over n^{\alpha-1} |x-\xi|^{\alpha+1}} \leq {C_p
2^{K} \over i^{\alpha+1}},
$$
where $\xi \in I$ and $x\in [{\pi i 2^{-K}},{\pi (i+1) 2^{-K}})$.
Consequently,
$$
|\tau_{n}^{\alpha} a(x)| \leq C_p 2^{K/p-K} {2^K \over i^{\alpha+1}}
$$
and
$$
(B) \leq C_p \sum_{i=1}^{2^K-1} 2^{-K} 2^{K} {1 \over
i^{(\alpha+1)p}} \leq C_p.
$$
If $\alpha>1$, then the theorem can be proved in the same way.
\end{proof}

We get by interpolation that
$$
\sup_{\rho >0} \rho \,\lambda(\tau_*^\alpha f > \rho) \leq C
\|f\|_{1} \qquad (f\in L_1(\T)).
$$
This, (\ref{e16.4}) and Theorem \ref{t16.1} yield
\begin{equation}\label{e16.5}
\|\sigma^\alpha_* f\|_{p} \leq C_{p} \|f\|_{{p}} \qquad (f\in
L_{p}(\T), 1<p<\infty)
\end{equation}
and
$$
\sup_{\rho >0} \rho \,\lambda(\sigma_*^\alpha f > \rho) \leq C
\|f\|_{1} \qquad (f\in L_1(\T)).
$$
Note that the last inequality is exactly Theorem \ref{t3}.

Now, we return to the higher dimensional case and verify the
$L_p(\T^d)$ boundedness of $\sigma_*^{\alpha}$.

\begin{thm}\label{t16.2}
If $\alpha>0$ and $1<p\leq \infty$, then
$$
\|\sigma_*^{\alpha} f \|_p \leq C_p \|f\|_p \qquad (f\in L_p(\T^d)).
$$
\end{thm}

\begin{proof}
We may suppose that $d=2$. Applying Theorem \ref{t16.1} and
(\ref{e16.5}), we have
\begin{eqnarray*}
\lefteqn{\int_\T \int_\T \sup_{n,m\in\N} \Big|\int_\T \int_\T
f(t,u) K_n^{\alpha}(x-t) K_m^{\alpha}(y-u) \dd t \dd u \Big|^p \dd x\dd y } \n\\
&\leq& \int_\T \int_\T \sup_{m\in \N} \Big(\int_\T
\Big(\sup_{n\in\N}
\Big|\int_\T f(t,u) K_n^{\alpha}(x-t) \dd t \Big|\Big) |K_m^{\alpha}(y-u)| \dd u\Big)^p \dd y\dd x \\
&\leq& C_p \int_\T \int_\T \sup_{n\in\N} \Big|\int_\T f(t,y) K_n^{\alpha}(x-t) \dd t \Big|^p \dd x\dd y \\
&\leq& C_p \int_\T \int_\T |f(x,y)|^p \dd x\dd y,
\end{eqnarray*}
which proves the theorem.
\end{proof}

The next result is due to the author (\cite{wca4,wk2}).

\begin{thm}\label{t52} If $\alpha>0$ and $1/(\alpha+1)< p\leq \infty$, then
$$
\|\sigma^\alpha_* f\|_{p} \leq C_{p} \|f\|_{H_{p}} \qquad (f\in
H_{p}(\T^d)).
$$
\end{thm}

\begin{proof}
We sketch the proof by giving only the main ideas. Similarly to
(\ref{e4.22}) and (\ref{e4.23}), we have for the Riesz kernels that
\begin{equation}\label{e16.26}
|(K_{n_j}^{\alpha})^{(k)}(x)| \leq {C \over n_j^{\alpha-k}
|x|^{\alpha+1}} \qquad (x\neq 0, k\in \N).
\end{equation}
We will prove the theorem only for $d=3$, because the proof is
similar for larger $d$ or for $d=2$. Choose a simple $H_p$-atom $a$
with support $R=I_1\times I_2 \times A$ where $I_1$ and $I_2$ are
intervals with
$$
2^{-K_i-1} < |I_i|/\pi \leq 2^{-K_i} \qquad (K_i\in\N,i=1,2)
$$
and
$$
[-\pi 2^{-K_i-2}, \pi 2^{-K_i-2}] \subset I_i \subset [-\pi
2^{-K_i-1}, \pi 2^{-K_i-1}].
$$
We assume that $r_i\geq 2$ are arbitrary integers. Theorem
\ref{t16.2} implies that the operator $\sigma^\alpha_*$ is bounded
from $L_2(\T^d)$ to $L_2(\T^d)$. By Theorem \ref{t50}, it is enough
to show that
\begin{equation}\label{e16.1}
\int_{(I_1^{r_1})^c} \int_{(I_2^{r_2})^c} \int_A
|\sigma_*^\alpha(x)|^{p} \dd x \leq C_{p} 2^{-\eta_1 r_1} 2^{-\eta_2
r_2},
\end{equation}
and, if $A=I_3$ is also an interval,
\begin{equation}\label{e16.2}
\int_{(I_1^{r_1})^c} \int_{(I_2^{r_2})^c} \int_{(I_2)^c}
|\sigma_*^\alpha(x)|^{p} \dd x \leq C_{p} 2^{-\eta_1 r_1}
2^{-\eta_{2} r_{2}}
\end{equation}
for all $1/(\alpha+1)<p\leq 1$.

First, we decompose the supremum as
\begin{eqnarray}\label{e16.3}
\sigma_*^{\alpha} a &\leq& \sup_{n_1 < 2^{K_1},n_2 < 2^{K_2} \atop
n_3\in\N} |\sigma_{n}^{\alpha} a|+
\sup_{n_1 \geq 2^{K_1},n_2 < 2^{K_2} \atop n_3\in\N} |\sigma_{n}^{\alpha} a| \n\\
&&{} + \sup_{n_1 < 2^{K_1},n_2 \geq 2^{K_2} \atop n_3\in\N}
|\sigma_{n}^{\alpha} a|+ \sup_{n_1 \geq 2^{K_1},n_2 \geq 2^{K_2}
\atop n_3\in\N} |\sigma_{n}^{\alpha} a|.
\end{eqnarray}
We will investigate only the second term. Obviously,
\begin{eqnarray*}
\lefteqn{\int_{(I_1^{r_1})^c} \int_{(I_2^{r_2})^c} \int_{A}
\sup_{n_1 \geq 2^{K_1},n_2 < 2^{K_2} \atop n_3\in\N} |\sigma_{n}^{\alpha} a(x)|^p \dd x } \n\\
&\leq& \sum_{|i_1|=2^{r_1-2}}^{2^{K_1}-1}
\sum_{|i_2|=2^{r_2-2}}^{2^{K_2}-1} \int_{\pi i_1 2^{-K_1}}^{\pi
(i_1+1) 2^{-K_1}} \!\!\!\int_{\pi i_2 2^{-K_2}}^{\pi (i_2+1)
2^{-K_2}} \!\!\! \int_{A} \sup_{n_1 \geq 2^{K_1},n_2 < 2^{K_2} \atop
n_3\in\bN} |\sigma_{n}^{\alpha} a(x)|^p \dd x,
\end{eqnarray*}
where we may suppose that $i_l>0$. Let $A_{0,0,0}(x):= a(x)$ and
$$
A_{k_1+1,k_2,k_3}(x):= \int_{-\pi}^{x_1} A_{k_1,k_2,k_3}(t,x_2,x_3)
\dd t \qquad (k_i\in\N).
$$
In the other indices, we use the same definition. By (iii) of the
definition of the simple $H_p$-atom, we can show that ${\rm supp} \
A_{k_1,k_2,0} \subset R$ and $A_{k_1,k_2,0}(x)$ is zero if $x_1$ is
at the boundary of $I_1$ or $x_2$ is at the boundary of $I_2$ for
$k_i=0,\ldots,N(p)+1$ $(i=1,2)$, where $N(p) \geq \lfloor 2/p-3/2
\rfloor $. Moreover, using (ii), we can compute that
\begin{equation}\label{e16.27}
\|A_{k_1,k_2,0}\|_2 \leq |I_1|^{k_1} |I_2|^{k_2} (|I_1| |I_2|
|A|)^{1/2-1/p} \qquad (k_i=0,\ldots,N(p)+1).
\end{equation}

We may suppose that $N(p)\geq \alpha+1$ and choose $N\in\N$ such
that $N<\alpha\leq N+1$. For $x_l\in [{\pi i_l 2^{-K_l}},{\pi
(i_l+1) 2^{-K_l}})$, $t_l\in [{-\pi 2^{-K_l-1}},{\pi 2^{-K_l-1}})$
$(l=1,2)$ inequality (\ref{e16.26}) implies
$$
|(K_{n_1}^{\alpha})^{(N)}(x_1-t_1)| \leq {C n_1^{N-\alpha}
2^{K_1(\alpha+1)} \over i_1^{\alpha+1}} \leq {C 2^{K_1(N+1)} \over
i_1^{\alpha+1}}
$$
and
$$
|(K_{n_2}^{\alpha})^{(N+1)}(x_2-t_2)| \leq {C n_2^{N+1-\alpha}
2^{K_2(\alpha+1)} \over i_2^{\alpha+1}} \leq {C 2^{K_2(N+2)} \over
i_2^{\alpha+1}}.
$$
Recall that in the first case $n_1\geq 2^{K_1}$ and in the second
one $n_2<2^{K_2}$.

Integrating by parts, we can see that
\begin{eqnarray*}
\lefteqn{|\sigma_{n}^{\alpha} a(x)| } \n\\
&=& |\int_{I_1} \int_{I_2} \int_A A_{N,N+1,0}(t)
(K_{n_1}^{\alpha})^{(N)}(x_1 - t_1) (K_{n_2}^{\alpha})^{(N+1)}
(x_2 - t_2) K_{n_3}^{\alpha}(x_3 - t_3)\dd t| \\
&\leq& {C 2^{K_1(N+1)} 2^{K_2(N+2)} \over
i_1^{\alpha+1}i_2^{\alpha+1}} \int_{I_1} \int_{I_2} |\int_A
A_{N,N+1,0}(t) K_{n_3}^{\alpha}(x_3 - t_3)\dd t_3| \dd t_1 \dd t_2
\end{eqnarray*}
whenever $x_l\in [{\pi i_l 2^{-K_l}},{\pi (i_l+1) 2^{-K_l}})$.
Hence, by H{\"o}lder's inequality,
\begin{eqnarray*}
\lefteqn{\int_{(I_1^{r_1})^c} \int_{(I_2^{r_2})^c} \int_{A}
\sup_{n_1 \geq 2^{K_1},n_2 < 2^{K_2} \atop n_3\in\N}
|\sigma_{n}^{\alpha}
a(x)|^p \dd x } \n\\
&\leq& C_p \sum_{i_1=2^{r_1-2}}^{2^{K_1}-1}
\sum_{i_2=2^{r_2-2}}^{2^{K_2}-1}
2^{-K_1} 2^{-K_2} {2^{K_1(N+1)p} 2^{K_2(N+2)p} \over i_1^{(\alpha+1)p}i_2^{(\alpha+1)p}} \\
&&{} \int_{A}\Big(\int_{I_1} \int_{I_2} \sup_{n_3\in\N} |\int_A
A_{N,N+1,0}(t)
K_{n_3}^{\alpha}(x_3 - t_3)\dd t_3| \dd t_1 \dd t_2 \Big)^p \dd x_3 \\
&\leq& C_p |A|^{1-p} \sum_{i_1=2^{r_1-2}}^{2^{K_1}-1}
\sum_{i_2=2^{r_2-2}}^{2^{K_2}-1}
{2^{K_1((N+1)p-1)} 2^{K_2((N+2)p-1)} \over i_1^{(\alpha+1)p}i_2^{(\alpha+1)p}} \\
&&{} \Big( \int_{A}\int_{I_1} \int_{I_2} \sup_{n_3\in\N} |\int_A
A_{N,N+1,0}(t) K_{n_3}^{\alpha}(x_3 - t_3)\dd t_3| \dd t_1 \dd t_2
\dd x_3 \Big)^p.
\end{eqnarray*}
Using again H{\"o}lder's inequality and the fact that
$\sigma_*^\alpha$ is bounded on $L_2(\T^d)$ for all $d\geq 1$, we
conclude
\begin{eqnarray*}
\lefteqn{\int_{(I_1^{r_1})^c} \int_{(I_2^{r_2})^c} \int_{A}
\sup_{n_1 \geq 2^{K_1},n_2 < 2^{K_2} \atop n_3\in\N}
|\sigma_{n}^{\alpha}
a(x)|^p \dd x } \n\\
&\leq& C_p |A|^{1-p/2} \sum_{i_1=2^{r_1-2}}^{2^{K_1}-1}
\sum_{i_2=2^{r_2-2}}^{2^{K_2}-1}
{2^{K_1((N+1)p-1)} 2^{K_2((N+2)p-1)} \over i_1^{(\alpha+1)p}i_2^{(\alpha+1)p}} \\
&&{} \Big( \int_{I_1} \int_{I_2} \Big(\int_{\T} \sup_{n_3\in\N}
|\int_A A_{N,N+1,0}(t)
K_{n_3}^{\alpha}(x_3 - t_3)\dd t_3|^2 \dd x_3 \Big)^{1/2} \dd t_1 \dd t_2  \Big)^p \\
&\leq& C_p |A|^{1-p/2} \sum_{i_1=2^{r_1-2}}^{2^{K_1}-1}
\sum_{i_2=2^{r_2-2}}^{2^{K_2}-1}
{2^{K_1((N+1)p-1)} 2^{K_2((N+2)p-1)} \over i_1^{(\alpha+1)p}i_2^{(\alpha+1)p}} \\
&&{} \Big(\int_{I_1} \int_{I_2} \Big(\int_\T
|A_{N,N+1,0}(t_1,t_2,x_3)|^2 \dd x_3 \Big)^{1/2}\dd t_1 \dd t_2
\Big)^p.
\end{eqnarray*}
Then (\ref{e16.27}) implies
\begin{eqnarray*}
\lefteqn{\int_{(I_1^{r_1})^c} \int_{(I_2^{r_2})^c} \int_{A}
\sup_{n_1 \geq 2^{K_1},n_2 < 2^{K_2} \atop n_3\in\N} |\sigma_{n}^{\alpha} a(x)|^p \dd x } \n\\
&\leq& C_p |A|^{1-p/2} \sum_{i_1=2^{r_1-2}}^{2^{K_1}-1}
\sum_{i_2=2^{r_2-2}}^{2^{K_2}-1}
2^{-K_1p/2} 2^{-K_2p/2} {2^{K_1((N+1)p-1)} 2^{K_2((N+2)p-1)} \over i_1^{(\alpha+1)p}i_2^{(\alpha+1)p}} \\
&&{} \Big(\int_{I_1} \int_{I_2} \int_\T |A_{N,N+1,0}(t_1,t_2,x_3)|^2 \dd x_3 \dd t_1 \dd t_2 \Big)^{p/2} \\
&\leq& C_p \sum_{i_1=2^{r_1-2}}^{2^{K_1}-1}
\sum_{i_2=2^{r_2-2}}^{2^{K_2}-1} {1 \over
i_1^{(\alpha+1)p}i_2^{(\alpha+1)p}} \leq C_p 2^{-r_1((\alpha+1)p-1)}
2^{-r_2((\alpha+1)p-1)}.
\end{eqnarray*}
The other terms of (\ref{e16.3}) can be handled in the same way,
which shows (\ref{e16.1}). Obviously, the same ideas show
(\ref{e16.2}).
\end{proof}

Corollary \ref{c51} implies

\begin{cor}\label{c53}
If $\alpha>0$ and $f \in H_1^i(\T^d)$ for some $i=1,\ldots,d$, then
$$
\sup_{\rho >0} \rho \lambda(\sigma^\alpha_* f > \rho) \leq C
\|f\|_{H_1^i}.
$$
\end{cor}

By the density argument, we get here almost everywhere convergence
for functions from the spaces $H_1^i(\T^d)$ instead of $L_1(\T^d)$.
In some sense, the Hardy space $H_1^i(\T^d)$ plays the role of
$L_1(\T^d)$ in higher dimensions.

\begin{cor}\label{c54}
If $\alpha>0$ and $f \in H_1^i(\T^d)$ for some $i=1,\ldots,d$, then
$$
\lim_{n\to\infty}\sigma^\alpha_{n}f = f \qquad \mbox{a.e.}
$$
The almost everywhere convergence is not true for all $f \in
L_1(\T^d)$.
\end{cor}

A counterexample, which shows that the almost everywhere convergence
is not true for all integrable functions, is due to G{\'a}t
\cite{ga11}. Recall that
$$
L_1(\T^d)\supset H_1^i(\T^d) \supset L(\log L)^{d-1}(\T^d) \supset
L_p(\T^d) \qquad (1<p\leq \infty).
$$

Let the \idword{conjugate Riesz means} and \idword{conjugate maximal
operators} of a distribution $f$ be defined by
$$
\tilde \sigma_{n}^{(j_1,\ldots,j_d);\alpha} f(x):= \tilde
f^{(j_1,\ldots,j_d)} * K_{n}^{\alpha} \qquad
(j_i=0,1)\index{\file-1}{$\tilde\sigma_n^{(j_1,\ldots,j_d);\alpha}f$}
$$
and
$$
\tilde \sigma_*^{(j_1,\ldots,j_d);\alpha}f := \sup_{n\in \N^d}
|\tilde \sigma_{n}^{(j_1,\ldots,j_d);\alpha}
f|.\index{\file-1}{$\tilde\sigma_*^{(j_1,\ldots,j_d);\alpha}f$}
$$
Then the following results hold.

\begin{thm}\label{t55}
If $\alpha>0$ and $1/(\alpha+1)< p<\infty$, then for all $j_i=0,1$,
$$
\|\tilde \sigma_*^{(j_1,\ldots,j_d);\alpha}f\|_{p} \leq C_{p}
\|f\|_{H_{p}} \qquad (f\in H_{p}(\T^d))
$$
and
$$
\|\tilde \sigma_{n}^{(j_1,\ldots,j_d);\alpha} f\|_{H_{p}} \leq C_{p}
\|f\|_{H_{p}} \qquad (n\in \N^d, f\in H_{p}(\T^d)).
$$
In particular, if $f \in H_1^i(\T^d)$ for some $i=1,\ldots,d$, then
$$
\sup_{\rho>0}\lambda(\tilde \sigma_*^{(j_1,\ldots,j_d);\alpha}f >
\rho) \leq C \|f\|_{H_1^i}.
$$
\end{thm}

\begin{cor}\label{c56}
If $\alpha>0$, $j_i=0,1$ and $f\in H_1^i(\T^d)$, then
$$
\lim_{n\to\infty}\tilde \sigma_{n}^{(j_1,\ldots,j_d);\alpha} f =
\tilde f^{(j_1,\ldots,j_d)} \qquad {\rm a.e.}
$$
Moreover, if $f \in H_{p}(\T^d)$ with $1/(\alpha+1)< p <\infty$,
then this convergence also holds in the $H_{p}(\T^d)$-norm.
\end{cor}

The proofs of the last two results are similar to those of Theorem
\ref{t30} and Corollary \ref{c31}.

\sect{Rectangular $\theta$-summability}\label{s17}

Given the $d$-dimensional function $\theta$, the \idword{rectangular
$\theta$-means} of $f\in L_1(\T^d)$ are defined by
\begin{eqnarray}\label{e17.2}
\sigma_n^{\theta}f(x) &:=&\sum_{k_1\in \Z} \cdots \sum_{k_d\in \Z}
\theta\Big(\frac{-k_1}{n_1},\ldots,\frac{-k_d}{n_d}\Big) \widehat
f(k) \ee^{\ii k \cdot x} \n \\ &=&\frac{1}{(2\pi)^d}\int_{\T^d}
f(x-u) K_n^{\theta}(u) \dd u\index{\file-1}{$\sigma_n^\theta f$}
\end{eqnarray}
$(x\in \T^d, n\in \N^d)$, where the \idword{$\theta$-kernels}
$K_{n}^{\theta}$ are given by
$$
K_{n}^{\theta}(u) := \sum_{k_1\in \Z} \cdots \sum_{k_d\in \Z}
\theta\Big(\frac{-k_1}{n_1},\ldots,\frac{-k_d}{n_d}\Big) \ee^{\ii k
\cdot u}  \qquad  (u\in\T^d). \index{\file-1}{$K_n^\theta$}
$$
Define the \dword{restricted} and \dword{unrestricted maximal
operators}\index{\file}{restricted maximal
operator}\index{\file}{unrestricted maximal operator} by
$$
\sigma_\Box^\theta f := \sup_{n \in \R_\tau^d} |\sigma_{n}^\theta
f|, \qquad \sigma_*^\theta f := \sup_{n \in \N^d} |\sigma_{n}^\theta
f|.\index{\file-1}{$\sigma_\Box^\theta
f$}\index{\file-1}{$\sigma_*^\theta f$}
$$

If $d=1$, then, instead of the third condition of (\ref{e8}), we may
suppose that $\theta\in W(C,\ell_1)(\R)$ (see the definition below).
A measurable function $f$ belongs to the \idword{Wiener amalgam
space} \inda{$W(L_\infty,\ell_1)(\R^d)$} if
$$
\|f\|_{W(L_\infty,\ell_1)} := \sum_{k\in \Z^d} \sup_{x\in [0,1)^d}
|f(x+k)| <\infty.
$$
It is easy to see that
$$
W(L_\infty,\ell_1)(\R^d)\subset L_p(\R^d) \quad \mbox{for all} \quad
1\leq p\leq \infty.
$$
The smallest closed subspace of $W(L_\infty,\ell_1)(\R^d)$
containing continuous functions is denoted by
\inda{$W(C,\ell_1)(\R^d)$} and is called the \idword{Wiener
algebra}. It is used quite often in Gabor analysis, because it
provides a convenient and general class of windows (see e.g.~Walnut
\cite{Wal92} and Gr{\"o}chenig \cite{gr2}).

If $\theta$ is continuous and $|\theta|$ can be estimated by an
integrable function $\eta$ which is non-decreasing on $(-\infty,c)$
and non-increasing on $(c,\infty)$, then $\theta\in
W(C,\ell_1)(\R)$. Since
\begin{eqnarray}\label{e17.3}
\sum_{k_1=-\infty}^{\infty} \cdots \sum_{k_d=-\infty}^{\infty}\Big|
\theta\Big({k_1 \over n_1}, \ldots,{k_d \over n_d}\Big) \Big| &\leq& \sum_{l\in \Z^d} \Big(\prod_{j=1}^{d}n_j\Big)\sup_{x\in [0,1)^d} |\theta(x+l)| \n \\
&=& \Big(\prod_{j=1}^{d}n_j \Big) \|\theta\|_{W(C,\ell_1)} <\infty,
\end{eqnarray}
the $\theta$-kernels $K_{n}^{\theta}$ and the $\theta$-means
$\sigma_n^{\theta}f$ are well defined.

We introduce \idword{Feichtinger's algebra} \inda{$\bS_0(\R^d)$},
which is a subspace of the Wiener algebra. The \idword{short-time
Fourier transform} of $f\in L_2(\R^d)$ with respect to a window
function $g\in L_2(\R^d)$ is defined by
$$
S_gf(x,\omega) := \frac{1}{(2\pi)^d} \int_{\R^d} f(t)
\overline{g(t-x)} \ee^{-\ii \omega \cdot t} \dd t \qquad (x,\omega
\in \R^d).\index{\file-1}{$S_gf$}
$$
Using the short-time Fourier transform with respect to the Gauss
function $g_0(x):= \ee^{-\pi\|x\|_2^2}$, we define $\bS_0(\R^d)$ by
$$
\bS_0(\R^d) := \left\{ f\in L^2(\R^d): \|f\|_{\bS_0}:=
\|S_{g_0}f\|_{L_1(\R^{2d})}<\infty \right\}.
$$

Any other non-zero Schwartz function defines the same space and an
equivalent norm. It is known that $\bS_0(\R^d)$ is isometrically
invariant under translation, modulation and Fourier transform (see
Feichtinger \cite{Fei81}). Actually, $\bS_0$ is the minimal Banach
space having this property (see Feichtinger \cite{Fei81}).
Furthermore, the embedding $\bS_0(\R^d)\hookrightarrow
W(C,\ell_1)(\R^d)$ is dense and continuous and
$$
\bS_0(\R^d)\subsetneq W(C,\ell_1)(\R^d) \cap \cF(W(C,\ell_1)(\R^d)),
$$
where $\cF$ denotes the Fourier transform and
$\cF(W(C,\ell_1)(\R^d))$ the set of Fourier transforms of the
functions from $W(C,\ell_1)(\R^d)$ (see Feichtinger and Zimmermann
\cite{FZ98}, Losert \cite{los} and Gr{\"o}chenig \cite{gr1}).

\subsection{Norm convergence}\label{s17.1}

First, we investigate the $L_2$-norm convergence of $\sigma_n^\theta
f$ as $n\to\infty$.

\begin{thm}\label{t17.12}
If $\theta\in W(C,\ell_1)(\R^d)$ and $\theta(0)=1$, then
$$
\lim_{n\to\infty} \sigma_n^\theta f = f \quad \mbox{in the
$L_2(\T^d)$-norm for all $f\in L_2(\T^d)$}.
$$
\end{thm}

\begin{proof}
It is easy to see that the norm of the operator $\sigma_n^\theta :
L_2(\T^d)\to L_2(\T^d)$ is
\begin{eqnarray*}
\sup_{f\in L_2(\T^d),\, \|f\|_2\leq 1} \|f*K_n^\theta\|_2 &=& \sup_{f\in L_2(\T^d),\, \|f\|_2\leq 1} \|\widehat f \widehat K_n^\theta\|_2 \\
&=& \sup_{\widehat f\in \ell_2(\Z^d),\, \|\widehat f\|_2\leq 1} \|\widehat f \widehat K_n^\theta\|_2 \\
&=& \|\widehat K_n^\theta\|_\infty \\
&=& \sup_{k\in \Z^d} \Big|\theta\Big({-k_1 \over n_1+1}, \ldots, {-k_d \over n_d+1}\Big)\Big| \\
&\leq& C.
\end{eqnarray*}
Thus the norms of $\sigma_n^\theta$ $(n\in \N^d)$ are uniformly
bounded. Since $\theta$ is continuous, the convergence holds for all
trigonometric polynomials. The set of the trigonometric polynomials
are dense in $L_2(\T^d)$, so the usual density theorem proves
Theorem \ref{t17.12}.
\end{proof}

The $\theta$-means can be written as a singular integral of $f$ and
of the Fourier transform of $\theta$ in the following way
(Feichtinger and Weisz \cite{feiwe1}).

\begin{thm}\label{t5}
If $\theta\in W(C,\ell_1)(\R^d)$ and $\widehat \theta\in L_1(\R^d)$,
then
$$
\sigma_n^\theta f(x)= \Big(\prod_{j=1}^{d} n_j \Big) \int_{\R^d}
f(x-t) \widehat \theta (n_1t_1,\ldots,n_dt_d) \dd t
$$
for a.e.~$x\in \T^d$ and for all $n\in\N^d$ and $f\in L_1(\T^d)$.
\end{thm}

\begin{proof}
If $f(t)=\ee^{\ii k \cdot t}$ $(k\in \Z^d, t\in \T^d)$, then
\begin{eqnarray*}
\Big(\prod_{j=1}^{d}n_j \Big) \int_{\R^d} \ee^{\ii k \cdot (x-t)}
\widehat \theta(n_1t_1,\ldots,n_dt_d) \dd t &=& \ee^{\ii k \cdot x}
\int_{\R^d} \Big(\prod_{j=1}^{d} \ee^{-\ii k_jt_j/n_j}\Big)
\widehat \theta(t) \dd t \\
&=& \theta\Big({-k_1 \over n_1}, \ldots , {-k_d \over n_d}\Big)
\ee^{\ii k \cdot x}  \\
&=& \sigma_n^\theta f(x),
\end{eqnarray*}
thus the theorem holds also for trigonometric polynomials. The proof
can be finished as in Theorem \ref{t64.5}.
\end{proof}

Now, we give a sufficient and necessary condition for the uniform
and $L_1$ convergence $\sigma_n^\theta f \to f$ (see Feichtinger and
Weisz \cite{feiwe1}). Note that the statement $\eword{(i)}
\Leftrightarrow \eword{(ii)}$ in the next theorem was shown in the
one-dimensional case by Natanson and Zuk \cite{nazsu} for $\theta$
having compact support. The situation in our general case is much
more complicated.

\begin{thm}\label{t6}
If $\theta\in W(C,\ell_1)(\R^d)$ and $\theta(0)=1$, then the
following conditions are equivalent:
\begin{enumerate}
\item []
\begin{enumerate}
\item $\widehat \theta\in L_1(\R^d)$,
\item $\sigma_n^\theta f \to f$ uniformly for all $f\in C(\T^d)$ as $n\to\infty$,
\item $\sigma_n^\theta f(x) \to f(x)$ for all $x\in \T^d$ and
$f\in C(\T^d)$ as $n\to\infty$,
\item $\sigma_n^\theta f \to f$ in the $L_1(\T^d)$-norm for all $f\in L_1(\T^d)$ as $n\to\infty$,
\item $\sigma_n^\theta f \to f$ uniformly for all $f\in C(\T^d)$ as $n\to\infty$ and $n\in \R_\tau^d$,
\item $\sigma_n^\theta f(x) \to f(x)$ for all $x\in \T^d$ and
$f\in C(\T^d)$ as $n\to\infty$ and $n\in \R_\tau^d$,
\item $\sigma_n^\theta f \to f$ in the $L_1(\T^d)$-norm for all $f\in L_1(\T^d)$ as $n\to\infty$ and $n\in \R_\tau^d$.
\end{enumerate}
\end{enumerate}
\end{thm}

Recall the definition of $R_\tau^d$ from (\ref{e29}).

\begin{proof}
First, we verify the equivalence between \eword{(i)}, \eword{(ii)},
\eword{(iii)} and \eword{(iv)}. We may suppose that $d=1$, since the
multi-dimensional case is similar. If \eword{(i)} holds, then by
Theorem \ref{t5},
$$
\|\sigma_n^\theta f\|_\infty \leq \|f\|_\infty \|\widehat
{\theta}\|_1 \qquad  (f\in C(\T), n\in \N)
$$
and so $\sigma_n : C(\T)\to C(\T)$ are uniformly bounded. Since
\eword{(ii)} holds for all trigonometric polynomials and the set of
the trigonometric polynomials are dense in $C(\T)$, \eword{(ii)}
follows easily. \eword{(ii)} implies \eword{(iii)} trivially.

Suppose that \eword{(iii)} is satisfied. We are going to prove
\eword{(i)}. For a fixed $x\in \T$, the operators
$$
U_n:C(\T)\to \R,\qquad U_nf:=\sigma_n^\theta f(x) \qquad (\nn)
$$
are uniformly bounded by the Banach-Steinhaus theorem. We get by
(\ref{e17.2}) that
$$
\|U_n\|=\int_\T |K_n^\theta(x-t)| \dd t=\|K_n^\theta\|_1 \qquad
(\nn).
$$
Hence
$$
\sup_{\nn}\|K_n^\theta\|_1\leq C.
$$
Since $K_n^\theta$ is $2\pi$-periodic, we have for $\alpha\leq
(n+1)/2$ that
\begin{eqnarray}\label{e17.4}
\lefteqn{\int_{-2\alpha \pi}^{2\alpha \pi}
\frac{1}{n+1}\Big|\sum_{k=- \infty}^{ \infty}
\theta\Big(\frac{-k}{n+1}\Big) \ee^{\ii t \frac{k}{n+1}} \Big|\dd t
} \n\\ &\leq& \int_{-(n+1)\pi}^{(n+1)\pi}
\frac{1}{n+1}\Big|\sum_{k=- \infty}^{ \infty}
\theta\Big(\frac{-k}{n+1}\Big) \ee^{\ii t \frac{k}{n+1}} \Big|\dd t \n\\
&=& \int_{-\pi}^{\pi} \Big|\sum_{k=- \infty}^{ \infty}
\theta\Big(\frac{-k}{n+1}\Big) \ee^{\ii k x} \Big|\dd x\n\\
&=&\int_\T |K_{n}^{\theta}(x)|\dd x \leq C.
\end{eqnarray}
For a fixed $t\in \R$, let
$$
h_n(t) := \frac{1}{n+1} \sum_{k=- \infty}^{ \infty}
\theta\Big(\frac{-k}{n+1}\Big) \ee^{\ii t \frac{k}{n+1}}
$$
and
$$
\varphi_n(t,u):= \sum_{k=- \infty}^{ \infty}
\theta\Big(\frac{-k}{n+1}\Big) \ee^{\ii t \frac{k}{n+1}}
1_{[\frac{k}{n+1},\frac{k+1}{n+1})}(u).
$$
It is easy to see that
$$
\lim_{n\to\infty} \varphi_n(t,u) = \theta(-u) \ee^{\ii t u}.
$$
Moreover,
$$
|\varphi_n(t,u)|\leq \sum_{l=- \infty}^{ \infty} \sup_{x\in [0,1)}
|\theta(x-l-1)| 1_{[l,l+1)}(u)
$$
and
$$
\int_{-\infty}^{\infty} \sum_{l=- \infty}^{ \infty} \sup_{x\in
[0,1)} |\theta(x-l-1)| 1_{[l,l+1)}(u)\dd u = \sum_{l=- \infty}^{
\infty} \sup_{x\in [0,1)} |\theta(x-l-1)| =
\|\theta\|_{W(C,\ell_1)}.
$$
Lebesgue's dominated convergence theorem implies that
$$
\lim_{n\to\infty} \int_{-\infty}^{\infty}\varphi_n(t,u) \dd u=
\int_{-\infty}^{\infty}\theta(-u) \ee^{\ii t u}\dd u = \widehat
\theta(t).
$$
Obviously,
$$
\int_{-\infty}^{\infty}\varphi_n(t,u)\dd u=h_n(t)
$$
and so
$$
\lim_{n\to\infty} h_n(t) = \widehat \theta(t).
$$
Of course, this holds for all $t\in \R$. We have by (\ref{e17.3})
that $|h_n(t)| \leq \|\theta\|_{W(C,\ell_1)}$. Thus
$$
\lim_{n\to\infty} \int_{-2\alpha\pi}^{2\alpha \pi} |h_n(t)| \dd t =
\int_{-2\alpha \pi}^{2\alpha \pi} |\widehat \theta(t)|\dd t.
$$
Inequality (\ref{e17.4}) yields that
$$
\int_{-2\alpha \pi}^{2\alpha \pi} |\widehat \theta(t)|\dd t \leq C
\qquad \mbox{for all} \qquad \alpha>0
$$
and so
$$
\int_{-\infty}^{\infty} |\widehat \theta(t)|\dd t \leq C,
$$
which shows \eword{(i)}.

If $\widehat \theta\in L_1(\R)$, then Theorem \ref{t5} implies
$$
\|\sigma_n^\theta f\|_1 \leq \|f\|_1 \|\widehat\theta\|_1 \qquad
(f\in L_1(\T), \nn).
$$
Hence \eword{(iv)} follows from \eword{(i)} because the set of the
trigonometric polynomials are dense in $L_1(\T)$. The fact that
\eword{(iv)} implies \eword{(i)} can be proved similarly as
$\eword{(iii)} \Rightarrow \eword{(i)}$, since, by duality, the norm
of the operator $\sigma_n^\theta : L_1(\T)\to L_1(\T)$ is again
$\|\sigma_n^\theta\|=\|K_n^\theta\|_1$.

It is easy to see that the equivalence between \eword{(i)},
\eword{(v)}, \eword{(vi)} and \eword{(vii)} can be proved in the
same way.
\end{proof}

One part of the preceding result is generalized for homogeneous
Banach spaces.

\begin{thm}\label{t20}
Assume that $B$ is a homogeneous Banach space on $\T^d$. If
$\theta(0)=1$, $\theta\in W(C,\ell_1)(\R^d)$ and $\widehat \theta\in
L_1(\R^d)$, then
$$
\|\sigma_n^{\theta} f\|_B \leq C \|f\|_B \qquad (n\in \N^d)
$$
and
$$
\lim_{n\to\infty} \sigma_n^{\theta} f=f \qquad \mbox{in the $B$-norm
for all $f\in B$}.
$$
\end{thm}

\begin{proof}
For simplicity, we show the theorem for $d=1$. Using Theorem
\ref{t5}, we conclude
\begin{eqnarray*}
\sigma_n^\theta f(x) - f(x) &=& n \int_{\R} (f(x-t)-f(x)) \widehat \theta (nt) \dd t \\
&=& \int_{\R} \Big( f(x-\frac{t}{n})-f(x) \Big) \widehat \theta(t)
\dd t
\end{eqnarray*}
and
$$
\|\sigma_n^\theta f - f\|_B = \int_{\R} \Big\| T_{\frac{t}{n}} f -f
\Big\|_B |\widehat \theta(t)| \dd t.
$$
The theorem follows from the definition of the homogeneous Banach
spaces and from the Lebesgue dominated convergence theorem.
\end{proof}

Since $\theta\in\bS_0(\R^d)$ implies $\theta\in W(C,\ell_1)(\R^d)$
and $\widehat \theta\in \bS_0(\R^d) \subset L_1(\R^d)$, the next
corollary follows from Theorems \ref{t6} and \ref{t20}.

\begin{cor}\label{c4}
If $\theta\in \bS_0(\R^d)$ and $\theta(0)=1$, then
\begin{enumerate}
\item []
\begin{enumerate}
\item $\sigma_n^\theta f \to f$ uniformly for all $f\in C(\T^d)$ as $n\to\infty$,
\item $\sigma_n^\theta f \to f$ in the $L_1(\T^d)$-norm for all $f\in L_1(\T^d)$ as
$n\to\infty$,
\item $\sigma_n^\theta f \to f$ in the $B$-norm for all $f\in B$ as $n\to\infty$ if
$B$ is a homogeneous Banach space.
\end{enumerate}
\end{enumerate}
\end{cor}

The next corollary follows from the fact that $\theta\in
\bS_0(\R^d)$ is equivalent to $\widehat \theta\in L_1(\R^d)$,
provided that $\theta$ has compact support (see e.g.~Feichtinger and
Zimmermann \cite{FZ98}).

\begin{cor}\label{c2}
If $\theta\in C(\R^d)$ has compact support and $\theta(0)=1$, then
the following conditions are equivalent:
\begin{enumerate}
\item []
\begin{enumerate}
\item $\theta\in \bS_0(\R^d)$,
\item $\sigma_n^\theta f \to f$ uniformly for all $f\in C(\T^d)$ as $n\to\infty$,
\item $\sigma_n^\theta f(x) \to f(x)$ for all $x\in \T^d$ and
$f\in C(\T^d)$ as $n\to\infty$,
\item $\sigma_n^\theta f \to f$ in the $L_1(\T^d)$-norm for all $f\in L_1(\T^d)$ as
$n\to\infty$,
\item $\sigma_n^\theta f \to f$ uniformly for all $f\in C(\T^d)$ as $n\to\infty$ and $n\in \R_\tau^d$,
\item $\sigma_n^\theta f(x) \to f(x)$ for all $x\in \T^d$ and
$f\in C(\T^d)$ as $n\to\infty$ and $n\in \R_\tau^d$,
\item $\sigma_n^\theta f \to f$ in the $L_1(\T^d)$-norm for all $f\in L_1(\T^d)$ as $n\to\infty$ and $n\in \R_\tau^d$.
\end{enumerate}
\end{enumerate}
\end{cor}

In the rest of this subsection, we give some sufficient conditions
for a function $\theta$ to satisfy $\widehat \theta\in L_1(\R^d)$,
resp. $\theta\in \bS_0(\R^d)$. Several such conditions are already
known. For example, if $\theta\in L_\infty(\R^d)$ and $\widehat
\theta\geq 0$, then $\widehat \theta \in L_1(\R^d)$ (see Bachman,
Narici and Beckenstein \cite{bach}). As mentioned before, $\theta\in
\bS_0(\R^d)$ implies also that $\widehat \theta\in L_1(\R^d)$.
Recall that $\bS_0(\R^d)$ contains all Schwartz functions. If
$\theta\in L_1(\R^d)$ and $\widehat \theta$ has compact support or
if $\theta\in L_1(\R^d)$ has compact support and $\widehat \theta\in
L_1(\R^d)$, then $\theta\in \bS_0(\R^d)$. If $\theta\in L_1(\R)$ has
compact support and $\theta\in \mbox{Lip}(\alpha)$ for some
$\alpha>1/2$, then $\widehat \theta\in L_1(\R)$ (see Natanson and
Zuk \cite[p.~176]{nazsu}) and so $\theta\in \bS_0(\R)$. If $\theta
v_s,\widehat \theta v_s \in L_2(\R^d)$ for some $s>d$ or if $\theta
v_s,\widehat \theta v_s \in L_\infty(\R^d)$ for some $s>3d/2$, then
$\theta\in \bS_0(\R^d)$. The weight function $v_s$ is given by
$v_s(\omega):= (1+|\omega|)^s$ $(\omega\in \R^d, s\in \R)$.

Sufficient conditions can also be given with the help of Sobolev,
fractional Sobolev and Besov spaces. For a detailed description of
these spaces, see Triebel \cite{tr2}, Runst and Sickel \cite{rusi},
Stein \cite{st} and Grafakos \cite{gra}. The \idword{Sobolev space}
\inda{$W_p^k(\R^d)$} $(1\leq p\leq \infty,k\in \N)$ is defined by
$$
W_p^{k}(\R^d):= \{\theta\in L_p(\R^d): D^\alpha \theta \in
L_p(\R^d), |\alpha|\leq k\}
$$
and endowed with the norm
$$
\|\theta\|_{W_p^{k}}:= \sum_{|\alpha|\leq k} \|D^\alpha \theta\|_p,
$$
where $D$ denotes the distributional derivative.

This definition can be extended to every real $s$ in the following
way. The \idword{fractional Sobolev space} \inda{$\cL_p^s(\R^d)$}
$(1\leq p\leq \infty,s\in \R)$ consists of all tempered
distributions $\theta$ for which
$$
\|\theta\|_{\cL_p^{s}}:= \| \cF^{-1} ( (1+|\cdot|^2)^{s/2} \widehat
\theta)\|_p<\infty,
$$
where $\cF$ denotes the Fourier transform. It is known that
$$
\cL_p^s(\R^d)=W_p^k(\R^d) \quad \mbox{if} \quad s=k\in \N \quad
\mbox{and} \quad 1<p<\infty
$$
with equivalent norms.

In order to define the Besov spaces, take a non-negative Schwartz
function $\psi\in \cS(\R)$ with support $[1/2,2]$ that satisfies
$$
\sum_{k=-\infty}^{\infty} \psi(2^{-k}s)=1 \quad \mbox{for all} \quad
s\in \R\setminus \{0\}.
$$
For $x\in \R^d$, let
$$
\phi_k(x):=\psi(2^{-k}|x|)\quad \mbox{for} \quad k\geq 1 \quad
\mbox{and} \quad \phi_0(x)=1-\sum_{k=1}^{\infty} \phi_k(x).
$$
The \idword{Besov space} \inda{$B_{p,r}^s(\R^d)$} $(0<p,r\leq
\infty,s\in \R)$ is the space of all tempered distributions $f$ for
which
$$
\|f\|_{B_{p,r}^s}:=\Big(\sum_{k=0}^{\infty} 2^{ksr}
\|(\cF^{-1}{\phi}_k)*f\|_p^r\Big)^{1/r}<\infty.
$$
The Sobolev, fractional Sobolev and Besov spaces are all
quasi-Banach spaces, and if $1\leq p,r\leq \infty$, then they are
Banach spaces. All these spaces contain the Schwartz functions. The
following facts are known: in the case $1\leq p,r\leq \infty$ one
has
$$
W_p^{m}(\R^d), B_{p,r}^s(\R^d) \hookrightarrow L_p(\R^d)\qquad
\mbox{if} \quad s>0,m\in \N,
$$
\begin{equation}\label{e17.15}
W_p^{m+1}(\R^d)\hookrightarrow B_{p,r}^s(\R^d) \hookrightarrow
W_p^{m}(\R^d) \qquad \mbox{if} \quad m<s<m+1,
\end{equation}
\begin{equation}\label{e17.16}
B_{p,r}^s(\R^d) \hookrightarrow B_{p,r+\epsilon}^s(\R^d),
B_{p,\infty}^{s+\epsilon}(\R^d) \hookrightarrow B_{p,r}^{s}(\R^d)
\qquad  \mbox{if} \quad  \epsilon>0,
\end{equation}
\begin{equation}\label{e17.17}
B_{p_1,1}^{d/p_1}(\R^d) \hookrightarrow
B_{p_2,1}^{d/p_2}(\R^d)\hookrightarrow C(\R^d)\qquad \mbox{if} \quad
1\leq p_1\leq p_2<\infty.
\end{equation}
For two quasi-Banach spaces $\bX$ and $\bY$, the embedding
$\bX\hookrightarrow \bY$ means that $\bX\subset\bY$ and $\|f\|_{\bY}
\leq C\|f\|_{\bX}$.

The connection between Besov spaces and Feichtinger's algebra is
summarized in the next theorem.

\begin{thm}\label{t17.14}
We have
\begin{enumerate}\item []
\begin{enumerate}
\item If $1\leq p\leq 2$ and $\theta\in B_{p,1}^{d/p}(\R^d)$, then $\widehat \theta\in L_1(\R^d)$
and
$$
\|\widehat \theta\|_1 \leq C\|\theta\|_{B_{p,1}^{d/p}}.
$$
\item If $s>d$, then $\cL_1^s(\R^d)\hookrightarrow \bS_0(\R^d)$.
\item If $d'$ denotes the smallest even integer which is larger than $d$ and $s>d'$, then
$$
B_{1,\infty}^s(\R^d)\hookrightarrow W_1^{d'}(\R^d)\hookrightarrow
\bS_0(\R^d).
$$
\end{enumerate}
\end{enumerate}
\end{thm}

\begin{proof}
(i) was proved in Girardi and Weis \cite{giwe} and (ii) in Okoudjou
\cite{ok}. The first embedding of (iii) follows from (\ref{e17.15})
and (\ref{e17.16}). If $k$ is even, then
$W_1^{k}(\R^d)\hookrightarrow \cL_1^k(\R^d)$ (see Stein
\cite[p.~160]{st}). Then (ii) proves (iii).
\end{proof}

It follows from (i) and (\ref{e17.15}) that $\theta\in
W_p^{j}(\R^d)$ $(j>d/p, j\in \N)$ implies $\widehat \theta\in
L_1(\R^d)$. If $j\geq d'$, then even $W_1^{j}(\R^d)\hookrightarrow
\bS_0(\R^d)$ (see (iii)). Moreover, if $s>d'$ as in (iii), then
$$
B_{1,\infty}^s(\R^d)\hookrightarrow B_{1,1}^{d}(\R^d)\hookrightarrow
B_{p,1}^{d/p}(\R^d)\qquad (1<p<\infty)
$$
by (\ref{e17.16}) and (\ref{e17.17}). Theorem \ref{t17.14} says that
$B_{1,\infty}^s(\R^d) \subset \bS_0(\R^d)$ $(s>d')$ and if we choose
$\theta$ from the larger space $B_{p,1}^{d/p}(\R^d)$ $(1\leq p\leq
2)$, then $\widehat \theta$ is still integrable.

The embedding $W_1^{2}(\R)\hookrightarrow\bS_0(\R)$ follows from
(iii). With the help of the usual derivative, we give another useful
sufficient condition for a function to be in $\bS_0(\R^d)$.

\begin{dfn}\label{d17.11}
 A function $\theta$ is in \inda{$V_1^k(\R)$} if there are numbers
$-\infty=a_0<a_1<\cdots<a_n<a_{n+1}=\infty$, where $n=n(\theta)$
depends on $\theta$ and
$$
\theta\in C^{k-2}(\R), \qquad \theta\in C^k(a_i,a_{i+1}), \qquad
\theta^{(j)}\in L_1(\R)
$$
for all $i=0,\ldots,n$ and $j=0,\ldots,k$. Here, $C^k$ denotes the
set of $k$ times continuously differentiable functions. The norm of
this space is defined by
$$
\|\theta\|_{V_1^k}:= \sum_{j=0}^{k} \|\theta^{(j)}\|_1 +
\sum_{i=1}^{n} |\theta^{(k-1)}(a_i+0)-\theta^{(k-1)}(a_i-0)|,
$$
where $\theta^{(k-1)}(a_i\pm 0)$ denotes the right and left limits
of $\theta^{(k-1)}$.
\end{dfn}

These limits do exist and are finite because $\theta^{(k)}\in
C(a_i,a_{i+1}) \cap L_1(\R)$ implies
$$
\theta^{(k-1)}(x)=\theta^{(k-1)}(a)+ \int_{a}^{x} \theta^{(k)}(t)\dd
t
$$
for some $a\in (a_i,a_{i+1})$. Since $\theta^{(k-1)}\in L_1(\R)$ we
establish that
$$
\lim_{x\to-\infty}\theta^{(k-1)}(x)=\lim_{x\to\infty}\theta^{(k-1)}(x)=0.
$$
Similarly, $\theta^{(j)}\in C_0(\R)$ for $j=0,\ldots,k-2$.

Of course, $W_1^2(\R)$ and $V_1^2(\R)$ are not identical. For
$\theta\in V_1^2(\R)$, we have $\theta'=D\theta$, however,
$\theta''=D^2\theta$ only if $\lim_{x\to a_i+0}\theta'(x)=\lim_{x\to
a_i-0}\theta'(x)$ $(i=1,\ldots,n)$.

\begin{thm}\label{t17.16}
We have $V_1^2(\R) \hookrightarrow \bS_0(\R)$.
\end{thm}

\begin{proof}
Integrating by parts, we have
\begin{eqnarray*}
S_{g_0}\theta(x,\omega) &=& \frac{1}{2\pi}\int_{\R} \theta(t) \overline{g_0(t-x)} \ee^{-\ii \omega t} \dd t \\
&=& \frac{1}{2\pi} \sum_{i=0}^n \int_{a_i}^{a_{i+1}} \theta(t) \ee^{-\pi(t-x)^2} \ee^{-\ii \omega t} \dd t \\
&=& \frac{1}{2\pi} \sum_{i=0}^n \Big[ \theta(t) \ee^{-\pi(t-x)^2}
\frac{\ee^{-\ii \omega t}}{-\ii
\omega} \Big]_{a_i}^{a_{i+1}} \\
&&{}- \frac{1}{2\pi} \sum_{i=0}^n \int_{a_i}^{a_{i+1}}
\Big(\theta'(t) \ee^{-\pi(t-x)^2} - 2\pi\theta(t)
\ee^{-\pi(t-x)^2}(t-x)\Big) \frac{\ee^{-\ii \omega t}}{-\ii \omega }
\dd t.
\end{eqnarray*}
Observe that the first sum is 0. In the second sum, we integrate by
parts again to obtain
\begin{eqnarray*}
S_{g_0}\theta(x,\omega) &=& \frac{1}{2\pi} \sum_{i=0}^n \Big[
\Big(\theta'(t) \ee^{-\pi(t-x)^2} - 2\pi\theta(t)
\ee^{-\pi(t-x)^2}(t-x)\Big) \frac{\ee^{-\ii \omega t}}
{\omega^2}\Big]_{a_i}^{a_{i+1}}\\
&&{} - \frac{1}{2\pi} \sum_{i=0}^n \int_{a_i}^{a_{i+1}}
\Big(\theta''(t) \ee^{-\pi(t-x)^2} -
4\pi\theta'(t) \ee^{-\pi(t-x)^2}(t-x) \\
&&{} - 2\pi\theta(t) \Big(- 2\pi \ee^{-\pi(t-x)^2}(t-x)^2
+\ee^{-\pi(t-x)^2} \Big) \Big) \frac{\ee^{-\ii \omega t}}{\omega^2}
\dd t.
\end{eqnarray*}
The first sum is equal to
$$
\frac{1}{2\pi} \sum_{i=1}^n \Big(\theta'(a_i+0) - \theta'(a_i-0)
\Big) \ee^{-\pi(a_i-x)^2} \frac{\ee^{-\ii \omega a_i}}{\omega^2} \ .
$$
Hence
$$
\int_{\R} \int_{\{|\omega|\geq 1\}} |S_{g_0}\theta(x,\omega)| \dd
x\dd \omega \leq C_s \|\theta\|_{V_1^2}.
$$
On the other hand,
$$
\int_{\R} \int_{\{|\omega|< 1\}} |S_{g_0}\theta(x,\omega)| \dd x\dd
\omega \leq C_s \int_{\R} \int_{\{|\omega|< 1\}} \int_{\R}
|\theta(t)| g_0(t-x) \dd t \dd x\dd \omega \leq C_s
\|\theta\|_{V_1^2},
$$
which finishes the proof of Theorem \ref{t17.16}.
\end{proof}

The next Corollary follows from the definition of $\bS_0(\R^d)$ and
from Theorem \ref{t17.16}.

\begin{cor}\label{c17.16}
If each $\theta_j \in V_1^2(\R)$ $(j=1,\ldots,d)$, then
$$
\theta:=\prod_{j=1}^{d} \theta_j \in \bS_0(\R^d).
$$
\end{cor}

It is easy to see that $\theta\in V_1^2(\R)\subset \bS_0(\R)$ in all
examples of Subsection \ref{s10.1}. Moreover, in Example \ref{x11}
(the Riesz summation), $\theta\in \bS_0(\R)$ for all
$0<\alpha<\infty$. In the next examples, $\theta$ has $d$ variables
and $\theta\in \bS_0(\R^d)$.

\begin{exa}[\ind{Riesz summation}]\label{x21}\rm
Let
$$
\theta(t)=\cases{(1-\|t\|_2^\gamma)^\alpha &if $\|t\|_2\leq 1$ \cr 0
&if $\|t\|_2>1$ \cr} \qquad (t\in \R^d)
$$
for some $(d-1)/2<\alpha<\infty,\gamma\in \N_+$ (see Figure
\ref{f25}).
\end{exa}
\begin{figure}[ht] 
   \centering
   \includegraphics[width=0.6\textwidth]{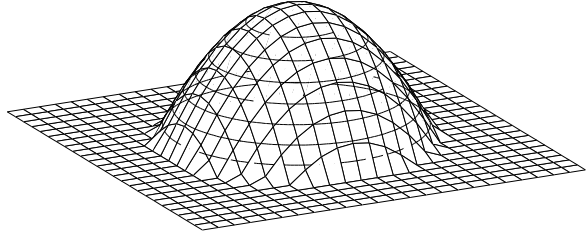}
   \caption{Riesz summability function with $d=2$, $\alpha=1$, $\gamma=2$.}
   \label{f25}
\end{figure}

\begin{exa}[\ind{Weierstrass summation}]\label{x22}\rm
Let
$$
\theta(t)=\ee^{-\|t\|_{2}} \quad \mbox{or} \quad
\theta(t)=\ee^{-\|t\|_{2}^{2}} \quad \mbox{or even} \quad
\theta(t)=\ee^{-\|t\|_{2l}^{2l}}
$$
for some $l\in \N_+$ $(t\in \R^d)$ (see Figure \ref{f26}).
\end{exa}
\begin{figure}[ht] 
   \centering
   \includegraphics[width=0.6\textwidth]{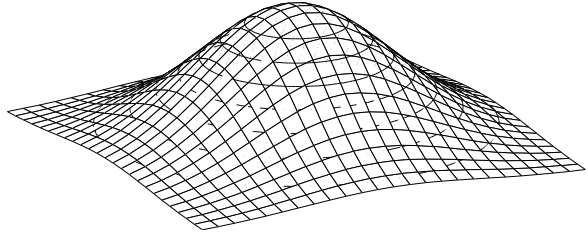}
   \caption{Weierstrass summability function $\theta(t)=\ee^{-\|t\|_{2}^{2}}$.}
   \label{f26}
\end{figure}

\begin{exa}\label{x23}\rm Let
$$
\theta(t)=\ee^{-(1+\|t\|_{2l}^{2l})^\gamma} \qquad (l\in \N_+, 0<
\gamma< \infty)
$$
(see Figure \ref{f27}).
\end{exa}
\begin{figure}[ht] 
   \centering
   \includegraphics[width=0.6\textwidth]{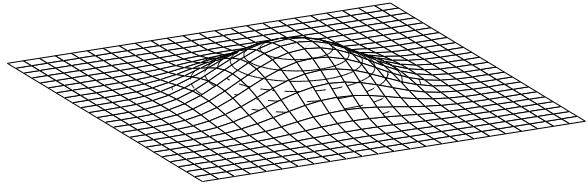}
   \caption{The summability function in Example \ref{x23} with $d=2$, $l=1$, $\gamma=2$.}
   \label{f27}
\end{figure}

\begin{exa}[Picard and Bessel summations\index{\file}{Picard summation}\index{\file}{Bessel summation}]\label{x24}\rm Let
$$
\theta(t)=(1+\|t\|_\gamma^\gamma)^{-\alpha} \qquad (t\in
\R^d,0<\alpha< \infty, 1\leq \gamma< \infty,\alpha\gamma>d)
$$
(see Figure \ref{f28}).
\end{exa}

\begin{figure}[ht] 
   \centering
   \includegraphics[width=0.6\textwidth]{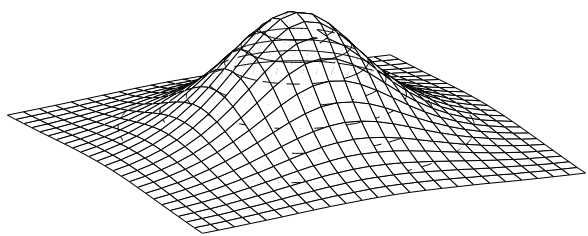}
   \caption{Picard-Bessel summability function with $d=2$, $\alpha=2$, $\gamma=2$.}
   \label{f28}
\end{figure}

\subsection{Almost everywhere convergence}\label{s17.2}

In this subsection, we suppose that
$$
\theta(0)=1, \qquad \theta=\theta_1\otimes \cdots \otimes \theta_d,
\qquad \theta_j\in W(C,\ell_1)(\R), \qquad j=1,\ldots ,d.
$$
For the restricted convergence, we suppose in addition that
$$
\cI\,\theta_j\in W(C,\ell_1)(\R), \qquad j=1,\ldots ,d.
$$
Note that $\cI$ denotes the identity function, so $\cI(x)=x$  and
$(\cI\,\theta_j)(x)=x\theta_j(x)$. Then (\ref{e17.3}) implies that
$$
|K_{n}^{\theta_j}|\leq C n \qquad (n\in \N).
$$
Similarly,
$$
\sum_{k=-\infty}^{\infty} \Big| {k \over n} \theta_j({k \over n})
\Big| \leq n \|\cI\,\theta_j\|_{W(C,\ell_1)} <\infty,
$$
from which we get immediately that
$$
|(K_{n}^{\theta_j})'|\leq C n^2 \qquad (n\in \N).
$$
These two inequalities were used several times in the proofs of
Theorems \ref{t43} and \ref{t14.1}. By Theorem \ref{t5},
$$
K_{n_j}^{\theta_j}(x)=2\pi n_j \sum_{k=-\infty}^\infty \widehat
\theta_j(n_j(x+2k\pi)) \qquad (x\in \T)
$$
as in (\ref{e6.3}). If each $\theta_j$ satisfies (\ref{e10}) and
(\ref{e11}) with $d=1$, $N=0$ and $0<\beta_j\leq 1$, then
$$
|K_{n_j}^{\theta_j}(x)| \leq {C \over n_j^{\beta_j} |x|^{\beta_j+1}}
\qquad (x\neq 0)
$$
and
$$
|(K_{n_j}^{\theta_j})'(x)| \leq {C \over n_j^{\beta_j-1}
|x|^{\beta_j+1}} \qquad (x\neq 0).
$$
Under these conditions, one can verify that the generalizations of
the results in the restricted sense of Section \ref{s14} hold with
$$
\max\{d/(d+1),1/(\beta_j+1)\}< p<\infty
$$
(see Weisz \cite{wcone1,wphi2-local}). As we have seen in Section
\ref{s14}, the Riesz summation in Example \ref{x11} satisfies all
conditions just mentioned with $\beta_j=\alpha\wedge 1$.

\begin{lem}\label{l17.2}
Let $\theta\in W(C,\ell_1)(\R)$, $\cI\,\theta\in W(C,\ell_1)(\R)$
and $\theta$ be even and twice differentiable on the interval
$(0,c)$, where $[-c,c]$ is the support of $\theta$ $(0<c\leq
\infty)$. Suppose that
$$
\lim_{x \to c-0} x \theta(x)=0, \quad \lim_{x \to +0} \theta'\in \R,
\quad \lim_{x \to c-0} \theta'\in \R \quad \mbox{and} \quad \lim_{x
\to \infty} x\theta'(x)=0.
$$
If $\theta'$ and $(\cI\vee 1)\theta''$ are integrable, then
$$
|\widehat\theta(x)| \leq {C\over x^{2}}, \qquad |\widehat\theta'(x)|
\leq {C\over x^{2}} \qquad (x\neq 0).
$$
\end{lem}

\begin{proof}
By integrating by parts, we have
\begin{eqnarray*}
\widehat\theta(x) &=& \frac{2}{{2\pi}} \int_0^c \theta(t)\cos tx \dd t \\
&=& \frac{1}{{\pi x}} \int_0^c \theta'(t)\sin tx \dd t \\
&=& \frac{-1}{\pi x^{2}} [\theta'(t) \cos tx]_0^c + \frac{1}{\pi
x^{2}} \int_0^c \theta''(t) \cos tx \dd t.
\end{eqnarray*}
Similarly,
\begin{eqnarray*}
(\widehat\theta)'(x) &=& \frac{2}{{2\pi}} \int_0^c t \theta(t)\cos tx \dd t \\
&=& \frac{1}{\pi x} \int_0^c (t \theta(t))'\sin tx \dd t \\
&=& \frac{-1}{\pi x^{2}} [(t\theta(t))' \cos tx]_0^c + \frac{1}{\pi
x^{2}} \int_0^c (t \theta(t))^{''} \cos tx \dd t,
\end{eqnarray*}
which proves the lemma.
\end{proof}

Examples \ref{x18}--\ref{x8} all satisfy Lemma \ref{l17.2}, thus
$\beta_j=1$. In Example \ref{x8}, let $\alpha\gamma>2$. One can
easily see that the same holds for Examples \ref{x14}--\ref{x17}.

For the unrestricted convergence, we can allow more general
conditions for $\theta_i$. If each $\theta_i$ satisfies (\ref{e8})
and (\ref{e9}) or (\ref{e10}) and (\ref{e11}) with $d=1$, then the
generalizations of Theorems \ref{t52}, \ref{t55} and Corollaries
\ref{c53}, \ref{c54} and \ref{c56} hold with
$$
\max\{1/(\alpha_i+1)\}< p<\infty
$$
(see Weisz \cite{wphi3,wk2}). All examples in Section \ref{s10}
satisfy these conditions.

\subsection{Hardy-Littlewood maximal functions}

Let $\X$ denote either $\T$ or $\R$. The \idword{Hardy-Littlewood
maximal function} is defined by
$$
Mf(x):= \sup_{x\in B} \frac{1}{|B|}\int_{B} |f| \dd \lambda \qquad
(x\in \X^d),\index{\file-1}{$Mf$}
$$
where the supremum is taken over all Euclidean balls $B=B_2(c,h)$
containing $x$, and $f\in L_p(\X^d)$ $(1\leq p\leq \infty)$. We
denote by $B_r(c,h)$ $(c\in \R^d,h>0)$ the ball
$$
B_r(c,h):=\{x\in \R^d: \|x-c\|_r<h\} \qquad (1\leq r\leq \infty).
$$
For $r=2$, we omit the index and write simply $B=B_2$. We can also
define the centered version of the maximal function:
$$
\tilde Mf(x):= \sup_{h>0} \frac{1}{|B(x,h)|}\int_{B(x,h)} |f| \dd
\lambda \qquad (x\in \X^d).
$$
Of course, $\tilde Mf \leq Mf$. On the other hand, if $x\in B(y,h)$,
then $B(y,h)\subset B(x,2h)$ and so $Mf \leq 2^d \tilde Mf$. For a
ball $B(x,h)$, let $2B(x,h):=B(x,2h)$. We need the following
covering lemma.

\begin{lem}[{Vitali covering lemma}]\label{l176.1}
Let $E$ be a measurable subset of $\X^d$ that is the union of a
finite collection of Euclidean balls $\{B_j\}$. Then we can choose a
disjoint subcollection $B_1,\ldots,B_m$ such that
$$
\sum_{k=1}^m |B_k| \geq 2^{-d} |E|.
$$
\end{lem}

\begin{proof}
Let $B_1$ be a ball of the collection $\{B_j\}$ with maximal radius.
Next choose $B_2$ to have maximal radius among the subcollection of
balls disjoint from $B_1$. We continue this process until we can go
no further. Then the balls $B_1,\ldots,B_m$
are disjoint. Observe that $2B_k$ contains all balls of the original collection that intersect 
$B_k$ $(k=1,\ldots,m)$. From this, it follows that $\cup_{k=1}^m 2B_k$ contains all balls from 
the original collection. Thus
$$
|E| \leq \Big|\bigcup_{k=1}^m 2B_k \Big| \leq \sum_{k=1}^m |2B_k|
\leq 2^d \sum_{k=1}^m |B_k|,
$$
which shows the lemma.
\end{proof}

\begin{thm}\label{t6.2}
The maximal operator $M$ is of weak type (1,1), i.e.,
\begin{equation}\label{e6.61}
\sup_{\rho >0} \rho \lambda(Mf > \rho) \leq C \|f\|_1 \qquad (f \in
L_1(\X^d)).
\end{equation}
Moreover, if $1<p \leq \infty$, then
\begin{equation}\label{e6.62}
\|Mf\|_p\leq C_p \|f\|_p \qquad (f \in L_p(\X^d)).
\end{equation}
\end{thm}

\begin{proof}
Let $E_\rho:=\{Mf>\rho\}$ and $E\subset E_\rho$ be a compact subset.
For each $x\in E$, there exists a ball $B_x$ such that $x\in B_x$
and
\begin{equation}\label{e6.60}
|B_x| \leq \frac{1}{\rho} \int_{B_x} |f| \dd \lambda.
\end{equation}
Since $x\in B_x$ and $E$ is compact, we can select a finite
collection of these balls covering $E$. By Lemma \ref{l176.1}, we
can choose a finite disjoint subcollection $B_1,\ldots,B_m$ of this
covering with
$$
|E|\leq 2^d \sum_{k=1}^m |B_k|.
$$
Since each ball $B_k$ satisfies (\ref{e6.60}), adding these
inequalities, we obtain
$$
|E| \leq \frac{C}{\rho} \int_{\X^d} |f| \dd \lambda.
$$
Taking the supremum over all compact $E\subset E_\rho$, we get
(\ref{e6.61}). Since $M$ is evidently bounded on $L_ \infty(\X^d)$,
we get from (\ref{e6.61}) and from interpolation that (\ref{e6.62})
holds.
\end{proof}

Note that the inequality $\|f\|_p \leq \|Mf\|_p$ $(1<p \leq \infty)$
is trivial. If we use in the definition of the Hardy-Littlewood
maximal function the $r$-norm and the balls $B_r(c,h)$, then we get
an equivalent maximal function. In the special case when $r=\infty$,
we have
$$
M_cf(x):= \sup_{x\in I} \frac{1}{|I|}\int_I |f| \dd \lambda \qquad
(x\in \X^d),\index{\file-1}{$M_cf$}
$$
where the supremum is taken over all cubes with sides parallel to
the axes. Of course,
$$
C_1 M_cf \leq Mf \leq C_2 M_cf.
$$
Let
$$
M_\Box f(x):= \sup_{x\in I, \tau^{-1} \leq |I_i|/|I_j|\leq \tau
\atop i,j=1,\ldots,d} \frac{1}{|I|}\int_I |f| \dd \lambda \qquad
(x\in \X^d)\index{\file-1}{$M_\Box f$}
$$
for some $\tau\geq 1$, where the supremum is taken over all
appropriate rectangles
$$
I= I_1\times \cdots \times I_d
$$
with sides parallel to the axes. Again, it is easy to see that
$$
C_1 M_\Box f \leq Mf \leq C_2 M_\Box f.
$$
From this follows

\begin{cor}\label{c176.2}
We have
$$
\sup_{\rho >0} \rho \lambda(M_\Box f > \rho) \leq C \|f\|_1 \qquad
(f \in L_1(\X^d))
$$
and, for $1<p \leq \infty$,
$$
\|M_\Box f\|_p\leq C_p \|f\|_p \qquad (f \in L_p(\X^d)).
$$
\end{cor}

\begin{cor}\label{c176.3}
If $f \in L_1(\X^d)$ and $\tau \geq 1$, then
$$
\lim_{x\in I, |I_j|\to 0 \atop \tau^{-1} \leq |I_i|/|I_j|\leq \tau,
i,j=1,\ldots,d} \frac{1}{|I|}\int_I f \dd \lambda = f(x)
$$
for a.e.~$x\in\X^d$.
\end{cor}

\begin{proof}
The result is clear for continuous functions. Since the continuous
functions are dense in $L_1(\X^d)$, the corollary follows from the
density Theorem \ref{t4}.
\end{proof}

Let us consider the \idword{strong maximal function}
$$
M_s f(x):= \sup_{x\in I} \frac{1}{|I|}\int_I |f| \dd \lambda \qquad
(x\in \X^d),\index{\file-1}{$M_sf$}
$$
where $f\in L_p(\X^d)$ $(1\leq p\leq \infty)$ and the supremum is
taken over all rectangles with sides parallel to the axes.

\begin{thm}\label{t6.3}
If $f\in L(\log L)^{d-1}(\X^d)$ and $C_0>0$, then
$$
\sup_{\rho >0} \rho \lambda(x: M_s f(x) > \rho, \|x\|_\infty\leq
C_0) \leq C + C \| \, |f| (\log^+|f|)^{d-1}\|_1.
$$
Moreover, for $1<p \leq \infty$, we have
$$
\|M_s f\|_p\leq C_p \|f\|_p \qquad (f \in L_p(\X^d)).
$$
\end{thm}

To the proof, we need the following lemma (see e.g.~Weisz
\cite[p.~12]{wk2}).

\begin{lem}\label{l1.2}
If a sublinear operator $T$ is bounded on $L_\infty(\T)$ and of weak
type $(1,1)$, then for every $k=1,2,\ldots,$
$$
\|\, |Tf| (\log^+|Tf|)^{k-1}\|_1 \leq C+C\|\, |f|
(\log^+|f|)^{k}\|_1 \qquad (f \in L(\log L)^{k}).
$$
\end{lem}

\begin{proof*}{Theorem \ref{t6.3}}
Let us denote the one-dimensional Hardy-Littlewood maximal function
with respect to the $j$th coordinate by $M^{(j)}$. By Theorem
\ref{t6.2},
\begin{eqnarray*}
\lefteqn{\sup_{\rho >0} \rho \lambda(x: M_s f(x) > \rho,
\|x\|_\infty\leq C_0) } \n\\ &=& \sup_{\rho >0} \rho \lambda(x:
M^{(1)} \circ M^{(2)}\circ \cdots\ \circ M^{(d)}f(x) > \rho,
\|x\|_\infty\leq C_0) \\
&\leq& \| M^{(2)}\circ \cdots\ \circ M^{(d)}f \, 1_{B_\infty(0,C_0)} \|_1 \\
&\leq& C+C\|\, |M^{(3)}\circ \cdots\ \circ M^{(d)}f|
(\log^+|M^{(3)}\circ \cdots\ \circ M^{(d)}f|)\,
1_{B_\infty(0,C_0)}\|_1 \\
&\leq& \cdots \leq C+C\|\, |f| (\log^+|f|)^{d-1} \,
1_{B_\infty(0,C_0)}\|_1.
\end{eqnarray*}
The second inequality of Theorem \ref{t6.3} follows similarly.
\end{proof*}

Note that the condition $\|x\|_\infty\leq C_0$ in Theorem \ref{t6.3}
is important, because the measure space in Lemma \ref{l1.2} has
finite measure. The operators $M_\Box$ and $M_s$ are not bounded
from $L_1(\X^d)$ to $L_1(\X^d)$.

Similarly to Corollary \ref{c176.3}, we obtain

\begin{cor}\label{c6.3}
If $f\in L(\log L)^{d-1}(\X^d)$, then
$$
\lim_{x\in I, |I_j|\to 0 \atop j=1,\ldots,d} \frac{1}{|I|}\int_I f
\dd \lambda = f(x)
$$
for a.e.~$x\in\X^d$.
\end{cor}

\subsection{Restricted convergence at Lebesgue points}\label{s17.3}

Under some conditions on $\theta$, we can characterize the set of
almost everywhere convergence. The well known theorem of Lebesgue
\cite{leb} says that, for the one-dimensional Fej{\'e}r means and
for all $f \in L_1(\T)$,
\begin{equation}\label{e22}
\lim_{n\to\infty}\sigma_{n}f(x)=f(x)
\end{equation}
at each Lebesgue point of $f$. In this subsection, we generalize
this result to higher dimensions. Here, we investigate the almost
everywhere convergence for other summability functions than in the
preceding subsection. We do not suppose that $\theta=\theta_1\otimes
\cdots \otimes \theta_d$.

First of all, we introduce the Herz spaces. We say that a function
belongs to the \idword{homogeneous Herz space} \inda{$E_q(\R^d)$}
$(1\leq q\leq \infty)$ if
\begin{equation}\label{e17.232}
\|f\|_{E_q}:= \sum_{k=-\infty}^\infty 2^{kd(1-1/q)} \|f
1_{P_k}\|_q<\infty,
\end{equation}
where
$$
P_k:=\{x\in \R^d: 2^{k-1}\pi \leq
\|x\|_\infty<2^{k}\pi\}=B_\infty(0,2^k\pi)\setminus
B_\infty(0,2^{k-1}\pi) \qquad (k\in \Z).
$$
It is easy to see that, using other norms of $\R^d$ in the
definition of $P_k$, like
$$
P_k^r:=\{x\in \R^d: 2^{k-1}\pi \leq
\|x\|_r<2^{k}\pi\}=B_r(0,2^k\pi)\setminus B_r(0,2^{k-1}\pi) \qquad
(k\in \Z),
$$
we obtain the same spaces $E_q(\R^d)$ with equivalent norms for all
$1\leq r\leq \infty$. If we modify the definition of $P_k^r$,
$$
P_k^r=\{x\in \R^d: 2^{k-1}\pi \leq \|x\|_r<2^{k}\pi\} \cap \T^d
\qquad (k\in \Z),
$$
then we get the definition of the space \inda{$E_q(\T^d)$}. This
means that for $r=\infty$, we have to take the sum in
(\ref{e17.232}) only for $k\leq 0$. These spaces are special cases
of the Herz spaces \cite{he3} (see also Garcia-Cuerva and Herrero
\cite{gar1}). It is easy to see that
$$
L_1(\X^d)= E_1(\X^d) \hookleftarrow {E}_q(\X^d) \hookleftarrow
{E}_{q'}(\X^d)\hookleftarrow {E}_\infty(\X^d) \qquad
(1<q<q'<\infty),
$$
where $\X$ denotes either $\R$ or $\T$. Moreover,
\begin{equation}\label{e17.22}
{E}_\infty(\T^d) \hookleftarrow L_\infty(\T^d).
\end{equation}

It is known in the one-dimensional case (see e.g.~Torchinsky
\cite{to}) that if there exists an even function $\eta$ such that
$\eta$ is non-increasing on $\R_+$, $|\widehat \theta|\leq \eta$,
$\eta\in L_1(\R)$, then $\sigma_*^\theta$ is of weak type $(1,1)$.
Under similar conditions, we will generalize this result to the
multi-dimensional setting. First, we introduce an equivalent
condition (see Feichtinger and Weisz \cite{feiwe2}).

\begin{thm}\label{t23}
For a measurable function $f$, let the non-increasing majorant be
defined by
$$
\eta(x):= \sup_{\|t\|_r\geq \|x\|_r} |f(t)|
$$
for some $1\leq r\leq \infty$. Then $f\in {E}_\infty(\R^d)$ if and
only if $\eta\in L_1(\R^d)$ and
$$
C^{-1}\|\eta\|_1 \leq \|f\|_{{E}_\infty} \leq C\|\eta\|_1.
$$
\end{thm}

\begin{proof}
If $\eta\in L_1(\R^d)$, then
$$
\|f\|_{{E}_\infty} \leq \|\eta\|_{{E}_\infty} =
\sum_{k=-\infty}^\infty 2^{kd} \|\eta 1_{P_k}\|_\infty =
\sum_{k=-\infty}^\infty 2^{kd} \eta(2^{k-1}\pi)\leq C \|\eta\|_1.
$$

For the converse, denote by
$$
a_k:=\sup_{B_r(0,2^k\pi)\setminus B_r(0,2^{k-1}\pi)}|f| \quad
\mbox{and} \quad \nu':=\sum_{k=-\infty}^{\infty} a_k
1_{B_r(0,2^k\pi)\setminus B_r(0,2^{k-1}\pi)}.
$$
Let
$$
\nu(x):= \sup_{\|t\|_r\geq \|x\|_r} \nu'(t)  \qquad (x\in \R^d).
$$
Since $f\in {E}_\infty(\R^d)$ implies $\lim_{k\to\infty} a_k=0$, we
conclude that there exists an increasing sequence $(n_k)_{k\in \Z}$
of integers such that $(a_{n_k})_{k\in\Z}$ is decreasing and $\nu$
can be written in the form
$$
\nu=\sum_{k=-\infty}^{\infty} a_{n_{k}}
1_{B_r(0,2^{n_{k}}\pi)\setminus B_r(0,2^{n_{k-1}}\pi)}.
$$
Thus
$$
\|\eta\|_1 \leq \|\nu\|_1 = \sum_{k=-\infty}^{\infty}
a_{n_k}\int_{B_r(0,2^{n_{k}}\pi)\setminus B_r(0,2^{n_{k-1}}\pi)} \dd
\lambda = C\sum_{k=-\infty}^{\infty}
\Big(2^{dn_{k}}-2^{dn_{k-1}}\Big) a_{n_k}.
$$
By Abel rearrangement,
$$
\|\eta\|_1 \leq C\sum_{k=-\infty}^{\infty} 2^{dn_{k-1}}
(a_{n_{k-1}}-a_{n_k}) \leq C\|f\|_{{E}_\infty},
$$
which proves the theorem.
\end{proof}

Obviously, the result holds for the space ${E}_\infty(\T^d)$ as
well.

\begin{thm}\label{t17.210} If $\theta\in W(C,\ell_1)(\R^d)$ and
\begin{equation}\label{e17.20}
\sup_{n\in \R^d_\tau} \|K_n^\theta\|_{E_\infty(\T^d)}\leq C,
\end{equation}
then
\begin{equation}\label{e17.236}
\sigma_\Box^\theta f \leq C \Big(\sup_{n\in \R^d_\tau}
\|K_n^\theta\|_{E_\infty(\T^d)} \Big) M_\Box f \qquad \mbox{a.e.}
\end{equation}
for all  $f\in L_1(\T^d)$.
\end{thm}

\begin{proof}
By (\ref{e17.2}),
$$
|\sigma_n^\theta f(x)| = \frac{1}{(2\pi)^d}\Big|\int_{\T^d} f(x-t)
K_n^{\theta}(t) \dd t \Big| \leq \frac{1}{(2\pi)^d}
\sum_{k=-\infty}^0 \int_{P_{k}} |f(x-t)| |K_n^{\theta}(t)| \dd t.
$$
Then
$$
|\sigma_n^\theta f(x)| \leq \frac{1}{(2\pi)^d} \sum_{k=-\infty}^0
\sup_{P_{k}} |K_n^{\theta}| \int_{P_{k}} |f(x-t)| \dd t.
$$
It is easy to see that if
$$
G(u):= \int_{|t_j|<u_j,j=1,\ldots ,d} |f(x-t)| \dd t \qquad (u\in
\R_+^d),
$$
then
$$
\frac{G(u)}{\prod_{j=1}^{d} u_j} \leq C M_\Box f(x) \qquad (u\in
\R_\tau^d).
$$
Therefore
\begin{eqnarray*}
|\sigma_n^\theta f(x)| &\leq& C \sum_{k=-\infty}^0 \sup_{P_{k}} |K_n^{\theta}| \, G(2^{k}\pi,\ldots ,2^{k}\pi) \\
&\leq& C \sum_{k=-\infty}^0 2^{kd} \sup_{P_{k}} |K_n^{\theta}| \, M_\Box f(x)\\
&=& C \|K_n^{\theta}\|_{{E}_\infty(\T^d)} M_\Box f(x),
\end{eqnarray*}
which shows the theorem.
\end{proof}

Note that $\theta\in W(C,\ell_1)(\R^d)$ implies $K_n^\theta\in
L_\infty(\T^d) \subset {E}_{\infty}(\T^d)$ for all $n\in \N^d$,
because of (\ref{e17.3}). Corollary \ref{c176.2} implies immediately

\begin{thm}\label{t17.210a} If $\theta\in W(C,\ell_1)(\R^d)$ and
$$
\sup_{n\in \R^d_\tau} \|K_n^\theta\|_{E_\infty(\T^d)}\leq C
$$
holds, then for every $1<p\leq\infty$
$$
\|\sigma_\Box^\theta f\|_p \leq C_p \Big(\sup_{n\in \R^d_\tau}
\|K_n^\theta\|_{E_\infty(\T^d)} \Big) \|f\|_p \qquad (f\in
L_p(\T^d)).
$$
Moreover,
$$
\sup_{\rho >0} \rho \lambda(\sigma_\Box^\theta f > \rho) \leq C
\Big(\sup_{n\in \R^d_\tau} \|K_n^\theta\|_{E_\infty(\T^d)} \Big)
\|f\|_1 \qquad (f\in L_1(\T^d)).
$$
\end{thm}

\begin{cor}\label{c17.19} If $\theta\in W(C,\ell_1)(\R^d)$, $\theta(0)=1$ and
$$
\sup_{n\in \R^d_\tau} \|K_n^\theta\|_{E_\infty(\T^d)}\leq C
$$
holds, then
$$
\lim_{n\to \infty, \, n\in \R_\tau^d} \sigma_n^\theta f =  f \quad
\mbox{ a.e.}
$$
for all $f\in L_1(\T^d)$.
\end{cor}

In the next theorem, we suppose a little bit more than in Theorem
\ref{t6}, namely instead of $\widehat \theta\in L_1(\R^d)$, we
suppose that $\widehat \theta\in {E}_{\infty}(\R^d)$.

\begin{thm}\label{t17.2100} If $\theta\in W(C,\ell_1)(\R^d)$ and $\widehat \theta\in {E}_{\infty}(\R^d)$, then
$$
\sigma_\Box^\theta f \leq C \|\widehat \theta\|_{{E}_\infty(\R^d)}
M_\Box f \qquad \mbox{a.e.}
$$
for all  $f\in L_1(\T^d)$.
\end{thm}

\begin{proof}
Since by Theorem \ref{t5}
$$
\sigma_n^\theta f(x) =\frac{1}{(2\pi)^d}\int_{\T^d} f(x-t)
K_n^{\theta}(t) \dd t = \Big(\prod_{j=1}^{d} n_j \Big) \int_{\R^d}
f(x-t) \widehat \theta (n_1t_1,\ldots,n_dt_d) \dd t,
$$
we get similarly to (\ref{e6.3}) that
$$
K_n^{\theta}(t)=(2\pi)^d \Big(\prod_{j=1}^{d} n_j \Big) \sum_{j\in
\Z^d} \widehat \theta (n_1(t_1+2j_1\pi),\ldots,n_d(t_d+2j_d\pi)).
$$
We will prove that $\widehat \theta\in {E}_{\infty}(\R^d)$ implies
$$
\|K_n^{\theta}\|_{E_\infty(\T^d)} \leq  C \|\widehat
\theta\|_{E_\infty(\R^d)} \quad \mbox{for all} \quad n\in \R^d_\tau.
$$
Suppose that $2^{s-1}< \tau \leq 2^s$ and $2^{l-1}< n_1 \leq 2^l$.
Since $n\in \R^d_\tau$, we have
$$
2^{l-s-1}< \frac{1}{\tau} n_1 \leq n_j \leq \tau n_1 \leq 2^{l+s}
\quad \mbox{for all} \quad j=1,\ldots ,d.
$$
First, we investigate the term $j=0$ from the sum:
\begin{eqnarray}\label{e17.21}
\lefteqn{\Big\|\Big(\prod_{j=1}^{d} n_j \Big) \widehat \theta
(n_1t_1,\ldots,n_dt_d)\Big\|_{E_\infty(\T^d)}  } \n\\
&=&\sum_{k=-\infty}^0 2^{kd}
\Big(\prod_{j=1}^{d} n_j \Big) \sup_{2^{k-1}\pi \leq \|t\|_\infty<2^{k}\pi} |\widehat \theta (n_1t_1,\ldots,n_dt_d)| \n\\
&\leq & C \sum_{k=-\infty}^0 2^{(k+l)d} \sup_{2^{k-2+l-s}\pi \leq \|t\|_\infty<2^{k+l+s}\pi} |\widehat \theta (t_1,\ldots,t_d)| \n\\
&\leq & C \sum_{i=-\infty}^{l+s} 2^{id} \sup_{2^{i-1}\pi \leq \|t\|_\infty<2^{i}\pi} |\widehat \theta (t_1,\ldots,t_d)| \nonumber \\
&\leq&  C \|\widehat \theta\|_{E_\infty(\R^d)}.
\end{eqnarray}
Moreover,
\begin{eqnarray*}
\lefteqn{\Big\|\Big(\prod_{j=1}^{d} n_j \Big) \sum_{j\in \Z^d, \, j\neq 0} \widehat \theta (n_1(t_1+2j_1\pi),\ldots,n_d(t_d+2j_d\pi))\Big\|_{E_\infty(\T^d)}  } \n\\
&=& \sum_{k=-\infty}^0 2^{kd} \Big(\prod_{j=1}^{d} n_j \Big) \sup_{2^{k-1}\pi \leq \|t\|_\infty<2^{k}\pi} \sum_{j\in \Z^d,\, j\neq 0} |\widehat \theta (n_1(t_1+2j_1\pi),\ldots,n_d(t_d+2j_d\pi))| \n\\
&\leq & C \Big(\prod_{j=1}^{d} n_j \Big) \sup_{\|t\|_\infty<\pi} \sum_{j\in \Z^d,\, j\neq 0} |\widehat \theta (n_1(t_1+2j_1\pi),\ldots,n_d(t_d+2j_d\pi))| \nonumber \\
&\leq & C \Big(\prod_{j=1}^{d} n_j \Big) \sum_{j\in \Z^d,\, j\neq 0}
\sup_{\|t\|_\infty<\pi} |\widehat \theta
(n_1(t_1+2j_1\pi),\ldots,n_d(t_d+2j_d\pi))|.
\end{eqnarray*}
Since
$$
|n_j(t_j+2j_j\pi)| \geq \frac{1}{\tau} n_1\pi > 2^{l-s-1}\pi \qquad
(j=1,\ldots ,d),
$$
we conclude
\begin{eqnarray*}
\lefteqn{\Big\|\Big(\prod_{j=1}^{d} n_j \Big) \sum_{j\in \Z^d,\, j\neq 0} \widehat \theta (n_1(t_1+2j_1\pi),\ldots,n_d(t_d+2j_d\pi))\Big\|_{E_\infty(\T^d)}  } \n\\
&\leq & C \Big(\prod_{j=1}^{d} n_j \Big) \sum_{i=(l-s)\vee
0}^{\infty}\ \
\sum_{j\in \Z^d,\, j\neq 0,n_1(\T+2j_1\pi)\times \cdots \times n_d(\T+2j_d\pi) \cap \{t\in \R^d: 2^i\pi\leq \|t\|_\infty<2^{i+1}\pi\} \neq 0} \nonumber \\
&&{} \sup_{2^i\pi\leq \|t\|_\infty<2^{i+1}\pi} |\widehat \theta (t)| \nonumber \\
&\leq & C \sum_{i=(l-s)\vee 0}^{\infty} 2^{id}
\sup_{2^i\pi\leq \|t\|_\infty<2^{i+1}\pi} |\widehat \theta (t)| \\
&\leq&  C \|\widehat \theta\|_{E_\infty(\R^d)},
\end{eqnarray*}
which yields indeed that $\|K_n^{\theta}\|_{E_\infty(\T^d)} \leq  C
\|\widehat \theta\|_{E_\infty(\R^d)}$ for all $n\in \R^d_\tau$. The
theorem follows from Theorem \ref{t17.210}.
\end{proof}

\begin{thm}\label{t17.210a0}
If $\theta\in W(C,\ell_1)(\R^d)$ and $\widehat \theta\in
{E}_{\infty}(\R^d)$, then for every $1<p\leq\infty$
$$
\|\sigma_\Box^\theta f\|_p \leq C_p \|\widehat
\theta\|_{{E}_\infty(\R^d)} \|f\|_p \qquad (f\in L_p(\T^d)).
$$
Moreover,
$$
\sup_{\rho >0} \rho \lambda(\sigma_\Box^\theta f > \rho) \leq C
\|\widehat \theta\|_{{E}_\infty(\R^d)} \|f\|_1 \qquad (f\in
L_1(\T^d)).
$$
\end{thm}

\begin{cor}\label{c17.190}
If $\theta\in W(C,\ell_1)(\R^d)$, $\theta(0)=1$ and $\widehat
\theta\in {E}_{\infty}(\R^d)$, then
$$
\lim_{n\to \infty, \, n\in \R_\tau^d} \sigma_n^\theta f =  f \quad
\mbox{ a.e.}
$$
for all $f\in L_1(\T^d)$.
\end{cor}

If $f\in L_1(\T^d)$ implies the almost everywhere convergence of
Corollary \ref{c17.19}, then $\sigma_\Box^\theta$ is bounded from
$L_1(\T^d)$ to $L_{1,\infty}(\T^d)$, as in Theorem \ref{t17.210a}
(see Stein \cite{st2}). The partial converse of Theorem
\ref{t17.210} is given in the next result. More exactly, if
$\sigma_\Box^\theta f$ can be estimated pointwise by $M_\Box f$,
then (\ref{e17.20}) holds.

\begin{thm}\label{t17.216} If $\theta\in W(C,\ell_1)(\R^d)$ and
\begin{equation}\label{e17.224}
\sigma_\Box^\theta f(x) \leq C M_\Box f(x)
\end{equation}
for all  $f\in L_1(\T^d)$ and $x\in \T^d$, then
$$
\sup_{n\in \R^d_\tau} \|K_n^\theta\|_{E_\infty(\T^d)}\leq C.
$$
\end{thm}

\begin{proof}
We define the space \inda{$D_p(\T^d)$} $(1\leq p<\infty)$ by the
norm
\begin{equation}\label{e17.226}
\|f\|_{{D}_p(\T^d)}:=\sup_{0<r\leq \pi} \Big(\frac{1}{r^d}
\int_{[-r,r]^d}|f|^p\dd \lambda \Big)^{1/p}.
\end{equation}
One can show that the norm
\begin{equation}\label{e17.225}
\|f\|_*= \sup_{k\leq 0} 2^{-kd/p} \|f 1_{P_k}\|_p
\end{equation}
is an equivalent norm on ${D}_p(\T^d)$. Indeed, choosing $r=2^k\pi$
$(k\leq 0)$, we conclude $\|f\|_*\leq C\|f\|_{{D}_p}$. On the other
hand, for $n\leq 0$
$$
2^{-nd} \int_{[-2^n\pi,2^n\pi]^d}|f|^p\dd \lambda = 2^{-nd}
\sum_{k=-\infty}^{n} \int_{P_k} |f|^p\dd \lambda \leq 2^{-nd}
\sum_{k=-\infty}^{n} 2^{kd} \|f\|_{*}^p \leq C \|f\|_{*}^p,
$$
or, in other words $\|f\|_{{D}_p(\T^d)}\leq C\|f\|_*$. Choosing
$n=0$, we can see that ${D}_p(\T^d)\subset L_p(\T^d)$ and
$\|f\|_p\leq C\|f\|_{{D}_p(\T^d)}$. Taking the suprema in
(\ref{e17.226}) and (\ref{e17.225}) for all $0<r<\infty$ and $k\in
\Z$, we obtain the space ${D}_p(\R^d)$.

It is easy to see by (\ref{e17.225}) that
\begin{equation}\label{e17.233}
\sup_{\|f\|_{{D}_1(\T^d)}\leq 1} \Big| \int_{\T^d} f(-t)
K_n^\theta(t) \dd t \Big| = \|K_n^\theta\|_{{E}_\infty(\T^d)}.
\end{equation}
There exists a function $f\in {D}_1(\T^d)$ with $\|f\|_{{D}_1}\leq
1$ such that
$$
\frac{\|K_n^\theta\|_{{E}_\infty(\T^d)}}{2}\leq  \Big| \int_{\T^d}
f(-t) K_n^\theta(t) \dd t \Big|.
$$
Since $f\in L_1(\R^d)$, by (\ref{e17.224}),
$$
|\sigma_n^\theta f(0)| = \Big| \int_{\T^d} f(-t) K_n^\theta(t) \dd t
\Big| \leq C M_\Box f(0) \qquad (n\in \R_\tau^d),
$$
which implies
$$
\|K_n^\theta\|_{{E}_\infty(\T^d)} \leq C M_\Box f(0) \leq C M_c f(0)
\leq C\|f\|_{{D}_1}\leq C.
$$
This proves the result.
\end{proof}

Note that the norm of $D_p(\T^d)$ in (\ref{e17.226}) is equivalent
to
$$
\|f\|=\sup_{r\in [0,\pi]^d,\, \tau^{-1} \leq r_i/r_j \leq \tau\atop
i,j=1,\ldots,d} \Big(\frac{1}{\prod_{j=1}^{d}r_j}
\int_{-r_1}^{r_1}\cdots \int_{-r_d}^{r_d} |f|^p\dd \lambda
\Big)^{1/p}.
$$

Now, we introduce the concept of Lebesgue points in higher
dimensions. Corollary \ref{c176.3} says that
\begin{equation}\label{e17.11}
\lim_{h\to 0,\, \tau^{-1} \leq h_i/h_j \leq \tau\atop
i,j=1,\ldots,d} \frac{1}{\prod_{j=1}^{d}(2h_j)}
\int_{-h_1}^{h_1}\cdots \int_{-h_d}^{h_d} f(x+u) \dd u=f(x)
\end{equation}
for a.e.~$x\in \T^d$, where $f\in L_1(\T^d)$. A point $x\in \T^d$ is
called a \idword{Lebesgue point} of $f\in L_1(\T^d)$ if
$$
\lim_{h\to 0,\, \tau^{-1} \leq h_i/h_j \leq \tau\atop
i,j=1,\ldots,d} \frac{1}{\prod_{j=1}^{d}(2h_j)}
\int_{-h_1}^{h_1}\cdots \int_{-h_d}^{h_d} |f(x+u)-f(x)| \dd u =0.
$$
One can see that this definition is equivalent to
$$
\lim_{h\to 0} \frac{1}{(2h)^d} \int_{-h}^{h}\cdots \int_{-h}^{h}
|f(x+u)-f(x)| \dd u=0.
$$

\begin{thm}\label{t17.2}
Almost every point is a Lebesgue point of $f\in L_1(\T^d)$.
\end{thm}

\begin{proof}
Since $f$ is integrable, we may suppose that $f(x)$ is finite. Let
$q$ be a rational number for which $|f(x)-q|<\epsilon$. Then
\begin{eqnarray*}
\frac{1}{(2h)^d} \int_{-h}^{h}\cdots \int_{-h}^{h} |f(x+u)-f(x)| \dd
u &\leq &
\frac{1}{(2h)^d} \int_{-h}^{h}\cdots \int_{-h}^{h} |f(x+u)-q| \dd u \\
&&{}+ \frac{1}{(2h)^d} \int_{-h}^{h}\cdots \int_{-h}^{h} |q-f(x)|
\dd u.
\end{eqnarray*}
The second integral is equal to $|q-f(x)|<\epsilon$. Applying
(\ref{e17.11}) to the function $|f(\cdot)-q|$, we can see that for
a.e.~$x$, the first integral is less than $\epsilon$ if $h$ is small
enough.
\end{proof}

The next theorem generalizes Lebesgue's theorem (\ref{e22}) (see
Feichtinger and Weisz \cite{feiwe2}).

\begin{thm}\label{t32}
Suppose that $\theta\in W(C,\ell_1)(\R^d)$, $\theta(0)=1$ and
$$
\sup_{n\in \R^d_\tau} \|K_n^\theta\|_{E_\infty(\T^d)}\leq C.
$$
If there exists $1\leq r\leq \infty$ such that for all $\delta>0$
\begin{equation}\label{e17.23}
\lim_{n\to\infty,n\in \R^d_\tau} \sup_{\T^d\setminus B_r(0,\delta)}
|K_n^\theta| =0,
\end{equation}
then
$$
\lim_{n\to \infty, \, n\in \R_\tau^d}\sigma_n^\theta f(x) = f(x)
$$
for all Lebesgue points of $f\in L_1(\T^d)$.
\end{thm}

\begin{proof}
Now, set
$$
G(u):= \int_{|t_j|<u_j,j=1,\ldots ,d} |f(x-t)-f(x)| \dd t \qquad
(u\in \R_+).
$$
Since $x$ is a Lebesgue point of $f$, for all $\epsilon>0$, there
exists $m\in \N$ such that
\begin{equation}\label{e17.215}
\frac{G(u)}{\prod_{j=1}^{d} u_j} \leq \epsilon \qquad \mbox{if}
\qquad 0<u_j\leq 2^{-m},j=1,\ldots ,d,u\in \R_\tau^d.
\end{equation}
Note that
$$
\sigma_n^{\theta} f(x)-f(x) = \frac{1}{(2\pi)^d}\int_{\T^d}
(f(x-t)-f(x)) K_n^{\theta}(t) \dd t.
$$
Thus
\begin{eqnarray*}
|\sigma _n^{\theta} f(x)-f(x)| &\leq& C \int_{\T^d} |f(x-t)-f(x)| |K_n^{\theta}(t)| \dd t \\
&=& C \int_{(-2^{-m}\pi,2^{-m}\pi)^d} |f(x-t)-f(x)| |K_n^{\theta}(t)| \dd t \\
&&{}+ C \int_{\T^d\setminus (-2^{-m}\pi,2^{-m}\pi)^d} |f(x-t)-f(x)| |K_n^{\theta}(t)| \dd t \\
&=:& A_0(x)+A_1(x).
\end{eqnarray*}
We estimate $A_0(x)$ by
\begin{eqnarray*}
A_0(x)&=& C\sum_{k=-\infty}^{-m} \int_{P_{k}} |f(x-t)-f(x)|
|K_n^{\theta}(t)|
\dd t\\
&\leq& C\sum_{k=-\infty}^{-m} \sup_{P_{k}} |K_n^{\theta}| \int_{P_{k}} |f(x-t)-f(x)| \dd t \\
&\leq& C\sum_{k=-\infty}^{-m} \sup_{P_{k}} |K_n^{\theta}| \,
G(2^{k}\pi,\ldots ,2^{k}\pi).
\end{eqnarray*}
Then, by (\ref{e17.215}),
$$
A_0(x) \leq C \epsilon \sum_{k=-\infty}^{-m} 2^{kd} \sup_{P_{k}}
|K_n^{\theta}| \leq C\epsilon \|K_n^{\theta}\|_{{E}_\infty(\T^d)}.
$$

Let $B_r(0,\delta)$ be the largest ball for which
$B_r(0,\delta)\subset (-2^{-m}\pi,2^{-m}\pi)^d$. Then
\begin{eqnarray*}
A_1(x)&\leq & C \int_{\T^d\setminus B_r(0,\delta)} |f(x-t)-f(x)| |K_n^{\theta}(t)| \dd t \\
&\leq& C \sup_{\T^d\setminus B_r(0,\delta)} |K_n^\theta|
(\|f\|_1+|f(x)|),
\end{eqnarray*}
which tends to $0$ as $n\to\infty,n\in \R^d_\tau$.
\end{proof}

Observe that (\ref{e17.22}) and $B_r(0,\delta')\subset
(-2^k\pi,2^k\pi)^d \subset B_r(0,\delta)$ imply
\begin{equation}\label{e17.24}
\|K_n^\theta\|_{E_\infty(\T^d\setminus B_r(0,\delta))} \leq
\|K_n^\theta\|_{L_\infty(\T^d\setminus (-2^k\pi,2^k\pi)^d)} \leq
C_\delta \|K_n^\theta\|_{E_\infty(\T^d\setminus B_r(0,\delta'))}.
\end{equation}
In other words, condition (\ref{e17.23}) is equivalent to
$$
\lim_{n\to\infty,n\in \R^d_\tau}
\|K_n^\theta\|_{E_\infty(\T^d\setminus B_r(0,\delta))} =0.
$$

In the case $\widehat \theta\in {E}_{\infty}(\R^d)$, we can
formulate a somewhat simpler version of the preceding theorem.

\begin{thm}\label{t17.320}
Suppose that $\theta\in W(C,\ell_1)(\R^d)$, $\theta(0)=1$ and
$\widehat \theta\in {E}_{\infty}(\R^d)$. Then
$$
\lim_{n\to \infty, \, n\in \R_\tau^d}\sigma_n^\theta f(x) = f(x)
$$
for all Lebesgue points of $f\in L_1(\T^d)$.
\end{thm}

\begin{proof}
We have seen in Theorem \ref{t17.2100} that $\widehat \theta\in
{E}_{\infty}(\R^d)$ implies
$$
\|K_n^{\theta}\|_{E_\infty(\T^d)} \leq  C \|\widehat
\theta\|_{E_\infty(\R^d)} \quad \mbox{for all} \quad n\in \R^d_\tau,
$$
so the first condition of Theorem \ref{t32} is satisfied. On the
other hand, let $B_\infty(0,2^{k_0})$ be the largest ball for which
$B_\infty(0,2^{k_0})\subset B_r(0,\delta)$, let $2^{s-1}< \tau \leq
2^s$ and $2^{l-1}< n_1 \leq 2^l$. Obviously, if $n\to \infty, \,
n\in \R_\tau^d$, then $l\to \infty$. We get similarly to
(\ref{e17.21}) that
$$
\|K_n^\theta\|_{E_\infty(\T^d\setminus B_r(0,\delta))} \leq C
\sum_{i=k_0+l-s-1}^{\infty} 2^{id} \sup_{2^{i-1}\pi \leq
\|t\|_\infty<2^{i}\pi} |\widehat \theta (t_1,\ldots,t_d)|,
$$
which tends to $0$ as $n\to \infty, \, n\in \R_\tau^d$, since
$\widehat \theta\in {E_\infty(\R^d)}$. Then (\ref{e17.23}) follows
by (\ref{e17.24}).
\end{proof}

Since each point of continuity is a Lebesgue point, we have

\begin{cor}\label{c7}
If the conditions of Theorem \ref{t32} or Theorem \ref{t17.320} are
satisfied and if $f\in L_1(\T^d)$ is continuous at a point $x$, then
$$
\lim_{n\to \infty, \, n\in \R_\tau^d}\sigma_n^\theta f(x) = f(x).
$$
\end{cor}

The converse of Theorem \ref{t32} holds also.

\begin{thm}\label{t17.217}
Suppose that $\theta\in W(C,\ell_1)(\R^d)$, $\theta(0)=1$ and
(\ref{e17.23}) holds. If
$$
\lim_{n\to \infty, \, n\in \R_\tau^d}\sigma_n^\theta f(x) = f(x)
$$
for all Lebesgue points of $f\in L_1(\T^d)$, then
$$
\sup_{n\in \R^d_\tau} \|K_n^\theta\|_{E_\infty(\T^d)}\leq C.
$$
\end{thm}

\begin{proof}
The space \inda{$D_1^0(\T^d)$} consists of all functions $f\in
{D}_1(\T^d)$ for which $f(0)=0$ and $0$ is a Lebesgue point of $f$,
in other words
$$
\lim_{h\to 0} \frac{1}{(2h)^d} \int_{-h}^{h}\cdots \int_{-h}^{h}
|f(u)| \dd u=0.
$$
We will show that ${D}_1^0(\T^d)$ is a Banach space. Let $(f_n)$ be
a Cauchy sequence in ${D}_1^0(\T^d)$, i.e.,
$\|f_n-f_m\|_{D_1^0(\T^d)}\to 0$ as $n,m\to\infty$. Then there
exists a subsequence $(f_{\nu_n})$ such that
$$
\|f_{\nu_{n+1}}-f_{\nu_n}\|_{D_1^0(\T^d)}\leq 2^{-n}.
$$
Then $(f_{\nu_n})$ is a.e.~convergent, let $f=\lim_{n\to\infty}
f_{\nu_n}$ and $f(0)=0$. For all $\epsilon>0$, there exists $N$ such
that
$$
\|f-f_{\nu_N}\|_{D_1^0(\T^d)} \leq \sum_{n=N}^{\infty}
\|f_{\nu_{n+1}}-f_{\nu_n}\|_{D_1^0(\T^d)}\leq \sum_{n=N}^{\infty}
2^{-n} <\epsilon.
$$
If $h>0$ is small enough, then
$$
\frac{1}{(2h)^d} \int_{-h}^{h}\cdots \int_{-h}^{h} |f_{\nu_N}(u)|
\dd u <\epsilon.
$$
Hence
$$
\frac{1}{(2h)^d} \int_{-h}^{h}\cdots \int_{-h}^{h} |f(u)| \dd u \leq
\|f-f_{\nu_N}\|_{D_1^0(\T^d)} + \frac{1}{(2h)^d} \int_{-h}^{h}\cdots
\int_{-h}^{h} |f_{\nu_N}(u)| \dd u  <2\epsilon,
$$
whenever $h$ is small enough. From this it follows that $f\in
D_1^0(\T^d)$ and $0$ is a Lebesgue point of $f$. Thus $D_1^0(\T^d)$
is a Banach space.

We get from the conditions of the theorem that
$$
\lim_{n\to \infty, \, n\in \R_\tau^d}\sigma_n^\theta f(0) = 0 \qquad
\mbox{for all} \qquad f\in {D}_1^0(\T^d).
$$
Thus the operators
$$
U_n:{D}_1^0(\T^d)\to \R, \qquad U_nf:=\sigma_n^\theta f(0) \qquad
(n\in \R_\tau^d)
$$
are uniformly bounded by the Banach-Steinhaus theorem. Observe that
in (\ref{e17.233}), we may suppose that $f$ is $0$ in a neighborhood
of $0$. Then
\begin{eqnarray*}
C&\geq& \|U_n\| \\
&=&\sup_{\|f\|_{{D}_1^0(\T^d)}\leq 1} \Big| \int_{\T^d} f(-t)
K_n^\theta(t)
\dd t \Big| \\
&=& \sup_{\|f\|_{{D}_1(\T^d)}\leq 1} \Big| \int_{\T^d} f(-t) K_n^\theta(t) \dd t \Big| \\
&=& \|K_n^\theta\|_{{E}_\infty(\T^d)}
\end{eqnarray*}
for all $n\in \R_\tau^d$.
\end{proof}

\begin{cor}\label{c17.24}
Suppose that $\theta\in W(C,\ell_1)(\R^d)$, $\theta(0)=1$ and
(\ref{e17.23}) holds. Then
$$
\lim_{n\to \infty, \, n\in \R_\tau^d}\sigma_n^\theta f(x) = f(x)
$$
for all Lebesgue points of $f\in L_1(\T^d)$ if and only if
$$
\sup_{n\in \R^d_\tau} \|K_n^\theta\|_{E_\infty(\T^d)}\leq C.
$$
\end{cor}

A one-dimensional version of this theorem can be found in the book
of Alexits \cite{{al}}. Now, we present some sufficient condition on
$\theta$ such that $\widehat {\theta}\in E_\infty(\R^d)$. The next
theorem was proved in Herz \cite{he3}, Peetre \cite{pe} and Girardi
and Weis \cite{giwe}.

\begin{lem}\label{l17.24}
If $f\in B_{1,1}^{d}(\R^d)$, then $\widehat f\in E_\infty(\R^d)$ and
$$
\|\widehat f\|_{E_\infty} \leq C_p\|f\|_{B_{1,1}^{d}}.
$$
\end{lem}

A function $f$ belongs to the \idword{weighted Wiener amalgam space}
\inda{$W(L_\infty,\ell_1^{v_s})(\R^d)$} if
$$
\|f\|_{W(L_\infty,\ell_1^{v_s})} := \sum_{k\in \Z^d} \sup_{x\in
[0,1)^d} |f(x+k)| v_s(k) <\infty,
$$
where $v_s(x):=(1+\|x\|_\infty)^s$ $(x\in \R^d)$.

\begin{lem}\label{l17.21}
If $f\in W(L_\infty,\ell_1^{v_{d}})(\R^d)$, then $f\in
{E}_\infty(\R^d)$ and
$$
\|f\|_{{E}_\infty} \leq  C \|f\|_{W(L_\infty,\ell_1^{v_{d}})}.
$$
\end{lem}

\begin{proof}
The inequalities
\begin{eqnarray*}
\|f\|_{{E}_\infty} &=& \sum_{k=-\infty}^\infty 2^{kd} \sup_{P_k} |f| \\
&\leq& C \sum_{k=0}^\infty 2^{kd} \sum_{j: [-\pi,\pi]^d+2j\pi\cap P_k\neq \emptyset} \sup_{[-\pi,\pi]^d+2j\pi} |f|\\
&\leq& C \sum_{j\in \Z^d} (1+\|j\|_\infty)^{d} \sup_{[-\pi,\pi]^d+2j\pi} |f|\\
&=& C \|f\|_{W(L_\infty,\ell_1^{v_{d}})}
\end{eqnarray*}
prove the result.
\end{proof}

Note that if $\theta\in W(C,\ell_1)(\R^d)$ and $\widehat \theta$ has
compact support, then all the results above hold. Actually,
$\theta\in \bS_0(\R^d)$ in this case (see e.g.~Feichtinger and
Zimmerman \cite{FZ98}). We can generalize these results if $\theta$
is in a suitable modulation space. We define the \idword{weighted
Feichtinger's algebra} or \idword{modulation space}
\inda{$M_1^{v_s}(\R^d)$} (see e.g.~Feichtinger \cite{Fei81} and
Gr{\"o}chenig \cite{gr1}) by
$$
M_1^{v_s}(\R^d) := \left\{ f\in L^2(\R^d): \|f\|_{M_1^{v_s}}:=
\|S_{g_0}f \cdot v_s\|_{L_1(\R^{2d})}<\infty \right\} \qquad (s\geq
0),
$$
where $v_s(x,\omega):=v_s(\omega)=(1+\|\omega\|_\infty)^s$
$(x,\omega\in \R^d)$.

\begin{lem}\label{l17.22}
If $f \in M_1^{v_d}(\R^d)$, then $\widehat {f}\in {E}_\infty(\R^d)$
and
$$
\|\widehat f\|_{{E}_\infty} \leq C\|f\|_{M_1^{v_d}}.
$$
\end{lem}

\begin{proof}
By Lemma \ref{l17.21},
$$
\|\widehat f\|_{{E}_\infty} \leq C \|\widehat
f\|_{W(L_\infty,\ell_1^{v_d})}\leq C \|f\|_{M_1^{v_d}},
$$
where the second inequality can be found in Gr{\"o}chenig
\cite[p.~249]{gr1}.
\end{proof}

\begin{cor}\label{c17.25}
Suppose that $\theta\in W(C,\ell_1)(\R^d)$ and $\theta(0)=1$. If
$\theta\in M_1^{v_d}(\R^d)$, then
$$
\lim_{n\to \infty, \, n\in \R_\tau^d}\sigma_n^\theta f(x) = f(x)
$$
for all Lebesgue points of $f\in L_1(\R^d)$. Moreover,
$$
\sup_{\rho>0} \rho\, \lambda(\sigma_\Box^\theta>\rho) \leq C
\|\theta\|_{M_1^{v_d}} \|f\|_1\qquad (f\in L_1(\R^d))
$$
and, for every $1<p \leq \infty$,
$$
\|\sigma_\Box^\theta f\|_p \leq C_p \|\theta\|_{M_1^{v_d}} \|f\|_p
\qquad (f\in L_p(\R^d)).
$$
\end{cor}

\begin{thm}\label{t17.26}
If $\theta\in V_1^k(\R)$ for some $k\geq 2$, then $\theta\in
M_1^{v_s}(\R)$ for all $0\leq s<k-1$ and
$$
\|\theta\|_{M_1^{v_s}} \leq C_s \|\theta\|_{V_1^k}.
$$
\end{thm}

This theorem can be proved as was Theorem \ref{t17.16}. Note that
$V_1^k(\R^d)$ was defined in Definition \ref{d17.11}.

\begin{cor}\label{c17.26}
If each $\theta_j \in V_1^k(\R)$ $(j=1,\ldots,d)$, then
$$
\theta:=\prod_{j=1}^{d} \theta_j \in M_1^{v_s}(\R^d) \qquad (0\leq
s<k-1).
$$
\end{cor}

The space $V_1^2(\R)$ is not contained in $M_1^{v_1}(\R)$. However,
the same results hold as in Corollary \ref{c17.25}.

\begin{cor}\label{c3}
If $\theta \in V_1^2(\R)$, then $\widehat \theta\in E_\infty(\R)$.
\end{cor}

\begin{proof}
The inequality
$$
|\widehat \theta(x)|\leq C/x^2 \qquad (x\neq 0)
$$
can be shown similarly to Theorem \ref{t17.16}. $\widehat \theta\in
E_\infty(\R)$ follows from Theorem \ref{t23}.
\end{proof}

All examples of Subsection \ref{s10.1}, respectively Subsection
\ref{s17.1}, satisfy the condition $\widehat \theta\in
E_\infty(\R)$, respectively $\widehat \theta\in E_\infty(\R^d)$.

\subsection{Unrestricted convergence at Lebesgue points}\label{s17.4}

To formulate the unrestricted version of the preceding theorems, we
have to modify slightly the definition of the space ${E}_q(\R^d)$.
The \idword{homogeneous Herz space} \inda{$E_q'(\R^d)$} contains all
functions $f$ for which
$$
\|f\|_{{E}_q'}:= \sum_{k_1=-\infty}^\infty \cdots
\sum_{k_d=-\infty}^\infty \Big(\prod_{j=1}^{d} 2^{k_j(1-1/q)} \Big)
\|f 1_{P_k'}\|_q,
$$
where
$$
P_k':=\{x\in \R^d: 2^{k_j-1} \leq |x_j|<2^{k_j}, j=1,\ldots,d\}
\qquad (k\in \Z^d).
$$
The spaces \inda{$E_q'(\T^d)$} can be defined analogously. Again,
$$
L_1(\X^d)= E_1'(\X^d) \hookleftarrow E'_q(\X^d) \hookleftarrow
E'_{q'}(\X^d)\hookleftarrow E'_\infty(\X^d), \qquad 1<q<q'<\infty,
$$
where $\X=\R$ or $\T$ and
$$
E'_\infty(\T^d) \hookleftarrow L_\infty(\T^d).
$$
It is easy to see that ${E}_q'(\X^d)\supset {E}_q(\X^d)$ and
$$
\|f\|_{{E}_q'} \leq  C\|f\|_{{E}_q} \qquad  (1\leq q\leq \infty).
$$

Except for converse type results, all theorems of the preceding
subsection can also be proved for the unrestricted convergence (see
Feichtinger and Weisz \cite{feiwe2}). We point out some of these
theorems.

\begin{thm}\label{t17.410} If $\theta\in W(C,\ell_1)(\R^d)$ and
$$
\sup_{n\in \N^d} \|K_n^\theta\|_{E'_\infty(\T^d)}\leq C,
$$
then
$$
\sigma_*^\theta f \leq C \Big(\sup_{n\in \N^d}
\|K_n^\theta\|_{E'_\infty(\T^d)} \Big) M_s f \qquad \mbox{a.e.}
$$
for all  $f\in L_1(\T^d)$.
\end{thm}

\begin{thm}\label{t17.4100}
If $\theta\in W(C,\ell_1)(\R^d)$ and $\widehat \theta\in
E'_{\infty}(\R^d)$, then
$$
\sigma_*^\theta f \leq C \|\widehat \theta\|_{E'_\infty(\R^d)} M_s f
\qquad \mbox{a.e.}
$$
for all  $f\in L_1(\T^d)$.
\end{thm}

\begin{thm}\label{t17.410a0} If $\theta\in W(C,\ell_1)(\R^d)$ and $\widehat \theta\in E'_{\infty}(\R^d)$, then for every $1<p\leq\infty$,
$$
\|\sigma_*^\theta f\|_p \leq C_p \|\widehat
\theta\|_{E'_\infty(\R^d)} \|f\|_p \qquad (f\in L_p(\T^d)).
$$
Moreover,
$$
\sup_{\rho >0} \rho \lambda(\sigma_*^\theta f > \rho) \leq C
\|\widehat \theta\|_{E'_\infty(\R^d)} (1+\|f\|_{L(\log L)^{d-1}})
\qquad (f\in L(\log L)^{d-1}(\T^d)).
$$
\end{thm}

\begin{cor}\label{c17.490}
If $\theta\in W(C,\ell_1)(\R^d)$, $\theta(0)=1$ and $\widehat
\theta\in E'_{\infty}(\R^d)$, then
$$
\lim_{n\to \infty} \sigma_n^\theta f =  f \quad \mbox{ a.e.}
$$
for all $f\in L(\log L)^{d-1}(\T^d)$.
\end{cor}

Now, we also modify the definition of Lebesgue points. By Corollary
\ref{c6.3},
$$
\lim_{h\to 0} \frac{1}{\prod_{j=1}^{d}(2h_j)}
\int_{-h_1}^{h_1}\cdots \int_{-h_d}^{h_d} f(x+u) \dd u = f(x)
$$
for a.e.~$x\in \T^d$, where $f\in L(\log L)^{d-1}(\T^d)$. A point
$x\in \T^d$ is called a \idword{strong Lebesgue point} of $f$ if
$M_sf(x)$ is \eword{finite} and
$$
\lim_{h\to 0} \frac{1}{\prod_{j=1}^{d}(2h_j)}
\int_{-h_1}^{h_1}\cdots \int_{-h_d}^{h_d} |f(x+u)-f(x)| \dd u =0.
$$
Similarly to Theorem \ref{t17.2}, one can show that almost every
point $x\in \T^d$ is a strong Lebesgue point of $f\in L(\log
L)^{d-1}(\T^d)$.

\begin{thm}\label{t17.432}
Suppose that $\theta\in W(C,\ell_1)(\R^d)$, $\theta(0)=1$ and
$$
\sup_{n\in \N^d} \|K_n^\theta\|_{E'_\infty(\T^d)}\leq C.
$$
If for all $\delta>0$
$$
\lim_{n\to\infty} \|K_n^\theta\|_{E'_\infty(\T^d\setminus
B_\infty(0,\delta))} =0,
$$
then
$$
\lim_{n\to \infty}\sigma_n^\theta f(x) = f(x)
$$
for all strong Lebesgue points of $f\in L(\log L)^{d-1}(\T^d)$.
\end{thm}

\begin{thm}\label{t17.420}
Suppose that $\theta\in W(C,\ell_1)(\R^d)$, $\theta(0)=1$ and
$\widehat \theta\in {E}_{\infty}'(\R^d)$. Then
$$
\lim_{n\to \infty}\sigma_n^\theta f(x) = f(x)
$$
for all Lebesgue points of $f\in L(\log L)^{d-1}(\T^d)$.
\end{thm}

\begin{cor}\label{c17.7}
If the conditions of Theorem \ref{t17.432} or Theorem \ref{t17.420}
are satisfied and if $f\in L(\log L)^{d-1}(\T^d)$ is continuous at a
point $x$, then
$$
\lim_{n\to \infty}\sigma_n^\theta f(x) = f(x).
$$
\end{cor}

If $\theta_j\in M_1^{v_1}(\R)$ for each $j=1,\ldots,d$ (for example
$\theta_j\in V_1^k(\R)$ $(k>2)$) and
$\theta=\prod_{j=1}^{d}\theta_j$, then $\widehat \theta_j\in
{E}_\infty'(\R)$ and so $\widehat \theta\in {E}_\infty'(\R^d)$.
Corollary \ref{c3} shows that also $\theta_j\in V_1^2(\R)$ implies
$\widehat \theta_j\in {E}_\infty'(\R)$. This verifies the next
corollary.

\begin{cor}\label{c17.433}
If $\theta=\prod_{j=1}^{d}\theta_j$ and each $\theta_j \in
V_1^2(\R)$ or $\theta_j\in M_1^{v_1}(\R)$ $(j=1,\ldots,d)$, then the
theorems of this subsection hold.
\end{cor}

The converse to Theorems \ref{t17.410} and \ref{t17.432} do not hold
in this case. However, converse type results can be found in
Feichtinger and Weisz \cite{feiwe2}.

\subsection{Ces{\`a}ro summability}

We define the \dword{rectangular Ces{\`a}ro (or
$(C,\alpha)$)-means}\index{\file}{rectangular Ces{\`a}ro
means}\index{\file}{rectangular $(C,\alpha)$-means} of a function
$f\in L_1(\T^d)$ by
\begin{eqnarray*}
\sigma_n^{(c,\alpha)} f(x) &:=&
\frac{1}{\prod_{i=1}^{d}A_{n_i-1}^{\alpha}} \sum_{|k_1|\leq n_1}
\cdots \sum_{|k_d|\leq n_d} \prod_{i=1}^{d}A_{n_i-1-|k_i|}^{\alpha}
\widehat f(k) \ee^{\ii k\cdot x} \\
&=& \frac{1}{(2\pi)^d}\int_{\T^2} f(x-u) K_n^{(c,\alpha)}(u) \dd
u,\index{\file-1}{$\sigma_n^{(c,\alpha)}f$}
\end{eqnarray*}
where
\begin{eqnarray*}
K_n^{(c,\alpha)}(u) &:=& \frac{1}{\prod_{i=1}^{d}A_{n_i-1}^{\alpha}}
\sum_{|k_1|\leq n_1} \cdots \sum_{|k_d|\leq n_d} \prod_{i=1}^{d}A_{n_i-1-|k_i|}^{\alpha} \ee^{\ii k\cdot x}\\
&=& \frac{1}{\prod_{i=1}^{d}A_{n_i-1}^{\alpha}} \sum_{j_1=0}^{n_1-1}
\cdots \sum_{j_d=0}^{n_d-1} \prod_{i=1}^{d} A_{n_i-1-j_i}^{\alpha-1}
D_j(u).\index{\file-1}{$K_n^{(c,\alpha)}$}
\end{eqnarray*}
Thus
$$
\sigma_n^{(c,\alpha)} f(x) =
\frac{1}{\prod_{i=1}^{d}A_{n_i-1}^{\alpha}} \sum_{k_1=0}^{n_1-1}
\cdots \sum_{k_d=0}^{n_d-1} \prod_{i=1}^{d} A_{n_i-1-k_i}^{\alpha-1}
s_{k} f(x).
$$

The results of Subsection \ref{s17.2} for \ieword{Ces{\`a}ro
summation} can be found in Weisz \cite{wcesf5,wca4,wk2}.

\sect{$\theta$-summability of Fourier transforms}\label{s18}

Analogously to Sections \ref{s11} and \ref{s17}, we introduce now
the \idword{Dirichlet integral} by
$$
s_{t} f(x) := \int_{-t_1}^{t_1} \cdots \int_{-t_d}^{t_d} \widehat
f(v) \ee^{\ii x \cdot v} \dd v \qquad (t=(t_1,\ldots,t_d)\in
\R_+^d).\index{\file-1}{$s_tf$}
$$

For $T>0$, the \idword{rectangular $\theta$-means} of a function
$f\in L_p(\R^d)$ $(1\leq p \leq 2)$ are defined by
$$
\sigma_T^{\theta}f(x):= \int_{\R^d}
\theta\Big(\frac{-v_1}{T_1},\ldots,\frac{-v_d}{T_d}\Big) \widehat
f(v)  \ee^{\ii x \cdot v} \dd v.
$$
It is easy to see that
$$
\sigma_T^{\theta}f(x)=\frac{1}{(2\pi)^d}\int_{\R^d} f(x-u)
K_{T}^{\theta}(u) \dd u\index{\file-1}{$\sigma_T^\theta f$}
$$
where
$$
K_{T}^{\theta}(u) = \int_{\R^d}
\theta\Big(\frac{-v_1}{T_1},\ldots,\frac{-v_d}{T_d}\Big) \ee^{\ii
u\cdot v} \dd v = (2\pi)^d \Big(\prod_{i=1}^{d} T_i \Big)
\widehat\theta(T_1u_1,\ldots,T_du_d).\index{\file-1}{$K_T^\theta$}
$$
Note that for the Fej{\'e}r means (i.e., if each
$\theta_i(t)=\max((1-|t|),0)$), we obtain
$$
\sigma_T^{\theta} f(x) =
\frac{1}{\prod_{i=1}^{d}T_i}\int_{0}^{T_1}\cdots \int_{0}^{T_d}
s_{t} f(x) \dd t.
$$
We extend the definition of the \idword{$\theta$-means} to tempered
distributions by
$$
\sigma_T^{\theta} f := f * K_T^{\theta} \qquad (T>0).
$$
Again, $\sigma_T^{\theta} f$ is well defined for all tempered
distributions $f\in H_p(\R^d)$ $(0<p\leq \infty)$, for all functions
$f\in L_p(\R^d)$ $(1 \leq p\leq \infty)$ and for all $f\in B$, where
$B$ is a homogeneous Banach space on $\R^d$.

Now, the \idword{product Hardy space} \inda{$H_p(\R^d)$} is defined
with the help of the one-dimensional \idword{non-periodic Poisson
kernel}
$$
P_t(x):= {c t \over t^2 + |x|^2} \qquad (t>0, x \in
\R).\index{\file-1}{$P_t$}
$$
The Hardy space $H_p(\R^d)$ satisfies the same properties as the
periodic space $H_p(\T^d)$, except (\ref{e16}) (see Weisz
\cite{wk2}). The \idword{conjugate distributions} are defined by
$$
(\tilde f^{(j_1,\ldots,j_d)})^\land(t) := \Big(\prod_{i=1}^d (-\ii\
{\rm sign} \ t_{i})^{j_i} \Big) \widehat f(t) \qquad (j_i=0,1,t\in
\R^d).\index{\file-1}{$\tilde f^{(j_1,\ldots,j_d)}$}
$$
If $f$ is integrable, then
\begin{eqnarray*}
\tilde f^{(j_1,\ldots,j_d)}(x) &=& {\rm p.v.}
\ \int_\T \cdots \int_\T {f(x_1-t_1^{j_1},\ldots,x_d-t_d^{j_d}) \over \prod_{i=1}^{d} (\pi t_i)^{j_i}} \dd t^{j_1}\cdots \dd t^{j_d}\\
&=& \lim_{\epsilon\to 0} \int_{\epsilon_1<|t_1|<\pi} \cdots
\int_{\epsilon_d<|t_d|<\pi} {f(x_1-t_1^{j_1},\ldots,x_d-t_d^{j_d})
\over \prod_{i=1}^{d} (\pi t_i)^{j_i}} \dd t^{j_1}\cdots \dd t^{j_d}
\qquad \mbox{a.e.}
\end{eqnarray*}

The same results are true for the \dword{maximal
operators}\index{\file}{maximal operator}
$$
\sigma_\Box^\theta f := \sup_{T \in \R_\tau^d} |\sigma_{T}^\theta
f|, \qquad \sigma_\gamma^{\theta}f := \sup_{T\in \R_{\tau,\gamma}^d}
|\sigma_{T}^\theta f|, \qquad \sigma_*^\theta f := \sup_{T \in \R^d}
|\sigma_{T}^\theta f|\index{\file-1}{$\sigma_\Box^\theta
f$}\index{\file-1}{$\sigma_\gamma^\theta
f$}\index{\file-1}{$\sigma_*^\theta f$}
$$
as in Sections \ref{s13}--\ref{s17}. We point out some of them.

\begin{thm}\label{t18.210}
If $\theta\in L_1(\R^d)$ and $\widehat \theta\in
{E}_{\infty}(\R^d)$, then
$$
\sigma_\Box^\theta f \leq C \|\widehat \theta\|_{{E}_\infty} M_\Box
f \qquad \mbox{a.e.}
$$
for all  $f\in L_1(\R^d)$.
\end{thm}

\begin{thm}\label{t18.216} If $\theta\in L_1(\R^d)$, $\widehat \theta\in L_1(\R^d)$ and
$$
\sigma_\Box^\theta f(x) \leq C M_\Box f(x)
$$
for all  $f\in L_1(\R^d)$ and $x\in \R^d$, then $\widehat \theta\in
{E}_{\infty}(\R^d)$.
\end{thm}

\begin{thm}\label{t18.32}
Suppose that $\theta\in L_1(\R^d)$, $\theta(0)=1$ and $\widehat
\theta\in {E}_{\infty}(\R^d)$. Then
$$
\lim_{T\to \infty,\, T\in \R_\tau^d}\sigma_T^\theta f(x) = f(x)
$$
for all Lebesgue points of $f\in L_1(\R^d)$.
\end{thm}

\begin{thm}\label{t18.217}
Suppose that $\theta\in L_1(\R^d)$, $\theta(0)=1$ and $\widehat
\theta\in L_1(\R^d)$. If
$$
\lim_{T\to \infty,\, T\in \R_\tau^d}\sigma_T^\theta f(x) = f(x)
$$
for all Lebesgue points of $f\in L_1(\R^d)$, then $\widehat
\theta\in {E}_{\infty}(\R^d)$.
\end{thm}

\begin{cor}\label{c18.24}
Suppose that $\theta\in L_1(\R^d)$, $\theta(0)=1$ and $\widehat
\theta\in L_1(\R^d)$. Then
$$
\lim_{T\to \infty,T\in \R_\tau^d}\sigma_T^\theta f(x) = f(x)
$$
for all Lebesgue points of $f\in L_1(\R^d)$ if and only if $\widehat
\theta\in {E}_{\infty}(\R^d)$.
\end{cor}

With the help of (\ref{e17.21}), these results can be proved as were
the corresponding results in Subsection \ref{s17.3}.

\def\sectiontitle{}

\printindex{\file}{Index}

\printindex{\file-1}{Index of Notations}

{

\bigskip
\hskip1.4 em\vbox{\noindent Ferenc Weisz \\ Department of Numerical Analysis \\ E\"otv\"os L. University \\
H-1117 Budapest, P\'azm\'any P. s\'et\'any 1/C., Hungary \\
{\tt weisz@numanal.inf.elte.hu }\\
{\tt http://numanal.inf.elte.hu/\~{}weisz }}

}

\endddoc
\def\updated{15jun11}